\documentclass[a4paper,11pt,leqno]{article}
\usepackage{enumerate}
\usepackage[shortlabels]{enumitem}
\usepackage[utf8]{inputenc}
\usepackage[T1]{fontenc}
\usepackage{amsthm}
\usepackage[font={footnotesize}]{caption}
\usepackage{amssymb}
\usepackage{mathrsfs}
\usepackage{amsmath}
\usepackage{amsfonts}
\usepackage{mathtools}
\usepackage{graphicx}
\usepackage{float}
\usepackage{dsfont}
\usepackage[a4paper]{geometry}
\usepackage{url}
\usepackage{csquotes}
\usepackage{IEEEtrantools}
\usepackage{appendix}

\usepackage{hyperref}
\hypersetup{linktocpage,
  colorlinks   = true, 
  urlcolor     = blue,  
  linkcolor    = blue,
  citecolor   = blue 
}

\usepackage{tikz-cd}
\usepackage{multicol}
\usepackage{parcolumns}
\usepackage{stackrel}
\usepackage{comment}
\usepackage{constants}
\usepackage{centernot}

\usepackage{mathabx} %

\newtheorem{The}{Theorem}[section]

\newtheorem{Lemme}[The]{Lemma}
\newtheorem{Prop}[The]{Proposition}
\newtheorem{Cor}[The]{Corollary}
\theoremstyle{definition}

\theoremstyle{remark}
\newtheorem{Rk}[The]{Remark}

\usepackage{titlesec}
\titleformat{\subsection}[runin]{\normalfont\bfseries}{\thesubsection.}{.5em}{}[.]\titlespacing{\subsection}{0pt}{2ex plus .1ex minus .2ex}{.8em}
\titleformat{\subsubsection}[runin]{\normalfont\bfseries}{\thesubsubsection.}{.5em}{}[.]
\titlespacing{\subsubsection}{0pt}{2ex plus .1ex minus .2ex}{.8em}

\title{\normalsize 
\textbf{
FRACTAL DIMENSION %
OF %
CRITICAL GAUSSIAN FREE FIELD CLUSTERS}}
\author{}
\date{}
\makeatletter
\newcommand{\E}{\mathbb{E}}

\newcommand{\Q}{\tilde{\mathbb{P}}}
\newcommand{\R}{\mathbb{R}}
\newcommand{\Z}{\mathbb{Z}}
\newcommand{\N}{\mathbb{N}}

\newcommand{\Mac}{\mathsf{Mac}}
\newcommand{\Mact}{\widetilde{\mathsf{Mac}}{}} 
\newcommand{\Dia}{\mathrm{D}}     %
\newcommand{\Capa}{\mathrm{C}}    %
\newcommand{\Vol}{\mathrm{V}}     %
\newcommand{\Xt}{\mathrm{X}}      %
\newcommand{\Yt}{\mathrm{Y}}
\newcommand{\Zt}{\mathrm{Z}}

\newcommand{\G}{\mathcal{G}}
\newcommand{\K}{\mathcal{K}}
\newcommand{\A}{\mathcal{A}}

\renewcommand{\P}{\mathbb{P}}
\newcommand{\eps}{\varepsilon}

\newcommand{\I}{{\cal I}}

\renewcommand{\phi}{\varphi}
\renewcommand{\tilde}{\widetilde}
\renewcommand{\hat}{\widehat}
\renewcommand{\epsilon}{\varepsilon}

\@addtoreset{equation}{section}

\newcommand{\tend}[2]{\displaystyle\mathop{\longrightarrow}_{#1\rightarrow#2}}
\newconstantfamily{c}{symbol=c}

\newcommand{\PK}{\P^K}
\newcommand{\PPK}{P^K}

\newcommand{\OP}{\overline{\P}}

\newcommand{\PKK}[1]{\P^{#1}}
\newcommand{\PPKK}[1]{P^{#1}}
\newcommand{\EKK}[1]{\E^{#1}}

\newcommand{\LV}{M}

\newcommand{\rr}{\bar{r}}
\newcommand{\RR}{\bar{R}}
\newcommand{\dd}{\overline{\delta}}

\makeatother

\usepackage{color}
\definecolor{Red}{rgb}{1,0,0} %
\definecolor{Blue}{rgb}{0,0,1}
\definecolor{Olive}{rgb}{0.41,0.55,0.13}
\definecolor{Yarok}{rgb}{0,0.5,0}
\definecolor{Green}{rgb}{0,1,0}
\definecolor{MGreen}{rgb}{0,0.8,0}
\definecolor{DGreen}{rgb}{0,0.55,0}
\definecolor{Yellow}{rgb}{1,1,0}
\definecolor{Cyan}{rgb}{0,1,1}
\definecolor{Magenta}{rgb}{1,0,1}
\definecolor{Orange}{rgb}{1,.5,0}
\definecolor{Violet}{rgb}{.5,0,.5}
\definecolor{Purple}{rgb}{.75,0,.25}
\definecolor{Brown}{rgb}{.75,.5,.25}
\definecolor{Grey}{rgb}{.7,.7,.7}
\definecolor{Black}{rgb}{0,0,0}

\definecolor{MergeOrange}{rgb}{0.85,0.4,0}

\usepackage{titletoc}

\newconstantfamily{index_prob_cont_time}{
	symbol=\mathcal{T},
}

\newconstantfamily{index_prob}{
	symbol=\mathcal{N},
}
\newconstantfamily{constant}{
	symbol=\mathcal{C},
}
\newconstantfamily{small}{
	symbol=c,
}

\begin{document}
\maketitle
\thispagestyle{empty}
\vspace{0.01cm}
\begin{center}
\vspace{-1.9cm}

Alexander Drewitz\footnote{Universit\"at zu K\"oln, Department Mathematik/Informatik, Weyertal 86--90, 50931 Köln, Germany, \url{adrewitz@uni-koeln.de}}, Alexis Pr\'evost\footnote{University of Bonn, Institute for Applied Mathematics,
Endenicher Allee 60,
53115 Bonn,
Germany, \url{prevost@iam.uni-bonn.de}} and Pierre-Fran\c cois Rodriguez\footnote{University of Cambridge, Department of Pure Mathematics and Mathematical Statistics, Wilberforce Road, Cambridge CB3 0WB, United Kingdom, \url{pfr26@cam.ac.uk}} 
\end{center}
\begin{abstract}
We study metric graph sign clusters of the Gaussian free field on a large class of transient weighted graphs $\G$ that are sufficiently \emph{low}-dimensional, i.e.~below the mean-field regime, and prove a \emph{macroscopic equivalence property}: several natural ways to define macroscopic clusters -- involving any of diameter, volume, or capacity functionals -- are in fact the same. As a corollary, macroscopic clusters are \emph{characterized} in low dimensions as those containing large loops in the corresponding loop soup picture, in contrast with the typical behavior in high dimensions. We then apply these results to the small-mesh scaling limit of critical clusters on $\varepsilon\Z^d$, $d=3,4,5$, as $\varepsilon \downarrow 0$. We  prove that every subsequential scaling limit consists of sets with positive Brownian capacity whose Hausdorff and Minkowski dimensions are both equal to $1+\tfrac d2$, as conjectured by Werner. We also derive two-point crossing estimates in the limit, and show that each limiting loop soup cluster is the closure of the union of the scaling limits of the macroscopic loops it contains. Remarkably, unlike with planar Bernoulli percolation for instance, the determination of the Hausdorff dimension does not emerge from direct calculations involving a limiting object (such as $\text{SLE}_6$), but rather from the distinctive features of the discrete model itself.
\end{abstract}

\vspace{-0.8cm}

{\setcounter{tocdepth}{2} %
\tableofcontents}

\section{Main results}\label{sec:mainresults}

This article concerns Brownian loop soup clusters at intensity $1/2$, or, equivalently, sign clusters of the Gaussian free field on metric graphs in low transient dimensions, in a sense to be made precise below. We refer to \cite{DrePreRod8,DrePreRod9,cai2024onearm,cai2025heterochromatic,cai2025separation,cai2025gapdimensions,jego2023crossing,rodriguez2025cable,werner2025switching,dremaux2025singularity, jego2026threedim}, and references below, for a sample of very recent related results.
The main result of this article is Theorem~\ref{the:equivmacro} below, which exhibits several characterizations of macroscopic clusters in a box, in terms of either their diameter, (renormalized) volume, or capacity. Theorem~\ref{the:equivmacro} has numerous consequences. It implies for instance that macroscopic clusters contain a macroscopic loop (in the corresponding loop soup picture), see Corollary~\ref{cor:macrocontainsbigloops}. This is in stark contrast with the high-dimensional case, where macroscopic clusters are known to proliferate, but the number of macroscopic loops remains of unit order. In Theorem~\ref{the:Hausdorff} (see Section~\ref{subsec:scalinglimitintro}), we use Theorem~\ref{the:equivmacro} on $\Z^d$, $d\in\{3,4,5\}$, to prove tightness of the families of macroscopic clusters in the scaling limit and to determine their fractal (Hausdorff) dimension, and derive various other properties of these limiting clusters. 

\medskip

 We start by specifying the setup for this article. Our results hold on a general class of graphs satisfying suitable volume growth and Green function decay assumptions, that we introduce shortly. These graphs need to be sufficiently low-dimensional, as captured by the condition $\sup q<\infty$, see \eqref{eq:qPsiBd}, which concerns the decay of the one-arm probability; in particular, this condition is satisfied throughout the entire regime below the upper critical dimension $d=6$ on $\Z^d$. Without a low-dimension condition, the results we prove do not hold; see Remark~\ref{R:main-thm}. 
 
 Our choice of setup is not just a matter of generality. Although the example of $\Z^d$, for $d\in \{3,4,5\}$, to which our results apply, is of primary interest, let us stress that the consequences alluded to above -- even in this specific example -- rely on versions of Theorem~\ref{the:equivmacro} that immediately depart from the framework of $\Z^d$, since they are typically applied  on ``perforated'' versions of the (metric) graph $\Z^d$, due to exploration arguments/use of so-called obstacles, etc, which are potentially ``rough'' sets that effectively act as a Dirichlet boundary condition. This warrants a robust approach. Notwithstanding, the reader may choose to focus on the case of $G=\Z^d$, $d\in \{3,4,5\}$, with uniform weights between nearest neighbors, in what follows.

We consider a countable weighted graph $\G = (G, \lambda)$. That is, $G$ is a countable (vertex) set and $\lambda_{x,y}=\lambda_{y,x} \geq 0$ are non-negative symmetric weights. Two points $x,y \in G$ are neighbors, denoted by $x\sim y$, if and only if $\lambda_{x,y}>0$; in this case, we refer to $\{x,y\}$ as an edge. 
We define $\lambda_x = \sum_{y \sim x} \lambda_{x,y}$ for $x \in G$ and think of $\lambda_{\cdot}$ as a positive measure on $G$. Thus we write $\lambda(A)=\sum_{x\in A}\lambda_x$ for $A\subset G$.

We will \emph{always} work under three assumptions on $\G$, traditionally \cite{MR1853353,MR3616731} called  \eqref{eq:ellipticity}, \eqref{eq:intro_sizeball} and \eqref{eq:intro_Green}. The first of these, also referred to as ``controlled weights,'' is a standard ellipticity assumption on the weight function $\lambda$. It requires that for some constant $c_0 \in (0,\infty)$,
\begin{align}
&\label{eq:ellipticity} \tag{$p_0$}
\ \, \lambda_{x,y}/\lambda_x\geq c_0 \text{ for all $x,y \in G$ such that $\lambda_{x,y}>0$}. 
\end{align}
 The remaining two assumptions on $\G$ depend on the choice of a metric $d(\cdot,\cdot)$ defined on $G$. For simplicity, we will assume that $d$ is the graph distance on $\G$, i.e.~$d(x,y)=1$ if and only if $\lambda_{x,y}>0$ (extended to a geodesic distance on $G$). We write $B(x,r)=\{ y \in G : d(x,y) \leq r\}$, $x \in G$, $r>0$, for the discrete balls in the distance $d$. We further arbitrarily fix some vertex $0$ as the origin of the graph, and abbreviate $B(r)=B(0,r)$. Note that we use $0$ only to simplify the exposition, and our results will not depend on the choice of this point. With a slight abuse of notation, in the special case of $\Z^d$ (see e.g.~Sections~\ref{subsec:scalinglimitintro} and~\ref{sec:proofhausdorffdim}), we depart from this convention and use $d$ to denote the $\ell^{\infty}$-distance, which is a matter of convenience.

 There are interesting examples which may require different choices of $d$, see \cite[Section 3]{DrePreRod2} for an extensive discussion. In fact all our results hold under weaker assumptions than $d$ being the graph distance, see Remarks~\ref{rk:ballsconnected} and~\ref{rk:generaldistance},\ref{rk:generaldistance1}. In particular, as detailed in  Remark~\ref{rk:generaldistance},\ref{rk:generaldistance2}, they hold for all examples  from \cite[(1.4)]{DrePreRod2}, such as the product of the Sierpinski gasket with $\Z$, or finitely generated Cayley graphs with polynomial growth.

Equipped with the (graph) distance $d$, we require the weighted graph $\G$ to be $\alpha$-Ahlfors regular. That is,~there exist $\alpha>0$ and $\Cl[c]{c:alpha-lb}, \Cl{c:alpha-ub} \in (0 ,\infty)$ such that the volume growth condition
\begin{equation}
\label{eq:intro_sizeball} \tag{$V_{\alpha}$}
\Cr{c:alpha-lb}r^{\alpha}\leq \lambda(B(x,r))\leq \Cr{c:alpha-ub}r^{\alpha}\quad \text{ for all }x\in{G}\text{ and }r\geq1,
\end{equation}
is satisfied. In particular, \eqref{eq:intro_sizeball} entails that 
the sets $B(x,r)$ have finite cardinality for all $x \in G$ and $r>0$, and that the degree of all vertices in $G$ is bounded uniformly; see \cite[Lemma 2.3]{MR3616731}.

The third condition involves the \emph{random walk} on $\G$, which is the continuous-time Markov chain $X$ on $G$ with generator
$Lf(x) = \sum_{y \in G} \lambda_{x,y} (f(y) - f(x)),$
for $f: G \to \R$. We write $P_x$ for its canonical law when started in $x \in G.$ We require that there exist constants $\Cl[c]{c:green-lb}, \Cl{c:green-ub} \in (0,\infty)$ and an exponent $\nu \in (0,\infty)$ such that the Green function $g$ on $G$ defined as 
\begin{align} \label{eq:green}
g(x,y) = \int_{0}^\infty P_x(X_t =y)\, {\rm  d}t, \quad x,y \in G,
\end{align}
 satisfies
\begin{equation}
\label{eq:intro_Green}\tag{$G_{\nu}$}
\begin{split}
&  \Cr{c:green-lb} (d(x,y) \vee1)^{-\nu}\leq g(x,y)\leq  \Cr{c:green-ub} (d(x,y) \vee1)^{-\nu}, \quad \text{ for all }x, y\in{G}.
\end{split}
\end{equation}
In particular \eqref{eq:intro_Green} implies that $X$ is transient. Together, the conditions \eqref{eq:ellipticity}, \eqref{eq:intro_sizeball} and \eqref{eq:intro_Green} imply that $0<\nu\leq \alpha-2$, see~\cite{MR2076770}, and, as shown in \cite[(2.10)]{DrePreRod2}, they also imply that
\begin{equation}
\label{eq:lambdabounded}
\begin{split} 
c\leq \lambda_{x,y}\leq \lambda_x\leq C\text{ for all }x\sim y\in{G}.
\end{split}
\end{equation}
 In particular, combining \eqref{eq:intro_sizeball} and \eqref{eq:lambdabounded}, the cardinality of a ball of size $r$ is of order $r^{\alpha}$. 
 
 The conditions \eqref{eq:ellipticity}, \eqref{eq:intro_sizeball} and \eqref{eq:intro_Green} are for instance  clearly true for $G=\Z^d$, $d \geq 3$, with $d=\alpha =\nu+2$ and equal weights between $x$ and $y$ at unit graph distance such that $\lambda_{\cdot,\cdot}=1/2$ (say). When referring to $\Z^d$ in the sequel we will for definiteness implicitly assume these choices of normalization, which are inconsequential.

The Gaussian free field on ${\G}$ is the centered Gaussian field $(\varphi_x)_{x\in {G}}$, with covariance $g(\cdot,\cdot)$. We denote its canonical law by \(\P\) and the corresponding expectation by \(\E\), under which we also define an independent family $U_e$, $e$ edge of $\G$, of i.i.d.\ uniform random variables in $[0,1]$. Let $a \in \mathbb{R}$ be a (height) parameter. 
Each edge $e=\{x,y\}$ is declared open (at height $a$) if and only if 
\begin{equation}
\label{eq:edge-perco}
U_e \leq 1-\exp\big(-
 2\lambda_{x,y}(\varphi_x-a)_+(\varphi_y-a)_+ \big), 
\end{equation}
where $t_+ = t\vee0$ denotes the positive part of $t \in \R$. Hence, an edge being open requires that the field at both its endpoints exceeds value $a$, and in addition, the uniform variable attached to this edge cannot be too large, in a  manner depending on the field at the endpoints. 

Let $E^{\geq a}$ denote the random subset of open edges of $\G$ at height $a$ thus obtained. 
It is clear from \eqref{eq:edge-perco} that  $E^{\geq a}$ is decreasing in $a$. The set $E^{\geq a}$ 
induces (maximal) connected components of edges, called \textit{clusters} (of $E^{\geq a}$): two vertices belong to the same cluster if there is a path of edges in $E^{\geq a}$ linking them (two successive edges in a path must share a common vertex). One could also equivalently (for the purposes we have in mind) define $E^{\geq a}$ as the level sets associated to the metric graph Gaussian free field, and we refer to the discussion around \eqref{eq:newnotation} for details. All our results hold for either choice.

\bigskip

The assumptions  \eqref{eq:ellipticity}, \eqref{eq:intro_sizeball} and \eqref{eq:intro_Green} will always be in force in the sequel. As mentioned, an additional condition is needed to ensure the ``low-dimensionality'' of $\G$. This is conveniently encapsulated in terms of the function $q: [1,\infty)\to  (0,\infty]$, defined as
\begin{equation}
    \label{eq:qPsiBd}
q(R)\stackrel{\text{def.}}{=}\sup_{x \in G} R^{\nu/2} \P \big( x \leftrightarrow \partial {B}(x,R) \big);
\end{equation}
here, with hopefully obvious notation, $\partial B$ denotes the interior vertex boundary of $B\subset {G}$, and we denote by $\{U {\leftrightarrow} V\}$  for $U,V\subset G$, the event that $U$ and $V$ are connected by a path in $E^{\geq 0}$.

 It is known that $q$ is bounded away from zero for all $\nu > 0$, cf.\ \cite[(1.22)–(1.23)]{DrePreRod5}, and our condition on ``low dimension'' is that $\sup q < \infty$, where here and in the sequel, we abbreviate $\sup q \equiv \sup_{R \geq 1} q(R)$ with $q(\cdot)$ as in \eqref{eq:qPsiBd}.
 As a result of the series of works \cite{DrePreRod8,DrePreRod9,cai2024onearm} one now knows that $\sup q< \infty$ when $\alpha > 2\nu$ and on $\Z^d$, $d=3,4,5$. On the other hand $q(R)\asymp R^{(d-6)/2}$ on $\Z^d$ for $d\geq 7$ by \cite{werner2020clusters,cai2023onearm}, yielding a decay for the connection probability appearing in \eqref{eq:qPsiBd} (with $\nu=d-2$) independent of $d$, indicative of mean-field behaviour. The assumption $\sup q <
\infty$ is crucial for our main results to hold, but various intermediate results do not require it; hence it will be mentioned explicitly whenever it forms part of the assumptions.
 
\subsection{Equivalence of macroscopic clusters}
\label{subsec:equivalenceintro}
Our main results assert that for graphs satisfying \eqref{eq:ellipticity}, \eqref{eq:intro_sizeball} and \eqref{eq:intro_Green}, and for which $\sup q <
\infty$, various ways of measuring macroscopic sign clusters of $\varphi$ are equivalent, and that the number of these clusters is tight. We will mainly use three different notions of size, involving diameter ($\Dia$), capacity ($\Capa$) and (renormalized) volume ($\Vol$), respectively; cf.~\eqref{eq:defMRs} below. A fourth one will involve clusters containing large loops; cf.~\eqref{eq:defMRSL}. We now introduce these notions of size.  The first one is straightforwardly expressed in terms of the diameter of $K\subset G$, defined as
 \begin{equation}\label{eq:diam}
 \delta(K)\stackrel{\text{def.}}{=}\sup_{x,y\in{K}}d(x,y).
 \end{equation}
 with the convention $\delta(\emptyset)=0$. The second notion of size involves the capacity $\text{cap}(\cdot)$,  which we briefly introduce. For a finite set $K\subset G$, we write $e_K(x)$ for the equilibrium measure of $K$, that is $\lambda_x$ times the probability that the random walk on $\G$ starting in $x$ never returns to $K$ if $x\in{K}$, and zero otherwise. The capacity $\mathrm{cap}(K)$ of $K$ is the total equilibrium measure:
\begin{equation}
\label{eq:capintro}
\mathrm{cap}(K) \stackrel{\text{def.}}{=}\sum_{x\in{K}}e_K(x).
\end{equation}
The third relevant notion of size is the (renormalized) volume.  For $r \geq 0$
and an ``approximate lattice'' $\Lambda_r (\subset G)$ at scale $r$ (to be introduced in \eqref{eq:defLambdaL}), and $K\subset G$, we let
\begin{equation}\label{eq:vol}
|K|_r\stackrel{\text{def.}}{=}\big|\big\{y\in{\Lambda_{r}}:B(y,r)\cap K\neq\emptyset\big\}\big|,
\end{equation}
where $|\cdot|$ denotes the cardinality of a subset of $G$.  It remains to specify $\Lambda_r$. As we now briefly explain, one can find for each $r\geq1$ a set $\Lambda_{r}\subset G$ such that
\begin{equation}
\label{eq:defLambdaL}
\begin{cases} 
\bigcup_{y\in{\Lambda_{r}}}B(y,r)=G,\ B(y,\tfrac{r}{3}) \cap B(y',\tfrac{r}{3})=\emptyset ,\forall  y \neq  y'\in{\Lambda_{r}},\\ \text{and } cN^{\alpha } \leq |\Lambda_{r}\cap B(x,\bar{r}N)|\leq CN^\alpha, \ \forall x\in{G},\,  r, N \geq 1;
\end{cases}
\end{equation}
indeed, as observed in \cite[(2.2)]{DrePreRod8}, a set $\Lambda_r$ satisfying the properties in the first line of \eqref{eq:defLambdaL} and the upper bound on $|\Lambda_{r}\cap B(x,rN)|$ exists under \eqref{eq:ellipticity}, \eqref{eq:intro_sizeball} and \eqref{eq:intro_Green}. Furthermore  $|B(x,r(N-1))|\leq |\bigcup_{y\in{\Lambda_{r}\cap B(x,rN)}}B(y,r)|\leq |\Lambda_{r}\cap B(x,rN)|\sup_{y\in{G}}|B(y,r)|$, and so \eqref{eq:intro_sizeball} and \eqref{eq:lambdabounded} imply the remaining lower bound. For convenience, we  define $\Lambda_{r}=G$ for $r\in{[0,1)}$ 
so that \eqref{eq:defLambdaL} holds for any $\Lambda_{r}$, $r \geq 0$. To improve legibility, we have introduced in \eqref{eq:defLambdaL} the notation 
\begin{equation}
\label{eq:r1}
\begin{split} 
\rr \stackrel{\text{def.}}{=}r\vee1  
\end{split}
\end{equation}
that we will use frequently throughout the article 
(and similarly $\RR = R\vee1$, $\rr' = r' \vee 1$, $\dd(K)=\delta(K)\vee1$ etc.). We now fix once and for all a choice of $(\Lambda_r)_{r \geq 0}$, and in the case of $\Z^d$ we simply let $\Lambda_{r}=  \lceil \rr \rceil \Z^d$, 
which clearly satisfies \eqref{eq:defLambdaL} (recall that we fix $d(\cdot,\cdot)$ to be the sup-distance in this case).

Let us denote by $\mathcal{C}_{x,R}$ the set of connected components of non-isolated vertices linked by  open edges in $E^{\geq 0}$ whose endpoints are both within  
${B}(x,R)$. We introduce three different subsets of $\mathcal{C}_{x,R}$ that extract its macroscopic clusters, in terms of diameter, capacity, or volume on the renormalized lattice $\Lambda_{r}$ (cf.~\eqref{eq:diam}, \eqref{eq:capintro} or \eqref{eq:vol}, resp.), as follows. For each $R>r\geq0$, $x\in{G}$ and $s\in{(0,1]}$, let 
\begin{equation}
\label{eq:defMRs}
\begin{split} 
&\Mac_{x,R}^s(\Dia)=\{K\in{\mathcal{C}_{x,R}}:\, \delta(K)\geq sR\},
\\&\Mac_{x,R}^s(\Capa)=\{K\in{\mathcal{C}_{x,R}}:\, \mathrm{cap}(K)\geq (sR)^{\nu}\},
\\&\Mac_{x,R}^s(\Vol_r)=\{K\in{\mathcal{C}_{x,R}}:\, (sR/\rr)^{\alpha-\frac\nu2}\leq |K|_r\leq (R/(s\rr))^{\alpha-\frac\nu2}\}.
\end{split}
\end{equation}
Observe on the one hand that the scaling of $\delta(K)$ and $\mathrm{cap}(K)$ in \eqref{eq:defMRs} is of maximal order when $s$ is a constant (in the latter case, by virtue of \eqref{eq:capBallBd} below), that is for $s\leq c$ one can add therein w.l.o.g.\ upper bounds by $R/s$ and $(R/s)^{\nu}$ respectively, and the results below all remain true. On the other hand, there is no deterministic bound to prevent the existence of clusters of size larger than $(R/(s\rr))^{\alpha-\frac\nu2}$ when $s$ is a small constant. Such clusters actually only exist with vanishing probability as $s\rightarrow0$, and the largest cluster in ${B}(R)$ has renormalized volume of order $(R/\rr)^{\alpha-\frac\nu2}$, see \eqref{eq:MR0tight} in Proposition~\ref{thm:macroscopiclocalcluster} below, which sheds light on the more intricate scaling in the definition of $\Mac_{x,R}^s(\Vol_r)$. Note that the results from this section would remain true if one only required the lower bound condition on $|K|_r$ in the above definition of $\Mac_{x,R}^s(\Vol_r)$, see Remark~\ref{rk:otherchoicesforE},\ref{rk:Vr-}.

The following result states that any cluster with either large diameter, volume or capacity,  actually also has large diameter, volume and capacity. It further yields a bound on the numbers of macroscopic clusters of any of the three types, which entails their tightness.

The positive constants $c,C, C_i$ etc.~appearing below may implicitly depend on the choice of $\alpha$, $\nu$ and the various constants entering conditions \eqref{eq:ellipticity}, \eqref{eq:intro_sizeball}, \eqref{eq:intro_Green}, as well as the value of $\sup q$; see the end of this introduction regarding our convention with constants.

\begin{The}
\label{the:equivmacro}
 For all graphs  $\G$ satisfying  \eqref{eq:ellipticity}, \eqref{eq:intro_sizeball} and \eqref{eq:intro_Green} and for which $\sup q< \infty$, all $t\in{(0,1]}$,  $s\in{(0,ct^{\Cl{C:condst}})}$, $r\geq0$ and all $\Xt,\Yt\in{{\{\Dia,\Capa,\Vol_r\}}}$, $x \in G$  and $R\geq C\rr/s$,
\begin{equation}
\label{eq:equivmacro}
\begin{split} 
\P\big(\Mac^t_{ x,R}(\Xt)\subset \Mac^s_{ x, R}(\Yt)\big)\geq 1-Cs^c.
\end{split}
\end{equation}
Moreover, for all $\Xt\in{{\{\Dia,\Capa,\Vol_0\}}}$, 
\begin{equation}
\label{eq:volumeMrs}
\begin{split} 
 \inf_{x \in G} \inf_{s\in{(0,1]}}\inf_{R\geq C/s}\P\big(\lfloor\eta s^{-\alpha}\rfloor\leq |\Mac^s_{ x,R}(\Xt)|\leq \eta^{-1}s^{-\alpha}\big)\tend{\eta}{0}1.
\end{split}
\end{equation}
\end{The}

We start by discussing an important consequence of the above results. Further comments on Theorem~\ref{the:equivmacro} are given in Remark~\ref{R:main-thm} below. One can use Theorem~\ref{the:equivmacro} to better understand the structure of the macroscopic critical clusters from a loop soup point of view.

We say that $\gamma$ is a rooted loop on $\G$ of length $n$ if $\gamma=(\gamma_k)_{0\leq k\leq n}$ is a path in $G$ with $\gamma_0=\gamma_n$.  Let us denote by $\mu$ the measure on rooted loops on $\G$, which gives weight $\tfrac1n\prod_{i=1}^n(\lambda_{\gamma_{i-1},\gamma_i}/\lambda_{\gamma_{i-1}})$ to any rooted loop $\gamma$ of length $n$. Under a probability measure $\tilde{\P}$, denote by $\mathcal{L}$ the associated loop soup with parameter $1/2$, that is the point process of rooted loops on $G$ with intensity $\tfrac12\mu$.  We denote by $(L_x)_{x\in{G}}$ the associated field of local times, that is $L_x$ is a sum of independent exponential variables with mean $\lambda_x^{-1}$, one per visit of $x$ by any of the loops in $\mathcal{L}$, plus an additional independent $\Gamma(1/2,\lambda_x^{-1})$-distributed variable for each $x$. Let us further assume that $\tilde{\P}$ is suitably extended so as to carry independent uniform random variables $U_e$ on $[0,1]$ for each edge $e$ of $\G$.  An edge $e=\{x,y\}$ for the loop soup is declared open if either there is a loop in $\mathcal{L}$ which crosses $e$, or \eqref{eq:edge-perco} for $a=0$ occurs, with $\phi_x$, $\phi_y$ therein replaced by $\sqrt{L_x}$, $\sqrt{L_y}$. For $A\subset G$, a cluster of $\mathcal{L}$ within $A$ is a maximal connected component of at least $2$ points in $A$, all connected by a finite number of open edges with endpoints in $A$. This corresponds to clusters associated to the metric graph loop soup, see the discussion around \eqref{eq:newnotation} for details. Let 
\begin{equation}
\label{eq:deftildeMRS}
\Mact_{x,R}^s(\Xt), \quad \Xt\in{\{\Dia,\Capa,\Vol_r\}},
\end{equation}
denote the loop clusters in  ${B}(x,R)$ with either large diameter, capacity or renormalized volume,  defined similarly as in \eqref{eq:defMRs} but replacing $\mathcal{C}_{x,R}$ by the set $\tilde{\mathcal{C}}_{x,R}$, consisting of the clusters of $\mathcal{L}$ within ${B}(x,R)$. 

The set $\Mact_{x,R}^s(\Xt)$ has the same law as the union of two (dependent) copies of $\Mac_{x,R}^s(\Xt)$, see \eqref{eq:isoloopK}, and inspection of the proofs shows that Theorem~\ref{the:equivmacro} continues to hold with $\Mact$ in place of $\Mac$ everywhere. It then readily follows from Theorem~\ref{the:equivmacro} that, for some $s=s(\varepsilon,t)>0$ and large enough $R$, with probability at least $1-\varepsilon$,
\begin{equation}\label{eq:loopinclusion-easy}
\text{any loop $\gamma$  in ${B}(x,R)$ with $\delta(\text{range}(\gamma)) \geq tR$ is part of a cluster in $\bigcap_{\Xt} \Mact_{x,R}^s(\Xt)$,}
\end{equation}
where the intersection ranges over $\Xt \in \{\Dia,\Capa,\Vol_r\}$, and $\text{range}(\gamma)$ denotes the set of vertices visited by $\gamma$. With slightly more work, essentially to show that loops in $\mathcal{L}$ with large capacity or volume also typically have large diameter, one can deduce variants of \eqref{eq:loopinclusion-easy}, with the condition on the diameter of $\text{range}(\gamma)$ replaced by the requirement that its capacity be at least $ (tR)^{\nu\wedge(\alpha-\nu)}\log(tR)^{-1\{\alpha=2\nu\}}$, or its volume at least $(tR)^{\alpha-\nu}$. These scalings correspond to the typical capacity/volume of the range of $X$ started in the bulk prior to exiting  ${B}(x,R)$, cf.~\cite[(5.8), (5.8')]{DrePreRod5} regarding capacity, and the proof of \cite[Lemma 3.1]{DrePreRod8} as to how to transfer such estimates to statements concerning the intensity measure of $\mathcal{L}$.

Actually, one can deduce that the converse of \eqref{eq:loopinclusion-easy} is also true: as the next statement asserts, any macroscopic cluster contains a macroscopic loop. Indeed, denoting by
\begin{equation}
\label{eq:defMRSL}
\begin{split} 
\Mact_{x,R}^s(\mathcal{L})&\stackrel{\text{def.}}{=}\big\{K\in{\tilde{\mathcal{C}}_{x,R}}:\, \exists\,\gamma\in{\mathcal{L}}:\text{range}(\gamma)\subset K\text{ and }\delta(\text{range}(\gamma))\geq sR\big\},
\end{split}
\end{equation}
we have the following result.

\begin{Cor}
\label{cor:macrocontainsbigloops}
For $\G$ as in Theorem~\ref{the:equivmacro}, all $\eps>0$ and $t\in{(0,1]}$, there exists $s=s(\eps,t)$ such that for all $x\in G$, $r\geq0$, $R\geq C\rr/s$ and $\Xt\in{{\{\Dia,\Capa,\Vol_r\}}}$, 
\begin{equation}
\label{eq:macrocontainsbigloops}
\begin{split} 
\Q\big(\Mact_{x,R}^t(\Xt)\subset \Mact_{x,R}^s(\mathcal{L})\big)\geq 1-\eps.
\end{split}
\end{equation}
\end{Cor}
\begin{Rk} \label{R:main-thm}
\begin{enumerate}[1)]
\item The assumptions of Theorem~\ref{the:equivmacro} (and Corollary~\ref{cor:macrocontainsbigloops}) are met on $\Z^d$ when $d \in \{3,4,5\}$. The tightness of the number of macroscopic clusters for $d\leq 5$ expressed by Theorem~\ref{the:equivmacro} is in sharp contrast with the situation for $d\geq7$, where there are at least $R^{d-6}/\log(R)^2$ disjoint macroscopic clusters, see \cite[Proposition~4]{werner2020clusters}.
\item The inclusion \eqref{eq:loopinclusion-easy} is complementary to \eqref{eq:macrocontainsbigloops} and together, they imply that macroscopic clusters under the low-dimensionality condition $\sup q<\infty$ can also be \emph{characterized} by the property that they contain a loop of macroscopic diameter. Similarly as with the discussion below \eqref{eq:loopinclusion-easy}, one could also change the definition of $\Mact_{x,R}^s(\mathcal{L})$ in \eqref{eq:defMRSL} by requiring the loops to have large capacity or volume instead of diameter, and so the conclusions of the previous sentence remain true with these substitutions.
\item (High dimensions). Since for any possible values of $\alpha$ and $\nu$ the number of macroscopic loops in ${B}(R)$ is tight -- which follows for instance from \cite[Lemma~3.1]{DrePreRod8} -- the conclusion of Corollary~\ref{cor:macrocontainsbigloops} does not hold true on $\Z^d$ for $d\geq7$ by \cite[Proposition~4]{werner2020clusters}; that is, there are macroscopic clusters which do not intersect any macroscopic loop, and in fact the vast majority of them don't. In fact, more is true. By recent results \cite{lupu2025intensity,werner2025switching} one knows that the clusters containing a macroscopic loop on $\Z^d$, $d \geq 7$, are (asymptotically) indistinguishable (in law) from clusters where the macroscopic cycle of the cluster is achieved by chaining many small (at most mesoscopic) loops instead. This means not only that, unlike  Corollary~\ref{cor:macrocontainsbigloops}, big loops play (essentially) no role in characterizing macroscopic clusters, but even fail to do so when adding the (atypical in high dimension) constraint that the macroscopic cluster must contain a large cycle. See also \cite{carpenter2025loopscriticalhighdimensionalpercolation} for related results in high-dimensional Bernoulli percolation, implying for instance that the number of clusters having a large cycle is tight. We refer to \cite{cai2023onearm,ganguly2024ant} for other recent results on critical clusters in the high-dimensional regime.
\end{enumerate}
\end{Rk}

\subsection{Application: scaling limits on $\Z^d$}\label{subsec:scalinglimitintro}

In this subsection, we consider
\begin{equation}
\label{eq:lowdim}
G=\Z^d,\,d\in{\{3,4,5\}},
\end{equation}
and we recall that we chose the weights $\lambda_{\cdot,\cdot}=1/2$ on $\Z^d$.  We write $d(x,y)=|x-y|_{\infty}$ for all $x,y\in{\R^d}$.  With a slight abuse of notation, we identify each cluster in $E^{\geq a}$ or $\mathcal{L}$ with the subset of $\R^d (\supset \Z^d)$ obtained by replacing each edge $e=\{x,y\}$ in the cluster by the unit length line segment joining $x$ and $y$. In particular, each cluster thus becomes a connected subset of $\R^d$.

Our interest is in describing properties of the rescaled clusters of $E^{\geq 0}$ in the fine-mesh scaling limit, i.e.~on $\tfrac{1}{n}\Z^d$ as $n \to \infty$ (as the \lq ultraviolet\rq\ cut-off is removed, in the parlance of physicists). We start by introducing a space $\mathcal{S}$ of families of subsets of $[-1,1]^d$, which will be the natural configuration space for the rescaled clusters of interest. 
We will always consider families as unordered, that is two families $(K_i)_{i\in{I}}$ and $(K'_j)_{j\in{J}}$ are equal if and only if there is a bijection $p$ between the index sets $I$ and $J$ such that $K_i=K'_{p(i)}$ for all $i\in{I}$.  An unordered family $\mathscr{E}=(K_i)_{i \in I}$ is in $\mathcal{S}$ if the index set $I$ is countable, $K_i$ is a closed subset of $[-1,1]^d$ with positive diameter for all $i\in{I}$, and if the number of indices $i\in{I}$ such that $K_i$ has diameter larger than $\varepsilon$ is finite for each $\varepsilon >0$. 
Note that these properties do not depend on the specific choice of the representative family $(K_i)_{i \in I}$ for ${\mathscr{E}}$.  For technical reasons, the family $\mathcal{S}$ includes the empty family $\emptyset$ when $I=\emptyset$, containing no set, for which the previous requirements trivially hold. 
  
  On $\mathcal{S}$, we introduce a distance $\mathbf{d}$,  defined rigorously in \eqref{eq:defdistance} below, such that the resulting metric space $(\mathcal{S}, \mathbf{d})$ has desirable topological properties, see Lemma~\ref{L:d-metric}. For the purposes of this introduction, it is sufficient to know that for two non-empty families $\mathscr{E}=(K_i)_{i\in{I}}$ and $\mathscr{E}'=(K'_j)_{j\in{J}}$ in $\mathcal{S}$, one has $\mathbf{d}(\mathscr{E},\mathscr{E}')<\eps$ roughly if and only if one can match each set $K_i$ with diameter larger or equal to $\eps$ to a set $K'_j$ at Hausdorff distance strictly smaller than $\eps/2$, and vice-versa. The factor $1/2$ relative to the bound on the diameter is linked to the factor $2$ appearing in \eqref{eq:d_H-diam}; see Section~\ref{sec:dist} for details. Proximity to the empty family $\emptyset$ entails that all clusters have small diameter. For all $n\in{\N}$, we introduce the notation
\begin{equation}
\label{eq:renormalizedlevelset-intro}
\begin{split} 
\mathscr{E}^{\geq 0}_n\stackrel{\text{def.}}{=}\text{ an enumeration of }\big\{n^{-1}K:\,K\text{ is a cluster of }E^{\geq 0}\cap [-n,n]^d\big\},
\end{split}
\end{equation}
where $E^{\geq 0}\cap [-n,n]^d$ corresponds to the edges in $E^{\geq 0}$ having both endpoints in $[-n,n]^d$. 
In view of our above identification, the components (or clusters) of $\mathscr{E}^{\geq 0}_n$ are viewed as closed, connected subsets of $[-1,1]^d \subset \R^d$. In particular, as $\mathscr{E}^{\geq 0}_n$ comprises finitely many clusters, this entails that $\mathscr{E}^{\geq 0}_n$ is an $\mathcal{S}$-valued random variable (recall that we always consider families, and so enumerations, as unordered).

Our main results concern clusters in $\mathscr{E}^{\geq 0}_n$ and their scaling limits as $n \to \infty$. We refer to Remark~\ref{R:ext-signclusters} below for extensions to other clusters of interest. In what follows, for a closed and connected set $K\subset[-1,1]^d$, we denote by $\mathrm{cap}_{\R^d}(K)$ its Brownian capacity, defined as in \cite[p.58]{zbMATH03642454}, by $\text{dim}_H(K)$ its Hausdorff dimension, by $\text{dim}_M(K)$ its Minkowski dimension, and by $\delta(K)$ its diameter (relative to the sup-norm); see Section~\ref{sec:proofhausdorffdim} for details. Furthermore, we write $\{U {\leftrightarrow} V\text{ in }K\}$  for $U,V,K \subset\R^d$ in the sequel to denote the event that $U$ and $V$ are connected by a (continuous) path in $K$. For $\mathscr{E}=(K_i)_{i\in{I}}\in{\mathcal{S}}$, we also write $\{U {\leftrightarrow} V\text{ in }\mathscr{E}\}$  for the union of the events $\{U {\leftrightarrow} V\text{ in }K_i\}$ over $i\in{I}$. Our first result deals with clusters in \eqref{eq:renormalizedlevelset-intro}.

\begin{The}
\label{the:Hausdorff}
Assume $d\in{\{3,4,5\}}$. There exists a strictly increasing sequence $\sigma:\N\rightarrow\N$ and a probability measure $\P^{\infty}_{\sigma}$ on $\mathcal{S}$ such that, with respect to $\mathbf{d}$, 
\begin{equation}
\label{eq:Hausdorffconvergence}
\begin{split} 
\text{ the law of }\mathscr{E}^{\geq0}_{\sigma(n)}\text{ converges weakly to }\P^{\infty}_{\sigma} \text{ as } n \to \infty.
\end{split}
\end{equation}
Moreover, for any $(\sigma,\P^{\infty}_{\sigma})$ such that \eqref{eq:Hausdorffconvergence} is satisfied, if $\mathscr{E}=(K_i)_{i\in{I}}$ has law $\P^{\infty}_{\sigma}$, then a.s.
\begin{align}
&\label{eq:TIGHT} \big(\eps^{d}(|\{i\in{I}\text{ such that } {\delta}(K_i)\geq \eps\}|\vee1)\big)_{\eps\in{(0,1]}} \text{ is tight in $(0,\infty)$},\\
&\label{eq:DIM} \text{dim}_H(K_i)=\text{dim}_M(K_i)=1+\tfrac{d}{2}, \text{ for all $i\in{I}$},\\
&\label{eq:CAP} \mathrm{cap}_{\R^d}(K_i)>0 \text{ and } K_i \cap (-1,1)^d \neq \emptyset \text{ for all $i\in{I}$,}
\end{align}
and for all $\varepsilon \in (0,1)$, $r,R>0$ and $x,y\in{[-1,1]^d}$ such that $d(x,y)\geq 2(r+R)$ and  $B(x,r),B(y,R)\subset [-1+\varepsilon, 1-\varepsilon]^d$,
\begin{align}
&\label{eq:CROS}
c(\varepsilon)d(x,y)^{2-d}(Rr)^{\frac{d-2}2}\leq \P_{\sigma}^{\infty}\big({B}(x,r)\leftrightarrow {B}(y,R)\text{ in }\mathscr{E}\big)\leq C(\varepsilon)d(x,y)^{2-d}(Rr)^{\frac{d-2}2}.
\end{align}
\end{The}

In particular, notice that \eqref{eq:TIGHT} implies that $\P^{\infty}_{\sigma}$ is non-trivial, in the sense that~
if $\mathscr{E}$ has law $\P^{\infty}_{\sigma}$, then $\mathscr{E}\neq\emptyset$-a.s. It is also natural to ask what the crossing formula \eqref{eq:CROS} becomes when $x=y$, which  corresponds to the annulus crossing probability. This can in fact be deduced directly from \eqref{eq:CROS}, but the formulas we obtain are then slightly different from the one in \eqref{eq:CROS}, and we refer to \eqref{eq:crossingseccrit} and Remark~\ref{rk:othercrossingbounds} for more details. 

During the preparation of this article, it  was announced in \cite{cai2026scalinglimitcriticalloop} that the family $\hat{\mathscr{E}}_n^{\geq 0}$ of sets $K\in{\mathscr{E}_{n}^{\geq 0}}$ such that $nK$ is a cluster of $E^{\geq 0}$ has a unique scaling limit along dyadic scales (note that $\mathscr{E}_n^{\geq 0}$ additionally contains renormalized clusters of $E^{\geq 0}\cap [-n,n]^d$ which are part of a cluster of $E^{\geq 0}$ intersecting $([-n,n]^d)^{\mathsf{c}}$). More precisely, for Lebesgue almost all $\eps>0$ \cite{cai2026scalinglimitcriticalloop} asserts convergence in law for the distance $\mathbf{d}$ (see Remark~\ref{rk:choiceofd}) for the subfamily of $\hat{\mathscr{E}}_n^{\geq 0}$ consisting of all its clusters with diameter larger than $\eps$, which, combined with tightness (see \eqref{eq:tightness-pf1}), trivially implies convergence in law for the full family $\hat{\mathscr{E}}_n^{\geq 0}$, see \eqref{eq:defdistance}. In particular, \eqref{eq:TIGHT}--\eqref{eq:CROS} are also in force for this scaling limit.

\bigskip

We now discuss some extensions of Theorem~\ref{the:Hausdorff}. One can replace the clusters defining $\mathscr{E}^{\geq 0}_n$ in \eqref{eq:renormalizedlevelset-intro} by the excursion sets above height $0$ of the (continuous) Gaussian free field on the metric graph corresponding to $n^{-1}\Z^d$ (see Section~\ref{subsec:metric} for relevant definitions), and the conclusions of Theorem~\ref{the:Hausdorff} still hold. This is because the two sets of clusters in question are at $\mathbf{d}$-distance $o(1)$ as $n \to \infty$, cf.~the discussion around \eqref{eq:new-clusters-ok}. Other natural modifications are to consider the Gaussian free field killed on $B(n)^{\mathsf{c}}$ in the definition of $\mathscr{E}_{n}^{\geq 0}$, or to replace $\mathscr{E}_{n}^{\geq 0}$ by all the connected components of $\frac1nE^{\geq 0}\subset\R^d$. Up to adapting the distance $\mathbf{d}$, most of the results from Theorem~\ref{the:Hausdorff} are still in force for these modifications, and we refer to Remark~\ref{rk:otherchoicesforE} for a related discussion.

As detailed in Remark~\ref{R:ext-signclusters} below, Theorem~\ref{the:Hausdorff} also applies to the closure of the rescaled sign clusters of the metric graph Gaussian free field (of \emph{any} sign). By means of Lupu's extension of Le Jan's isomorphism theorem \cite{zbMATH05757657,MR3502602}, the latter can also be identified with the (closures of) the clusters of loops stemming from the metric-graph loop soup on $\Z^d$ at intensity $1/2$. Theorem~\ref{the:Hausdorff} therefore equally applies to them. Actually, one can even deduce from Corollary~\ref{cor:macrocontainsbigloops} that in the scaling limit, each of these macroscopic clusters of loop soup entirely consists of macroscopic loops. In other words, the microscopic loops are used to connect together macroscopic loops, but do not create new parts in the cluster away from these macroscopic loops.

We now make this statement precise. Let us define the renormalized family of loop clusters within $[-n,n]^d$ as
\begin{equation}
\label{eq:deftildeEnC}
{\mathscr{E}}_n\stackrel{\text{def.}}{=}\text{ an enumeration of }\big\{n^{-1}{K}:\,K\text{ is a cluster of }\mathcal{L}\text{ within }[-n,n]^d\big\},
\end{equation} 
and the renormalized family of loops within $[-n,n]^d$ as
\begin{equation}
\label{eq:deftildeEL}
{\mathscr{L}}_n\stackrel{\text{def.}}{=}\text{ an enumeration of }\big\{n^{-1}\text{range}(\gamma):\,\gamma\text{ is a loop in }\mathcal{L}\text{ with }\text{range}(\gamma)\subset[-n,n]^d\big\}.
\end{equation}
By definition, see above \eqref{eq:renormalizedlevelset-intro}, both ${\mathscr{E}}_n$ and ${\mathscr{L}}_n$ a.s.\ belong to $\mathcal{S}$. Note however that contrary to the families  appearing in \eqref{eq:renormalizedlevelset-intro} and \eqref{eq:deftildeEnC}, the sets (i.e.\ traces of loops) in the family ${\mathscr{L}}_n$ can intersect each other. 
 We also write $\mathbf{d}_2$ for the $\ell^{\infty}$-product metric on $(\mathcal{S},\mathbf{d})\times(\mathcal{S},\mathbf{d})$, that is $\mathbf{d}_2((\mathscr{E}_1,\mathscr{E}'_1),(\mathscr{E}_2,\mathscr{E}'_2))=\max(\mathbf{d}(\mathscr{E}_1,\mathscr{E}_2),\mathbf{d}(\mathscr{E}'_1,\mathscr{E}'_2))$.

\begin{Cor}
\label{cor:loopscaling}
Assume $d\in{\{3,4,5\}}$. There exists a strictly increasing sequence $\sigma:\N\rightarrow\N$ and a probability measure $\tilde{\P}_{\sigma}^{\infty}$ on $\mathcal{S}\times\mathcal{S}$ such that, with respect to $\mathbf{d}_2$,
\begin{equation}
\label{eq:Hausdorffconvergenceloops}
\begin{split} 
\text{ the law of }({\mathscr{E}}_{\sigma(n)},{\mathscr{L}}_{\sigma(n)})\text{ converges weakly to }\tilde{\P}^{\infty}_{\sigma} \text{ as } n \to \infty.
\end{split}
\end{equation}
Moreover, for any $(\sigma,\tilde{\P}^{\infty}_{\sigma})$ such that \eqref{eq:Hausdorffconvergenceloops} is satisfied, if $({\mathscr{E}},{\mathscr{L}})=((K_i)_{i\in{I}},(K'_j)_{j\in{J}})$ has law $\tilde{\P}^{\infty}_{\sigma}$, then a.s.\ the family $(K_i)_{i\in{I}}$ satisfies \eqref{eq:TIGHT}--\eqref{eq:CROS}. Moreover, $\tilde{\P}^{\infty}_{\sigma}$-a.s.\ for each $i\in{I}$
\begin{equation}
\label{eq:loopbuildclusters}
K_i\text{ is equal to the closure of }\bigcup_{j\in{J}:\,K'_j\subset K_i}K'_j.
\end{equation}
\end{Cor}

Note that \eqref{eq:loopbuildclusters} implies in particular that each non-trivial scaling limit of clusters of loops contains the non-trivial scaling limit of a loop. Moreover, one can prove using  \cite{zbMATH05120529,qian2026couplingbrownianloopsoups}  that the family ${\mathscr{L}}$ from Corollary~\ref{cor:loopscaling} has the same law as the family consisting of the traces of the loops in the Brownian loop soup  on $\R^d$ from \cite{MR2045953}. 
In particular, its law does not depend on the choice of the sequence $\sigma$ as in \eqref{eq:Hausdorffconvergenceloops}. These Brownian loops have Hausdorff dimension two, whereas each set $K_i$ in ${\mathscr{E}}$ has Hausdorff dimension $1+d/2$ in view of \eqref{eq:DIM}. Therefore by \eqref{eq:loopbuildclusters}, there are actually infinitely many continuous loops $K'_j$ included in $K_i$. In particular in dimensions $d\in{\{4,5\}}$, for which the Brownian loops never intersect, one readily deduces that $K_i$ is not a connected component of loops in the Brownian loop soup. This confirms conjectures from \cite[Section~4.2.1]{werner2020clusters}. In $d=3$, the analogous conclusion follows by combining \eqref{eq:DIM} with \cite{cai2025gapdimensions}, where it is shown that the upper box-counting dimension of the clusters of the Brownian loop soup on $\R^3$ is strictly smaller than $5/2$. This is also in line with \cite{cai2025separation}, which suggests that already in dimension three, microscopic loops play a crucial role in the construction of macroscopic loop clusters.

In contrast to \eqref{eq:loopbuildclusters}, when $d\geq7$, there are clusters in the scaling limit which do not contain any macroscopic loop. This is due to the fact that the number of macroscopic clusters diverges in the scaling limit, see \cite[Proposition~4]{werner2020clusters}, whereas the number of macroscopic loops remains tight.  
In fact even if one only considers macroscopic clusters which contain a cycle, only half of these clusters will contain a macroscopic loop when $d\geq7$, see \cite{lupu2025intensity}, and \cite{carpenter2025loopscriticalhighdimensionalpercolation} for related results in high-dimensional Bernoulli percolation.

\medskip

We now describe how this article is organized. In Section~\ref{sec:prepa}, we introduce the metric graph, including with killing on a set $K$ and collect preliminary facts from potential theory. Section~\ref{sec:boxtobox} is devoted to connectivity bounds in the presence of obstacles, and in particular Section~\ref{sec:volneg} gathers box-counting bounds for the largest cluster, which are adaptations of results from \cite{DrePreRod10}. In Section~\ref{sec:largevolumecapacity}, we control the volume and capacity of clusters in a box. Section~\ref{sec:absence} shows that macroscopic clusters do not stay close to the boundary of a ball, and Section~\ref{sec:equimacro} combines these ingredients to prove Theorem~\ref{the:equivmacro} and Corollary~\ref{cor:macrocontainsbigloops}. Sections~\ref{sec:largevolumecapacity}--\ref{sec:equimacro} form the technical core of this article. Finally, in Section~\ref{sec:proofhausdorffdim} we specialize to $\Z^d$, $d\in\{3,4,5\}$, and prove Theorem~\ref{the:Hausdorff} and Corollary~\ref{cor:loopscaling}; the proof of a two-point estimate used therein is deferred to Appendix~\ref{App:lowMink}, and certain key properties of the distance $\mathbf{d}$ are relegated to Appendix~\ref{App:dmetric}.

Our rules in dealing with constants are as follows: constants without indices such as $c,C, c'$ etc.\ are used to denote generic positive constants that can change from one expression to the next. They may implicitly depend on the parameters $\alpha$ and $\nu$ appearing in conditions \eqref{eq:intro_sizeball} and \eqref{eq:intro_Green}, on the various constants entering conditions \eqref{eq:ellipticity}, \eqref{eq:intro_sizeball}, \eqref{eq:intro_Green}, and on the value of $\sup q$ whenever it is assumed to be finite. Numbered constants are fixed where they first appear.

\section{Some preparation}
\label{sec:prepa}
In this section, we introduce the metric graph $\tilde{\G}$, and explain how to reinterpret the results from the introduction using objects defined on the metric graph directly. We also introduce the metric graph $\tilde{\G}_K$ with killing on a set $K \subset \tilde{\G}$ belonging to a suitably large class, and collect a few preliminary facts from potential theory for the diffusion $X$ on $\tilde{\G}_K$, including a generalization of the classical sweeping identity (Lemma~\ref{L:sweeping-extended}), which is of independent interest. We also introduce the free field on $\tilde{\G}_K$ with law $\P^K$ and state the corresponding version of the isomorphism with the loop soup.  Throughout, we tacitly assume that $\G$ is a weighted graph as introduced in Section~\ref{sec:mainresults}, on which the assumptions \eqref{eq:ellipticity}, \eqref{eq:intro_sizeball} and \eqref{eq:intro_Green} are satisfied. 

 \subsection{The metric graph $\tilde{\G}$}
 \label{subsec:metric}
 The above general setup extends to the metric (or cable) graph $\tilde{\G}$ associated to $\G$, in which each edge $\{x,y\}$ is replaced by a continuous line segment $I_{\{x,y\}}$ of (Euclidean) length $1/(2\lambda_{x,y})$ glued at its endpoints. Thus $\tilde{\G}$, endowed with the resulting (geodesic) distance, is viewed as a metric space.  We refer to \cite[Section 2]{DrePreRod3} for a comprehensive introduction. In this manner the vertex set $G$ is naturally viewed as a subset of $\tilde{\G}$. In particular the Markov chain $X$ introduced above \eqref{eq:green} extends to a continuous process, the Brownian motion on $\tilde{\G}$. With a slight abuse of notation, we continue to use $X$, $P_{x}$, $x \in \tilde\G$, and $g(x,y)=g_{\tilde{\G}}(x,y)$, $x,y \in \tilde{\G}$ for the extension of quantities defined around \eqref{eq:green} to $\tilde{\G}$. We denote by $\tilde{B}(x,r)\subset\tilde{\G}$ the union of $B(x,r)$ and the set of points in $\tilde{\G}$  that belong to a segment whose endpoints are both contained in $B(x,r)$, $\partial B$ the topological boundary of $B\subset \tilde{\G}$, and $\{U {\longleftrightarrow} V\text{ in }K\}$  for $U,V,K \subset \tilde\G$,  the event that $U$ and $V$ are connected by a (continuous) path in $K$.
 
For a compact set $K\subset \tilde{\mathcal{G}}$ with finitely many connected components, we write $e_{K}= e_{K,\tilde{\G}}$ for the equilibrium measure of $K$ relative to $\tilde{\G}$ (see for instance \cite[(2.16)]{DrePreRod3} for its definition in the present context). Its support is a finite set included in $\partial K$, %
and its total mass 
 \begin{equation}
 \label{eq:cap}
 \text{cap}(K)= \text{cap}_{\tilde{\G}}(K) \stackrel{\text{def.}}{=} \int {\rm d}e_{K,\tilde{\G}} \ (< \infty)
 \end{equation}
 is the capacity of $K$. If $K\subset G$, it simply corresponds to the capacity introduced in \eqref{eq:capintro}. We also abbreviate $\delta(K)=\delta(K\cap G)$, $|K|_r=|K\cap G|_r$, as well as $|K|\coloneqq |K|_0$ (see \eqref{eq:vol}) for all $K\subset\tilde{\G}$, which, if $K\subset G$, yields consistent notation.
 
 Under the probability measure $\P$, we extend the definition of the Gaussian free field to $\tilde{\G}$:  $(\varphi_x)_{x\in \tilde{\G}}$ is the centered Gaussian field with covariance $g(\cdot,\cdot)$  (which agrees with $(\varphi_x)_{x\in{G}}$ on $G$).
 
  Under the probability measure $\Q$, let us denote by $\mathcal{L}_{\tilde{\G}}$ the loop soup on $\tilde{\G}$ with intensity $1/2$, see \cite{MR3238780,MR3502602}, which is a Poisson point process of random continuous loops on $\tilde{\G}$. For $A \subset \widetilde \G$ we write $\mathcal{L}_A$ for the restriction of $\mathcal L_{\tilde{\G}}$ to those loops whose range is completely contained in $A$. Note that the restriction of $\mathcal{L}_{\tilde{\G}}$ to the loops which intersect at least two distinct vertices of $G$ corresponds to the loop soup $\mathcal{L}$ introduced above \eqref{eq:deftildeMRS}. The loop soup $\mathcal{L}_{\tilde{\G}}$ also contains loops that do not intersect $G$, as well as loops that only intersect a single vertex of $G$, and the total time spent by these last loops at a given vertex $x$ is $\Gamma(1/2,\lambda_x^{-1})$-distributed, see \cite[(4.12) and (4.17)]{MR2932978}. We denote by $(L_x)_{x\in{\tilde{\G}}}$ the associated field of local times, whose restriction to $G$ thus agrees with the process introduced above \eqref{eq:deftildeMRS}.

The Gaussian free field on $\tilde{\G}$ induces a percolation model  at level $a$: two points $x,y\in{\tilde{\G}}$ are connected if and only if there is a continuous path between $x$ and $y$ along which $\phi\geq a$. An equivalent definition for the restriction to $G$ of these connected components  is to declare each edge open at level $a$ if and only if the field along that edge stays above $a$. The resulting families of percolation configurations on edges have the same laws as those defined by \eqref{eq:edge-perco} for every $a$, as follows readily by applying the Markov property of the free field when conditioning on the discrete field $(\varphi_x)_{x \in G}$, see \cite[(1.6)]{DrePreRod5}.

One can also define a percolation model associated to the loop soup on $\tilde{\G}$: two points $x,y \in \tilde{\G}$ are in the same cluster of loops in $\mathcal{L}_A$ \emph{within $B$} if and only if there is a path from $x$ to $y$ which is entirely included in the intersection of $B$ and the union of the traces of a finite number of loops in $\mathcal{L}_A$. A cluster of loops in $\mathcal{L}_A$ is simply a cluster of loops in $\mathcal{L}_A$ within $\tilde{\mathcal G}$. Similarly as for the Gaussian free field, the restriction to $G$ of the clusters of loops can be obtained by declaring an edge $e$ open if and only if $\{L_x>0\text{ for all }x\in{I_e}\}$. The corresponding clusters of open edges then have the same law as those defined above \eqref{eq:deftildeMRS}, see \cite[Corollary~3.6]{MR3502602}.

In other words, we can replace in the definitions \eqref{eq:defMRs}, \eqref{eq:deftildeMRS} and \eqref{eq:defMRSL} the discrete level sets and loop soups therein by their metric graph version (note that in case of capacity, this changes the definition by a constant factor for $s$). With a slight abuse of notation, in the remainder of the article we will thus write 
\begin{equation}
\label{eq:newnotation}
\mathcal{L}=\mathcal{L}_{\tilde{\G}}\text{ and }E^{\geq a}=\{x\in{\tilde{\G}}:\,\phi_x\geq a\},
\end{equation}
and forget altogether about the discrete notation from the introduction when proving the results from Section~\ref{subsec:equivalenceintro}, namely Theorem~\ref{the:equivmacro} and Corollary~\ref{cor:macrocontainsbigloops}. In particular, from now on $\mathcal{C}_{x,R}$, resp.\ $\tilde{\mathcal{C}}_{x,R}$, will be defined as the set of connected components of $E^{\geq 0}\cap \tilde{B}(x,R)$, resp.\ of loops in $\mathcal{L}_{\tilde{\G}}$ within $\tilde{B}(x,R)$. We will similarly only need to prove the results of Section~\ref{subsec:scalinglimitintro} for the Gaussian free field and the loop soup on metric graphs, and we refer to the discussion around \eqref{eq:new-clusters-ok} for details.

 \subsection{The graph $\mathcal G_K$} \label{subsec:G_K} We start by introducing the weighted graph $\G_K$ that we will mostly work with in the sequel, along with its associated metric graph $\tilde{\G}_K$. To put things into context, $K$ will fulfil various duties, as ``obstacle'' set $\mathcal{O}$ in various contexts, which -- depending on the situation -- will comprise e.g.~the range of loops of a certain type, or also involve the exploration of (finite) boundary clusters. To account for all our needs, we assume that 
\begin{equation}\label{eq:K-ass}
\text{$K\subset\tilde{\G}$ is compact and has finitely many connected components}
\end{equation}
(recall from Section~\ref{subsec:metric} that $\tilde{\G}$ is viewed as a metric space). Note that \eqref{eq:K-ass} includes the possibility $K=\emptyset$.  
Informally, $\G_K$ is such that its associated metric graph $\tilde{\G}_K$ corresponds to the union of the connected components of $K^{\mathsf c}=\tilde{\G} \setminus K$ which intersect $G$, and
with infinite killing on $K$. Our definition differs slightly from the one in \cite[Section~3.1]{DrePreRod10}, where $\tilde{\G}_K$ was defined as the (unique) unbounded connected component of $K^{\mathsf{c}}$, but we will here sometimes need to consider the connected component of $0$ in the case where $K$ disconnects $0$ from infinity. Formally, if $K\subset G$, $\G_K$ is the graph $(\bar{G}^M,\bar{\lambda}^M,\bar{\kappa}^M)$ introduced in \cite[(2.12)]{DrePreRod3}, when setting therein $\bar{G}=G$, $\bar{\lambda}=\lambda$, $\bar{\kappa}_x=\infty$ if $x\in{K}$ and $\bar{\kappa}_x=0$ otherwise, which is obtained from $G\cap K^{\mathsf{c}}$ by adding midpoints to any edge between $K$ and $K^{\mathsf{c}}$ with a killing measure adapted so that the total length of that edge does not change. If $K\subset\tilde{\G}$, we define $\G_K=(G_K,\lambda^K,\kappa^K):=(\G^{\partial K})_{\partial K}$, where $\G^{\partial K}$ is the graph from \cite[Lemma~2.1]{DrePreRod3}, which is equivalent to $G$ but has additionally the points in $\partial K$ as vertices. The metric graph $\tilde{\G}_K$ is then obtained similarly as $\tilde{\G}$ is from $\G$, except that for each $x\in{G_K}$ with $\kappa_x^K>0$, one adds a half-open interval $I_x$ of length $1/(2\kappa_x^K)$ identified with $x$ through its closed end.

By adapting the reasoning of \cite{DrePreRod3}, which can be extended to the setting of $K$ satisfying \eqref{eq:K-ass},
one can then identify $\tilde{\G}_K$ with the union of the connected components of $\widetilde{\G}\setminus K$ (an open set), and identify the diffusion on $\tilde{\G}_K$, whose law will be denoted by $\PPK_{\cdot}$ (so $P_{\cdot}^{\emptyset}= P_{\cdot}$), with the diffusion on $\widetilde{\G}\setminus K$ killed when entering $K$; that is, the law of $X_{\cdot \wedge H_{K}}$ under $P_{\cdot}$, where $H_{K} =  \inf\{t \geq 0 : X_t \in K \}$. Note that $\tilde\G_K$ may well comprise several (but finitely many since $\G$ is locally finite under our assumptions, cf.~the discussion following \eqref{eq:intro_sizeball}) connected components.  Various useful results concerning $\G_K$ previously assumed that it is connected. One can safely transfer them from the case of connected graphs to the case of graphs with finitely many connected components by considering each finite connected component as its own graph, and this slight extension is implicit in various references to previous literature below.  

For $K$ as in \eqref{eq:K-ass}, the Gaussian free field on $\tilde{\G}_K$ is the continuous centered Gaussian process $(\varphi_x)_{x\in \tilde\G_K}$ with law $\P^K$ such that
\begin{equation}\label{eq:introGFF}
\E^K[\varphi_x \varphi_y] = g_{\tilde{\G}_K}(x,y), \quad \text{ for all }x,y \in \tilde\G_K,
\end{equation}
and we simply write $\PKK{\emptyset}=\P$ in the same vein as above, which is consistent with our notation from the previous subsection.  Note that $\varphi_x=0$ $\PK$-a.s.~for all $x\in K$ so $\P^K$ is the Gaussian free field with zero Dirichlet boundary condition on $K$. 

Recall from Section~\ref{subsec:metric} the loop soup $\mathcal{L}$ on $\tilde{\G}$ under $\Q$ and its restrictions $\mathcal{L}_A$. Up to extending the probability space $\Q$, we further attribute to each cluster of loops in $\mathcal{L}_A$ a sign $\pm$ with probability $1/2$, independently for each cluster. The positive clusters of loops are the clusters that have sign $+1$. By the isomorphism \cite[Proposition~2.1]{MR3502602} between loop soups and the Gaussian free field, applied on the graph $\G_K$, we have for all $K$ as in \eqref{eq:K-ass} (including $K=\emptyset$, in which case $\P^K=\P$ and $\mathcal{L}_{\tilde{\G}_K}=\mathcal{L}$):
\begin{equation}
\label{eq:isoloopK}
\begin{gathered}
\{x\in{\tilde{\G}}:\phi_x\geq0\}\text{ under }\P^K\text{ has the same law as the}
\\\text{closure of the positive clusters of loops in $\mathcal{L}_{\tilde{\G}_K}$ under }\Q.
\end{gathered}
\end{equation}

\subsection{Potential theory} \label{subsec:pot} We now collect a few aspects of potential theory, and start with a useful generalization of the classical sweeping identity. We will eventually apply what follows with the choice $\tilde{\G}=\tilde{\G}_K$, which is perfectly fine as $\tilde{\G}_K$ has been introduced as \emph{the} metric graph of some weighted graph (namely, $(G_K,\lambda^K,\kappa^K)$ above), but this is irrelevant for now. 

Let $Z=(Z_0, Z_1,\dots)$ under $P_{\cdot}$ be the successive vertices in $G$ visited by the diffusion $X$ on $\tilde{\G}$, disregarding successive visits to the same vertex, which in law is the discrete random walk on $G$ with %
transition probabilities equal to $\lambda_{x,y}/ \lambda_x$ and probability $\kappa_{x}/\lambda_x$ to be killed at $x$. For $B \subset G$ let $L_B=\sup\{n\geq0:\,Z_n\in{B}\}$ denote the time of the last visit of $Z$ to $B$, and $\tilde{H}_B =\inf\{ n \geq 1: Z_n \in B\}$ the hitting time of $B$, with $\sup\emptyset=-\infty$ and  $\inf\emptyset=+\infty$ by convention.  
The following result holds for any %
transient weighted graph $\G = (G, \lambda)$, without any further assumptions on $\G$. %

\begin{Lemme}[Generalized sweeping identity] \label{L:sweeping-extended} 
For all %
$A,B\subset{G}$ and $y\in A$,
\begin{equation}
\label{eq:sweeping-gen}
\begin{split}
P_{e_{B,\tilde{\G}}}(X_{H_A}=y,H_A<\infty)&=\lambda_y P_y ({L_{B}}\leq \tilde{H}_A,L_{B}\geq0).
\end{split}
\end{equation}
\end{Lemme}
In particular, when $A\subset B$, one has under $P_y$ for any $y \in A$ that $\{0 \leq L_{B}\leq \tilde{H}_A\} = \{\tilde{H}_A=\infty\}$, and one recovers the usual sweeping identity, see for instance \cite[(1.59)]{MR2932978}.
\begin{proof}
Decomposing over the (integer) value $n \geq 0$ of $L_B$ and the corresponding location $x (
\in B)$ of the discrete skeleton $Z$, one obtains for any $y \in A$ that
\begin{multline}\label{eq:sweep-pf}
P_y(0 \leq {L_{B}}\leq \tilde{H}_A)=\sum_{n,x} P_y(Z_k \notin A, 1 \leq k < n, Z_n=x, Z_m \notin B, \, m >n)\\
=\sum_{n,x} P_y(Z_k \notin A, 1 \leq k < n, Z_n=x) \lambda_x^{-1} e_{B,\tilde{\G}}(x),
\end{multline}
where the second line follows by applying the Markov property at time $n$, and by convention $P_y(Z_k \notin A, 1 \leq k \leq n, Z_n=x) = \delta_{y,x}$ when $n=0$. By time-reversal, one readily sees, with $\gamma= (\gamma_i)_{0 \leq i \leq \ell(\gamma)}$ below ranging over all finite length (i.e.~$0 \leq \ell(\gamma) < \infty$) nearest-neighbor paths in $G$ such that $\gamma_0=y$, $\gamma_{\ell(\gamma)}=x$ and $\gamma_i \notin A$ for $1\leq i < \ell(\gamma)$, that
\begin{multline*}
\sum_{n \geq 0} \lambda_y P_y(Z_k \notin A, 1 \leq k < n, Z_n=x) =
\sum_{\gamma} \Big( \lambda_{\gamma_0, \gamma_1}\prod_{1 \leq i < \ell(\gamma)} \tfrac{\lambda_{\gamma_i, \gamma_{i+1}}}{\lambda_{\gamma_i}} \Big)
\\
= \sum_{\gamma} \Big( \lambda_{\gamma_{\ell(\gamma)}, \gamma_{\ell(\gamma)-1}}\prod_{1 \leq i < \ell(\gamma)} \tfrac{\lambda_{\gamma_{i}, \gamma_{i-1}}}{\lambda_{\gamma_i}} \Big)=
\lambda_x P_x(X_{H_A}=y,H_A<\infty),
\end{multline*}
using symmetry of $\lambda_{\cdot,\cdot}$ in the second line. Feeding this into \eqref{eq:sweep-pf} completes the proof.
\end{proof}

We now gather a few random walk estimates tailored to our needs that will be used repeatedly. Unlike in Lemma~\ref{L:sweeping-extended}, for what follows it is crucial that \eqref{eq:ellipticity}, \eqref{eq:intro_sizeball} and \eqref{eq:intro_Green} be in force. We consider $\tilde{\G}_K$ as introduced in Section \ref{subsec:G_K} with $K$ satisfying \eqref{eq:K-ass}; in particular, $K=\emptyset$ possibly. The Green function $g_{\tilde{\G}_K}$ on $\tilde{\G}_K$ is defined in the same manner as the Green function $g_{\tilde{\G}} (= g_{\tilde{\G}_\emptyset})$ on $\tilde{\G}$ (cf.~above \eqref{eq:cap}), and by \cite[Lemma~3.1]{DrePreRod2} it satisfies
\begin{equation}
\label{eq:GreenGK}
cd(x,y)^{-\nu}\leq g_{\tilde{\G}_K}(x,y)\leq Cd(x,y)^{-\nu},\text{ for all }x\neq y\in{\tilde{\G}}\text{ with }d(x,K)\geq Cd(x,y)
\end{equation}  %
(note that for $K= \emptyset$ this is just  \eqref{eq:intro_Green}, from which \eqref{eq:GreenGK} derives). For a compact set $B\subset \tilde{\mathcal{G}}$ with finitely many connected components, the total mass of the equilibrium measure $ e_{B,\tilde{\G}_K}$ relative  to $\tilde{\G}_K$ is the capacity ${\rm cap}_{\tilde{\G}_K} (B)$, cf.~\eqref{eq:cap}. In view of \cite[(5.7)]{DrePreRod5}, the capacity of a ball satisfies
\begin{equation}
\label{eq:capBallBd}
\begin{split} 
cr^\nu \leq {\rm cap}_{\tilde{\G}_K} (B(x,r)) \leq Cr^\nu, \ \text{ for all $K\subset \tilde{B}(0,r)$, $x\in{G}$ s.t.~$d(x,0)\geq Cr1\{K\neq\emptyset\}$}.
\end{split}
\end{equation}
 Further, combining a reasoning similar to \cite[(3.6)]{DrePreRod10} for the lower bound, and \cite[(1.57)]{MR2932978} and \eqref{eq:intro_Green} for the upper bound, there exists $\Cl{C:exit}\in{(1,\infty)}$ such that for all $A, K \subset\tilde{\G}$ each satisfying \eqref{eq:K-ass}, $r\geq1$, $x,y\in{{G}_K}$, $A\subset \tilde{B}(x,r)$, $y\in{B(x,r)}$ with $d(x,K)\geq 2\Cr{C:exit}r$, one has
 \begin{equation}
\label{eq:hittingvscap}
\begin{split} 
 cr^{-\nu}\mathrm{cap}_{\tilde{\G}_K}(A)\leq P_y^K\big(H_{A}<H_{\tilde{B}(x,\Cr{C:exit}r)^{\mathsf c}}\big)\leq P_y^K\big(H_{A}<\infty)\leq Cd(y,A)^{-\nu}\mathrm{cap}_{\tilde{\G}_K}(A).
\end{split}
\end{equation}
Note that \eqref{eq:capBallBd}, \eqref{eq:hittingvscap} remain true if $K=\emptyset$, in which case $\G_{\emptyset}=\G$ and $\PPKK{\emptyset}_{\cdot}=P_{\cdot}$ as introduced above \eqref{eq:defLambdaL}. 

We also recall the following elliptic Harnack inequality. For $B \subset G$ let $\overline{B}= B \cup \partial (B^{\mathsf{c}})$. Below a function $v: \overline{B} \to \R$ is called \emph{harmonic in $B$} if $v(y)=E_y[v(Z_1)]$ for all $y \in B$. %
By \cite[(3.3) and Lemma~6.5]{DrePreRod2}, there exists $\Cr{c:harnack}>0$ and $C<\infty$ such that for any $x\in{G}$ and $r\geq C$, if $v$ is non-negative and  %
harmonic in $B(x,Cr)\setminus B(x,\Cl[c]{c:harnack}r)$, resp.\ $B(x,r)$, writing $A_r=B(x,9r)\setminus B(x,r/3)$, resp.~$A_r=B(x,\Cr{c:harnack}r)$, one has 
\begin{equation}
\label{eq:Harnack}
    \sup_{y\in{A_r}}v(y)\leq C\inf_{y\in{A_r}}v(y)
\end{equation}
(in case $A_r=B(x,9r)\setminus B(x,r/3)$, \eqref{eq:Harnack} is actually stated in \cite[Lemma~6.5]{DrePreRod2} for the sphere $\partial B(x,r)$ instead of the annulus $B(x,9r)\setminus B(x,r/3)$, but the proof therein can easily be adapted).

\section{Connectivity bounds with obstacles \texorpdfstring{$\mathcal{O}$}{O}}
\label{sec:boxtobox}

The remainder of this article is concerned with proving the results of Section~\ref{sec:mainresults}. Throughout, we tacitly assume that $\G$ is a weighted graph as introduced in Section~\ref{sec:mainresults}, on which the assumptions \eqref{eq:ellipticity}, \eqref{eq:intro_sizeball} and \eqref{eq:intro_Green} are satisfied. As in Section~\ref{sec:mainresults}, we further assume for simplicity that $d(\cdot,\cdot)$ is the graph distance on $\G$. The assumption on $d(\cdot,\cdot)$ can be relaxed, see Remark~\ref{rk:ballsconnected}.

The main result of this section, Proposition~\ref{pro:boxtobox} below, is an upper bound on the probability to connect two sets $A$ and $A'$ using positive sign clusters of loops while avoiding an obstacle set $\mathcal{O}$, which by \eqref{eq:isoloopK} is equivalent to connecting $A$ and $A'$ via the positive sign clusters of the GFF on the graph $\tilde \G_{\mathcal{O}}$ with killing on $\mathcal{O}$, introduced in Section~\ref{sec:prepa}.
The proof is inspired by that of \cite[Proposition~3.1]{cai2024onearm}, 
which deals with $\mathcal{O}=\emptyset$  in the case where $A, A'$ are concentric annuli and when the base graph is $\Z^d$. The generalization to $\G$ satisfying the above conditions (still with  $\mathcal{O}=\emptyset$) is relatively straightforward, and we seize this opportunity to streamline some of the arguments from \cite{cai2024onearm}. In addition, Proposition~\ref{pro:boxtobox} generalizes this bound on the connection probability between $A$ and $A'$ by allowing for the presence of the obstacle $\mathcal{O} \neq \emptyset$, which produces an extra factor, roughly corresponding to the probability to hit $A'$ when starting a random walk from $A$ with killing on $\mathcal{O}$, see \eqref{eq:crossingAB} below; see also \cite[Lemma 3.2]{DrePreRod9} for the emergence of a similar factor in a different context.
 
 The upper bound derived in Proposition~\ref{pro:boxtobox} will be essential both, when $\mathcal{O}=\emptyset$ to obtain box-to-box estimates (see Proposition~\ref{cor:crossingboxseccrit}), and when $\mathcal{O}$ consists of large clusters of loops to prove that it is unlikely for a cluster with large diameter to have small capacity or (renormalized) volume, see Section~\ref{sec:volneg}, which gathers box-counting bounds for the largest cluster, as well as
 Proposition~\ref{pro:radimpliesOavoided} and its ramifications in Section~\ref{sec:equimacro}.

Recall from \eqref{eq:diam} that $\delta(\cdot)$ denotes diameter and the notations $q(\cdot)$ from \eqref{eq:qPsiBd} and $\rr$ from \eqref{eq:r1}. 
The set $\tilde{\G}_{\mathcal{O}}$ refers to the metric graph with killing on $\mathcal{O}$ as introduced in Section~\ref{sec:prepa}. Following is the main result of this section. It yields an upper bound on the probability to connect two sets $A$ and $A'$ outside $\mathcal{O}$, subject to certain geometric conditions on these sets that involve the scales $r,R$ and $R'$. A guiding example to keep in mind is to pick $A=\tilde{B}(x,r)$, $A'$ a set (e.g.$\tilde{B}(y,R)$ with $d(x,y)=R'\geq C\rr$ and $R\leq R'/8$) whose $R$-neighborhood $B(A',R)=\bigcup_{z\in A' \cap G} B(z,R)$ of $A'$ lies in the ``annulus'' $B(A,2R')\setminus B(A,R'/2)$ at distance of order $R'$ from $A$, and an obstacle $\mathcal{O} \subset B(x,cR') \setminus B(x,C\rr)$, i.e.,~in between $A$ and that annulus, but sufficiently far away from both.%
Recall the notation 
$\{A\leftrightarrow A' \, \text{(in $K$)}\}$ from Section~\ref{subsec:metric}. It will be convenient to give a name to the ``local neighborhoods'' of $A$, $A'$, 
$$
U = B(A,r), \quad U' = B(A', R),
$$
so that $A \subset U$, $A' \subset U'$.

\begin{Prop}
\label{pro:boxtobox}
Let $\mathcal{O}, A , A'$ be compact sets with finitely many connected components each and such that $\mathcal{O}\subset\tilde{\G}$, $\emptyset\neq A,A'\subset\tilde{\G}_{\mathcal{O}}$.
Then, if $r,R\geq0$ and $R'\geq C\rr$ are such that  $\delta(A)\leq r$, $U'\subset B(A,{ 2}R')\setminus B(A,{R'}/2)$, and $\mathcal{O}\subset B(A,cR')\setminus B(A,C\rr)$, we have
\begin{equation}
\label{eq:crossingAB}
\begin{split} 
\P^{\mathcal{O}}(A\leftrightarrow A')\leq Cq(\rr)q(\RR)\left(\frac{\rr}{\RR}\right)^{\frac\nu2}\inf_{z\in{U}}P_z\big(H_{U'}<H_{\mathcal{O}}\big).
\end{split}
\end{equation}
\end{Prop}

 The remainder of this section is organized as follows. Proposition~\ref{pro:boxtobox} is proved in Section~\ref{subsec:pf-boxtobox}. Prior to that, we discuss in Section~\ref{subsec:consequ-boxtobox} some consequences of Proposition~\ref{pro:boxtobox} in the case $\mathcal{O}=\emptyset$.

\subsection{Consequences of Proposition~\ref{pro:boxtobox} for $\mathcal{O}=\emptyset$} \label{subsec:consequ-boxtobox}
 Proposition~\ref{pro:boxtobox} implies the upper bounds in the following estimates for box-to-box and annulus crossings. In the sequel we denote by $x_0 \in G$ an arbitrary reference point and abbreviate $B(N)=B(x_0, N)$, $\tilde{B}(N)=\tilde{B}(x_0,N)$ and $\P^N= \P^{\partial B(N)}$  the law of the field killed on $\partial B(N)$, see \eqref{eq:introGFF}. We stress that all results (in particular, the constants $c,C$ appearing below) hold uniformly in the choice of $x_0$ without further mention.

\begin{Prop}
\label{cor:crossingboxseccrit} Assume that $\sup q < \infty$, and let $\eps\in{(0,1)}$. For all $N\geq1$,  $x,y\in{G}$ and $R,r\geq0$ such that $d(x,y)\geq 2(\RR+\rr)$ and $B(x,r),B(y,R)\subset B((1-\eps)N)$,  we have 
\begin{equation}
\label{eq:crossingboxseccrit}
\begin{split} 
c(\varepsilon)d(x,y)^{-\nu}(\RR\rr)^{\frac\nu2}\leq \P^N({B}(x,r)\leftrightarrow {B}(y,R))\leq Cd(x,y)^{-\nu}(\RR\rr)^{\frac\nu2}.
\end{split}
\end{equation}
\end{Prop}

\begin{Rk}[Further consequences]
\label{rk:othercrossingbounds}
\begin{enumerate}[1)] 
\item Without the assumption $\sup q < \infty$, the upper bound in \eqref{eq:crossingboxseccrit} remains true with an additional factor $q(\rr)\sup_{s\in{[c\RR,C\RR]}}q(s)$(cf.~\eqref{eq:crossingAB}), while the lower bound does not require this assumption. This follows directly from the proof.
One immediately deduces from this the crossing estimates
\begin{equation}
\label{eq:crossingannuli}
c(\varepsilon)\left(\frac{\rr}{R}\right)^{\frac\nu2}\leq \P^{N}({B}(x,r)\leftrightarrow \partial {B}(x,R))\leq Cq(\rr)q(cR)\left(\frac{\rr}{R}\right)^{\frac\nu2},
\end{equation}
valid for $N \geq 1$, $\varepsilon \in (0,1)$, $B(x,R)\subset B((1-\eps)N)$ and $R \geq \rr$. It suffices to argue when $R\geq C'\rr$ for a large constant $C'$: if $\rr\leq R\leq C'\rr$, the upper bound is trivial, and the lower bound follows from the one for $C'\rr$ instead of $R$, since the probability in \eqref{eq:crossingannuli} is non-increasing in $R$. For the lower bound in \eqref{eq:crossingannuli}, one places a ball $B(y, c\varepsilon R)$ inside $B((1-\tfrac\eps2)N)\setminus B(x,R)$ in such a way that $(R< )d(x,y) \leq CR $, so that $d(x,y)^{-\nu} \geq cR^{-\nu}$. This can be achieved by  resorting to the approximate lattice $\Lambda_{c\varepsilon R}$, see \eqref{eq:defLambdaL}, and picking $y$ therein. Similarly, using a bounded number of balls with centers in $\Lambda_{c R}$ and radius $cR$, all contained in $B(x,R)$, one covers $\partial B(x,R/2)$ and the upper bound in \eqref{eq:crossingannuli} follows via a union bound.
\item \label{eq:crossing-infinitevol} Letting $N\rightarrow\infty$ (with $\varepsilon=\tfrac12$ say), analogues of both \eqref{eq:crossingboxseccrit} and \eqref{eq:crossingannuli} remain true with $\P$ in place of $\P^N$, which hold with a uniform constant $c=c(\frac12)$ in the lower bounds and removing the restriction involving $B((1-\eps)N)$.
\item %
 (Nearby balls in \eqref{eq:crossingboxseccrit} and conformal invariance). The restriction $d(x,y)\geq 2(\RR+\rr)$ in \eqref{eq:crossingboxseccrit} is not essential for the upper bound. Assume that $\sup q<\infty$, $r\leq R$ and $d(x,y)\geq R+C\rr$. On the event $\{B(x,r)\leftrightarrow B(y,R)\}$, one has $B(x,r)\leftrightarrow\partial B(x,d(x,y)-R-1)$, and thus by \eqref{eq:crossingannuli} (with $N\rightarrow\infty$, see item~\ref{eq:crossing-infinitevol} above) and monotonicity,
\begin{equation}
\label{eq:crossingnearby}
\P^N\big(B(x,r)\leftrightarrow B(y,R)\big)\leq C\Big(\frac{\rr}{d(x,y)-R}\Big)^{\frac\nu2}\leq C'\Big(\frac{\rr\RR}{d(x,y)^2-(\RR-\rr)^2}\Big)^{\frac\nu2},
\end{equation}
where the last inequality follows from $d(x,y)^2-(\RR-\rr)^2=(d(x,y)-\RR+\rr)(d(x,y)+\RR-\rr)$ when $d(x,y)\leq 2(\RR+\rr)$, and is a consequence of \eqref{eq:crossingboxseccrit} otherwise. The right-hand side of \eqref{eq:crossingnearby} is thus of the same order as the bounds in \eqref{eq:crossingboxseccrit} when $d(x,y)\geq2(\RR+\rr)$, but also captures the behavior of the connection probability when $B(x,r)$ and $B(y,R)$ are close to each other relative to $R$, in which case it is of order $(\rr/(d(x,y)-R))^{\nu/2}$. We expect a matching lower bound to hold in this regime as well. It could presumably be obtained by adapting the proof of the lower bound in \eqref{eq:crossingboxseccrit} below, exploring only the clusters of the part of $\partial B(y,R)$ within distance $C(d(x,y)-R)$ of $x$, for which the variance of the martingale at time $0$ is of order $(d(x,y)-R)^{-\nu}$, while the bound on its quadratic variation is unchanged; we do not pursue this here. Finally, when $d$ is the Euclidean distance (for instance for balls in $\R^d$ in the scaling limit), one has
\begin{equation*}
\frac{\rr\RR}{d(x,y)^2-(\RR-\rr)^2}=\frac1{2(1+\mathrm{I})},\quad\text{where }\mathrm{I}\stackrel{\text{def.}}{=}\frac{d(x,y)^2-\rr^2-\RR^2}{2\rr\RR}
\end{equation*}
is the inversive product of the spheres $\partial B(x,\rr)$ and $\partial B(y,\RR)$, see \cite[Section~3.2]{beardon1983geometry}, which is invariant under Möbius transformations of $\R^d\cup\{\infty\}$. It thus seems natural to ask whether the connection probabilities in the scaling limit, see Theorem~\ref{the:Hausdorff}, are functions of $\mathrm{I}$ only. 
\end{enumerate}

\end{Rk}

\subsection{Box counting for the largest cluster}
\label{sec:volneg}
In this subsection, we gather bounds on the renormalized volume of the largest cluster in a ball. They extend results from \cite{DrePreRod10} to the renormalized volume $|\cdot|_r$ from \eqref{eq:vol}, and to clusters having simultaneously large volume and large capacity. Their proofs follow closely those of \cite[Sections~2, 4 and~5]{DrePreRod10}; we give the short proof of Lemma~\ref{lem:critvol-ubren}, and only sketch the proof of Proposition~\ref{thm:macroscopiclocalcluster}, indicating the necessary changes. %

For  $r \in [0,\infty)$, $A\subset \tilde{\G}$ and $x \in \tilde{\mathcal{G}}$ we define the sets
\begin{align} \label{eq:clustersDef}
\begin{split}
\mathcal{K}^0_A(x, r) &\stackrel{\text{def.}}{=} \{ y \in \tilde{\mathcal{G}} : B(x, r) {\leftrightarrow} y\text{ in }E^{\geq 0}\cap A \}.
\end{split}
\end{align}
We further write $\K^0_A(x)=\K^0_A(x,0)$ for the cluster of $x$ in $E^{\geq 0}\cap A$, and in case $A=\tilde{\G}$ we abbreviate $\K^0(x,r)=\K^0_{\tilde{\G}}(x,r)$ and $\K^0(x)=\K^0_{\tilde{\G}}(x)$. For $R\geq1$ and $r\in{[0,R)}$, let us introduce the random variable
\begin{equation}
\label{eq:defVra}
\begin{split} 
\LV_{R,r}^0\stackrel{\text{def.}}{=}\sup\big\{|K|_r:\,K\text{ is a connected component of }E^{\geq 0}\cap\tilde{B}(R)\big\}.
\end{split}
\end{equation}
In words, $\LV_{R,r}^0$ measures the maximal (renormalized) cluster volume at scale $r$ inside $\tilde{B}(R)$, that is, the maximal number of balls of radius $r$ centered in $\Lambda_{r}$ (see \eqref{eq:defLambdaL}) that any cluster of $\{ \varphi \geq  0\} \cap \tilde{B}(R)$ intersects.

We start with a simple but useful bound on the (renormalised) cardinality of $\mathcal{K}^0(x, r)$, and on $\LV^0_{R,r}$. They are versions of the bound \cite[(2.18)]{DrePreRod10} and of \cite[Proposition~2.4]{DrePreRod10} on the renormalized lattice $\Lambda_{r}$, and follow from Proposition~\ref{cor:crossingboxseccrit} below. Recall $|\cdot|_r$ from \eqref{eq:vol}.

\begin{Lemme}
\label{lem:critvol-ubren}
Assume $\sup q<\infty$. For all $x\in{G}$, $R\geq1$, $r\in{[0,R)}$ and $r'\in{[0,r]}$,%
\begin{align}
&\E\big[|\K^{0}(x,r')\cap B(x,R)|_r\big]\leq C R^{\alpha-\nu} \rr^{\tfrac\nu2-\alpha}(\rr')^{\tfrac\nu2}, \label{eq:expectedvolumeinballren} \\
&\E\big[\LV^{0}_{R,r}\big]  \le C(R/\rr)^{\alpha-\frac{\nu}2}. \label{eq:upperMR0}
\end{align}

\end{Lemme}    
\begin{proof}
By \eqref{eq:crossingboxseccrit} and Remark~\ref{rk:othercrossingbounds},\ref{eq:crossing-infinitevol}, the left side of \eqref{eq:expectedvolumeinballren} is bounded by
\begin{equation*}
\begin{split} 
& \sum_{y\in{B(x,R+r)\cap \Lambda_{r}}}\P\big({B}(x,r')\leftrightarrow {B}(y,r)\big)
\leq C+C(\rr\rr')^{\tfrac{\nu}2}\sum_{y\in{\Lambda_{r}\cap B(x,R+r)\setminus B(x,4\rr)}} d(x,y)^{-\nu}.
\end{split}
\end{equation*}
Using \eqref{eq:defLambdaL} one readily finds that the last sum is bounded by $C' \rr^{-\alpha}R^{\alpha-\nu}$, which proves \eqref{eq:expectedvolumeinballren}. Turning to \eqref{eq:upperMR0}, following ideas from \cite[Proposition~2.4]{DrePreRod10}, we have by \eqref{eq:clustersDef}, \eqref{eq:defVra}, \eqref{eq:defLambdaL},\begin{equation*}
\begin{split} 
\E\big[(\LV_{R,r}^{0})^2\big]\leq \E\Big[\sum_{x\in{B(2R)\cap\Lambda_{r}}}|\K_{\tilde{B}(R)}^{0}(x,r)|_r\Big]\leq C(R/\rr)^{\alpha}\sup_{x\in{G}}\E\big[|\K^{0}_{\tilde{B}(x,3R)}(x,r)|_r\big].
\end{split}
\end{equation*}
Bounding the last expectation by means of \eqref{eq:expectedvolumeinballren} (with $3R$ instead of $R$ and $r'=r$), we obtain $\E[(\LV_{R,r}^{0})^2]\leq C(R/\rr)^{2\alpha-\nu}$, and \eqref{eq:upperMR0} follows by the Cauchy--Schwarz inequality.
\end{proof}

The lower bound on $\LV^0_{R,r}$, and more generally Proposition~\ref{thm:macroscopiclocalcluster} below, rely on the following condition, which ensures that clusters inside a ball have large capacity with sufficiently large probability, see also \eqref{eq:LBL2} in the proof of Proposition~\ref{thm:macroscopiclocalcluster}. Define for all $R\geq 1$, $\eta>0$ and $r\in{[0,R)}$ the quantity 
\begin{align}
\label{eq:defprR}
\begin{split} 
& p(r,R,\eta) \\
&\quad \stackrel{\text{def.}}{=} \inf_{x,y} \sup\Big\{v\geq0:\,\P\big(\exists z \in B(y,\rr) \, : \, \mathrm{cap}\big(\K^0_{\tilde{B}(x,R)}(z)\big)\geq (vR)^{\nu}\big)\geq \eta \left({R}/{\rr}\right)^{-\frac\nu2}\Big\},
\end{split}
\end{align}
with the sup ranging over $x \in G, y \in B(x,R/2)$. Note that due to $\K^0_{\tilde{B}(x,R)}(x,r) \subset \tilde{B}(x,R)$,  the monotonicity of the capacity in combination with \eqref{eq:capBallBd} directly supplies us with the finiteness of  $p(r,R,\eta)$ for all admissible choices of parameters. What is more relevant to us, however, is that $p$ is bounded away from $0$, i.e.
\begin{equation} \label{eq:pPos}
\underline{p}\stackrel{\text{def.}}{=} \inf_{R'\geq1, r \in [1,cR']}p(r,R',\eta)>0 \quad \text{ for some $c,\eta > 0$.}
\end{equation}
As we will use \eqref{eq:pPos} frequently in the following, we note here that the validity of \eqref{eq:pPos} will be established in Corollary \ref{cor:supqimpliessupp} under the sole assumption $\sup q<\infty$. 
 Nevertheless, as we have not established that corollary yet, we will  keep the dependence on $\underline{p}$ explicitly for the remainder of this subsection.

Recall the definition of $\LV^{0}_{R,r}$ from \eqref{eq:defVra}, the notation $\mathcal{L}_{A}$, $A \subset \tilde{\G}$, from Section~\ref{subsec:metric}, and that we attribute to each cluster of loops of $\mathcal{L}_{\tilde{B}(x,R)}$ a sign $\pm$ with probability $1/2$, independently for each cluster. The following result generalizes \cite[Proposition~2.4 and Theorem~2.5]{DrePreRod10}, which correspond to the special case $r=0$, and asserts in particular that the largest cluster in $\tilde{B}(x,R)$ has renormalized volume of order $(R/\rr)^{\alpha-\frac\nu2}$.

\begin{Prop}
    \label{thm:macroscopiclocalcluster}
 Assume that $\sup q<\infty$ and that \eqref{eq:pPos}
 holds. Then there exist constants $c,C \in (0,\infty)$ such that  for all $x\in{G}$, $t \in (0, c]$, $r>0$ and $R\geq C\rr/t$, 
\begin{equation}
\label{eq:tailEstlargestVolgeninB(r)}
\begin{split} 
\Q \Big(\exists\text{ a cluster }K\text{ in }\mathcal L_{\tilde{B}(x,R)}\text{ such that } |K|_{\rr}\geq (tR/\rr)^{\alpha-\nu/2}&\text{ and }\mathrm{cap}(K)\geq (tR)^{\nu}  \Big)
\\&\geq 1-\exp\{-ct^{-\alpha}\},
\end{split}
\end{equation}
and for all $t\geq1$,  $r\geq0$ and $R\geq C\rr$,
\begin{equation}\label{eq:MR0tight}
\P\Big( t^{-1}(R/\rr)^{\alpha-\frac{\nu}2}\leq \LV_{R,r}^{0}\leq t(R/\rr)^{\alpha-\frac{\nu}2}\Big)\geq 1-Ct^{-1}.
\end{equation}
Moreover, under the sole assumption $\sup q<\infty$, i.e.~without assuming \eqref{eq:pPos}, the upper bound on $\LV_{R,r}^{0}$ in \eqref{eq:MR0tight} remains true, and so does \eqref{eq:tailEstlargestVolgeninB(r)} upon removing the condition on $|K|_{\rr}$, that is, for all $x\in{G}$, $t \in (0, c]$, $r>0$ and $R\geq C\rr/t$,
\begin{equation}\label{eq:tailEstlargestCapinB(r)}
\Q \Big(\exists\text{ a cluster }K\text{ in }\mathcal L_{\tilde{B}(x,R)}\text{ such that } \mathrm{cap}(K)\geq (tR)^{\nu}  \Big)
\geq 1-\exp\{-ct^{-\alpha}\}.
\end{equation}
\end{Prop}

\begin{proof}[Proof (sketch)]
We first explain how \eqref{eq:MR0tight} follows from \eqref{eq:tailEstlargestVolgeninB(r)}. The upper bound on $\LV_{R,r}^{0}$ follows from \eqref{eq:upperMR0} and Markov's inequality, which only requires $\sup q<\infty$. The lower bound is a direct consequence of \eqref{eq:tailEstlargestVolgeninB(r)} and the isomorphism \eqref{eq:isoloopK}.

We now sketch the proof of \eqref{eq:tailEstlargestVolgeninB(r)}, which is obtained by adapting the proofs of \cite[Proposition~4.1 and Theorem~2.5]{DrePreRod10}; we only indicate the necessary changes. %
\begin{itemize}
\item \emph{Renormalized volume.} Cardinalities are replaced throughout by the renormalized volume $|\cdot|_{\rr}$ from \eqref{eq:vol}, and sums over the vertices of $B(x,R)$ by sums over $B(x,R)\cap \Lambda_{\rr}$, see \eqref{eq:defLambdaL}. Accordingly, the event from \cite[(4.2)]{DrePreRod10} is replaced, for $y\in{\Lambda_{\rr}}$, by the event that there exists $z\in{B(y,\rr)}$ with $\mathrm{cap}(\K^0_{\tilde{B}(x,R)}(z))\geq s$; the bound \cite[(2.18)]{DrePreRod10} is replaced by \eqref{eq:expectedvolumeinballren}, and \cite[Proposition~2.4]{DrePreRod10} by \eqref{eq:upperMR0}.
\item \emph{Capacity of local clusters.} The input \cite[Lemma~3.1]{DrePreRod10}, which bounds from below the probability that a cluster inside a ball has large capacity in terms of the function $q$, is replaced by the assumption \eqref{eq:pPos}, which directly supplies the corresponding lower bound for the events in the first bullet point. This is why the constants below depend on $\underline{p}$; recall that \eqref{eq:pPos} is verified under $\sup q<\infty$ in Corollary~\ref{cor:supqimpliessupp}.
\item \emph{Second moment.} In the upper bound on the second moment, cf.~\cite[(4.10)--(4.12)]{DrePreRod10}, the probability that $B(y,2\rr)$ and $B(z,2\rr)$ are connected, $y\neq z\in{\Lambda_{\rr}}$, is bounded by $C(\rr/d(y,z))^{\nu}$ using Proposition~\ref{pro:boxtobox} and \eqref{eq:hittingvscap}, and each of the two capacity events has probability at most $C\rr^{\frac\nu2}s^{-\frac12}$ by \eqref{eq:capBallBd} and \eqref{eq:crossingannuli} (with $N\to\infty$). This requires $s\geq (\underline{p}\rr)^{\nu}$ and $cs^{1/\nu}\geq \rr$. %
Actually, only the case $K=\emptyset$ (no killing) of \cite[Section~4]{DrePreRod10} is needed, i.e.~the free field on $\tilde{\G}$ rather than on $\tilde{\G}_K$.
\end{itemize}
With these changes, the argument of \cite[Section~4]{DrePreRod10}, which relies on the isomorphism between the free field and random interlacements \cite[(3.20)]{DrePreRod10} at level $u=a^2/2$, yields the following analogue of \cite[Proposition~4.1]{DrePreRod10}: assuming $\sup q<\infty$ and \eqref{eq:pPos}, for all choices of parameters $a\in{(0,1]}$, $R\geq  C $, and $r\in{[0,cR]}$  with
\begin{equation} \label{eq:Rcond}
a \ge \sqrt{2} (\underline pR/4)^{-\frac{\nu}{2}}
\end{equation}
we have for all $x\in G$ that
\begin{align} \label{eq:LBL2} 
\begin{split}
        &\P \Big(\exists\, y\in B(x,R):\, \big| {\K}^{-a}_{\tilde B(x,R)}(y,\rr)\big|_{\rr} \ge  c \underline{p} a\rr^\frac{\nu}{2} \big( {R}/{\rr}\big)^\alpha,\,
  \mathrm{cap}\big({\K}^{-a}_{\tilde B(x,R)}(y,\rr)\big)\geq cR^{\nu}\Big)  
        \ge 
   c \underline{p}^{2},
\end{split}
\end{align}
where $\K_{\tilde{B}(x,R)}^{-a}(y,\bar{r})$ is defined similarly as in \eqref{eq:clustersDef} but for $E^{\geq -a}$ instead of $E^{\geq 0}$.
One then deduces \eqref{eq:tailEstlargestVolgeninB(r)} as follows. Applying \eqref{eq:LBL2} at scale $\eps R$ with $a=(c\eps R)^{-\nu/2}$, the entropy bound \cite[(3.7)]{DrePreRod10} with $b=a$ and $A=\tilde{B}(x,\eps R)$ transfers the resulting cluster to level $0$, at the cost of a multiplicative constant, see \eqref{eq:capBallBd}. On the event that no cluster of $\{\varphi\geq0\}$ intersects both $\tilde{B}(x,\eps R)$ and $\tilde{B}(x,R)^{\mathsf c}$, whose complement has probability at most $C\eps^{\frac\nu2}$ by \eqref{eq:crossingannuli} (with $N\to\infty$) and $\sup q<\infty$, this cluster is also a cluster of $\mathcal{L}_{\tilde{B}(x,R)}$ by the isomorphism \eqref{eq:isoloopK}. Choosing $\eps$ a small enough constant, this proves \eqref{eq:tailEstlargestVolgeninB(r)} with the right-hand side replaced by a positive constant. Finally, the exponential bound follows by considering the balls $\tilde{B}(y,tR)$, $y\in{B(x,R/2)\cap \Lambda_{3tR}}$, which are disjoint by \eqref{eq:defLambdaL}, so that the corresponding events for the loop soup are independent, similarly as below \cite[(5.2)]{DrePreRod10}.

Finally, \eqref{eq:tailEstlargestCapinB(r)} is proved in the same way, but without the need for \eqref{eq:pPos}: with the right-hand side replaced by a positive constant and at level $0$, the existence of a cluster of $\{\varphi\geq0\}\cap\tilde{B}(x,\eps R)$ with capacity at least $c(\eps R)^{\nu}$ with probability bounded away from $0$ follows from \cite[(4.6)--(4.7)]{DrePreRod10} and $\sup q<\infty$, and one then concludes as above by passing to the clusters of $\mathcal{L}_{\tilde{B}(x,R)}$ and boosting the probability via independence, as below \cite[(5.2)]{DrePreRod10}. 
\end{proof}

\subsection{Proof of Proposition~\ref{pro:boxtobox}} \label{subsec:pf-boxtobox}We broadly follow the general strategy from \cite[Section~3]{cai2024onearm} involving the formula \cite[(18)]{LuWe}, suitably adapted (see Lemma~\ref{lem:mainboxtobox} below) to yield the more general statement \eqref{eq:crossingAB}, and to work on general graphs with a non-empty obstacle set $\mathcal{O}$.

We first recall the strong Markov property for $\varphi$, see e.g.~\cite[(1.19)]{MR3492939} for details. We fix some set $K\subset\tilde{\G}$ as in \eqref{eq:K-ass}. For  $O\subset\tilde{\G}_K$ let $\mathcal{A}_O$ denote the $\sigma$-algebra $\sigma({\phi}_x,\,x\in{O})$ under $\P^K$. A (random) %
set $\K \subset \tilde{\G}_K$ is said to be \textit{compatible} if it is almost surely compact  and has finitely many connected components, and if $\{\K\subset O\}\in{\A_O}$ for any open set $O\subset\tilde{\G}.$ We also introduce the augmented filtration $\mathcal{A}_{K'}^+=\bigcap_{\eps>0}\A_{K'_\eps}$ for any compact $K'\subset\tilde{\G}_K$, where $K'_\eps$ is the set of point at distance at most $\eps$ from $K'$, for the geodesic distance on $\tilde{\G}_K$ which assigns unit length to each edge.  For any compatible set $\K\subset\tilde{\G}_K$, the $\sigma$-algebra associated to $\K$ is then
\begin{equation}
\label{eq:Markov1}
    \begin{split}
    \mathcal{A}_{\K}^+\stackrel{\text{def.}}{=}\big\{A\in{\A_{\tilde{\G}_K}}:\,A\cap\{\K\subset K'\}\in{\A_{K'}^+}\ &\text{for all compact }K'\subset\tilde{\G}_K  \text{ with } \mathring{K}' \neq \emptyset\big\}.
        \end{split}
\end{equation}
Note that all the above notation depends implicitly on the choice of the set $K$ satisfying \eqref{eq:K-ass} (possibly with $K=\emptyset$). Its choice will always be clear from context. 
The strong Markov property on the graph $\G_K$ then asserts that for any such set $K$ and any random compatible (not necessarily connected) set $\K \subset \tilde{\G}_K,$ 
\begin{equation}
\label{eq:Markov2}
(\phi_x-\eta_{\K}^{K}(x))_{x\in{\tilde{\G}_{\K\cup K}}}\text{ under }\PK(\, \cdot\,|\,\mathcal{A}_{\K}^+)\text{ has the same law as $(\phi_x)_{x\in{\tilde{\G}_{\K\cup K}}}$ under } \PKK{K\cup\K};
\end{equation}
where, for a compact set $K'\subset\tilde{\G}_K$, we write (recall the notation $P_x^K$ from Section~\ref{subsec:G_K})
\begin{equation}
\label{eq:defeta}
    \eta_{K'}^{K}(x)=\sum_{y\in{{\partial K'}}}\PPK_x(X_{H_{K'}}=y,H_{K'}<\infty)\phi_y=\sum_{y\in{{\partial K'}}}P_x(X_{H_{K'}}=y,H_{K'}<H_K)\phi_y
\end{equation}
for the harmonic extension of $\phi$ on $K'$ for the graph $\G_K$. We sometimes abbreviate $\eta_{K'}=\eta_{K'}^{\emptyset}$ (e.g.~in Appendix~\ref{App:lowMink}). Note that $\eta_{K'}^K=\phi$ on $K'$, and that the sum in \eqref{eq:defeta} is finite since $K'$ has finitely many connected components. 

For the purpose of proving Proposition~\ref{pro:boxtobox},  $\mathcal{K}$ will be a subset of $\mathcal{C}_{K}^-$, for some (compact) set $K\subset\tilde{\G}$ satisfying \eqref{eq:K-ass}, where
\begin{equation}
\label{eq:Cn-}
\begin{split} 
\mathcal{C}_{K}^-\stackrel{{\rm def.}}{=} K \cup \big\{y\in{\tilde{\G}}:\,y\leftrightarrow \partial K\text{ in }\{z\in{\tilde{\G}}:\phi_z\leq 0\}\big\},
\end{split}
\end{equation}
consisting of $K$ and the closure of negative sign clusters of $\phi$ intersecting it. The key to the proof of Proposition~\ref{pro:boxtobox} is the following:

\begin{Lemme}
\label{lem:mainboxtobox}
For all sets $\mathcal{O}\subset\tilde{\G}$ and $\emptyset\neq A,A'\subset\tilde{\G}_{\mathcal{O}}$ as in \eqref{eq:K-ass},  and for all $R,r\geq C$ and $R'\geq Cr$ such that $U'\subset B(A,4R')\setminus B(A,R'/3)$, $\delta(A)\leq 3r$, and $\mathcal{O}\subset B(A,cR')$, abbreviating
\begin{equation}
\label{eq:defcalH}
\begin{split} 
\mathcal{H}_{A,A'}^{\mathcal{O}}(r,R)\stackrel{\text{def.}}{=} \sum_{x\in{ U\setminus \mathcal{C}_{A}^-}} \lambda_x^{\mathcal{O}}P_x^{\mathcal{O}}\big(\tilde{H}_{U'}< 
\tilde{H}_{U}\big)\eta_{\mathcal{C}_A^-}^{{ \mathcal{O}}}(x)
\end{split}
\end{equation}
(see above~\eqref{eq:sweeping-gen} for the notation $\tilde{H}_{\cdot}$ and  below \eqref{eq:defeta} for $\eta$), we have
\begin{equation}
\label{eq:mainboxtobox}
\begin{split} 
\text{$\P^{\mathcal{O}}$-a.s., } \quad \P^{\mathcal{O}}\big(A\leftrightarrow A'\,|\,\mathcal{A}_{\mathcal{C}_A^-}^+\big)\leq CR^{-\frac\nu2}q(R)\mathcal{H}_{A,A'}^{\mathcal{O}}(r,R). %
\end{split}
\end{equation}
\end{Lemme}

We postpone the proof of Lemma~\ref{lem:mainboxtobox} for a few lines. In view of \eqref{eq:mainboxtobox}, we need a bound on the expectation of $\mathcal{H}_{A,A'}^{\mathcal{O}}(r,R)$.

\begin{Lemme}
\label{lem:HbarvsH}
For all sets $\mathcal{O}\subset\tilde{\G}$ and $\emptyset\neq A,A'\subset\tilde{\G}_{\mathcal{O}}$ as in \eqref{eq:K-ass},  and for all $R,R',r\geq C$ such that $(U'\cup\mathcal{O})\subset B(A,Cr)^{\mathsf{c}}$, and $\delta(A)\leq r$, we have
\begin{equation}
\label{eq:expH}
\begin{split}
    \E^{ \mathcal{O}}\big[\mathcal{H}_{A,A'}^{\mathcal{O}}(r,R)\big]\leq Cr^{\frac\nu2}q(r)\inf_{y\in{U}}P_y\big(H_{U'}<H_{\mathcal{O}}\big).
\end{split}
\end{equation}
\end{Lemme}

With both Lemmas~\ref{lem:mainboxtobox} and~\ref{lem:HbarvsH} at hand, we can readily finish the: 
\begin{proof}[Proof of Proposition~\ref{pro:boxtobox}]
When $r,R\geq C$, \eqref{eq:crossingAB} is a direct consequence of \eqref{eq:mainboxtobox} and \eqref{eq:expH}, and actually \eqref{eq:crossingAB} holds under the (weaker) assumption that $U'=B(A',R) \subset B(A,3R')\setminus B(A,R'/3)$. The case $r\leq C$ is easily deduced from \eqref{eq:crossingAB} for $r=C$ (which also implies $\delta(A) \leq C$), using monotonicity of the connection event in $A$ and appropriate union bounds. The case $R\leq C$ is handled similarly, upon changing the condition $U'\subset B(A,3R')\setminus B(A,R'/3)$ to $U'\subset B(A,2R')\setminus B(A,R'/2)$ (recall that $R'\geq C$), and noting that $P_z(H_{B(A',C)}<H_{\mathcal{O}})\leq CP_z(H_{U'}<H_{\mathcal{O}})$ for all $z\in{U}$ by combining the strong Markov property, \eqref{eq:hittingvscap} and \eqref{eq:capBallBd}.
\end{proof}

The task is now reduced to Lemmas~\ref{lem:mainboxtobox} and~\ref{lem:HbarvsH}. The following will be used in both proofs.
\begin{Lemme}
For any $K\subset\tilde{\G}$, $K'\subset\tilde{\G}_{K}$ both compact with finitely many connected components, any random compatible set $\K$ with $ {K'} \subset  \K\subset\mathcal{C}_{K'}^-$, we have (see \eqref{eq:defeta} for notation)
\label{lem:Cn-onearm}
\begin{equation}
\label{eq:expectationphiCA-}
\begin{split} 
\E^{K}\big[\eta_{\K}^K(z)1\{z\notin{\K}\}\big]\leq C\P(z\leftrightarrow K'), \quad z\in{\G_K\setminus K'}.
\end{split}
\end{equation}
\end{Lemme}
\begin{proof}
    By the Markov property \eqref{eq:Markov2} of the Gaussian free field, we have 
\begin{equation*}
\begin{split} 
    \E^{K}\big[\eta_{\K}^K(z)1\{z\notin{\K}\}\big]=\E^{K}\big[\E^{K}\big[\eta_{\K}^K(z)\,|\,\mathcal{A}_{\K}^+\big]1\{z\notin{\K}\}\big]=\E^{K}[\phi_z1\{z\notin{\K}\}],
\end{split}
\end{equation*}
where the last equality uses that $\EKK{K}[(\eta_{\K}^K(z)-\varphi_z) \,|\,\mathcal{A}_{\K}^+]= \EKK{K\cup\K}[\varphi_z]=0$.
Furthermore, since $\phi$ is centered under $\P^{K}$, the last term in the previous equation is equal to
\begin{equation*}
 \E^{K}[\phi_z] -\E^{K}\big[\phi_z1\{z\in{\K}\}\big] =-\E^{K}\big[\phi_z1\{z\in{\K}\}\big]\leq -\E^K\big[\phi_z1\{z\in{\mathcal{C}_{K'}^-}\}\big]=\E^K\big[\phi_z1\{z\leftrightarrow K'\}\big],
\end{equation*}
where to obtain the inequality we used that $\phi_z\leq 0$ for  $z\in (\mathcal{C}_{K'}^- { \setminus K'})\supset ({\K}{ \setminus K'})$, %
and where we took advantage of the symmetry of $\phi$ in the last equality. This last term can be bounded from above by $\E[\phi_z1\{z\leftrightarrow K'\}]$ due to the monotonicity of the loop soup and \eqref{eq:isoloopK}, and we can therefore conclude using \cite[Lemma~3.3]{cai2024onearm}. Note that this last result is stated on $\Z^d$, but its proof actually works on any graph $\G$ with bounded Green function and bounded weights, which is the case in view of \eqref{eq:intro_Green} and \eqref{eq:lambdabounded}.
\end{proof}

We now prove Lemmas~\ref{lem:mainboxtobox} and~\ref{lem:HbarvsH}, starting with the latter, which is simpler.
\begin{proof}[Proof of Lemma~\ref{lem:HbarvsH}] We will first argue under the assumptions of the lemma that
\begin{equation}
\label{eq:expH2}
\begin{split} 
    P_x^{\mathcal{O}}(\tilde{H}_{U'}< \tilde{H}_{U})\leq Ce_{U}(x)\inf_{y\in{U}}P_y\big(H_{U'}<H_{\mathcal{O}}\big)\text{ for all }x\in{\partial U},
\end{split}
\end{equation}
with $e_U= e_{U,\tilde{\G}}$, see above \eqref{eq:cap}. For $t\geq1$ to be chosen, let $U_t= B(A,tr)$, so that $U=U_1 \subset U_t$. For any $x\in{\partial U}$, by the strong Markov property at time $H_{\partial U_t}$ we have if $(U'\cup\mathcal{O})\subset (U_t)^{\mathsf{c}}$ 
\begin{equation}
\label{eq:fifthPOAA'}
\begin{split} 
P_x^{\mathcal{O}}\big(\tilde{H}_{U'}< \tilde{H}_{U}\big)\leq P_x\big(\tilde{H}_{\partial U_t}\leq \tilde{H}_{U}\big)\sup_{y\in{U_t}}P_{y}(H_{U'}<H_{\mathcal{O}}),
\end{split}
\end{equation}
where we used that $X_{\cdot}$ under $P_y^{\mathcal{O}}$ has the same law as $X_{\cdot \wedge H_{\mathcal{O}}}$ under $P_y$. By the Harnack inequality \eqref{eq:Harnack}, one can replace the supremum on the right-hand side of \eqref{eq:fifthPOAA'} by an infimum as long as $(U'\cup\mathcal{O})\subset B(A,Ctr)^{\mathsf{c}}$. Moreover, by the strong Markov property at time $\tau= \tilde{H}_{\partial U_t}$,\begin{multline*}
P_x(\tau\leq \tilde{H}_{U}) \leq P_x(\tilde{H}_{U}=\infty)+\sum_{y\in{\partial {U_t}}}P_x\big(Z_{\tau}=y,\tau\leq \tilde{H}_{U}\big)P_y(\tilde{H}_{U}<\infty)
\\\leq \frac{e_{U}(x)}{\lambda_x}+\frac12P_x(\tau\leq \tilde{H}_{U}),
\end{multline*}
where we used \cite[(1.57)]{MR2932978}, \cite[(2.8)]{DrePreRod2} and \eqref{eq:capBallBd} in the last inequality to bound $P_y(\tilde{H}_{U}<\infty)\leq Ct^{-\nu}\leq 1/2$ for all $y\in{\partial{U_t}}$ upon fixing $t=C$ for a large enough constant $C$. Rearranging terms in the last display, noting that $\lambda_x^{\mathcal{O}}=\lambda_x$ for all $x\in{U}$ if $\mathcal{O}\subset B(A,Cr)^{\mathsf{c}}$, and combining it with \eqref{eq:fifthPOAA'} gives \eqref{eq:expH2}. To deduce \eqref{eq:expH}, one feeds \eqref{eq:expH2} into \eqref{eq:defcalH} (note that only $x\in{\partial U}$ contribute to the sum in \eqref{eq:defcalH}, since $P_x^{\mathcal{O}}(\tilde{H}_{U'}<\tilde{H}_U)=0$ otherwise), uses \eqref{eq:expectationphiCA-} with $\mathcal{K}=\mathcal{C}_A^-$, $K'=A$, $K=\mathcal{O}$ to bound the resulting expectation of harmonic average in terms of a connection probability, and combines the resulting estimate with \eqref{eq:qPsiBd} (recall here that $d(x,A)\geq cr$ for $x\in{\partial U}$) and \eqref{eq:capBallBd} to bound $\sum_{x\in{\partial U}}e_U(x)=\mathrm{cap}(U)\leq Cr^{\nu}$, thus concluding the proof.
\end{proof}

\begin{proof}[Proof of Lemma~\ref{lem:mainboxtobox}]
Let us abbreviate $\mathcal{C}^- = \mathcal{C}_{A\cup A'}^-$, see \eqref{eq:Cn-} for notation. Conditionally on $\mathcal{A}_{\mathcal{C}^-}^+$, the event $\{A\leftrightarrow A'\}$ only depends on the field in the complement of $\mathcal{C}^-$, which by \eqref{eq:Markov2} is a free field with boundary values $\phi$ on $\partial\mathcal{C}^-$. The set $\mathcal{C}^-$ is $\P^{\mathcal{O}}$-a.s.~bounded, because $A \cup A'$ is compact, so the number of clusters of $\{\varphi \leq 0 \}$ intersecting $\partial(A\cup A')$ is finite, and each of these clusters is bounded, as implied by our standing assumptions on $\G$ (combining \cite[Corollary 3.3(1) and Lemma 3.4(2)]{DrePreRod3}, which together imply that $\G^{\mathcal{O}}$ satisfies the relevant condition (Sign)). In particular, $\partial \mathcal{C}^-$ is a finite set.

We first extend a formula from \cite{LuWe} valid on finite graphs to our (infinite volume) setup. Let ${\G}^-=(G^-,\lambda^-)$ denote the random graph under $\P^{\mathcal{O}}$ formally defined as $(\G_{\mathcal{O}})^{\partial \mathcal{C}^-}$ in \cite[Lemma~2.1]{DrePreRod3}. In words, $\G^{-}$ is a graph with vertex set $G^-=\partial \mathcal{C}^-\cup G_{\mathcal{O}}$ (see Section~\ref{subsec:G_K} regarding $G_{\mathcal{O}}$) and with weights $\lambda^-$ on $G^-$ chosen so that the canonical discrete random walk $(Z_n)_{n\in{\N}}$ on $\G^-$, the law of which we denote by $P_x^-$ when started in $x \in G^-$, has the same law as the successive visits to $G^-$ by the skeleton of the diffusion $X$ under $P_x^{\mathcal{O}}$. 

For $K \subset \tilde{\G}$ compact (and possibly random) with $\partial K \subset G^-$, let
$$
h_K(x,y)= \lambda_x^- P_x^- (Z_{\tilde{H}_K} =y, \,\tilde{H}_K< \infty), \quad x,y \in G^-,
$$
where recall that $ \tilde{H}_K=\inf\{n\geq1:\,Z_n\in{K}\}$,  cf.~above Lemma~\ref{L:sweeping-extended}. By \cite[Proposition~1.11]{MR2932978}, $h_K(x,y)=h_K(y,x)$ for $x \neq y$, $x,y \in K$.
By network equivalence, we infer that the effective conductance (as defined in \cite[Section~2.2]{LuWe}) between  $x,y \in \partial\mathcal{C}^-$ with $x \ne y$, for the graph $\G^-$ with boundary set $\partial \mathcal{C}^-$ is precisely equal to $h_{\mathcal{C}^-}(x,y)$. Applying \cite[(18)]{LuWe} on that graph with $\hat{A}=\partial\mathcal{C}_A^-$ and $\widecheck{A}=\partial\mathcal{C}_{A'}^-\setminus \mathcal{C}_A^-$ therein, we thus have in view of \eqref{eq:Cn-} that 
\begin{equation}
\label{eq:firstPOAA'}
\begin{split} 
\P^{\mathcal{O}}\big(A\leftrightarrow A'\,|\,\mathcal{A}_{\mathcal{C}^-}^+\big)&=\P^{\mathcal{O}}\big(\partial\mathcal{C}^-_A\leftrightarrow (\partial\mathcal{C}_{A'}^-\setminus \mathcal{C}_A^-)\,|\,\mathcal{A}_{\mathcal{C}^-}^+\big)
= 1-\exp\Big\{-2\sum_{x,y}h_{\mathcal{C}^-}(x,y) \phi_x\phi_y\Big\},
\end{split}
\end{equation}
with the sums ranging over $x \in{\partial\mathcal{C}_A^-}$, $y\in{\partial\mathcal{C}_{A'}^-\setminus \mathcal{C}_A^-}$;
actually, \cite[(18)]{LuWe}, and hence the deduction of \eqref{eq:firstPOAA'}, is only proved therein for finite graphs, but \eqref{eq:firstPOAA'} can easily be extended to any infinite graph for which $\mathcal{C}^-$ is bounded, as is the case here (see the beginning of this proof), by defining  $\G_n^-$ as the  finite graph which is the intersection of $\G^{-}$ with a ball of size $n$, with infinite killing on the boundary of that ball, and noting that the event $A\leftrightarrow A'$ for the graph $\G_n^-$ is increasing in $n$ toward the event $A\leftrightarrow A'$ for the graph $\G^-$ in view of \eqref{eq:isoloopK}. Note that $\phi_x=0$ for all $x\in{\partial\mathcal{C}_A^-\setminus A}$ and $\phi_y=0$ for all $y\in{\partial\mathcal{C}_{A'}^-\setminus A'}$, and that $\phi\geq0$ on $\partial\mathcal{C}^-$. 
Combining this with \eqref{eq:firstPOAA'}, the inequality $1-e^{-t}\leq 1\wedge t$  and the tower property, we thus obtain, with $\P^{\mathcal{O}}(\cdot\,|\,{\mathcal{C}_{A}^-})=\P^{\mathcal{O}}(\cdot \,|\,\mathcal{A}_{\mathcal{C}_{A}^-}^+)$, that
\begin{equation}
\label{eq:secondPOAA'}
\begin{split} 
\P^{\mathcal{O}}\big(A\leftrightarrow A'\,|\, \mathcal{C}_{A}^-\big)\leq \E^{\mathcal{O}}\Big[1\wedge\Big(2 \sum_{x,y}h_{\mathcal{C}^-}(x,y)\phi_x\phi_y\Big)\,\Big|\, \mathcal{C}_{A}^-\Big],
\end{split}
\end{equation}
with the sums now ranging over $x\in{\partial\mathcal{C}_A^-\cap A}$ and $y\in{(\partial\mathcal{C}_{A'}^-\cap A'})\setminus\mathcal{C}_A^-$.

Let us briefly outline the remainder of the proof. In order to make the quantity $\mathcal{H}_{A,A'}^{\mathcal{O}}(r,R)$ from \eqref{eq:defcalH} appear out of the expectation in \eqref{eq:secondPOAA'}, we ``push out'' the sums over $x,y$ to sums over $z \in U(=B(A,r))$ and $ z' \in U' (=B(A'R))$ using last-exit calculations. This will transform $\phi_x$, $\phi_y$ into harmonic averages at $z,z'$, one of which appears plainly in \eqref{eq:defcalH}. The other one is hiding behind $q(R)$ in \eqref{eq:mainboxtobox} via an application of Lemma~\ref{lem:Cn-onearm}.

We now proceed to implement this. We assume that $r,R\geq C$ so that $\partial U\cap A=\emptyset$ and $\partial U'\cap A'=\emptyset$. Abbreviate $\tilde{H}= \tilde{H}_{\mathcal{C}^-}$, and let $L=\sup\{k\in{[0,\tilde{H}]}:\,Z_k\in U\}$. %
For $x,y$ as in \eqref{eq:secondPOAA'}, notice that $L<\infty$ under $P_x^-$ on the event $\{Z_{\tilde{H}}=y, \,  \tilde{H}< \infty\}$. Thus, for such $x,y$ and all $z\in \partial U$, $k \geq 0$, applying the simple Markov property at time $k$, we get
\begin{equation*}
 P_x^-(L=k,Z_k=z,Z_{\tilde{H}}=y,  \tilde{H} < \infty )=  P_x^-(Z_k=z,\tilde{H}> k) \, P_z^-(\tilde{H}_{U} >  \tilde{H},Z_{\tilde{H}}=y)
\end{equation*}
(note that $1\leq L<\tilde{H}$ on the event on the left-hand side, since $x\in{A}$, $\partial U\cap A=\emptyset$ and $y\notin{U}$; in particular the terms with $z\in{\mathcal{C}^-}$ vanish).
By reversibility, the first term on the right-hand side equals $\frac{\lambda_z^-}{\lambda_x^-}P_z^-(Z_k^-=x,\tilde{H}= k)$. Using this, multiplying by $\lambda_x^-$ on both sides, and noting that the sum over $k \geq 0$ and $z \in \partial U$ (even only over $z \in U\setminus (\mathcal{C}_A^-\cup\mathcal{O})$) on the left yields $h_{\mathcal{C}^-}(x,y)$, it follows that
\begin{equation*}
\begin{split} 
h_{\mathcal{C}^-}(x,y) &=\sum_{z\in{\partial U\setminus (\mathcal{C}_A^-\cup\mathcal{O})}} h_{\mathcal{C}^-}(z,x) P_z^-(\tilde{H}_{U}> \tilde{H},Z_{\tilde{H}}=y)
\\&\leq\sum_{z\in{\partial U\setminus (\mathcal{C}_A^-\cup\mathcal{O})}} \sum_{z'\in{\partial U' \setminus \mathcal{C}^-
}} h_{\mathcal{C}^-}(z,x) P_z^-(Z_{\tilde{H}_{U'}}=z',\tilde{H}_{U'}< \tilde{H}_{U})(\lambda_{z'}^-)^{-1} h_{\mathcal{C}^-}(z',y),
\end{split}
\end{equation*}
where we used the strong Markov property at time $\tilde{H}_{U'}$. Next, observe for $z,z'$ as above that
$$ \lambda_z^-P_z^-(Z_{\tilde{H}_{U'}}=z',\tilde{H}_{U'} <  \tilde{H}_{U})= \lambda_z^{\mathcal{O}}P_z^{\mathcal{O}}(Z_{\tilde{H}_{U'}}=z',\tilde{H}_{U'}< \tilde{H}_{U});
$$
to see this, one recalls that $Z$ under $P_z^-$ has the same law as the sequence of vertices in $G^-= \partial\mathcal{C}^-\cup \partial\mathcal{O}\cup G$ visited by the diffusion $X$ under $P_z$, stopped on hitting $\mathcal{O}$, and applies a similar reasoning as above separately for the walk $Z$ under $P_z^-$ and under $P_z^{\mathcal{O}}$ to infer that either side equals $P_z^{\mathcal{O}}(X_{L'}=z,X_{H_{U'}}=z')/g_{\tilde{\G}_{\mathcal{O}}}(z,z)$ where $L'=\sup\{t\in{[0,H_{U'}]}:\,X_t\in{U}\}$. Moreover, for all $x\in{A}$ and $y\in{A'\setminus \mathcal{C}_A^-}$, by monotonicity one bounds, abbreviating $\mathcal{O}'=\mathcal{O}\cup\mathcal{C}_A^-$,
\begin{align*}
&{ h_{\mathcal{C}^-}(z,x)}/{\lambda_{z}^-} \leq P_z^{ \mathcal{O}}(X_{H_{\mathcal{C}_A^-}}=x, H_{\mathcal{C}_A^-} < \infty),\\
&  {h_{\mathcal{C}^-}(z',y)}/{\lambda_{z'}^-} \leq P_{z'}^{\mathcal{O}'}\big(X_{H_{\mathcal{C}_{A'}^-\setminus \mathcal{C}_A^-}}=y, H_{\mathcal{C}_{A'}^-\setminus \mathcal{C}_A^-} < \infty\big).
\end{align*}
Putting things together, plugging the resulting estimate for $h_{\mathcal{C}^-}(x,y)$ into \eqref{eq:secondPOAA'}, summing first over $x\in{\partial \mathcal{C}_A^-\cap A}$ and $y\in{(\partial \mathcal{C}_{A'}^-\cap A')\setminus \mathcal{C}_A^-}$ (or equivalently over $x\in{\partial \mathcal{C}_A^-}$ and $y\in{\partial \mathcal{C}_{A'}^-\setminus \mathcal{C}_A^-}$), using Jensen's inequality %
, and recalling \eqref{eq:defeta}, one further bounds $\P^{\mathcal{O}}\big(A\leftrightarrow A'\,|\, \mathcal{C}_{A}^-\big)$ by 
\begin{equation}
\label{eq:thirdPOAA'}
\begin{split} 
1\wedge 
2\sum_{{z\in{\partial U\setminus(\mathcal{C}_A^-\cup\mathcal{O})}}}  \sum_{{z'\in{\partial U'}}}   \lambda_z^{\mathcal{O}}\eta_{\mathcal{C}_A^-}^{ \mathcal{O} }(z)P_z^{\mathcal{O}}(Z_{\tilde{H}_{U'}}=z',\tilde{H}_{U'}< \tilde{H}_{U}) \bar{\eta}(z'), 
\end{split}
\end{equation}
where
$$
 \bar{\eta}(z') \stackrel{\text{def.}}{=} \E^{\mathcal{O}}\big[\eta_{\bar{\mathcal{C}}}^{\mathcal{O}'}(z')1\{z'\notin{\mathcal{C}^-}\}\,\big|\, \mathcal{C}_{A}^-\big], \quad \bar{\mathcal C} \stackrel{\text{def.}}{=} \mathcal{C}_{A'}^-\setminus \mathcal{C}_A^-.
$$
We now bound $\bar{\eta}$. For any $f:\tilde{\G}\rightarrow\R$, let $\K_f$ be the union of $A'\setminus \mathcal{C}_A^-$ and of the clusters of $\{z\in{\tilde{\G}_{\mathcal{O}'}}:\,\phi_z\leq -f(z)\}$ which intersect its boundary, viewed under $\P^{\mathcal{O}'}$, and abbreviate $\K=\K_{f}%
$ where $f= \eta_{\mathcal{C}_A^-}^{ \mathcal{O} }$;  here both $\mathcal{C}_A^-$ and $f$ are viewed as frozen under the probability $\P^{\mathcal{O}}( \cdot \, |\, \mathcal{C}_{A}^-)$. 
By \eqref{eq:Markov2} the set $\bar{\mathcal{C}}$ appearing in $\bar{\eta}$ has the same law under $\P^{\mathcal{O}}(\cdot\,|\, \mathcal{C}_{A}^-)$ as $\K$ under $\P^{\mathcal{O}'}$, and thus still using \eqref{eq:Markov2} we have for all $z'\in{\partial U'}$, $\P^{\mathcal{O}}(\cdot\,|\, \mathcal{C}_{A}^-)$-a.s., writing $\phi=\phi'+f$ on $\tilde{\G}_{\mathcal{O}'}$, where $\phi'$ has law $\P^{\mathcal{O}'}$, using linearity of the harmonic extension and that $1\{z'\notin{\mathcal{C}^-}\}=1\{z'\notin{\mathcal{C}_A^-}\}1\{z'\notin{\bar{\mathcal{C}}}\}$, $\bar{\eta}(z')$ is equal to %
\begin{equation}\label{eq:etabar2terms}
\begin{split}
&\Big(\E^{\mathcal{O}'}\big[\eta_{\K}^{\mathcal{O}'}(z')1\{z'\notin{\K}\}\big]\\
&\qquad +\E^{\mathcal{O}'}\big[1\{z'\notin{\K}\}E_{z'}^{\mathcal{O}'}[f(X_{H_{\K}})1\{ H_{\K} < \infty\}] \big]\Big)1\{z'\notin{\mathcal{C}_{A}^-}\}.
\end{split}
\end{equation}
We consider each term individually. The first term in \eqref{eq:etabar2terms} is bounded  by
 $C\P^{\mathcal{O}'}(z'\leftrightarrow A')$ using \eqref{eq:expectationphiCA-}, applied with $K=\mathcal{O}'$ and $K'=A'\setminus \mathcal{C}_A^- \subset \mathcal{K}$ (noting that $ \K\subset\mathcal{C}_{K'}^-$  under $\P^{\mathcal{O}'}$ since $f\geq 0$ on $\tilde{\G}_{\mathcal{O}'}$). Moreover, $\P^{\mathcal{O}'}(z'\leftrightarrow A')\leq \P(z'\leftrightarrow A')\leq q(R)R^{-\frac\nu2}$ for all $z'\in{\partial U'}$ by \eqref{eq:isoloopK} and \eqref{eq:qPsiBd}, so altogether the first term is bounded by $Cq(R)R^{-\frac\nu2}$.

The second term in \eqref{eq:etabar2terms} is bounded by $f(z')$: indeed, for every realization of $\K$ with $z'\notin{\K}$, applying the strong Markov property at time $H_{\K}$ to $f(z')=E_{z'}^{\mathcal{O}}[\phi_{X_{H_{\mathcal{C}_A^-}}}1\{H_{\mathcal{C}_A^-}<\infty\}]$, the sum inside the expectation is the contribution of the event $\{H_{\K}<H_{\mathcal{C}_A^-}\}$ to $f(z')$, and the remaining contribution is non-negative since $\phi\geq0$ on $\partial\mathcal{C}_A^-$. Now, abbreviating
 $\bar{e}_{U'}(\cdot)=\text{cap}_{\tilde{\G}_{\mathcal{O}}}(U')^{-1}{e}_{U',\tilde{\G}_{\mathcal{O}}}(\cdot)$, as we now explain, for all $z'\in{\partial U'}$,
\begin{equation*}
\begin{split} 
f(z')&=\sum_{y\in{A}}P_{z'}^{\mathcal{O}}(X_{H_{\mathcal{C}_A^-}}=y, H_{\mathcal{C}_A^-}< \infty)\phi_y\\
& \leq \sum_{z\in{ U}\setminus \mathcal{C}_A^-}P_{z'}^{\mathcal{O}}(X_{H_{U}}=z, H_U < \infty)f(z) \leq C\sum_{z\in{ U}\setminus \mathcal{C}_A^-}P_{\bar{e}_{U'}}^{\mathcal{O}}(X_{H_{U}}=z, H_U < \infty)f(z)
\\&\leq CR^{-\nu}\sum_{z\in{ U\setminus\mathcal{C}_A^-}}\lambda_z^{\mathcal{O}}P_z^{\mathcal{O}}(L_{U'}< 
 \tilde{H}_{U},L_{ U'}\geq0)f(z)
 \leq CR^{-\nu}\mathcal{H}_{A,A'}^{\mathcal{O}}(r,R).
\end{split}
\end{equation*}
Here, the first inequality follows from the strong Markov property at time $H_U$, noting that $\phi_y=0$ for all $y\in{\partial\mathcal{C}_A^-\setminus A}$, that $X_{H_{\mathcal{C}_A^-}}\in{A}$ implies $H_U\leq H_{\mathcal{C}_A^-}$ and, since $\partial U\cap A=\emptyset$, that $H_U<H_{\mathcal{C}_A^-}$ in this case unless $X_{H_U}\in{\partial\mathcal{C}_A^-\setminus A}$, and that $f\geq0$. For the second inequality, we apply the Harnack inequality \eqref{eq:Harnack} (in its annulus version) to the non-negative function $w\mapsto P_{w}^{\mathcal{O}}(X_{H_{U}}=z, H_U < \infty)$, for each fixed $z$, which is deterministic and harmonic outside of $U\cup\mathcal{O}$ (contrary to $f$, which is only harmonic outside of the random set $\mathcal{C}_A^-\cup\mathcal{O}$): fixing $a\in{A}$, since $\delta(A)\leq 3r$, $R'\geq Cr$ and $\mathcal{O}\subset B(A,cR')$, we have $U\cup\mathcal{O}\subset B(a,\Cr{c:harnack}R')$, whereas $\partial U'\subset B(a,9R')\setminus B(a,R'/3)$ since $U'\subset B(A,4R')\setminus B(A,R'/3)$, and $\bar{e}_{U'}$ is a probability measure supported on $\partial U'$. Moreover, in the third inequality we used \eqref{eq:capBallBd} to bound the capacity of $U'=B(A',R)$ from below, \eqref{eq:defeta}, and the sweeping identity \eqref{eq:sweeping-gen}, applied to $\G_\mathcal{O}$ with $A=U$ and $B=\partial U'$ (noting that $L_{U'} \leq 
 \tilde{H}_{U}$ implies the strict inequality); the last inequality follows from the fact that $\tilde{H}_{U'}\leq L_{U'}$ holds on the event $\{L_{U'}\geq0\}=\{\tilde{H}_{U'}<\infty\}$ under $P_z^{\mathcal{O}}$, $z\in{U}$, and the definition of $\mathcal{H}_{A,A'}^{\mathcal{O}}(r,R)$ in \eqref{eq:defcalH}. Returning to \eqref{eq:etabar2terms} and putting things together, using $q\geq1$, we deduce that for all $z'\in{\partial U'}$,
\begin{equation}
\label{eq:fourthPOAA'}
\begin{split} 
\bar{\eta}(z')\leq Cq(R)R^{-\frac\nu2}\big(1+q(R)R^{-\frac\nu2}{\mathcal{H}}_{A,A'}^{\mathcal{O}}(r,R)\big).
\end{split}
\end{equation}
Substituting \eqref{eq:fourthPOAA'} into \eqref{eq:thirdPOAA'}, summing over $z'$, recalling \eqref{eq:defcalH} and using the inequality $1\wedge (Cx(1+x))\leq Cx$, the bound \eqref{eq:mainboxtobox} follows.
\end{proof}

\subsection{Proof of Proposition~\ref{cor:crossingboxseccrit}} \label{subsec:pfcor4.2}The estimates for box-to-box and annulus crossings comprised in Proposition~\ref{cor:crossingboxseccrit} follow by combining Proposition~\ref{pro:boxtobox} with arguments similar to those used in the proof of the lower bound in \cite[Theorem~1.2]{cai2024onearm}.

\begin{proof}[Proof of Proposition~\ref{cor:crossingboxseccrit}]
Assuming w.l.o.g.\ that $R\geq C\rr$ for a large enough constant $C<\infty$, the upper bound in \eqref{eq:crossingannuli} for $\P$ instead of $\P^{N}$ follows from \eqref{eq:crossingAB} for $A=\tilde{B}(x,r)$, $A'=\partial\tilde{B}(x,R)$ and $\mathcal{O}=\emptyset$, which satisfy the hypotheses of Proposition~\ref{pro:boxtobox} for $R'=3R/4$ and replacing $R$ by $R/8$ and $r$ by $3\rr$ therein (since $\delta(\tilde{B}(x,r))\leq 2r+1\leq 3\rr$). If $d(x,y)\geq \Cl{C:boxtobox}(\RR\vee \rr)$ for a large enough constant $\Cr{C:boxtobox}\in{(1,\infty)}$, the upper bound in \eqref{eq:crossingboxseccrit} for $\P$ instead of $\P^{N}$ follows from \eqref{eq:crossingAB} for $A=\tilde{B}(x,r)$, $A'=\tilde{B}(y,R)$ and $\mathcal{O}=\emptyset$, which satisfy the hypotheses of Proposition~\ref{pro:boxtobox} for $R'=3d(x,y)/4$ and $3\rr$ instead of $r$, combined with the inequality $P_x(H_{B(y,2R)}<\infty)\leq CR^{\nu}d(x,y)^{-\nu}$, see \eqref{eq:capBallBd} and \eqref{eq:hittingvscap}. If $2(\RR+\rr)\leq d(x,y)\leq \Cr{C:boxtobox}(\RR\vee \rr)$, assuming w.l.o.g.\ that $R\geq r$ by symmetry, so that $\RR\geq d(x,y)/\Cr{C:boxtobox}$, we have $d(x,B(y,R))\geq d(x,y)-\RR\geq d(x,y)/2$, hence $B(x,r)\leftrightarrow B(y,R)$ implies $B(x,r)\leftrightarrow \partial B(x,d(x,y)/2)$. If $d(x,y)/2\geq C\rr$, the upper bound in \eqref{eq:crossingannuli} for $\P$, which we just proved, thus yields $\P(B(x,r)\leftrightarrow B(y,R))\leq C(\rr/d(x,y))^{\frac\nu2}\leq C\Cr{C:boxtobox}^{\frac\nu2}d(x,y)^{-\nu}(\rr\RR)^{\frac\nu2}$, and otherwise the upper bound in \eqref{eq:crossingboxseccrit} is trivial, since then $d(x,y)^{-\nu}(\rr\RR)^{\frac\nu2}\geq c$. Finally, since the connection events are increasing in the loop soup, $\P^N(B(x,r)\leftrightarrow B(y,R))\leq \P(B(x,r)\leftrightarrow B(y,R))$ by the isomorphism \eqref{eq:isoloopK}, which concludes the proof of the upper bounds.

It remains to prove the lower bound in \eqref{eq:crossingboxseccrit}.  The argument follows the same ideas as in \cite[Section~6.1]{cai2024onearm}, and we only detail the key differences. Under the probability measure $\P^{N}$, recall the exploration process $(\mathcal{I}_t^{\pm})_{t\geq0}=(\mathcal{I}_t^{\partial B(y,R),\pm})_{t\geq0}$ of the clusters of $\partial B(y,R)$ of both signs, as well as the martingale $(\mathcal{M}_t^{\pm})_{t\geq0}=(\mathcal{M}_{x,t}^{\partial B(y,R),\pm})_{t\geq0}$ introduced in \cite[Section~2.7]{cai2024onearm}, which is simply equal to $(\eta_{\mathcal{I}_t^{\pm}}^{\partial B(N)}(x))_{t\geq0}$ in view of \eqref{eq:Markov2}. We also abbreviate $P_x^N=P_x^{\partial B(N)}$ and $g_N=g_{\tilde{\G}_{\partial B(N)}}$. 
Since $\mathcal{I}_0^{\pm}=\partial B(y,R)$, we have by \eqref{eq:defeta} that
\begin{equation*}
\begin{split}
     \text{Var}_{\P^{N}}\big(\mathcal{M}^{\pm}_0\big)&=\sum_{z,z'\in{\partial B(y,R)}}P_x^N(X_{H_{B(y,R)}}=z)P_x^N(X_{H_{B(y,R)}}=z')g_N(z,z')
    \\& =\sum_{z\in{\partial B(y,R)}}P_x^N(X_{H_{B(y,R)}}=z)g_N(x,z)
    \geq cd(x,y)^{-\nu}P_x(H_{B(y,R)}<H_{\partial B(N)}),
\end{split}
\end{equation*}
where in the second equality we used Hunt's switching identity, see \cite[(1.50)]{MR2932978}, and in the last inequality we used \eqref{eq:GreenGK} and the fact that $d(x,z)\leq 2d(x,y)$. The last probability in the previous display is bounded below using \eqref{eq:capBallBd} and \eqref{eq:hittingvscap}, which are all in force if $B(x,r),B(y,R)\subset B(cN)$ for some small enough constant $c>0$, and one obtains that
\begin{equation}
\label{eq:toimproveforalleps}
P_x(H_{B(y,R)}<H_{\partial B(N)})\geq {c}\RR^{\nu}d(x,y)^{-\nu}.
\end{equation}
 Furthermore in view of \cite[(2.26)]{cai2024onearm}, on the event $\{\mathcal{I}_{\infty}^{\pm}\cap B(x,r)=\emptyset\}$
\begin{equation*}
\begin{split}
\langle \mathcal{M}^{\pm}\rangle_{\infty}=\sum_{z\in{\mathcal{I}_{\infty}^{\pm}}}P_x^N(X_{H_{\mathcal{I}_{\infty}^{\pm}}}=z)g_N(x,z)\leq C\rr^{-\nu}
\end{split}
\end{equation*}
by \eqref{eq:intro_Green} and since $g_N(x,z)\leq g(x,z)$ and $d(x,z)\geq r$. Proceeding similarly as in \cite[(6.6) and (6.7)]{cai2024onearm} one deduces from the last three equation that 
\begin{equation*}
\begin{split}
\P^{N}(\mathcal{I}_{\infty}^{\pm}\cap B(x,r)=\emptyset)&\leq \P^{N}(|X|\geq c\rr^{\frac\nu2}{|}\mathcal{M}^{\pm}_0{|})
\\&\leq \P^{N}(|XY^{-1}|\geq c(\rr \RR)^{\frac\nu2}d(x,y)^{-\nu}),
\end{split}
\end{equation*}
where $X,Y$ are independent centered Gaussian variables with variance $1$. The complement of this last probability can be lower bounded by $c(\rr \RR)^{\frac\nu2}d(x,y)^{-\nu}$ similarly as in \cite[(6.8)]{cai2024onearm} and since $(\rr \RR)^{\frac\nu2}\leq (\rr+\RR)^{\nu}\leq d(x,y)^{\nu}$ (recall that $d(x,y)\geq2(\RR+\rr)$). Since $\mathcal{I}_{\infty}^{\pm}$ is the set of points connected to $B(y,R)$ either in $E^{\geq0}$ or its complement, see \cite[Section~2.7]{cai2024onearm},  the probability that $B(x,r)\leftrightarrow B(y,R)$ is at least half the probability that $\mathcal{I}_{\infty}^{\pm}\cap B(x,r)\neq \emptyset$ by symmetry, which implies the lower bound in \eqref{eq:crossingboxseccrit} if $B(x,r),B(y,R)\subset B(cN)$ for some small enough constant $c>0$, as assume above \eqref{eq:toimproveforalleps}. %

We now fix some $\eps\in{(0,1)}$, and consider the general case where $B(x,r),B(y,R)\subset B((1-\eps)N)$. Inspecting the above argument, the assumption $B(x,r),B(y,R)\subset B(cN)$ was only used to ensure that $g_N(x,z)\geq cd(x,y)^{-\nu}$ for all $z\in{\partial B(y,R)}$, see \eqref{eq:GreenGK}, and to obtain \eqref{eq:toimproveforalleps}. It is therefore enough to show that both bounds remain true under the weaker assumption, with constants $c=c(\eps)$. Both follow from \eqref{eq:GreenGK}, \eqref{eq:capBallBd} and \eqref{eq:hittingvscap} if $d(x,y)\leq c\eps N$ for small enough $c$, and we now assume that $d(x,y)\geq c\eps N$. We first show that for a suitable constant $t\in{(0,\frac18]}$,
\begin{equation}
\label{eq:hittingballsrec}
P_{w}(H_{B(w',2t\eps N)}<H_{\partial B(N)})\geq c(\eps)\text{ for all }w,w'\in{B((1-\eps)N)}.
\end{equation}
If $t\eps N<1$, then $B(N)$ contains at most $C(\eps)$ vertices by \eqref{eq:intro_sizeball}, and \eqref{eq:hittingballsrec} follows from \eqref{eq:ellipticity} by forcing the walk to follow a path from $w$ to $w'$ in $B((1-\eps)N)$, which exists since $B((1-\eps)N)$ is connected when $d$ is the graph distance, or $\ell^{\infty}$ on $\Z^d$. Assume now that $t\eps N\geq1$, and let $\gamma=(x_0,\dots,x_n)$ be a path from $w$ to $w'$ included in $B((1-\eps)N)$. In view of \eqref{eq:defLambdaL}, there exist $y_0,\dots,y_n\in{\Lambda_{t\eps N}\cap B((1-\eps/2)N)}$ such that $x_i\in{B(y_i,t\eps N)}$ for all $0\leq i\leq n$, and thus $d(y_i,y_{i+1})\leq 2t\eps N+1\leq 3t\eps N$ for all $i<n$. Consider the (finite) graph with vertex set $\Lambda_{t\eps N}\cap B((1-\eps/2)N)$, in which two vertices are neighbors if they are at distance at most $3t\eps N$. Its number of vertices is at most $C(\eps)$ by \eqref{eq:defLambdaL}, and $(y_0,\dots,y_n)$ is a path in this graph, which can be loop-erased to obtain a self-avoiding path $v_0=y_0,v_1,\dots,v_p=y_n$ with $p\leq C(\eps)$. Fixing $t=c$ for a small enough constant $c$, so that $B(v_k,C't\eps N)\subset B(N)$ for all $k$ and a suitably large constant $C'$, it follows from \eqref{eq:capBallBd} and \eqref{eq:hittingvscap} that
$P_{z}(H_{B(v_{k+1},t\eps N)}<H_{\partial B(N)})\geq c$  for all $z\in{B(v_k,t\eps N)}, k\in\{0,\dots,p-1\}$. Applying this bound recursively together with the strong Markov property at the times $H_{B(v_{k},t\eps N)}$, $1\leq k\leq p-1$, and noting that $w\in{B(v_0,t\eps N)}$ and $B(v_p,t\eps N)\subset B(w',2t\eps N)$, we obtain \eqref{eq:hittingballsrec}.

We now deduce the two desired bounds from \eqref{eq:hittingballsrec}. For $z\in{\partial B(y,R)}$, by the strong Markov property at time $H_{B(z,2t\eps N)}$, \eqref{eq:hittingballsrec} and \eqref{eq:GreenGK} (noting that $d(w,\partial B(N))\geq \eps N/2\geq Cd(w,z)$ for all $w\in{B(z,2t\eps N)}$, upon decreasing $t$ if necessary),
\begin{equation*}
g_N(x,z)\geq P_x(H_{B(z,2t\eps N)}<H_{\partial B(N)})\inf_{w\in{\partial B(z,2t\eps N)}}g_N(w,z)\geq c(\eps)N^{-\nu}\geq c(\eps)d(x,y)^{-\nu},
\end{equation*}
since $d(x,y)\leq 2N$. Similarly, by the strong Markov property at time $H_{B(y,2t\eps N)}$, \eqref{eq:hittingballsrec}, \eqref{eq:capBallBd} and \eqref{eq:hittingvscap} (the latter being only needed if $R<2t\eps N$), $P_x(H_{B(y,R)}<H_{\partial B(N)})\geq P_x(H_{B(y,2t\eps N)}<H_{\partial B(N)})\inf_{w}P_w(H_{B(y,R)}<H_{\partial B(N)})$,
with $w$ ranging in $\in{B(y,2t\eps N)}$, and this is bounded from below by $c(\eps)(\RR/N)^{\nu}\geq c(\eps)\RR^{\nu}d(x,y)^{-\nu}$. 
That is, \eqref{eq:toimproveforalleps} holds with a constant $c(\eps)$.
\end{proof}

\begin{Rk}
\label{rk:ballsconnected}
Throughout this section, the only time we use the assumption that $d$ is the graph distance, or $\ell^{\infty}$-distance on $\Z^d$, is when we say that $B((1-\eps)N)$ is connected in the proof of \eqref{eq:hittingballsrec}. In particular, all the results in this section are still satisfied if we replace this assumption by the assumption that $B(R)$ is connected for all $R\geq1$. Furthermore,  for a general distance $d$ with balls which are not necessarily connected, all these results are still true except for the lower bound in \eqref{eq:crossingboxseccrit}  when $B(x,r)$ or $B(y,R)$ are in $B((1-\eps)N)\setminus B(cN)$. Moreover, for a general distance $d$, below \eqref{eq:secondPOAA'} the fact that one can choose $r,R\geq C$ such that $\partial U\cap A=\emptyset$ and $\partial U'\cap A'=\emptyset$ follows from \cite[(2.8)]{DrePreRod2}, which controls $d$ in terms of the graph distance. 
Note that in the rest of the article, this bound is only used once and on $\Z^d$, where it is in force, to obtain \eqref{eq:CROS}.
\end{Rk}

\section{Clusters with large volume and capacity}\label{sec:largevolumecapacity}

In this section, as a first step toward Theorem~\ref{the:equivmacro}, we deduce that if the cluster of any fixed ball has large diameter, then it also has large capacity and volume with high probability, see Corollary~\ref{cor:supqimpliessupp}. We actually first prove a slight generalization of this result: for any random set $\mathcal{O}$ with local capacity comparable to that of a ball, if the cluster of a ball has large diameter, then it intersects $\mathcal{O}$ with high probability.

 We call a random set $\mathcal{O}=\mathcal{O}(\psi)\subset G$ {\em increasing} if $\psi\in{[0,\infty)^{\widetilde \G}}\mapsto \mathcal{O}(\psi)$ is measurable and $\psi\leq \psi'$ implies $\mathcal{O}(\psi)\subset \mathcal{O}(\psi')$, and we say that $\mathcal{O}$  is a {\em random cluster set} if 
\begin{equation}
\label{eq:clusterset}
\begin{gathered}
\mathcal{O}=\bigcup_{K\in{\mathcal{P}\cap\mathfrak{C}^{|>0|}}}K, \quad \text{ for some }\mathcal{P}\subset2^{{G}}\text{ such that }
\\\text{ for all }K\subset{G},\{K\in{\mathcal{P}}\}\text{ is a function of }|\phi|_{|K},
\end{gathered}
\end{equation}
where $\mathfrak{C}^{|>0|}$ is the set of  intersections with $G$ of each of the (maximal) connected components of $\{x\in{\tilde{\G}}:|\phi_x|>0\}$ hitting $G$, i.e., 
\[
\mathfrak{C}^{|>0|} = \Big \{ K \cap G \, : \, K \text{ a maximal connected component of } \{x\in{\tilde{\G}}:|\phi_x|>0\} \Big\}.
\] 
Furthermore, $|\phi|_{|K}$ denotes the restriction of $|\phi|$ to the set $K$. Note that if $\{K\in{\mathcal{P}\}}$ only depends on $K$ %
 (which implies that $\mathcal{P}$ does not depend on $\phi$), then $\mathcal{O}$ as in \eqref{eq:clusterset} is a random cluster set. Recall the definition of $q$ from \eqref{eq:qPsiBd} as well as $\K^0(x,r)$ from \eqref{eq:clustersDef}. %

\begin{Prop}
\label{pro:radimpliesOavoided}
Assume that $\sup q<\infty$, and let $M\geq1$, $\eps\in{(0,1)}$. There exist constants $c,\Cl[c]{c:condcap}\in{(0,1)}$ and $C\in{(1,\infty)}$ depending on $M$ and $\eps$, such that for all  $s\in{(0,1]}$, $x\in{G}$, $r'\geq0$, $R\geq C(\rr'\vee s^{-1})$, sets $\emptyset\neq A,B\subset\tilde{\G}$ as in \eqref{eq:K-ass} satisfying $\tilde{B}(x,R)\subset B$, and all  increasing random cluster sets $\mathcal{O}=\mathcal{O}(|\phi|)\subset G$ satisfying 
\begin{equation}
\label{eq:condO}
\P^{\partial B(y,sMR)}\big(\mathrm{cap}(\mathcal{O})\geq (sR)^{\nu}\big)\geq \eps \quad \text{ for all }y\in{B(x,R/2)\setminus B(x,\Cr{c:condcap}R)},
\end{equation}
 we have, with $R':=d(x,A)\geq (R/2)\vee(\dd(A)/4)$, $\dd(A)=\delta(A)\vee1$,
\begin{equation}
\label{eq:bigradiusimpliesOavoided}
\begin{split} 
    \P\big(\K^0_B(x,r')\cap\mathcal{O}=\emptyset,B(x,r')\leftrightarrow A\big)
   \leq C(\dd(A)\rr')^{\frac\nu2}(R')^{-\nu}\exp(-cs^{-1}).
\end{split}
\end{equation}
\end{Prop}

Combining Proposition~\ref{thm:macroscopiclocalcluster} and Proposition~\ref{pro:radimpliesOavoided} for first the set $\mathcal{O}$ of clusters with large capacity, and then the set $\mathcal{O}$ of clusters with both large capacity and volume, we obtain under the assumption $\sup q<\infty$ that if the cluster of a ball is connected to a far away set, then it also has large volume and capacity. Recall the definition of $p$ from \eqref{eq:defprR}.

\begin{Cor}
\label{cor:supqimpliessupp}
Assume that $\sup q<\infty$. Then there exists $c>0$ such that
\begin{equation}
\label{eq:supqimpliessupp}
\begin{split} 
\inf_{R\geq 1,r\in{[0,cR]}}p(r,R,c)>0,
\end{split}
\end{equation}
i.e., \eqref{eq:pPos} holds true. Moreover, for all  $s\in{(0,1]}$, $x\in{G}$, $r,r'\geq0$, $R\geq \Cl{C:radimpliesvol3}(\rr'\vee \rr/s)$, sets $\emptyset\neq A,B\subset\tilde{\G}$ as in \eqref{eq:K-ass} satisfying $\tilde{B}(x,R)\subset B$, and $R':=d(x,A)\geq (R/2)\vee(\dd(A)/4)$,
\begin{equation}
\label{eq:bigradiusimpliessmallcapandvol}
\begin{split} 
    \P\Big(&\big\{\mathrm{cap}(\K^0_B(x,r'))\leq ({s R})^{\nu}\text{ or }|\K_B^0(x,r')|_r\leq (sR/\rr)^{\alpha-\frac\nu2}\big\},B(x,r')\leftrightarrow A\Big)
   \\&\leq C(\dd(A)\rr')^{\frac\nu2}(R')^{-\nu}\exp(-cs^{-1}).
\end{split}
\end{equation}
\end{Cor}
\begin{proof}
Write $\mathcal{P}_1=\mathcal{P}_1(R,s)=\{K\subset \tilde{B}(x,R):\,\mathrm{cap}(K)\geq (sR)^{\nu}\}$, and let $\mathcal{O}_1=\mathcal{O}_1(|\phi|)$ be the union of the sets $K\in{\mathcal{P}_1\cap \mathfrak{C}^{|>0|}}$. Then $\mathcal{O}_1$ is an increasing random cluster set (as a function of $|\phi|$), see \eqref{eq:clusterset} and above. Moreover, $\mathcal{O}_1$ satisfies \eqref{eq:condO} for  $\eps=1/2$ and $M$ a large enough constant if $s\leq 1/(2M)$ (which implies $\tilde{B}(y,sMR)\subset \tilde{B}(x,R)$ for all $y\in{B(x,R/2)}$), by \eqref{eq:tailEstlargestCapinB(r)} for $MsR$ instead of $R$ and $1/M$ instead of $t$, the isomorphism \eqref{eq:isoloopK}, and the assumption $\sup q<\infty$. Therefore, applying Proposition~\ref{pro:radimpliesOavoided} for this choice of $M$, $\eps=1/2$,  $A=\partial B(x,R)$, $B=\tilde{B}(x,R)$, (note that $R'=d(x,A)\geq R/2$ if $R\geq C$  in view of \cite[(2.8)]{DrePreRod2}), by monotonicity of the capacity we deduce that for all $x\in{G}$, $r'\geq0$, $s\in{(0,c]}$ and $R\geq C(\rr'\vee s^{-1})$
\begin{equation}
\label{eq:bigradiusimpliessmallcap2}
\begin{split} 
    \P&\big(\mathrm{cap}(\K^0_{\tilde{B}(x,R)}(x,r'))\leq ({s R})^{\nu},B(x,r')\leftrightarrow \partial B(x,R)\big)
    \\&\leq \P\big(\K^0_{\tilde{B}(x,R)}(x,r')\cap\mathcal{O}_1=\emptyset,B(x,r')\leftrightarrow \partial B(x,R)\big)
   \leq C(\rr'/R)^{\frac\nu2}\exp(-cs^{-1}).
\end{split}
\end{equation}
For any $R\geq C$, $r'\in{[0,cR]}$ and $s\in{(cR^{-1},c']}$ we have by \eqref{eq:crossingannuli} and \eqref{eq:bigradiusimpliessmallcap2}
\begin{equation*}
\begin{split} 
\P\big(\mathrm{cap}(\K^0_{\tilde{B}(x,R)}(x,r'))\geq sR^{\nu}\big)&\geq \P\big({B}(x,r')\leftrightarrow\partial \tilde{B}(x,R)\big)
\\&\quad -\P\Big(\mathrm{cap}(\K^0_{\tilde{B}(x,R)}(x,r'))\leq sR^{\nu},{B}(x,r')\leftrightarrow\partial {B}(x,R)\Big)
\\&\geq c \left({\rr'}/{R}\right)^{\frac\nu2}-C\big({\rr'}/{R}\big)^{\frac\nu2}\exp(-cs^{-1})\geq c \left({\rr'}/{R}\right)^{\frac\nu2},
\end{split}
\end{equation*}
where the last inequality holds upon taking $s=c$ for a small enough constant $c>0$, which is always possible for $R\geq C$. One readily deduces \eqref{eq:supqimpliessupp} in view of \eqref{eq:defprR}.

We now turn to the proof of \eqref{eq:bigradiusimpliessmallcapandvol}. Let 
\[
\mathcal{P}_2=\mathcal{P}_2(R,r,s)=\{K\subset \tilde{B}(x,R):\,\mathrm{cap}(K)\geq (sR)^{\nu}\text{ and }|K|_r\geq (sR/\rr)^{\alpha-\frac\nu2}\}
\]
as well as
\(
\mathcal{O}_2 = \bigcup_{K\in{\mathcal{P}_2\cap\mathfrak{C}^{|>0|}}} K.
\) 
Then $\mathcal{O}_2$ is an increasing random cluster set in the sense of \eqref{eq:clusterset}. Moreover, it satisfies 
\eqref{eq:condO} for  $\eps=1/2$ and $M$ a large enough constant if $s\leq 1/(2M)$ (which can be assumed w.l.o.g) and $R\geq C\rr/s$, by Proposition~\ref{thm:macroscopiclocalcluster} for $MsR$ instead of $R$ and $1/M$ instead of $t$, the isomorphism \eqref{eq:isoloopK}, the assumption $\sup q<\infty$, and \eqref{eq:supqimpliessupp}. By definition of $\mathcal{O}_2$, the event in the probability on the first line of \eqref{eq:bigradiusimpliessmallcapandvol} implies $\K_B^0(x,r')\cap\mathcal{O}_2=\emptyset$, and we can conclude using Proposition~\ref{pro:radimpliesOavoided} for $\mathcal{O}_2$. 
\end{proof}

\subsection{Proof for the intersection of the cluster with a box}
\label{sec:interbox}
Before proving Proposition~\ref{pro:radimpliesOavoided}, we first prove a version of this result when replacing $\K^0_B(x,r')$ therein by $\K^0(x,r')$, that is when considering the entire cluster of $B(x,r')$ instead of the cluster of $B(x,r')$ inside $B$. Let us first introduce some notation. %
Recall from above \eqref{eq:isoloopK} that $\mathcal{L}_A$ denotes the restriction of $\mathcal{L}$ to those loops entirely included in $A$, and for any $S\subset\tilde{\G}$, let us denote by $\mathcal{C}_S(x,r')$ the set which is equal to the cluster of $B(x,r')$ for the loops in  $\mathcal{L}_{S^{\mathsf c}}$; abbreviate $\mathcal{C}(x,r)=\mathcal{C}_{\emptyset}(x,r')$. Similarly as before, we call a 
measurable  function  $\mathcal{O}$ mapping from the set of (admissible) point configurations on continuous loops in $\widetilde \G$ (cf.\ before display \eqref{eq:isoloopK} for terminology) to $2^G$  increasing, if for two such point configurations $P$ and $P'$ with $P \le P'$ one has $\mathcal O(P)\subset \mathcal O(P')$. Somewhat  similar to the Gaussian free field setup in \eqref{eq:clusterset}, we also say for any loop soup $\mathcal L'$ on a graph $\widetilde \G' \subset \widetilde \G$, that  $\mathcal{O}=\mathcal{O}(\mathcal{L}')$ is a random cluster set if
\begin{equation}
\label{eq:clustersetL}
\begin{gathered}
\mathcal{O}=\bigcup_{K\in{\mathcal{P}\cap\mathfrak{C}^{\mathcal{L}}}}K\text{ for some }\mathcal{P}\subset2^{{G'}}\text{ such that }
\\\text{ for all }K\subset{G'},\{K\in{\mathcal{P}}\}\text{ is a function of }\mathcal{L}'_{K},
\end{gathered}
\end{equation}
where $\mathfrak{C}^{\mathcal{L'}}$ now denotes the set consisting of all intersections with $G'$ of each loop cluster for $\mathcal{L}'$. We now recall the notation $\rr$ and $\dd(A)$ from \eqref{eq:r1} and below, as well as $\Q$ from above \eqref{eq:isoloopK} in order to  state the main result of this subsection.

\begin{Lemme}
\label{lem:avoidmacroloops2}
Fix some $M\geq1$ and $\eps\in{(0,1)}$. There exist constants $c\in{(0,1)}$ and $C,\Cr{C:radimpliesvol2}\in{(1,\infty)}$ depending on $M$ and $\eps$, such that for all  $s\in{(0,c]}$, $x,y\in{G}$, $r'\geq0$, $R\geq C(\rr'\vee s^{-1})$, $R'\geq \Cl{C:radimpliesvol2}R$,  for all sets $\emptyset\neq A\subset\tilde{\G}$ satisfying the condition in \eqref{eq:K-ass} as well as the condition 
 $A\subset B(x,2R')\setminus B(x,R')$, and for all increasing random cluster sets $\mathcal{O}=\mathcal{O}(\mathcal{L})$ satisfying 
\begin{equation}
\label{eq:condOL}
\tilde{\P}\big(\mathrm{cap}(\mathcal{O}(\mathcal{L}_{\tilde{B}(y,sMR)}))\geq (sR)^{\nu}\big)\geq \eps \text{ for all }y\in{\Lambda_{2sMR}\cap B(x,R)\setminus B(x,R/2)},
\end{equation}
we have 
\begin{equation}
\label{eq:avoidmacroloops2}
\begin{split}
\Q\big(\mathcal{C}_{\mathcal{O}}(x,r')\cap A\neq\emptyset\big)\leq Cq(\rr')q(\dd(A))(\dd(A)\rr')^{\frac\nu2}(R')^{-\nu}\exp(-cs^{-1}),
\end{split}
\end{equation}
with $\dd(A)=\delta(A)\vee1$.
\end{Lemme}

Note that Lemma~\ref{lem:avoidmacroloops2} easily implies Proposition~\ref{pro:radimpliesOavoided} when replacing $\K_B^0(x,r')$ by $\K^0(x,r')$ therein. Indeed, if $\mathcal{C}(x,r')\cap \mathcal{O}=\emptyset$, then $\mathcal{C}(x,r')=\mathcal{C}_{\mathcal{O}}(x,r')$, and we can easily conclude since the positive clusters of $\mathcal{C}(x,r')$ have the same law as $\K^0(x,r')$ by the isomorphism \eqref{eq:isoloopK}. In fact to prove Theorem~\ref{the:equivmacro}, we will really need to consider $\K^0_B(x,r')$ (see \eqref{eq:clustersDef}) as in Proposition~\ref{pro:radimpliesOavoided} instead of $\K^0(x,r')$ when $B=\tilde{B}(CR)$, that is to consider connected components inside $B$, instead of connected components inside $\tilde{\G}$ (which in fact can be disconnected when intersected with $B$). We will come back to this in the next subsection, and first focus on the proof of Lemma~\ref{lem:avoidmacroloops2}. It is roughly inspired by the proof of  \cite[Theorem~1.1]{DrePreRod8} or \cite[Theorem~1.2]{DrePreRod9}, but uses Proposition~\ref{pro:boxtobox} to simplify the proof, see Remark~\ref{rk:othermethod6.1} for details. However in \cite{DrePreRod8,DrePreRod9} it is additionally assumed that $\nu<\alpha/2$, and only the set $\mathcal{O}$ of loops with large capacity is considered. Let us now explain how to adapt the proof to consider all values of $\nu$ and $\alpha$ from below \eqref{eq:intro_Green}, as well as how to extend the proof to any set $\mathcal{O}$ satisfying \eqref{eq:condOL}. For this purpose, the key new result is the following lemma.

\begin{Lemme}
\label{lem:avoidmacroloops}
Fix some $M\geq1$. For all $s\in{(0,1]}$, $r'\geq0$, $x,y\in{G}$, for all sets $A\subset\tilde{\G}$ as in \eqref{eq:K-ass} with $d(x,A\cap G)> r'$, and for all increasing random cluster sets $\mathcal{O}=\mathcal{O}(\mathcal{L})$, 
we have %
\begin{equation}
\label{eq:avoidmacroloops}
\begin{split} 
    \Q\big(\mathcal{C}_{\mathcal{O}}(x,r')\cap A\neq\emptyset\,|\,\mathcal{O}\big)\leq 2\P^{\mathcal{O}}\big(B(x,r')\leftrightarrow A\big).
\end{split}
\end{equation}
\end{Lemme}

We postpone the proof of Lemma~\ref{lem:avoidmacroloops} by a few lines, and first explain how to finish the proof of Lemma~\ref{lem:avoidmacroloops2}. Indeed, the proof is now very similar to the theory already developed in \cite{DrePreRod8,DrePreRod9} and we establish the auxiliary results Lemmas \ref{lem:goodobstacle} and \ref{lem:fpploops} now, which are needed in its proof. We say that $\pi=(\pi_0,\dots,\pi_n)$ is a path in $\Lambda_{L}$ from $K$ to $K'$ if $\pi_i\in{\Lambda_{L}}$ for all $0\leq i\leq n$, $\pi_0\in{K}$, $\pi_n\in{K'}$ and for all $0\leq i\leq n-1$ there exists $x\in{B(\pi_i,L)}$ and $y\in{B(\pi_{i+1},L)}$ such that $x$ and $y$ are neighbors in $\G$. We call a set $O\subset G$ an {\em $(x,R,N,\ell,s)$-good obstacle} if for any path $\pi=(\pi_0,\dots,\pi_n)$ in $\Lambda_{sNR}$ from $B(x,R/2)$ to $B(x,R)^{\mathsf{c}}$, there exists $A\subset\text{range}(\pi\cap B(x,R))$ with $|A|\geq \ell$ and $\mathrm{cap}(O\cap B(y,sNR))\geq (sR)^{\nu}$ for all $y\in{A}$. Let us start with the following simple adaptation of \cite[Lemma~2.1]{DrePreRod8}.

\begin{Lemme}
\label{lem:goodobstacle}
For all $N\geq 1$, $s\in{(0,c/N]}$, $R\geq Cs^{-1}$, $\ell\geq 1$, $x\in{G}$, $K\subset B(x,2R)^{\mathsf c}$, $y\in{B(x,R/2)}$ and $(x,R,N,\ell,s)$-good obstacle $O\subset G$, we have
\begin{equation}
\label{eq:hittingonG}
    P_y(H_{K}<H_{{O}})\leq  \sup_{z\in{B(x,2R)}}P_z(H_K<\infty)\exp(-c\ell N^{-\nu}).
\end{equation}
\end{Lemme}
The proof of Lemma~\ref{lem:goodobstacle} is similar to the one of \cite[Lemma~2.1]{DrePreRod8}, with $L=sNR$, $n=\ell$ and $\kappa=(sR)^{\nu}$ therein, and we omit it here. We only note that one needs to replace therein $0$ by $y\in{B(x,R/2)}$ and the event $\{H_x<\infty\}$ in \cite[(2.8)]{DrePreRod8},  which was giving rise to the term $Cd(x, B_{R+C_5L})^{-\nu}$ therein, by the event $\{H_K<\infty\}$, which gives rise to the term $\sup_{z\in{B(x,2R)}}P_z(H_K<\infty)$ here, noting that $B(x,R+CL)\subset B(x,2R)$ for $s\leq c/N$.

We will use Lemma~\ref{lem:goodobstacle} with $\mathcal O$ the increasing random cluster set from Lemma~\ref{lem:avoidmacroloops2}, and we thus need to prove that this set is a good obstacle, which follows ideas from \cite[Lemma~3.4]{DrePreRod9} . 

\begin{Lemme}
\label{lem:fpploops}
Fix some $M\geq1$ and $\eps>0$. Then there exist constants $c,\Cl[c]{c:N},\Cl{C:N}\in (0,\infty)$ such that for all $s\in{(0,c]}$, $R\geq C/s$, $x\in{G}$, and for all increasing random cluster sets $\mathcal{O}=\mathcal{O}(\mathcal{L})$ satisfying \eqref{eq:condOL}, writing 
\begin{equation}
\label{eq:defG}
\mathbf{G}=\big\{O\subset G:\,O\text{ is a }(x,R,\Cr{C:N},\Cr{c:N}s^{-1},s)\text{-good obstacle}\big\},
\end{equation}
we have 
\begin{equation}
\label{eq:probaG}
\begin{split} 
\tilde{\P}\big(\mathcal{O}\in{\mathbf{G}})\geq 1-\exp(-cs^{-1}).
\end{split}
\end{equation}
\end{Lemme}
\begin{proof}
For a parameter $\delta\in{(0,1)}$ to be fixed below, similarly as in \cite[(3.15)]{DrePreRod9}, introduce 
\begin{equation}
\label{eq:defLell}
\begin{split} 
L\stackrel{\text{def.}}{=}s\delta^{-\frac1\nu}R\text{ and }\ell\stackrel{\text{def.}}{=}\left\lceil {3R}/{40L}\right\rceil-3=\big\lceil{3\delta^{\frac1\nu}}/{40s}\big\rceil -3.
\end{split}
\end{equation}
Let $\tilde{\mathbf{G}}\subset 2^{G}$%
be the set of all $(x,R,\delta^{-\frac1\nu},\ell/2,s)$-good obstacle sets, and let us further write 
\begin{equation}
\label{eq:calO}
\tilde{\mathcal{O}}\stackrel{\text{def.}}{=}\bigcup_{y\in{\Lambda_{3sMR}}}\mathcal{O}(\mathcal{L}_{B(y,sMR)}).
\end{equation}
Note that the sets $\mathcal{O}(\mathcal{L}_{B(y,sMR)})$, $y\in{\Lambda_{3sMR}}$, are independent by \eqref{eq:defLambdaL}, the complete independence of Poisson point processes, and the definition of random cluster sets. In particular, it follows from \eqref{eq:defLell} and \eqref{eq:condOL} that 
for some function $p(\delta,M,\eps)$ satisfying $p(\delta,M,\eps)\rightarrow1$ as $\delta\rightarrow 0$, and
for all $y\in{\Lambda_{L}\cap B(x,R-L)\setminus B(x,R/2+L)}$, we have
\begin{equation*}
\tilde{\P}\big(\mathrm{cap}(\tilde{\mathcal{O}}\cap {B(y,L/2)})\geq (sR)^{\nu}\big)\geq p(\delta,M,\eps).
\end{equation*}
 We refer to the end of the proof of Proposition~\ref{thm:macroscopiclocalcluster} and to below \cite[(5.2)]{DrePreRod10} for a similar argument with more details. Note that the event $\{\mathrm{cap}(\tilde{\mathcal{O}}\cap {B(y,L/2)})\geq (sR)^{\nu}\}$ only depends on the loops entirely included in $B(y,L)$  by \eqref{eq:calO} upon assuming $N\geq C=C(M)$. Proceeding similarly as in the proof of \cite[(3.17)]{DrePreRod9} (with $b=1/2$ and $a=1/8$, and replacing the events $\{\exists\, \gamma_i\in{\mathcal{L}^{\text{big}}},\text{range}(\gamma_i)\subset B(x_i,L)\}$ by the events $\{\mathrm{cap}(\tilde{\mathcal{O}}\cap {B(y,L/2)})\geq (sR)^{\nu}\}$ for $y$ as above), one can easily deduce that $\tilde{\P}(\tilde{\mathcal{O}}\notin{\tilde{\mathbf{G}}})\leq \exp(-c\ell)$ upon fixing $\delta=c$ for a small enough constant $c=c(M,\eps)$. Since $\ell\geq cs^{-1}$, $\tilde{\mathbf{G}}\subset\mathbf{G}$ for suitable constants $\Cr{c:N}$ and $\Cr{C:N}$ in \eqref{eq:defG}, and $\tilde{\mathcal{O}}\subset \mathcal{O}$ by monotonicity, we can conclude.
\end{proof}

Combining Proposition~\ref{pro:boxtobox} with Lemmas~\ref{lem:avoidmacroloops}, \ref{lem:goodobstacle} and~\ref{lem:fpploops}, one can readily finish the

\begin{proof}[Proof of Lemma~\ref{lem:avoidmacroloops2}]
Let $\mathcal{O}'=\mathcal{O}\cap B(x,2R)\setminus B(x,R/4)$, which still satisfies \eqref{eq:condOL} (with $\mathcal O$ replaced by $\mathcal O'$) upon assuming $s\leq cM^{-1}$. By \eqref{eq:avoidmacroloops} for $\mathcal{O}'$, the assumption $\sup q<\infty$, %
and \eqref{eq:crossingAB} with $A,A',r$ and $R$ replaced respectively by $B(x,r')$, $A$, $2r'$, $\delta(A)/9$, noting that the assumptions $R'\geq C\rr'$, $B(A,\delta(A)/9)\subset B(x,r'+3R')\setminus B(x,r'+R'/2)$ and $\mathcal{O}'\subset B(x,r'+cR')\setminus B(x,r'+C\rr')$ therein are satisfied under the assumptions of Lemma~\ref{lem:avoidmacroloops2} (which imply $\delta(A)\leq 4R'$) if $R'\geq CR$ and $R\geq C\rr'$, we have
\begin{equation}
\label{eq:boundEonG}
\begin{split} 
\Q&\big(\mathcal{C}_{\mathcal{O}'}(x,r')\cap A=\emptyset\big)
\\&\leq Cq(\rr)q(\dd(A))\big(\frac{\rr'}{\dd(A)}\big)^{\frac\nu2}\tilde{\E}\Big[\inf_{z\in{B(x,3r')}}P_z(H_{B(A,\delta(A)/9)}<H_{\mathcal{O}'})\Big] 
\\&\leq Cq(\rr)q(\dd(A))\big(\frac{\rr'}{\dd(A)}\big)^{\frac\nu2}\sup_{z\in{B(x,2R)}}P_z(H_{B(A,\delta(A)/9)}<\infty)\exp(-cs^{-1})
\\&\leq  Cq(\rr)q(\dd(A))(\dd(A)\rr')^{\frac\nu2}(R')^{-\nu}\exp(-cs^{-1}),
\end{split}
\end{equation}
where the second inequality follows on the event $\mathcal{O}'\in{\mathbf{G}}$ (recall $\mathbf{G}$ from \eqref{eq:defG})  from  \eqref{eq:hittingonG} for $N=\Cr{C:N}$, $O=\mathcal{O}'$ and $\ell=\Cr{c:N}s^{-1}$, as well as on the event $\mathcal{O}'\notin{\mathbf{G}}$ from \eqref{eq:probaG} for $K=B(A,\delta(A)/9)$ and $\mathcal{O}'$ instead of $\mathcal{O}$, whose assumptions are satisfied if $R\geq Cs^{-1}$. Finally,  the last inequality is a consequence of \eqref{eq:capBallBd} and \eqref{eq:hittingvscap} since $d(z,B(A,\delta(A)/9))\geq R'-2R-\delta(A)/9\geq R'/4$ for all $z\in{B(x,2R)}$ if $R'\geq CR$.
\end{proof}

To conclude the proof of Lemma~\ref{lem:avoidmacroloops2}, it remains to give the
\begin{proof}[Proof of Lemma~\ref{lem:avoidmacroloops}]
In a slight abuse of standard notation, let us denote by $\overline{\mathcal{O}}$ the union of $\mathcal{O}=\mathcal{O}(\mathcal{L})$ with all line segments $I_{\{z,z'\}}\subset\tilde{\G}$ between any $z\sim z'$ such that both $z$ and $z'$ are in $\mathcal{O}$. Let $S\subset \tilde{\G}$ be a closed set such that $\partial S\subset G$. 
 We write $F_S^1$ for the event that  $\mathcal{O}(\mathcal{L}_{S^{\mathsf c}})=\emptyset$, and that  none of the connected components of $\mathcal{L}_{S^{\mathsf c}}$ intersect both $G$ and a loop in $\mathcal{L}^{\partial S}$ (recall the loop soup notation as well as its connected components from above \eqref{eq:isoloopK}). Write also $F_S^2$ for the event that $\mathcal{O}(\mathcal{L}^{S})=S\cap G$. Note that if none of the connected components of  $\mathcal{L}_{S^{\mathsf c}}$ intersects both $G$ and a loop in $\mathcal{L}^{\partial S}$, 
 then the set consisting of the intersections  of the connected components of $\mathcal{L}$ with $G$
 is the union of the set consisting of the intersections with $G$ of the connected components of $\mathcal{L}_{S^{\mathsf c}}$ 
 with the set consisting of the intersections with $G$ of the connected components of $\mathcal{L}^{S}$.
 As a consequence, $\mathcal{O}(\mathcal{L})=\mathcal{O}(\mathcal{L}_{S^{\mathsf c}})\cup \mathcal{O}(\mathcal{L}^{S})$ (where the graph $G'$ underlying the loop soup $\mathcal{L}^{S}$ in the context of the definition in \eqref{eq:clustersetL} is $S$) in view of \eqref{eq:clustersetL}. As a consequence of the above, we infer $\{\overline{\mathcal{O}}=S\} = F_S^1\cap F_S^2$.%
Moreover, $F_S^2$ only depends on $\mathcal{L}^{S}$, and is thus independent of $\mathcal{L}_{S^{\mathsf c}}$ conditionally on $\mathcal{L}^{\partial S}$ by properties of Poisson point processes. Therefore, 
since both $\mathcal{C}_S$ and $F_S^1$ are $\sigma(\mathcal{L}_{S^{\mathsf c}},\mathcal{L}^{\partial S})$-measurable,
 we have
\begin{equation*}
\begin{split} 
\Q\big(\mathcal{C}_S(x,r')\cap A\neq\emptyset,\overline{\mathcal{O}}=S\big)
&=\E_{\Q}\Big[\Q\big(\mathcal{C}_S(x,r')\cap A\neq\emptyset,F_S^1,F_S^2\,|\,\mathcal{L}^{\partial S}\big)\Big]
\\&=\E_{\Q}\Big[\Q\big(\mathcal{C}_S(x,r')\cap A\neq\emptyset,F_S^1\,|\,\mathcal{L}^{\partial S}\big)\Q\big(F_S^2\,|\,\mathcal{L}^{\partial S}\big)\Big].
\end{split}
\end{equation*}
As the event $F_S^1$ is decreasing and the event $\mathcal{C}_S(x,r')\cap A\neq\emptyset$ is increasing, by the FKG inequality (note that $F_S^1$ and $\mathcal{C}_S$ conditionally on $\mathcal{L}^{\partial S}$ only depend on $\mathcal{L}_{S^{\mathsf{c}}}$, which is positively associated) we can further upper bound 
\begin{equation*}
\begin{split} 
\Q\big(\mathcal{C}_S(x,r')\cap A\neq\emptyset,F_S^1\,|\,\mathcal{L}^{\partial S}\big)\leq \Q\big(\mathcal{C}_S(x,r')\cap A\neq\emptyset\big)\Q\big(F_S^1\,|\,\mathcal{L}^{\partial S}\big),
\end{split}
\end{equation*}
where we also used that $\mathcal{C}_S(x,r')$ is independent of $\mathcal{L}^{\partial S}$ (recall that $S$ is closed). Plugging this into the previous display, and recalling that, conditionally on $\mathcal{L}^{\partial S}$,  $F_S^1$ and $F_S^2$ are independent events  whose intersection is equal to $\{\overline{\mathcal{O}}=S\}$, we deduce that  
\begin{equation*}
\begin{split} 
  \Q\big(\mathcal{C}_S(x,r')\cap A\neq\emptyset&\,|\,\overline{\mathcal{O}}=S\big)\leq \Q\big(\mathcal{C}_S(x,r')\cap A\neq\emptyset\big).
\end{split}
\end{equation*}
If $B(x,r')\leftrightarrow A$ in $\mathcal{L}_{S^{\mathsf c}}$, then this connection occurs via a cluster with either positive or negative sign, and in both cases this event under $\Q$ has the same probability as the event $B(x,r')\leftrightarrow A$ under $\P^{S}$ by the isomorphism theorem \eqref{eq:isoloopK}. Since $d(x,A\cap G)>r'$, the event $\K^0(x,r')\cap A\neq\emptyset$ has the same law under $\P^{\overline{\mathcal{O}}}$ and $\P^{\mathcal{O}}$, see Section~\ref{subsec:G_K}, and we can conclude since $\overline{\mathcal{O}}$ is $\mathcal{O}$-measurable.
\end{proof}

\begin{Rk}
\label{rk:othermethod6.1}
Given Proposition~\ref{thm:macroscopiclocalcluster}, one could also prove Lemma~\ref{lem:avoidmacroloops2} using the methods developed in \cite{DrePreRod8,DrePreRod9}. Indeed, when $\alpha<2\nu$, \cite[Theorem~1.2]{DrePreRod9} is essentially equivalent to \eqref{eq:avoidmacroloops2} for $r=0$, $A=\partial B(x,R)^{\mathsf{c}}$  and $\mathcal{O}(\mathcal{L})$ the union of all clusters of $\mathcal{L}$ with capacity at least $(sR)^{\nu}$. In fact the condition $\alpha<2\nu$ in \cite[Theorem~1.2]{DrePreRod9} could easily be replaced therein by the condition $\sup q<\infty$ (which was the main goal of \cite{DrePreRod9}) by using a result similar to Lemma~\ref{lem:avoidmacroloops} instead of \cite[(3.14)]{DrePreRod9}. Our current proof of Proposition~\ref{pro:radimpliesOavoided} relying on Proposition~\ref{pro:boxtobox} is however simpler than the approach from \cite{DrePreRod9} relying on an extended Lupu's formula \cite[Proposition~2.2]{DrePreRod9}, since this last approach required an additional recurrence scheme, see \cite[Proposition~3.1]{DrePreRod9}. At least when $r'=0$ and $A=\{y\}$, instead of relying on Proposition~\ref{pro:boxtobox}, another possible approach would be to use the recent article \cite{werner2025switching} describing the law of $\mathcal{C}_S(x)$ conditionally on the event $x\leftrightarrow y$. 
\end{Rk}

\subsection{Proof for the cluster inside a box}
\label{sec:clusterinside}

We now explain how to deduce Proposition~\ref{pro:radimpliesOavoided} from Lemma~\ref{lem:avoidmacroloops2}, which, as explained below Lemma~\ref{lem:avoidmacroloops2}, essentially amounts to showing that not only the cluster of $B(x,r')$ typically intersects $\mathcal{O}$ when macroscopic, but in fact the same is true for the cluster of $B(x,r)$ within a large enough set $B$.

Let us define for all  $U,V$ and $O\subset G$ the following subset of all unrooted loop $\gamma$ on $\tilde{\G}$
\begin{equation}
\label{eq:defcalB}
    \mathcal{B}_{U,V}(O)=\big\{\gamma:\,\exists\, \gamma'\,\subset \text{trace}(\gamma),  U\leftrightarrow V\text{ in }\gamma', \gamma'\cap O=\emptyset\big\}.
\end{equation}
In words, $\mathcal{B}_{U,V}(O)$ corresponds to the (bad) loops $\gamma$ that contains a subpath connecting $U$ to $V$ and avoiding $O$. Recall that $\mu$ denotes the intensity measure of the loop soup $\mathcal{L}$.

\begin{Lemme}
\label{lem:nomacroloop}
There exist  $\Cl[c]{c:fpploop}\in{(0,1/8]}$, such that for all, $R\geq C$,  $x\in{G}$, $O\subset G$, $U_2\subset U_1\subset B(x,\Cr{c:fpploop}R)$ and $V_1\subset V_2\subset B(x,R)^{\mathsf{c}}$ 
\begin{equation}
\label{eq:nomacroloop}
\begin{split}
\mu\big(&\{\gamma\in{\mathcal{B}_{U_1,V_1}(O)}:\,U_2\leftrightarrow V_2\text{ in }\text{range}(\gamma)\}\big)
\\&\leq C\left(\frac{\mathrm{cap}(U_2)}{d(\partial U_1, U_2)^{\nu}}\wedge 1\right)\left(\frac{\mathrm{cap}(V_2)}{d(\partial V_1, V_2)^{\nu}}\wedge 1\right)\frac{\mathrm{cap}(U_1)}{d(x,V_2)^{\nu}}\sup_{z\in{U_1}}P_z(H_{V_1}<H_O).
\end{split}
\end{equation}
\end{Lemme}

In \eqref{eq:nomacroloop}, we use the convention $1/0=\infty$. In particular, in case $U_1=U_2$, $V_1=V_2=B(x,\zeta R)^{\mathsf{c}}$ for some $\zeta>1$ and $O=\emptyset$, Lemma~\ref{lem:nomacroloop} is equivalent to \cite[(3.6)]{DrePreRod8}, and is in fact proved using similar ideas as therein, see also \cite[Lemma~2.7]{zbMATH06566372}. The details are given at the end of the section, and we first explain how to combine Lemma~\ref{lem:nomacroloop} with Lemmas~\ref{lem:avoidmacroloops2}, \ref{lem:goodobstacle} and~\ref{lem:fpploops} to finish the 

\begin{proof}[Proof of Proposition~\ref{pro:radimpliesOavoided}]
Throughout this proof, we can assume w.l.o.g.\ that $R'\geq CR$ and $s\in{[0,c]}$ for some constants $C<\infty$ and $c>0$, depending on $M$ and $\eps$. Indeed the case $R\in{[R'/C,2R']}$ follows directly from the case $R=R'/C$ (with $2C'M$ instead of $M$), and the case $s\in{[c,1]}$ follows directly from the case $s=c$ (with $M/c$ instead of $M$), up to changing the constants appearing in  Proposition~\ref{pro:radimpliesOavoided}.

Write $\mathcal{O}_1=\mathcal{O}({L}_{\tilde{B}(x,R)})$ and $\mathcal{O}_2=\mathcal{O}({L}_{\tilde{B}(x,\Cr{c:fpploop}R)^{\mathsf{c}}})$, where $L_A$ denotes the local time process associated to $\mathcal{L}_A$. Let us denote by $\mathcal{L}^{\mathcal{B}}$ for the set of (bad) loops $\gamma$ in $\mathcal{L}\cap\mathcal{B}_{B(x,\Cr{c:fpploop}R),B(x,R)^{\mathsf{c}}}(\mathcal{O}_2)$, and for any $C\subset\tilde{\G}$, write $C(x,r')$ for the connected component of $B(x,r')$ in $C$. Let us further define the event
\begin{equation}
\label{eq:defEloop}
E=\left\{
\begin{gathered}
\text{there exists a connected set }C\text{ which is a union of traces of }
\\\text{loops in }\mathcal{L},\text{ such that }\emptyset\neq (C\cap B)(x,r')\subset  \mathcal{O}^{\mathsf{c}}\text{ and }C\cap A\neq \emptyset
\end{gathered}
\right\},
\end{equation}
which depends implicitly on the choice of $x,r',A$, $B$ and $\mathcal{O}$. For $R\geq Cr'$ and $R'\geq CR$, since $A\subset\tilde{B}(x,R)^{\mathsf{c}}$ and $\tilde{B}(x,R)\subset B$ and in view of  \eqref{eq:defcalB}, on the event $E$, we have that any loop included in the set $C$ from \eqref{eq:defEloop} and hitting both $(C\cap B)(x,r')\cap B(x,\Cr{c:fpploop}R)$ and $B(x,R)^{\mathsf c}$ must be in $\mathcal{B}_{B(x,\Cr{c:fpploop}R),B(x,R)^{\mathsf{c}}}(\mathcal{O})\subset \mathcal{B}_{B(x,\Cr{c:fpploop}R),B(x,R)^{\mathsf{c}}}(\mathcal{O}_2)$. In particular, if additionally $\mathcal{L}^{\mathcal{B}}=\emptyset$, then there is a path from $B(x,r')$ to $\partial B(x,\Cr{c:fpploop}R)$ which only consists of loops included in $\tilde{B}(x,R)$ and avoiding $\mathcal{O}$, and thus also $\mathcal{O}_1$, that is $\mathcal{C}_{\mathcal{O}_1\cup \tilde{B}(x,R)^{\mathsf{c}}}(x,r')\cap  \partial B(x,\Cr{c:fpploop}R)\neq\emptyset$.  Moreover, still on the previous events, we also have $A\leftrightarrow  B(x,R)$ using only loops in  $\mathcal{L}\setminus \mathcal{L}_{\mathrm{in}}$ by definition, where $\mathcal{L}_{\mathrm{in}}=\mathcal{L}_{\tilde{B}(x,R)}$ denotes the set of loops contained in $\tilde{B}(x,R)$. Since $\mathcal{C}_{\mathcal{O}_{1}\cup \tilde{B}(x,R)^{\mathsf c}}(x,r')$ is $\mathcal{L}_{\mathrm{in}}$ measurable, and thus independent of  $\mathcal{L}\setminus \mathcal{L}_{\mathrm{in}}$, we deduce that
\begin{equation}
\label{eq:LBempty}
\begin{split} 
\Q&\big(E,\mathcal{L}^{\mathcal{B}}=\emptyset\big)
\leq \Q\big(\mathcal{C}_{\mathcal{O}_{1}}(x,r')\cap \partial B(x,\Cr{c:fpploop}R)\neq\emptyset\big)\tilde{\P}\big(A\leftrightarrow  B(x,R)\big).
\end{split}
\end{equation}
Note that the clusters of loops correspond exactly to the clusters where the local time of the loop soup is positive, and so $\mathcal{O}_1$ satisfies \eqref{eq:clustersetL} since $\mathcal{O}$ satisfies \eqref{eq:clusterset}. Hence $\mathcal{O}_1$ is an increasing cluster set, and \eqref{eq:condOL} is satisfied for $\mathcal{O}_1$ instead of $\mathcal{O}$, $\Cr{c:fpploop}R/(2\Cr{C:radimpliesvol2})$ instead of $R$, and $2\Cr{C:radimpliesvol2}s/\Cr{c:fpploop}$ instead of $s$. Indeed, defining $\Cr{c:condcap}=\Cr{c:fpploop}/(4\Cr{C:radimpliesvol2})$, this follows easily from \eqref{eq:condO} and the isomorphism \eqref{eq:isoloopK}.
In particular,  \eqref{eq:avoidmacroloops2} holds for $A=\partial B(x,\Cr{c:fpploop}R)$, $\mathcal{O}_1$ instead of $\mathcal{O}$, $\Cr{c:fpploop}R/(2\Cr{C:radimpliesvol2})$ instead of $R$, $2\Cr{C:radimpliesvol2}s/\Cr{c:fpploop}$ instead of $s$, and $\Cr{c:fpploop}R/2$ instead of $R'$, as long as $R\geq C(\rr'\vee s^{-1})$, which implies that
\begin{equation}
\label{eq:LBempty2}
\begin{split} 
\Q\big(\mathcal{C}_{\mathcal{O}_{1}}(x,r')\cap \partial B(x,\Cr{c:fpploop}R)\neq\emptyset\big)\leq C(\rr'/R)^{\frac\nu2}\exp(-cs^{-1}),
\end{split}
\end{equation}
where we used the assumption $\sup q<\infty$. Combining \eqref{eq:LBempty}, \eqref{eq:LBempty2} and \eqref{eq:crossingboxseccrit}, which is in force for $x,r,y,R$ therein replaced by $z,2\delta(A),x,R$ for some $z\in{A}$ if $R'\geq CR\vee C\dd(A)$, we deduce that
\begin{equation}
\label{eq:LBempty3}
\begin{split} 
\Q&\big(E,\mathcal{L}^{\mathcal{B}}=\emptyset\big)\leq C(\dd(A)\rr')^{\frac\nu2}(R')^{-\nu}\exp(-cs^{-1}).
\end{split}
\end{equation}
Note that in case $\dd(A)/4\leq R'\leq C\dd(A)$, the inequality \eqref{eq:LBempty3} is still satisfied by combining \eqref{eq:LBempty}, \eqref{eq:LBempty2} and \eqref{eq:crossingannuli} with $x,r,R$ therein replaced by $x,R,R'$.

We now focus on the event $\mathcal{L}^{\mathcal{B}}\neq\emptyset$, and first consider the case where all the loops of $\mathcal{L}^{\mathcal{B}}$ are included in $B(x,R'/8)$. Let $K$ be the largest $k\in{[1,\log_2(R'/(8R))]\cap \N}$ such that $\text{trace}(\mathcal{L}^{\mathcal{B}})\cap B(x,2^{k-1}R)^{\mathsf c}\neq\emptyset$, and $P$ be the largest $p\in{[0,\log_2(R/\rr')]\cap \N}$ such that $B(x,r')\leftrightarrow \partial B(x,2^p\rr')$ in $\mathcal{L}_{\mathrm{in}}$, with $P=0$ if there is no such $p$. Then on the intersection of the events $E$, $\emptyset\neq\text{trace}(\mathcal{L}^{\mathcal{B}})\subset B(x,R'/8)$, $K=k$ and $P=p$, there is a path from $B(x,\rr')$ to $\partial B(x,2^p\rr')$ via loops in $\mathcal{L}_{\mathrm{in}}$,  from $A$ to $B(x,2^{k+1}R)$ in $\mathcal{L}\setminus (\mathcal{L}^{\mathcal{B}}\cup \mathcal{L}_{\mathrm{in}})$, and $\mathcal{L}^{\mathcal{B}}$ intersects both $B(x,2^{p+1}\rr')$ and $B(x,2^{k-1}R)^{\mathsf c}$. Here, we used that if a loop in $\mathcal{L}\setminus \mathcal{L}_{\mathrm{in}}$ intersects $(C\cap B)(x,r')$, with $C$ as in \eqref{eq:defEloop}, in $B(x,2^{p+1}\rr')$, then it must be in $\mathcal{L}^{\mathcal{B}}$ in view of \eqref{eq:defcalB} if $2^{p+1}\rr'\leq \Cr{c:fpploop}R$, and otherwise it is trivial that $\mathcal{L}^{\mathcal{B}}$ intersects $B(x,2^{p+1}\rr')$ since $\mathcal{L}^{\mathcal{B}}\neq\emptyset$.  Since $\mathcal{L}_{\mathrm{in}}$ and $\mathcal{L}^{\mathcal{B}}=\emptyset$ are supported on disjoint set of loops by definition,  $\mathcal{L}_{\mathrm{in}}$, $\mathcal{L}\setminus (\mathcal{L}^{\mathcal{B}}\cup \mathcal{L}_{\mathrm{in}})$ and $\mathcal{L}^{\mathcal{B}}$ are independent, and so for all $k$ and $p$ as above
\begin{equation}
\label{eq:LBnonempty1}
\begin{split} 
\Q\big(&E,\emptyset\neq\text{trace}(\mathcal{L}^{\mathcal{B}})\subset B(y,R'/8),K=k,P=p\big)
\\ & \quad \leq \Q\big(B(x,\rr')\leftrightarrow \partial B(x,2^p\rr'\big)\Q\big(A\leftrightarrow B(x,2^{k+1}R)\big)
\\& \qquad \times \Q\big(\text{trace}(\mathcal{L}^{\mathcal{B}})\cap B(x,2^{p+1}\rr')\neq\emptyset,\text{trace}(\mathcal{L}^{\mathcal{B}})\cap B(x,2^{k-1}R)^{\mathsf c}\neq\emptyset \big).
\end{split}
\end{equation}
Moreover, if $\mathcal{L}^{\mathcal{B}}$ intersects both $B(x,2^{p+1}\rr')$ and $B(x,2^{k-1}R)^{\mathsf c}$, then either there is a loop in $\mathcal{L}^{\mathcal{B}}$ intersecting both $B(x,2^{p+1}\rr')$ and $B(x,2^{k-1}R)^{\mathsf c}$, or there is a loop in $\mathcal{L}^{\mathcal{B}}$ intersecting $B(x,2^{p+1}\rr')$ but not $B(x,2^{k-1}R)^{\mathsf c}$, and another one intersecting $B(x,2^{k-1}R)^{\mathsf c}$. Conditionally on $\mathcal{O}_2$, choosing $U_1=B(x,\Cr{c:fpploop}R)$ and $V_1=B(x,R)^{\mathsf{c}}$, the probability of the first event can be bounded using \eqref{eq:nomacroloop} for $U_2=B(x,(2^{p+1}\rr')\wedge(\Cr{c:fpploop}R))$ and $V_2=B(x,2^{k-1}R)^{\mathsf c}$, and by independence the probability of the second event can be bounded using \eqref{eq:nomacroloop} first for $U_2=B(x,(2^{p+1}\rr')\wedge(\Cr{c:fpploop}R))$ and $V_2=B(x,R)^{\mathsf c}$, and then for $U_2=B(x,\Cr{c:fpploop}R)$ and $V_2=B(x,2^{k-1}R)^{\mathsf c}$, always with $O=\mathcal{O}_2$, which is independent of $\mathcal{L}^{\mathcal{B}}$ since any loop of $\mathcal{L}^{\mathcal{B}}$ intersects $B(x,\Cr{c:fpploop}R)$. Hence one can bound the last line of \eqref{eq:LBnonempty1} by
\begin{equation}
\label{eq:LBnonempty2}
\begin{split} 
C&\left(\frac{2^{p\nu}(\rr')^{\nu}2^{k\nu}R^{2\nu}}{R^{\nu}(2^{k\nu}R^{\nu})^2}+\frac{2^{p\nu}(\rr')^{\nu}R^{2\nu}}{R^{3\nu}}\times \frac{R^{\nu}2^{k\nu}R^{2\nu}}{R^{\nu}(2^{k\nu}R^{\nu})^2}\right)\tilde{\E}\Big[\sup_{z\in{B(x,\Cr{c:fpploop}R)}}P_z(H_{B(x,R)^{\mathsf{c}}}<H_{\mathcal{O}_2})\Big]
\\&\leq C2^{p\nu-k\nu}(\rr')^{\nu}R^{-\nu}\exp(-cs^{-1}),
\end{split}
\end{equation}
where in the last inequality we used Lemmas~\ref{lem:goodobstacle} and~\ref{lem:fpploops} for $R/2$ instead of $R$ and $2s$ instead of $s$. To apply these lemmas, one notes that $\mathcal{O}_2$ is a random cluster set satisfying \eqref{eq:condOL} for $R/2$ instead of $R$ and $2s$ instead of $s$, by \eqref{eq:condO} and the isomorphism \eqref{eq:isoloopK}. 
Combining this with \eqref{eq:crossingboxseccrit}, \eqref{eq:crossingannuli} (splitting the cases $R'\geq C\dd(A)$ and $\dd(A)/4\leq R'\leq C\dd(A)$ similarly as in the proof of \eqref{eq:LBempty3}) and summing \eqref{eq:LBnonempty1} over $k$ and $p$  yields
\begin{equation}
\label{eq:LBnonempty3}
\begin{split} 
\Q&\big(E,\emptyset\neq\text{trace}(\mathcal{L}^{\mathcal{B}})\subset B(y,R'/8)\big)
\\&\leq \sum_{k=1}^{\lfloor \log_2(R'/(8R)) \rfloor}\sum_{p=0}^{\lfloor \log_2(R/\rr') \rfloor }C2^{-\frac{p\nu}{2}}\dd(A)^{\frac\nu2}2^{\frac{k\nu}2}R^{\frac\nu2}(R')^{-\nu}2^{p\nu-k\nu}(\rr')^{\nu}R^{-\nu}\exp(-cs^{-1})
\\&\leq C(\dd(A)\rr')^{\frac\nu2}(R')^{-\nu}\exp(-cs^{-1}).
\end{split}
\end{equation}
It remains to consider the case where a loop in $\mathcal{L}^{\mathcal{B}}$ exits $B(x,R'/8)$. Fix some $z\in{A}$, and let $N$ be the largest $n\in{[1,\log_2(R'/(2\dd(A)))]\cap\mathbb{N}}$ such that $B(z,2\dd(A))\leftrightarrow \partial B(z,2^n\dd(A))$ in $\mathcal{L}\setminus\mathcal{L}^{\mathcal{B}}$, with $N=0$ in case $R'\leq 2\dd(A)$. Then on the intersection of the events $E$, $\text{trace}(\mathcal{L}^{\mathcal{B}})\cap B(x,R'/8)^{\mathsf c}\neq\emptyset$, $N=n$ and $P=p$,  there is a path from $B(x,r')$ to $\partial B(x,2^p\rr')$ via loops in $\mathcal{L}_{\mathrm{in}}$,  from $B(z,2\dd(A))$ to $\partial B(z,2^{n}\dd(A))$ in $\mathcal{L}\setminus (\mathcal{L}^{\mathcal{B}}\cup \mathcal{L}_{\mathrm{in}})$, and $\mathcal{L}^{\mathcal{B}}$ intersects both $B(x,2^{p+1}\rr')$ and $B(z,2^{n+1}\dd(A))$ for $n<\lfloor \log_2(R'/(2\dd(A)))\rfloor$, and both $B(x,2^{p+1}\rr')$ and $B(x,R'/8)^{\mathsf c}$ for $n=\lfloor \log_2(R'/(2\dd(A)))\rfloor$. Similarly as before, the previous events are independent. Moreover, using \eqref{eq:nomacroloop} three times similarly as in \eqref{eq:LBnonempty2}, the probability that $\mathcal{L}^{\mathcal{B}}$ intersects both $B(x,2^{p+1}\rr')$ and $B(z,2^{n+1}\dd(A))$ if $n<\lfloor \log_2(R'/(2\dd(A)))\rfloor$, and both $B(x,2^{p+1}\rr')$ and $B(x,R'/8)^{\mathsf c}$ if $n=\lfloor \log_2(R'/(2\dd(A)))\rfloor$, can be upper bounded by
\begin{equation*}
\begin{split} 
C&\left(\frac{2^{p\nu}(\rr')^{\nu}2^{n\nu}\dd(A)^{\nu}R^{\nu}}{R^{\nu}(R')^{2\nu}}+\frac{2^{p\nu}(\rr')^{\nu}R^{2\nu}}{R^{3\nu}}\times \frac{R^{\nu}2^{n\nu}\dd(A)^{\nu}R^{\nu}}{R^{\nu}(R')^{2\nu}}\right)
\\&\quad\times\tilde{\E}\Big[\sup_{z\in{B(x,\Cr{c:fpploop}R)}}P_z(H_{B(x,R)^{\mathsf{c}}}<H_{\mathcal{O}_2})\Big]
\\&\leq C2^{p\nu+n\nu}\dd(A)^{\nu}(\rr')^{\nu}(R')^{-2\nu}\exp(-cs^{-1}),
\end{split}
\end{equation*}
where we used Lemma~\ref{lem:fpploops} again in the last inequality. Thus using \eqref{eq:crossingannuli} similarly as in \eqref{eq:LBnonempty3} and above, we deduce that 
\begin{equation}
\label{eq:LBnonempty4}
\begin{split} 
\Q&\big(E,\text{trace}(\mathcal{L}^{\mathcal{B}})\cap B(y,R'/8)^{\mathsf c}\neq\emptyset\big)
\\&\leq \sum_{n=0}^{\lfloor \log_2(R'/(2\dd(A))) \rfloor}\sum_{p=0}^{\lfloor \log_2(R/\rr') \rfloor }C2^{-\frac{p\nu}{2}}2^{-\frac{n\nu}2}2^{p\nu+n\nu}\dd(A)^{\nu}(\rr')^{\nu}(R')^{-2\nu}\exp(-cs^{-1})
\\&\leq C(\dd(A)\rr')^{\frac\nu2}R^{\frac\nu2}(R')^{-\frac{3\nu}{2}}\exp(-cs^{-1}).
\end{split}
\end{equation}
Combining \eqref{eq:LBempty}, \eqref{eq:LBnonempty3} and \eqref{eq:LBnonempty4}, recalling that $R\leq R'$, one can thus bound $\tilde{\P}(E)$ by the right-hand side of \eqref{eq:bigradiusimpliesOavoided}. Let $E^+$ be the same event as $E$ in \eqref{eq:defEloop}, but with the additional condition that the loops in the cluster $C$ therein have positive sign. Then $\tilde{\P}(E^+)=\tilde{\P}(E)/2$ since all the loops in the cluster $C$ from \eqref{eq:defEloop} must have the same sign. Noting that if $B(x,r')$ is connected to $A$ by positive loops in $\mathcal{L}$, and the corresponding cluster of $B(x,r')$ within $B$ avoids $\mathcal{O}$, then $E^+$ occurs, we can conclude by the isomorphism \eqref{eq:isoloopK}. 
\end{proof}

It only remains to give the
\begin{proof}[Proof of Lemma~\ref{lem:nomacroloop}]
Recall the constant $\Cr{c:harnack}$ introduce above \eqref{eq:Harnack}, we will prove this result when choosing the constant $\Cr{c:fpploop}\leq \Cr{c:harnack}\wedge(1/8)$ such that 
\begin{equation}
\label{eq:defcfpploop}
    P_y(H_{B(x,\Cr{c:fpploop}R)}<\infty)\leq \frac12\text{ for all }R\geq 1/\Cr{c:fpploop},x\in{G}\text{ and }y\in{B(x,R)}^{\mathsf c},
\end{equation}
which exists in view of \eqref{eq:capBallBd} and \eqref{eq:hittingvscap}.

Let $\Gamma=(\Gamma_k)_{0\leq k< n}$ be either an infinite path if $n=\infty$, or if $n<\infty$ such that defining $\Gamma_n:=\Gamma_0$, $(\Gamma_k)_{0\leq k\leq  n}$ is a finite connected loop. Write $D_0=D_0(\Gamma)=0$, and define recursively on $k\geq0$ the successive return times to $U_1$ and $V_1$ by 
\begin{equation*}
\begin{split}
R_k=R_k(\Gamma)&=\inf\big\{k\geq D_k:\,\Gamma_k\in{U_1}\big\}\text{ and }
\\D_{k+1}=D_{k+1}(\Gamma)&=\inf\big\{k\geq R_k:\,\Gamma_k\in{V_1}\big\}.
\end{split}
\end{equation*} 
We take the convention $\inf\emptyset=n+1$($=\infty$ if $n=\infty$), and let $\tau=\sup\{k\geq0:R_k<n\}$. For any $k\leq \tau$, we further write $\Gamma^{(k)}=\{\Gamma_{i+T_k},0< i< D_{k+1}-T_k\}$ for the $k$-th excursion from $U_1$ to $V_1$ for $\Gamma$, where $T_k=\sup\{i\in{[R_k,D_{k+1}]}:\,\Gamma_i\in{U_1}\}$, with $\sup\emptyset =n+1$($=\infty$ if $n=\infty$). Let us define 
\begin{equation}
\label{eq:deftildeB}
    \tilde{\mathcal{B}}_{U_1,V_1}(O)=
    \big\{\text{finite loops }\Gamma:\,\exists\, k\leq \tau:\,\Gamma^{(k)}\cap O=\emptyset\big\}.
\end{equation}
Note that if $\Gamma$ is a loop and we reroot $\Gamma$ at any of the time $R_k$, $k\leq\tau$, which is equivalent to asking that $R_{\tau+1}=n$, the event $\Gamma\in{\tilde{\mathcal{B}}_{r,R}^{x}}$ does not depend on the choice of $k$. Moreover, if $\gamma$ is an unrooted loop on the metric graph $\tilde{\G}$ hitting $B(x,\Cr{c:fpploop}R)$, the event $\gamma\in{\mathcal{B}_{U_1,V_1}(O)}$ implies that either $\Gamma\in{\tilde{\mathcal{B}}_{U_1,V_1}(O)}$ or $\Gamma'\in{\tilde{\mathcal{B}}_{U_1,V_1}(O)}$, where $\Gamma$ corresponds to the discrete skeleton of $\gamma$ rooted at an arbitrary time so that $R_{\tau+1}(\Gamma)=n$, with $n=\text{length}(\Gamma)$, and $\Gamma'$ is the time-reversal of $\Gamma$, that is $\Gamma'_{k}=\Gamma'_{n-k}$ for all $0\leq k\leq n$. Recalling that $Z$ denotes the discrete time random walk on $\G$, a similar reasoning as the one leading to \cite[(3.8)]{DrePreRod8} then yields 
\begin{equation}
\label{eq:decompositionloops}
\begin{split} 
\mu&\big(\{\gamma\in{\mathcal{B}_{r,R}^{x,s}}:\,U_2\leftrightarrow V_2\text{ in }\text{range}(\gamma)\}\big)
\\&\leq 2\sum_{k\geq1}\frac1k\sum_{y\in{\partial U_1}}P_y\big(Z_{R_k}=y,H_{U_2}\leq R_k,H_{V_2}<R_k,(Z_i)_{0\leq i\leq R_k}\in{\tilde{B}_{U_1,V_1}(O)}\big),
\\&\leq 2\sum_{k\geq1}\frac1k\sum_{y\in{\partial U_1}}\sum_{p=1}^{k-1}P_y\big(Z_{R_k}=y,H_{U_2}\leq R_k,H_{V_2}<R_k,Z^{(p)}\cap O=\emptyset\big),
\end{split}
\end{equation}
where $R_k=R_k(Z),D_k=D_k(Z)$ and $T_k=T_k(Z)$, and we used \eqref{eq:deftildeB} and the fact that $\tau((Z_i)_{0\leq i\leq R_k})=k-1$ in the last inequality. Moreover in \eqref{eq:decompositionloops} to simplify notation we define $Z_{R_k}=\Delta$ for some cemetery point $\Delta$ if $R_k=\infty$, so that $Z_{R_k}=y$ for $y\in{G}$ implies in particular that $R_k<\infty$. Noting that if $H_{U_2}\leq R_k$ and $H_{V_2}< R_k$, then $H_{U_2}$ must happen between $R_n$ and $T_{n}$ for some $0\leq n<k$ since $U_2\subset U_1$, and $H_{V_2}$ must happen between time $D_m$ and $R_m$ for some $1\leq m\leq k$ since $V_1\subset V_2$, we further have
\begin{equation}
\label{eq:sumnm}
\begin{split} 
&P_y\big(Z_{R_k}=y,H_{U_2}\leq R_k,H_{V_2}<R_k,,Z^{(p)}\cap O=\emptyset\big)
\\&\leq \sum_{n=0}^{k-1}\sum_{m=1}^{k}P_y\big(D_m\leq H_{V_2}<R_{m}, R_n\leq H_{U_2}\leq T_{n},Z^{(p)}\cap O=\emptyset,Z_{R_{k}}=y\big).
\end{split}
\end{equation}
To bound the probability that $R_k<\infty$ applying the strong Markov property a time $R_0$, we have for all $j\geq0$ and  $z\in{G}$ 
\begin{equation}
\label{eq:boundRk}
\begin{split} 
P_z(R_j<\infty)\leq P_z(R_0<\infty)\sup_{z'\in{U_1}}P_{z'}(R_j<\infty)&\leq CP_z(R_0<\infty)2^{-j},
\end{split}
\end{equation}
where the last inequality follows from applying the strong Markov property at time $D_1$ and \eqref{eq:defcfpploop} similarly as in \cite[(3.9)]{DrePreRod8}, using that $U_1\subset B(x, \Cr{c:fpploop}R)$ and $V_1\subset B(x,R)^{\mathsf{c}}$. Hence if $m\leq   n\leq p<k$, applying \eqref{eq:boundRk} and the strong Markov property successively  at times $D_m$, $H_{V_2}$, $R_{n-m}$, $H_{U_2}$, $R_{p-n}$, $D_1$, and $D_{k-p-1}$, we have
\begin{equation}
\label{eq:splitviastrongMarkov}
\begin{split} 
&P_y\big(D_m\leq H_{V_2}<R_{m}, R_n\leq H_{U_2}\leq T_{n},Z^{(p)}\cap O=\emptyset,Z_{R_{k}}=y\big)
\\&\leq C2^{-m}\sup_{z_1,z_2}P_{z_1}\big(H_{V_2}<R_0)P_{z_2}(R_{n-m}\leq H_{U_2}\leq T_{n-m},Z^{(p-m)}\cap O=\emptyset,Z_{R_{k-m}}=y)
\\&\leq C2^{-n}\sup_{z_1,z_2,z_3,z_4}\!P_{z_1}\big(H_{V_2}<R_0\big)P_{z_2}(R_0<\infty)P_{z_3}\big(H_{U_2}<D_1)P_{z_4}(Z^{(p-n)}\cap O=\emptyset,Z_{R_{k-n}}=y)
\\&\leq C2^{-p}\sup_{z_1,z_2,z_3,z_4}P_{z_1}\big(H_{V_2}<R_0\big)P_{z_2}(R_0<\infty)P_{z_3}\big(H_{U_2}<D_1)P_{z_4}\big(Z^{(0)}\cap O=\emptyset,Z_{R_{k-p}}=y\big)
\\&\leq C2^{-k}\!\!\sup_{z_1,z'_1,z_2,z_3,z_4}\!\!\!P_{z_1}\big(H_{V_2}<R_0\big)P_{z_2}(R_0<\infty)P_{z_3}\big(H_{U_2}<D_1)P_{z_4}(Z^{(0)}\cap O=\emptyset)P_{z'_1}(Z_{R_0}=y),
\end{split}
\end{equation}
where the above supremums range for $z_1$ and $z'_1$ in ${\partial V_1}$, for $z_2$ in ${V_2}$, for $z_3\in{\partial U_1}$, and for $z_4$ in ${U_1}$. If $n<  m\leq p<k$, one can prove the same bound as in \eqref{eq:splitviastrongMarkov} by now applying the strong Markov property at time $R_n$, $H_{U_2}$, $D_{m-n}$, $H_{V_2}$, $R_{p-m}$, $D_1$ and $D_{k-p-1}$. One can proceed similarly in the cases $n\leq p<m\leq k$, $m\leq p<n< k$, $p<m\leq n<k$ or $p<n<m\leq k$, and we leave the details to the reader. Hence, combining this with \eqref{eq:sumnm}, we obtain
\begin{equation}
\label{eq:aftersummingnm}
\begin{split} 
&\sum_{y\in{\partial U_1}}P_y\big(Z_{R_k}=y,H_{U_2}<\infty,H_{V_2}<\infty,Z^{(p)}\in{E_{r,R}^s}\big)
\\&\leq  Ck^22^{-k}\sup_{z_1,z_2,z_3,z_4}P_{z_1}\big(H_{V_2}<\infty\big)P_{z_2}(R_0<\infty)P_{z_3}\big(H_{U_2}<\infty)P_{z_4}(Z^{(0)}\cap O=\emptyset),
\end{split}
\end{equation}
with $z_1,z_2,z_3$ and $z_4$ as before. Moreover, recalling that $R_0=H_{U_1}$, using three times \eqref{eq:hittingvscap}, we have for all $z_1\in{\partial V_1}$, $z_2\in{V_2}$ and $z_3\in{\partial U_1}$
\begin{equation}
\begin{split}
\label{eq:PyHVHU}
P_{z_1}\big(H_{V_2}<\infty\big) &P_{z_3}\big(H_{U_2}<\infty)P_{z_2}(R_0<\infty) \\
&\leq C\left(\frac{\mathrm{cap}(V_2)}{d(z_1,V_2)}\wedge1\right)\left(\frac{\mathrm{cap}(U_2)}{d(z_3,U_2)^{\nu}}\wedge 1\right) \frac{\mathrm{cap}(U_1)}{d(z_2,U_1)^{\nu}}
\\ &\leq C\left(\frac{\mathrm{cap}(V_2)}{d(\partial V_1, V_2)^{\nu}}\wedge 1\right)\left(\frac{\mathrm{cap}(U_2)}{d(\partial U_1,U_2)^{\nu}}\wedge 1\right)\frac{\mathrm{cap}(U_1)}{d(x,V_2)^{\nu}},
\end{split}
\end{equation}
where we used the inequality $d(z_2,U_1)\geq cd(x,V_2)$  for $R\geq C$ since $U_1\subset B(x,\Cr{c:fpploop}R)$ and $V_2\subset B(x,R)^{\mathsf c}$. Noting that $Z^{(0)}\cap O=\emptyset$ is equivalent to $H_{V_1}<H_O$ and combining \eqref{eq:decompositionloops}, \eqref{eq:aftersummingnm} and \eqref{eq:PyHVHU}, we can conclude.
\end{proof}

\section{Absence of clusters in the boundary of a ball}
\label{sec:absence}

In this section we prove the following:

\begin{Prop}
\label{pro:noclusterinboundary}
Assume that $\sup q<\infty$. For all $x \in G,$ $t\in{(0,1]}$, $s\in{(0,ct)}$ and $R\geq C/s$, 
\begin{equation}
\label{eq:noclusterinboundary}
\begin{split}
\P\big(\exists K\in{\Mac_{x,R}^t(\Dia)}, K\cap B(x,(1-s)R)= \emptyset\big)\leq Ct^{-C}s^c.
\end{split}
\end{equation}
\end{Prop}

The role of Proposition~\ref{pro:noclusterinboundary} is as follows: Corollary~\ref{cor:supqimpliessupp}, and more specifically \eqref{eq:bigradiusimpliessmallcapandvol}, shows that a cluster with large diameter typically also has large capacity and renormalized volume, provided that its diameter can be witnessed sufficiently far from the boundary of the ambient ball. Indeed, that result requires enough room to place a ball of macroscopic radius around one endpoint of the connection. It therefore does not directly apply to a component which has macroscopic diameter but lives tangentially inside a thin boundary layer of $B(R)$.

This is precisely the missing case in the implication from large diameter to large capacity and volume in Theorem~\ref{the:equivmacro}. In its proof, we split the clusters in $\Mac_R^t(\Dia)$ according to whether or not they intersect a slightly smaller ball $B((1-t')R)$. Proposition~\ref{pro:caseX=D} below, using Corollary~\ref{cor:supqimpliessupp}, treats the clusters which intersect this smaller ball, whereas Proposition~\ref{pro:noclusterinboundary} shows that, with high probability, no macroscopic cluster is entirely contained in the remaining boundary layer.  The proof of Proposition~\ref{pro:noclusterinboundary} relies on the following consequence of Lemma~\ref{lem:nomacroloop} concerning the loop measure $\mu$.

\begin{Lemme}
\label{lem:boundmukilled}
For all  $x\in{G}$, $R\geq C$, $t\in{[C/R,1]}$ and $s\in{[C/R,ct]}$, we have
\begin{equation}
\label{eq:boundmukilled}
\begin{split}
\mu&\big(\{\gamma:\,B(x,R)\leftrightarrow \partial B(x,(1+t)R)\text{ in }\text{range}(\gamma)\text{ and }B(x,(1-s)R)\cap\text{range}(\gamma)=\emptyset\}\big)
\\&\leq Ct^{-C}s^{c}.
\end{split}
\end{equation}
\end{Lemme}
\begin{proof}
For simplicity, we abbreviate $B_1=B(x,R)$ and $B_s=B(x,(1-s)R)$ throughout this proof. Recall the constant $\Cr{c:fpploop}\leq 1/8$ from Lemma~\ref{lem:nomacroloop}, and we assume that $s\leq \Cr{c:fpploop}t/2$. In view of \eqref{eq:defLambdaL}, if $B(x,R)\leftrightarrow \partial B(x,(1+t)R)$, then there exists $y\in{B(x,(1+\Cr{c:fpploop}t/2)R)\cap \Lambda_{\Cr{c:fpploop}tR/2}}$ such that $U_1^y\leftrightarrow V_1$, where $U_1^y=B(y,\Cr{c:fpploop}tR/2)\cap B_1$ and $V_1=\partial B(x,(1+t)R)$. Recalling \eqref{eq:defcalB}, using a union bound and \eqref{eq:defLambdaL} again, one can thus bound the first line of \eqref{eq:boundmukilled} by
\begin{equation}
\label{eq:boundmukilled2}
\begin{split}
Ct^{-\alpha}\sup_{y\in{B(x,(1+\Cr{c:fpploop}t/2)R)}}&\mu\big(\mathcal{B}_{U_1^y,V_1}^y(B_s)\big)
\\&\leq Ct^{\nu-\alpha}\sup_{y\in{B(x,(1+\Cr{c:fpploop}t/2)R)}}\sup_{z\in{U_1^y}}P_z\big(H_{V_1}<H_{B_s}\big),
\end{split}
\end{equation}
where the last inequality follows from \eqref{eq:capBallBd} and \eqref{eq:nomacroloop} for $U_2=U_1^y$, $V_2=V_1$ and $tR/2$ instead of $R$,  which is in force since $V_1=\partial B(x,(1+t)R)\subset B(y,tR/2)^{\mathsf{c}}$ for all $y\in{B(x,(1+\Cr{c:fpploop}t/2)R)}$ as $\Cr{c:fpploop}\leq 1/2$ and $tR\geq C$. We now bound the probability on the right-hand side of \eqref{eq:boundmukilled2}. By \eqref{eq:hittingvscap} for $r=2^ksR$ and $K=\emptyset$, recalling the constant $\Cr{C:exit}> 1$ from \eqref{eq:hittingvscap}, we have for all $k\in{s\times \{1,\dots,\lceil \log_2(t/s)\rceil\}}$, 
\begin{equation}
\begin{split}
\label{eq:largeobstaclebound2}
\inf_{y\in{B(B_1,2^{k-1}sR)}}&P_y\big(H_{B_s}<H_{\partial B(B_1,2^{k+1}\Cr{C:exit} sR)}\big)
\geq c(2^ksR)^{-\nu}\mathrm{cap}(B_s\cap B(y,2^ksR)),
\end{split}
\end{equation}
where we used that $\partial B(B_1,2^{k+1}\Cr{C:exit} sR)\subset B(y,\Cr{C:exit} 2^ksR)^{\mathsf{c}}$ for all $y\in{B(B_1,2^{k-1}sR)}$ if $sR\geq C$, see \cite[(2.8)]{DrePreRod2}. For each $k\in{s\times \{4,\dots,\lceil \log_2(t/s)\rceil\}}$ and $y\in{B(B_1,2^{k-1}sR)}$, if $z$ is the closest vertex in $B(x,(1-s-2^{k-3}s)R)$ to $y$, then since $d$ is the graph distance, or $\ell^{\infty}$ on $\Z^d$, we have  $d(z,y)\leq 2^{k-1}sR+2^{k-3}sR+sR+C\leq 3\times 2^{k-2}sR$ for $sR\geq C$. In particular, the set $B_s\cap B(y,2^ksR)$ contains the ball $B(z,2^{k-3}sR)$, and so $\mathrm{cap}(B_s\cap B(y,2^ksR))\geq c(2^ksR)^{\nu}$ in view of \eqref{eq:capBallBd}. This implies that one can bound the right-hand side of \eqref{eq:largeobstaclebound2} by a constant. Therefore, for all $p\geq 2$ with $(4\Cr{C:exit})^{p+1}\leq t/s$,  applying the strong Markov property at time $H_{\partial B(B_1,(4\Cr{C:exit})^{p+1}sR)}$ and taking $k=1+p\log_2(4\Cr{C:exit})$ in \eqref{eq:largeobstaclebound2}, we have
\begin{equation}
\label{eq:largeobstaclebound3}
\begin{split}
\sup_{y\in{B(B_1,(4\Cr{C:exit})^psR)}}&P_y\big(H_{\partial B(x,(1+t)R)}<H_{B_s}\big)
\\&\leq (1-c)\sup_{y\in{B(B_1,(4\Cr{C:exit})^{p+1}sR)}}P_y\big(H_{\partial B(x,(1+t)R)}<H_{B_s}\big).
\end{split}
\end{equation}
Iterating \eqref{eq:largeobstaclebound3} until $p=\log(t/s)/\log(4\Cr{C:exit})-1\geq c\log(t/s)+2$ if $s\leq ct$, using the inequality $1-x\leq \exp(-x)$, we deduce that
\begin{equation}
\label{eq:largeobstaclebound4}
\begin{split}
\sup_{y\in{B_1}}&P_y\big(H_{\partial B(x,(1+t)R)}<H_{B_s}\big)
\leq (s/t)^c.
\end{split}
\end{equation}
Noting that $U_1^y\subset B_1$ for all $y$ and recalling $V_1=\partial B(x,(1+t)R)$, we can conclude by combining \eqref{eq:boundmukilled2} and \eqref{eq:largeobstaclebound4}.
\end{proof}

Intuitively, Lemma~\ref{lem:boundmukilled} asserts that for $u$ fixed and $s$ small enough, with high probability all the loops hitting $B(x,R)$ and avoiding $B(x,(1-s)R)$ are included in $B(x,(1+u)R)$. In particular,  the event appearing in \eqref{eq:noclusterinboundary} (for the loop soup instead of the GFF) occurs with high probability for a connected component $K$ of loops killed not only on $B(x,(1-s)R)$, but also on $B(x,(1+u)R)^{\mathsf{c}}$. In other words, there is then a path of length $tR$ consisting of loops constrained to the \lq tube\rq\ $B(x,(1+u)R)\setminus B(x,(1-s)R)$ of side length $(u+s)R$. Taking $u\ll t$ and combining this observation with Proposition~\ref{pro:boxtobox} and Lemma~\ref{lem:goodobstacle}, we are  ready to do the:

\begin{proof}[Proof of Proposition~\ref{pro:noclusterinboundary}]
For each $u\geq s$, on the intersection of the events that there is a connected component $K$ of the intersection of loops in $\mathcal{L}$ with $B(R)$ such that ${\delta}(K)\geq tR$ and $K\cap B((1-s)R)=\emptyset$, and that no loop in $\mathcal{L}$ intersects both $B(R)$ and $\partial B((1+u)R)$ but not $B((1-s)R)$, we have that there is a connected component of loops in $\mathcal{L}_{B((1+u)R)\setminus (B((1-s)R)\cup \partial B((1+u)R))}$ such that $\delta(K)\geq tR$. In particular, if $u\in{[Cs,t]}$, using Lemma~\ref{lem:boundmukilled} and the isomorphism \eqref{eq:isoloopK} twice (first for $K=\emptyset$ and then for $K=B((1-s)R)\cup \partial B((1+u)R)$), we have
\begin{equation}
\label{eq:proof7.2.1}
\begin{split}
\P\big(&\exists K\in{\Mac_{R}^t(\Dia)}:K\cap B((1-s)R)=\emptyset\big)
\\&\leq Cu^{-C}s^{c}+\P^{B((1-s)R)\cup \partial B((1+u)R)}(|\Mac_{(1+u)R}^t(\Dia)|\geq1).
\end{split}
\end{equation}
Further note that if $|\Mac_{(1+u)R}^t(\Dia)|\geq1$, then by \eqref{eq:defLambdaL} there exist $y,y'\in{\Lambda_{\Cr{c:prop41}^3tR}\cap B((1+2t)R)}$ with $d(y,y')\geq tR/2$ such that $B(y,\Cr{c:prop41}^3tR)\leftrightarrow B(y',\Cr{c:prop41}^3tR)$ in $K$. Here, $\Cl[c]{c:prop41}\in{(0,1/2]}$ is a constant chosen small enough so that the assumptions of Proposition~\ref{pro:boxtobox} are satisfied for any $y,y'$ with $d(y,y')\geq tR/2$, when replacing $A,A',r,R,R'$ and $\mathcal{O}$ therein by $B(y,\Cr{c:prop41}^3tR)$, $B(y',\Cr{c:prop41}^3tR)$, $2\Cr{c:prop41}^3tR$, $\Cr{c:prop41}^3tR$, $d(y,y')$ and $O_y\stackrel{\text{def.}}{=}(B((1-s)R)\cup B((1+u)R)^{\mathsf{c}})\cap B(y,\Cr{c:prop41}tR)\setminus B(y,\Cr{c:prop41}^2tR)$. In particular, using a union bound, \eqref{eq:defLambdaL}, \eqref{eq:crossingAB}, the assumption $\sup q<\infty$, and the monotonicity of $K\mapsto \P^K$ following from \eqref{eq:isoloopK}, we deduce that
\begin{equation}
\label{eq:proof7.2.2}
\begin{split}
\P^{B((1-s)R)\cup \partial B((1+u)R)}(|\Mac_{(1+u)R}^t(\Dia)|\geq1)\leq Ct^{-2\alpha}\sup_{y,y':d(y,y')\geq tR/2}P_y(H_{B(y',2\Cr{c:prop41}^3tR)}<H_{O_y}).
\end{split}
\end{equation}
It remains to bound the probability on the right-hand side of \eqref{eq:proof7.2.2}. Assuming $Cs\leq u\leq t/4$, note that $\mathrm{cap}(B(z,4uR)\cap B((1-s)R))\geq c(uR)^{\nu}$ if $z\in{B((1+2u)R)}$ by a similar reasoning as below \eqref{eq:largeobstaclebound2} (replacing $2^ks$ by $4u$ therein). If $z\notin{B((1+2u)R)}$, then $B(z,uR)\subset B((1+u)R)^{\mathsf{c}}$ and  $\mathrm{cap}(B(z,uR))\geq c(uR)^{\nu}$ by \eqref{eq:capBallBd}. In particular, $\mathrm{cap}(B(z,4uR)\cap O_y)\geq c(uR)^{\nu}$ for all $z\in{B(y,\Cr{c:prop41}tR/2)\setminus B(y,\Cr{c:prop41}tR/4)}$, upon assuming $u\leq \Cr{c:prop41}t/16$ and choosing $\Cr{c:prop41}\leq 1/8$. Moreover, for each $y\in{G}$ and path $\pi=(\pi_0,\dots,\pi_n)$ in $\Lambda_{uR}$ (as defined above Lemma~\ref{lem:goodobstacle}) from $B(y,\Cr{c:prop41}tR/4)$ to $B(y,\Cr{c:prop41}tR/2)^{\mathsf{c}}$, we have $\Cr{c:prop41}tR/4\leq d(\pi_0,\pi_n)\leq CnuR$ by \cite[(2.8)]{DrePreRod2}, and so $n\geq ct/u$. Therefore, one directly deduces that for each $y\in{G}$, the set $O_y$ is a $(y,\Cr{c:prop41}tR/2,c,ct/u,c u/t)$-good obstacle set, as defined above Lemma~\ref{lem:goodobstacle}. In particular, for each $y,y'\in{G}$ with $d(y,y')\geq tR/2$, since $B(y',2\Cr{c:prop41}^3tR)\subset B(y,\Cr{c:prop41}tR)^{\mathsf{c}}$, if $R\geq C/s$ and $Cs\leq u\leq ct$ we have by Lemma~\ref{lem:goodobstacle}
\begin{equation}
\label{eq:proof7.2.3}
P_y(H_{B(y',2\Cr{c:prop41}^3tR)}<H_{O_y})\leq \exp(-ct/u).
\end{equation}
We now take $u=t/\log(1/s)$, so that the right-hand side of \eqref{eq:proof7.2.3} is smaller than $s^{c}$. Note that $Cs\leq u\leq ct$ as assumed above \eqref{eq:proof7.2.3} if $s\leq ct$. Using the inequality $\log(1/s)^{C}\leq s^{-c/2}$ if $s\leq c$, we can conclude by combining \eqref{eq:proof7.2.1}, \eqref{eq:proof7.2.2} and \eqref{eq:proof7.2.3}.
\end{proof}

\begin{Rk}
\label{rk:generaldistance}
\begin{enumerate}[1)]
\item \label{rk:generaldistance1}In the proof of Theorem~\ref{the:equivmacro}, the condition that $d$ is the graph distance,or $\ell^{\infty}$ on $\Z^d$, is only used in the proof of Lemma~\ref{lem:boundmukilled}. In fact, for any fixed $b\geq1$, Lemma~\ref{lem:boundmukilled} is still in force when replacing this condition by  
\begin{equation}
\label{eq:conditiondistance}
\begin{gathered}
\text{for all }s\in{(0,c]},R\geq C/s, x\in{G}\text{ and }y\in{B(B(x,R),4sR)},
\\\text{ there exists }z\in{G}\text{ with }B(z,sR)\subset B(y,8sR)\cap B(x,(1-s^b)R).
\end{gathered}
\end{equation}
Indeed, applying \eqref{eq:conditiondistance} for $2^{k-3}sR$ instead of $s$, we obtain for each $y\in{B(B(x,R),2^{k-1}sR)}$ that there exists some $z\in{G}$ with  $B(z,2^{k-3}sR)$ included in both $B(y,2^ksR)$ and $B(R(1-s^b2^{b(k-3)})\subset B(x,(1-s)R)$ if $2^{b(k-3)}\geq s^{1-b}$, as required below \eqref{eq:largeobstaclebound2}. In particular, \eqref{eq:largeobstaclebound3} is in force for all $p$ with $(b-1)\log(1/s)/b+3\log(4\Cr{C:exit})\leq \log(4\Cr{C:exit})(p+1)\leq \log(t/s)$, which is enough to deduce \eqref{eq:largeobstaclebound4} and in turn  Lemma~\ref{lem:boundmukilled}. Therefore, all the results in Sections~\ref{sec:mainresults} and \ref{sec:boxtobox}--\ref{sec:equimacro} remain true on any graph satisfying \eqref{eq:ellipticity}, \eqref{eq:intro_sizeball}, \eqref{eq:intro_Green} and \eqref{eq:conditiondistance} for a general distance $d$. The only exceptions are the lower bound in \eqref{eq:crossingboxseccrit} which additionally requires balls to be connected, see Remark~\ref{rk:ballsconnected}. Note however that there are graphs satisfying \eqref{eq:ellipticity}, \eqref{eq:intro_sizeball}  and  \eqref{eq:intro_Green}, but not satisfying \eqref{eq:conditiondistance} or that balls are connected. Typically one needs to choose a distance $d$ such that $\partial B(x,R)$ contains a point $z$ at distance at least $R/2$ from any other point in $B(x,R)$. 
\item \label{rk:generaldistance2} The classical examples $G_1,G_2,G_3$ and $G_4$ from \cite[(1.4)]{DrePreRod2} of graphs satisfying \eqref{eq:ellipticity}, \eqref{eq:intro_sizeball}  and  \eqref{eq:intro_Green},  actually all satisfy both \eqref{eq:conditiondistance} for some $b\geq1$ and that balls are connected, and so all our results apply to these examples. Indeed, see the proof of \cite[Corollary~3.9]{DrePreRod2}, this is clear for $G_1$, $G_3$ and $G_4$ since $d$ is the graph distance on those graphs. The graph $G_2$ is the product of two graphs, each endowed with the graph distance. We thus need to show that if $\G=\G'\times \G''$ for two graphs $\G'$ and $\G''$ endowed with respective distance $d'$ and $d''$ each satisfying \eqref{eq:conditiondistance} with respective parameters $b'$ and $b''$ and that balls are connected, then $\G$ also satisfies \eqref{eq:conditiondistance} for some $b>0$, and that balls are connected. In view of \cite[Proposition~3.5]{DrePreRod2}, up to exchanging the roles of $\G'$ and $\G''$, the distance $d$ on $\G$ that we need to consider is of the form $d=(d')^{1/a}\vee d''$ for some $a\geq1$. Therefore, denoting by $B'$ and $B''$ the balls for $d'$ and $d''$, we have $B((x',x''),R)=B'(x',R^{a})\times B(x'',R)$, which is clearly connected. Furthermore, for all $s\in{(0,c]}$ $R\geq C/s$, $x=(x',x'')\in{G}$ and $y=(y',y'')\in{B(B(x,R),4sR)=B'(B'(x',R^{a}),(4sR)^{a}))\times B''(B''(x'',R),4sR)}$, by \eqref{eq:conditiondistance} for $\G'$ with $4^{a-1}s^{a}$ instead of $s$ and $R^{a}$ instead of $R$, there exists $z'\in{G'}$ with $B'(z',4^{a-1}s^aR^{a})\subset B'(y',4^{a+1/2}s^aR^{a})\cap B'(x',(1-4^{b'(a-1)}s^{ab'})R^a)$. Furthermore by  \eqref{eq:conditiondistance} for $\G''$ there exists $z''\in{G}$ with $B''(z'',sR)\subset B''(y'',8sR)\cap B''(x'',(1-s^{b''})R)$, and then $B((z',z''),sR)=B'(z',s^aR^{a})\times B''(z'',sR)$ is included in $B(y,8sR)\cap B(x,(1-s^b)R)=\big(B'(y',8^as^aR^{a})\cap B'(x',(1-s^b)^aR^{a})\big)\times \big(B''(y'',8sR)\cap B''(x'',(1-s^b)R)\big)$ if $b=(ab')\vee b''$ since $a\geq 1$, which concludes the proof of \eqref{eq:conditiondistance} for $\G=\G'\times\G''$.
\end{enumerate}
\end{Rk}

\section{Equivalence of macroscopic components}
\label{sec:equimacro}
In this section, we prove Theorem~\ref{the:equivmacro} and Corollary~\ref{cor:macrocontainsbigloops}, that is we prove that with high probability, each cluster with either macroscopic diameter, volume, \emph{or} capacity, actually has macroscopic diameter, volume, \emph{and} capacity, and that from the point of view of the loop soup, it contains at least one macroscopic loop.

\subsection{Clusters with large diameter}

We start by proving \eqref{eq:equivmacro} in the case $\Xt=\Dia$. Actually, in view of the upper bound in \eqref{eq:MR0tight}, it will be enough to prove it when replacing $\Mac_{x,R}^s(\Vol_r)$ by 
\begin{equation}
\label{eq:defVR-}
\Mac_{x,R}^s(\Vol_r^{\geq})\stackrel{\text{def.}}{=}\{K\in{\mathcal{C}_{x,R}}:\, |K|_r\geq (sR/\rr)^{\alpha-\frac\nu2}\},
\end{equation}
and in view of Proposition~\ref{pro:noclusterinboundary}, we can focus on the clusters which do not always stay close to the boundary of $B(R)$. In other words, it will actually be enough to prove the following consequence of  Corollary~\ref{cor:supqimpliessupp}.

\begin{Prop}
\label{pro:caseX=D} 
Assume that $\sup q<\infty$. For all $t\in{(0,1]}$, $s\in{(0,t]}$, $x\in{G}$, $r\geq0$, $R\geq C\rr/s$, and all $\Yt\in{\{\Vol_r^{\geq},\Capa\}}$, abbreviating
\begin{equation*}
\begin{split} 
\Gamma_{x,R}^{t,s}(\Yt)=|\{K\in{\Mac_{x,R}^t(\Dia)\setminus \Mac_{x,R}^s(\Yt)}:\,K\cap {B}(x,(1-t)R)\neq\emptyset\}|,
\end{split}
\end{equation*}
we have
\begin{equation}
\label{eq:caseX=D}
\begin{split} 
\E\big[\Gamma_{x,R}^{t,s}(\Yt)\big]\leq Ct^{-2\alpha}\exp(-cts^{-1}).
\end{split}
\end{equation}
\end{Prop} 
\begin{proof}
Let us first consider the case $\Yt=\Vol_0^{\geq}$, and abbreviate $\bar{R}=R(1-t)$. Recalling the notation $\K^0(y)$ from \eqref{eq:clustersDef} and abbreviating $\K_{\mathrm{in}}(y)=\K^0_{\tilde{B}(x,R)}(y)$, for any $k\in{\N}$ we let %
\begin{equation*}
\begin{split} 
\mathbf{X}_k=\sum_{y\in{B(x,{R})}\cap\Lambda_{r}}1\left\{
\begin{gathered}
    \delta(\K^0(y))\geq tR,\K^0(y)\cap{B}(x,\bar{R})\neq\emptyset,
    \\|\K_{\mathrm{in}}(y)|\in{\big[(2^{-k}s{R})^{\alpha-\frac\nu2},(2^{-k+1}s{R})^{\alpha-\frac\nu2}\big]}
\end{gathered}
\right\}.
\end{split}
\end{equation*}
Recall the constant $\Cr{C:radimpliesvol3}$ from Corollary~\ref{cor:supqimpliessupp}, which we assume w.l.o.g.\ to satisfy $\Cr{C:radimpliesvol3}\geq1$. Let us denote by $S$ the set of $(y,z)\in{G^2}$ such that $y\in{B(x,{R})\setminus B(x,R(1-t/4))}$ and $z\in{\Lambda_{Rt/(2\Cr{C:radimpliesvol3})}\cap B(x,R(1-t/2))}$. Then by \eqref{eq:defLambdaL}, for any $y\in{B(x,{R})\setminus B(x,R(1-t/4))}$ such that $\delta(\K^0(y))\geq tR$ and $\K^0(y)\cap {B}(x,\bar{R})\neq\emptyset$, one can find $z\in{G}$ such that $(y,z)\in{S}$ and $y\leftrightarrow \partial B(z,Rt/(2\Cr{C:radimpliesvol3}))$. In particular, we have
\begin{equation}
\label{eq:boundXk1}
\begin{split} 
\E[\mathbf{X}_k]&\leq \sum_{(y,z)\in{S}}\P\big(|\K_{\mathrm{in}}(y)|\leq (2^{-k+1}sR)^{\alpha-\frac\nu2},y\leftrightarrow{B}(z,Rt/(2\Cr{C:radimpliesvol3}))\big)
\\&\,+\sum_{y\in{B(x,R(1-t/4))}}\P\big(|\K_{\mathrm{in}}(y)|\leq (2^{-k+1}sR)^{\alpha-\frac\nu2},y\leftrightarrow\partial{B}(y,Rt/4)\big)
\end{split}
\end{equation}
Therefore, using \eqref{eq:bigradiusimpliessmallcapandvol}, whose conditions are satisfied if $R\geq 2^k/(\Cl[c]{c:rsmallsR}s)$ for a small enough constant $\Cr{c:rsmallsR}>0$, when replacing therein $x$, $A$, $B$, $r$, $r'$, $R$ and $s$ by $z$, $\{y\}$, $\tilde{B}(x,R)$, $0$, $Rt/(2\Cr{C:radimpliesvol3})$, $Rt/2$ and $2^{-k+2}s/t$ respectively to bound the first line of \eqref{eq:boundXk1}, and by $y$, $\partial B(y,Rt/4)$, $\tilde{B}(x,R)$, $0$, $0$, $Rt/4$ and $2^{-k+3}s/t$ to bound the second line of \eqref{eq:boundXk1}, we have for all $R\geq 2^k/(\Cr{c:rsmallsR}s)$
\begin{equation}
\label{eq:boundXk2}
\begin{split} 
\E[\mathbf{X}_k]&\leq (|S|+|B(x,R(1-t/4))|)(tR)^{-\frac\nu2}\exp(-c2^kts^{-1})
\\&\leq CR^{\alpha-\frac\nu2}t^{-\alpha-\frac\nu2}\exp(-c2^kts^{-1}),
\end{split}
\end{equation}
where the last inequality follows from \eqref{eq:defLambdaL} and \eqref{eq:intro_sizeball}.
Note that for any connected set $K$ such that $\delta(K)\geq tR$ we have $|K|\geq ctR\geq (2^{-\lfloor\log_2(\Cr{c:rsmallsR}sR)\rfloor+2}sR)^{\alpha-\nu/2}$ if $tR\geq C$. Therefore, for every $K\in{\Mac_{x,R}^t(\Dia)}$ such that $K\notin{\Mac_{x,R}^s(\Vol_0)}$ and $K\cap{B}(x,\bar{R})\neq\emptyset$, 
there exists $k\in{[1,\log_2(\Cr{c:rsmallsR}sR)]\cap\mathbb{N}}$ and at least $(2^{-k}sR)^{\alpha-\frac\nu2}$ different $y\in{B(x,{R})}$ such that $K= \K_{\mathrm{in}}(y)$, $\delta(\K^0(y))\geq tR$, $|\K_{\mathrm{in}}(y)|\in{\big[(2^{-k}sR)^{\alpha-\frac\nu2},(2^{-k+1}sR)^{\alpha-\frac\nu2}\big]}$ and $\K^0(y)\cap{B}(x,\bar{R})\neq\emptyset$. Note that each of the previous $k\in{\N}$ and $y\in{B(x,{R})}$ can only correspond to a unique possible set $K\in{\Mac_{x,R}^t(\Dia)\setminus \Mac_{x,R}^s(\Vol_0)}$ (this might not be the case on the renormalized lattice $\Lambda_{r}$ for $r>0$ since two disjoint macroscopic cluster can intersect $B(y,r)$), and hence
\begin{equation}
\label{eq:finalDV}
\begin{split} 
    \E\big[\Gamma_{x,R}^{t,s}(\Vol_0^{\geq})\big]&\leq\sum_{k=1}^{\lfloor\log_2(\Cr{c:rsmallsR}sR)\rfloor}\E[\mathbf{X}_k](2^{-k}sR)^{-\alpha+\frac\nu2}
    \\&\leq Ct^{-\alpha-\frac\nu2}s^{-\alpha+\frac\nu2}\sum_{k\in{\N}}2^{k(\alpha-\frac\nu2)}\exp(-c2^kts^{-1}).
\end{split}
\end{equation}
The last sum above can be upper bounded for $0<s\leq t$ by 
\begin{equation}
\label{eq:finalDV2}
\begin{split} 
2\int_{0}^{\infty}2^{u(\alpha-\frac\nu2)}\exp(-c2^uts^{-1})\mathrm{d}u\leq C\int_{ts^{-1}}^{\infty}v^{\alpha-\frac\nu2-1}\exp(-cv)\mathrm{d}v\leq C\exp(-cts^{-1}),
\end{split}
\end{equation}
where we used the change of variable $v=2^uts^{-1}(\geq 2^u)$, and the part of \eqref{eq:caseX=D} concerning $\Yt=\Vol_0$ readily follows.

We now turn to the case $\Yt\in{\{\Capa,\Vol_r^{\geq}\}}$ with $r>0$, and define
\begin{equation*}
\begin{split} 
\mathbf{Y}=\sum_{y\in{B(x,{R})}}1\left\{
\begin{gathered}
\delta(\K^0(y))\geq tr,\K^0(y)\cap{B}(x,\bar{R})\neq\emptyset\text{ and }
\\\big\{|\K_{\mathrm{in}}(y)|_r\leq (sR/\rr)^{\alpha-\frac\nu2}\text{ or }\mathrm{cap}(\K_{\mathrm{in}}(y))\leq (sR)^{\nu}\big\}
\end{gathered}
\right\}.
\end{split}
\end{equation*}
Then  using \eqref{eq:defLambdaL}, \eqref{eq:bigradiusimpliessmallcapandvol}, and proceeding exactly in the same way as in \eqref{eq:boundXk1} and \eqref{eq:boundXk2} for $k=0$, but now for $r>0$ instead of $0$,  yields if $R\geq C\rr/s$
\begin{equation*}
\begin{split} 
\E[\mathbf{Y}]&%
\leq CR^{\alpha-\frac\nu2}t^{-\alpha-\frac\nu2}\exp(-cts^{-1}).
\end{split}
\end{equation*}
Note that for $\Yt\in{\{\Capa,\Vol_r^{\geq}\}}$ and each $K\in{\Mac_{x,R}^t(\Dia)\cap \Mac_{x,R}^s(\Vol_0^{\geq})}$ such that $K\notin{\Mac_{x,R}^s(\Yt)}$ and $K\cap {B}(x,\bar{R})\neq\emptyset$, there are at least  $(sR)^{\alpha-\frac\nu2}$ different $y\in{B(x,{R})}$ such that $K=\K_{\mathrm{in}}(y)$, $\delta(\K^0(y))\geq tR$, $\K^0(y)\cap {B}(x,\bar{R})\neq\emptyset$, and either $|\K_{\mathrm{in}}(y)|_r\geq (sR/\rr)^{\alpha-\frac\nu2}$ or $\mathrm{cap}(\K_{\mathrm{in}}(y))\leq (sR)^{\nu}$. Therefore using \eqref{eq:caseX=D} for the previous case $\Yt=\Vol_0^{\geq}$,  we have for each $\Yt\in{\{\Capa,\Vol_r^{\geq}\}}$, $r>0$,
\begin{equation}
\label{eq:finalDC}
\begin{split} 
    &\E[\Gamma_{x,R}^{t,s}(\Yt)]
    \\&\quad \leq \E\big[|\{K\in{(\Mac_{x,R}^t(\Dia)\cap \Mac_{x,R}^s(\Vol_0^{\geq}))\setminus \Mac_{x,R}^s(\Yt)}:\,K\cap {B}(x,\bar{R})\neq\emptyset\}|\big]+\E[\Gamma_{x,R}^{t,s}(\Vol_0)]
    \\& \quad \leq (sR)^{-\alpha+\frac\nu2}\E[\mathbf{Y}]+\E[\Gamma_{x,R}^{t,s}(\Vol_0^{\geq})] \leq Cs^{-\alpha+\frac\nu2}t^{-\frac\nu2-\alpha}\exp(-cts^{-1})+Ct^{-2\alpha}\exp(-cts^{-1}),
\end{split}
\end{equation}
and we readily conclude. 
\end{proof}

\subsection{Clusters with large volume}

We now show \eqref{eq:equivmacro} for $\Xt=\Vol_r^{\geq}$ and $\Yt=\Dia$, that is we prove that any connected component with large (renormalized) volume also has large diameter.

\begin{Lemme}
\label{lem:bigvolumeimpliesbigradius}
Assume $\sup q<\infty$. For all $s<t\in{(0,1]}$, $r\geq0$, and $R\geq C\rr/s$
\begin{equation*}
\begin{split} 
\P\big(|\Mac_R^t(\Vol_r^{\geq})\setminus \Mac_R^s(\Dia)|\geq1\big)\leq Cs^{\alpha-\nu}t^{\nu-2\alpha}.
\end{split}
\end{equation*}
\end{Lemme}
\begin{proof}
Let
\begin{equation*}
\begin{split} 
\mathbf{Y}=\sum_{x\in{B(R)\cap\Lambda_{r}}}1\{|\K^0_{\tilde{B}(R)}(x,r)|_r\geq (tR/\rr)^{\alpha-\frac\nu2},\delta(\K^0_{\tilde{B}(R)}(x,r))\leq sR\},
\end{split}
\end{equation*}
then by Markov's inequality we have
\begin{equation*}
\begin{split} 
\E[\mathbf{Y}]&\leq \sum_{x\in{B(R)\cap\Lambda_{r}}}\P\big(|\K^0(x,r)\cap \tilde{B}(x,sR)|_r\geq (tR/\rr)^{\alpha-\frac\nu2}\big)
\\&\leq (tR/\rr)^{-\alpha+\frac\nu2}\sum_{x\in{B(R)\cap\Lambda_{r}}}\E\big[|\K^0(x,r)\cap \tilde{B}(x,sR)|]
\\&\leq C(tR/\rr)^{-\alpha+\frac\nu2}(R/\rr)^{\alpha}(sR/\rr)^{\alpha-\nu}\leq C(tR/\rr)^{\alpha-\frac\nu2}t^{\nu-2\alpha}s^{\alpha-\nu},
\end{split}
\end{equation*}
where in the third inequality we used \eqref{eq:defLambdaL} and \eqref{eq:expectedvolumeinballren}. Noting that for each $K\in{\Mac_R^t(\Vol_r^{\geq})}$ such that $K\notin{\Mac_R^s(\Dia)}$, there are at least $(tR/\rr)^{\alpha-\frac\nu2}$ different $x\in{B(R)\cap\Lambda_{r}}$ such that $|\K^0_{\tilde{B}(R)}(x,r)|\geq (tR)^{\alpha-\frac\nu2}$ and $\delta(\K^0_{\tilde{B}(R)}(x,r))\leq sR$, and thus $\mathbf{Y}\geq(tR/\rr)^{\alpha-\frac\nu2}$. %
\end{proof}

Let us now turn to the upper bound in \eqref{eq:volumeMrs} for $\Xt=\Vol_0^{\geq}$ on $|\Mac_R^s(\Xt)|$ . The corresponding upper bounds for $\Xt\in{\{\Dia,\Capa\}}$ will eventually follow from this and Proposition~\ref{pro:caseX=D}, whereas the lower bounds will follow from Proposition~\ref{thm:macroscopiclocalcluster}.

\begin{Lemme}
\label{lem:volumeMrsV}
Assume $\sup q<\infty$. For all $s\in{(0,1]}$ $R\geq1/s$, and $x\in{G}$,
\begin{equation*}
\begin{split} 
\E\big[|\Mac_{x,R}^s(\Vol_0^{\geq})|\big]\leq Cs^{-\alpha}.
\end{split}
\end{equation*}
\end{Lemme}
\begin{proof}
Let us define
\begin{equation*}
\begin{split} 
\mathbf{X}=\sum_{y\in{B(x,R)}}1\{|\K^0_{\tilde{B}(x,R)}(y)|\geq (sR)^{\alpha-\frac\nu2}\}.
\end{split}
\end{equation*}
Then by \eqref{eq:intro_sizeball} and %
the estimate $\P\big(|\mathcal{K}^0(x)| \geq n\big) \leq  C n^{-\frac{\nu}{2\alpha - \nu}}$, valid for all $x \in G$ and $n \geq 1$, which can be derived similarly as \cite[Corollary~1.6]{DrePreRod5} by combining \eqref{eq:expectedvolumeinballren} at scales $r=r'=0,$ and $R= n^{\frac{2}{2\alpha-\nu}}$, Markov's inequality and the one-arm estimate to distance $R$ implied by $\sup q < \infty$, it follows that
\begin{equation*}
\begin{split} 
\E[\mathbf{X}]\leq \sum_{y\in{B(x,R)}}\P(|\K^0(y)|\geq (sR)^{\frac{2\alpha-\nu}2})&\leq CR^{\alpha}(sR)^{-\frac{2\alpha-\nu}2\times\frac{\nu}{2\alpha-\nu}}
\leq C(sR)^{\alpha-\frac\nu2}s^{-\alpha}.
\end{split}
\end{equation*}
Since $\mathbf{X}\geq (sR)^{\alpha-\frac\nu2}|\Mac_{x,R}^s(\Vol_0^{\geq})|$ by definition, see \eqref{eq:defVR-}, the claim follows.
\end{proof}

The proof of Lemma~\ref{lem:volumeMrsV} only works for $r=0$. Indeed, if $\mathbf{X}_r$ denotes the renormalized version of $\mathbf{X}$ on $\Lambda_{r}$, the bound $\mathbf{X}_r\geq (sR/\rr)^{\alpha-\frac\nu2}|\Mac_{x,R}^s(\Vol_r^{\geq})|$ might fail for $r>0$ since two distinct clusters in $\Mac_{x,R}^s(\Vol_r^{\geq})$ can intersect $\K^0_{\tilde{B}(x,R)}(y,r)$ for some $y\in{B(x,R)}$.

\subsection{Proof of Theorem~\ref{the:equivmacro} and Corollary~\ref{cor:macrocontainsbigloops}}

Combining Propositions~\ref{pro:noclusterinboundary} and \ref{pro:caseX=D} with Lemmas~\ref{lem:bigvolumeimpliesbigradius} and~\ref{lem:volumeMrsV}, we are now ready to finish the 
\begin{proof}[Proof of Theorem~\ref{the:equivmacro}]
We start with the proof of \eqref{eq:equivmacro}, and first note that the case $\Xt=\Yt$ is trivial since $s\mapsto\Mac_R^s(\Xt)$ is decreasing. As we now explain, we have actually essentially already proved all the remaining cases.
\begin{itemize}
    \item Case $\Xt=\Dia$ and $\Yt\in{\{\Vol_r,\Capa\}}$: In view of the upper bound in \eqref{eq:MR0tight}, we have that $\P(\Mac_R^s(\Vol_r^{\geq})\neq \Mac_R^s(\Vol_r))\leq Cs^{\alpha-\frac\nu2}$ as $s\rightarrow0$ uniformly in $r,R$, see \eqref{eq:defMRs} and \eqref{eq:defVR-} for definitions, and so we can consider $\Zt\in{\{\Vol_r^{\geq},\Capa\}}$ instead of $\Yt\in{\{\Vol_r,\Capa\}}$.  By Proposition~\ref{pro:noclusterinboundary},  there exists $c>0$ and $C<\infty$, such that for all $t\in{(0,1]}$, $t'\in{(0,ct]}$, and $R\geq C/t$
    \begin{equation*}
\begin{split}
\P&\big(\Mac_R^t(\Dia)\not\subset\Mac_R^s(\Zt)\big)
\\&\leq Ct^{-C}(t')^c+\P(\exists K\in{\Mac_R^t(\Dia)\setminus \Mac_R^s(\Yt)}:\,K\cap B(x,(1-t')R)\neq\emptyset)
\\&\leq Ct^{-C}(t')^c+\P\big(\Gamma_{x,R}^{t',s}(\Yt)\geq1\big)\leq Ct^{-C}(t')^c+C(t')^{-2\alpha}\exp(-ct's^{-1}),
\end{split}
\end{equation*}
where in the last inequality we used Proposition~\ref{pro:caseX=D} and Markov's inequality for $t=t'$ and Markov's inequality, which is in force if  $R\geq C\rr/s$. Taking $t'=\sqrt{s}$ and assuming that $s\leq ct^{\Cr{C:condst}}$ for a large enough constant $\Cr{C:condst}<\infty$, we can conclude.
\item Case $\Xt=\Vol_r$ and $\Yt=\Dia$: This follows from Lemma~\ref{lem:bigvolumeimpliesbigradius} upon additionally choosing $\Cr{C:condst}\geq 2(2\alpha-\nu)/(\alpha-\nu)$ .
\item Case $\Xt=\Vol_r$ and $\Yt=\Capa$: This follows by combining the previous case $\Xt=\Vol_r$ and $\Yt=\Dia$ (for $s'= c(\eps)t^{\Cr{C:condst}}$ instead of $s$), and the previous case $\Xt=\Dia$ and $\Yt=\Capa$ (for $s'$ instead of $t$, and $s\leq c(\eps)(s')^{\Cr{C:condst}}$), upon doubling the value of $\Cr{C:condst}$.
\item Case $\Xt=\Capa$ and $\Yt=\Dia$: This is trivial since $\mathrm{cap}(K)\geq (tr)^{\nu}$ implies $\delta(K)\geq c(tr)^{\nu}$ for all $K\subset\tilde{\G}$ in view of \eqref{eq:capBallBd}. 
\item Case $\Xt=\Capa$ and $\Yt=\Vol_r$: This follows by noting again that $\mathrm{cap}(K)\geq (tr)^{\nu}$ implies $\delta(K)\geq c(tr)^{\nu}$ for all $K\subset\tilde{\G}$, and using  the previous case $\Xt=\Dia$ and $\Yt=\Vol_r$ (for $ct$ instead of $t$).
\end{itemize}

It remains to prove \eqref{eq:volumeMrs}. Similarly as before, in view of the upper bound in \eqref{eq:MR0tight}, we can consider $\Xt\in{\{\Dia,\Vol_0^{\geq},\Capa\}}$ instead of $\Xt\in{\{\Dia,\Vol_0,\Capa\}}$.  We start with the lower bound, and will actually prove a slightly stronger statement \eqref{eq:MRsXlarge0}, whose proof is not more complicated.
 Abbreviating $\Lambda=B((1-s/10)R)\cap\Lambda_{sR}$, for $0<u<1/10$ to be fixed and $x\in{\Lambda}$,  we denote by $E_x$ the  event that there is a positive loop cluster $K$ in $\mathcal{L}_{\tilde{B}(x,usR)}$ with $|K|_r\geq (\eps sR/(10\rr))^{\alpha-\frac\nu2}$, $\mathrm{cap}(K)\geq (\eps sR/10)^{\nu}$, and $\delta(K)\geq \eps sR/10$. Then, noting that $\mathrm{cap}(K)\geq C(\eps sR)^{\nu}$ implies $\delta(K)\geq \eps sR$  by \eqref{eq:capBallBd}, if $R\geq C(\eta)\rr/(us)$ it follows from Proposition~\ref{thm:macroscopiclocalcluster} (the assumption $\inf p > 0$ therein follows from Cor. \ref{cor:supqimpliessupp} in combination with $\sup q < \infty$)  that 
$\Q(E_x)\geq 1-\eta/4$ upon fixing $\eps\leq cu$ for a small enough constant $c=c(\eta)$. Since the balls $\tilde{B}(x,sR/20)$ are disjoint for $x\in{\Lambda_{sR}}$ by \eqref{eq:defLambdaL}, the events $E_x$, $x\in{\Lambda}$, are independent. Hence by Hoeffding’s inequality, we have for $\eta\in{(0,1)}$, $R\geq C(\eta)\rr/(su)$,  $\eps\leq cu$ and $s\leq c$
\begin{equation}
\label{eq:Ebound}
\begin{split} 
\Q(E)\leq \exp\left(-c\eta^2|\Lambda|\right)\leq \exp(-c\eta^2s^{-\alpha}),\text{ with }E\stackrel{\text{def.}}{=}\Big\{\frac{1}{|\Lambda|}\sum_{x\in{\Lambda}}1\{E_x\}\leq 1-\frac{\eta}{2}\Big\},
\end{split}
\end{equation}
and where we used \eqref{eq:defLambdaL} in the second inequality. Let us also denote by $F_x$ the event that there is no loop cluster connecting $B(x,usR)$ to $\partial B(x,sR/10)$, then by Markov’s inequality and the isomorphism \eqref{eq:isoloopK} we have
\begin{equation}
\label{eq:firstmomentF}
\E_{\Q}\Big[\frac{1}{|\Lambda|}\sum_{x\in{\Lambda}}1_{F_x}\Big]\geq 1-\sup_{x\in{\Lambda}}\P\big(B(x,usR)\leftrightarrow \partial B(x,sR/10)\big)\geq 1-Cu^{\nu/2}\geq 1-\frac{\eta}{4}
\end{equation}
where the second inequality follows from \eqref{eq:crossingannuli} and the last one holds upon fixing $u=c$ for a small enough constant $c=c(\eta)>0$. Further note that for any $\ell\geq1$, on the event that there is no loop in $\mathcal{L}$ connecting $B(x,sR/10)$ to $B(x,\ell sR)$, the event $F_x$ can be replaced by the event $\tilde{F}_x$ that there is no loop cluster for $\mathcal{L}_{\tilde{B}(x,\ell sR)}$ connecting $B(x,usR)$ to $\partial B(x,sR/10)$. Therefore for all $x,y\in{\Lambda}$ with $d(x,y)>2\ell sR$, since $\tilde{F}_x$ and $\tilde{F}_y$ are independent we have
\begin{equation*}
\begin{split}
\text{Cov}\big(1_{F_x},1_{F_y}\big)&\leq C\sup_{z\in{\Lambda}}\P\big(\exists\,\gamma\in{\mathcal{L}}:\,B(z,sR/10)\leftrightarrow B(z,\ell sR)\text{ in }\text{trace}(\gamma)\big)
\leq C\ell^{-\nu}\leq \frac{\eta'\eta}{64},
\end{split}
\end{equation*}
where in the second inequality we used \cite[(3.6)]{DrePreRod8} (whose proof does not require the assumption $\nu\leq \alpha/2$ therein) and \eqref{eq:capBallBd}, and the last one holds upon fixing $\ell=C$ for a large enough constant $C=C(\eta',\eta)$. Therefore,
\begin{equation}
\label{eq:varianceF}
\text{Var}\Big(\frac{1}{|\Lambda|}\sum_{x\in{\Lambda}}1_{F_x}\Big)\leq \frac{1}{|\Lambda|^2}\sum_{x,y\in{\Lambda}:d(x,y)\leq 2\ell sR}\P(F_x)+\frac{1}{|\Lambda|^2}\sum_{x,y\in{\Lambda}}\frac{\eta'\eta}{64}\leq \frac{C}{(\ell s)^{\alpha}|\Lambda|}+\frac{\eta'\eta}{64}\leq  \frac{\eta'\eta}{32},
\end{equation}
where in the second inequality we used \eqref{eq:defLambdaL}, and the last one holds for $s\leq c(\eta,\eta')$. Denoting by $F$ the event that $\frac{1}{|\Lambda|}\sum_{x\in{\Lambda}}1_{F_x}$ is smaller than $1-\eta/2$, it follows from \eqref{eq:firstmomentF}, \eqref{eq:varianceF}, and Chebyshev's inequality that $\Q(F)\leq\eta'/2$. Further note that on the event $E_x\cap F_x$, there is a cluster $K$ in $\Mact_{x,sR/10}^{\eps}(\Xt)$ such that $K\cap\partial B(x,\tfrac{sR}{10})=\emptyset$  for any $\Xt\in{\{\Vol_r,\Capa,\Dia\}}$.  Hence combining this with \eqref{eq:Ebound} and the isomorphism \eqref{eq:isoloopK}, for $\Xt\in{\{\Vol_r,\Capa,\Dia\}}$, $\eta\in{(0,1)}$, $R\geq C(\eta)\rr/s$, $u\leq 1/10$ and $\eps\leq c(\eta)$  we have
\begin{equation}
\label{eq:MRsXlarge0}
\begin{split} 
\P&\Big(\frac{1}{|\Lambda|}\sum_{x\in{\Lambda}}1\{\exists K\in{\Mac_{x,\tfrac{sR}{10}}^{\eps}(\Xt),K\cap \partial B(x,\tfrac{sR}{10})=\emptyset}\}\leq 1-\eta \Big)
\\&\leq \Q\Big(\frac{1}{|\Lambda|}\sum_{x\in{\Lambda}}1\{E_x\cap F_x\}\leq 1-\eta \Big)\leq \exp\left(-c\eta^2s^{-\alpha}\right)+\frac{\eta'}{2}\leq\eta'
\end{split}
\end{equation}
for $s\leq c(\eta,\eta')$. Furthermore, note that each cluster $K\in{\Mac_{x,sR/10}^{\eps}(\Xt)}$ such that $K\cap \partial B(x,{sR/10})=\emptyset$ for some $x\in{\Lambda}$ appears exactly once in $\Mac_{R}^{s\eps/10}(\Xt)$. After a change of variable for $s$, the lower bound in \eqref{eq:volumeMrs} (note that we can assume w.l.o.g.\ $\eta\geq s^{\alpha}$ therein) follows readily from \eqref{eq:MRsXlarge0} (with $\eta=1/2$ therein) and \eqref{eq:defLambdaL}.

Let us now show that $\P(|\Mac_R^s(\Xt)|\geq  \eta^{-1} s^{-\alpha})\rightarrow0$ as $\eta\rightarrow0$, uniformly in $R,s$ for any $\Xt\in{\{\Dia,\Vol_0,\Capa\}}$. If $\Xt=\Vol_0$, this is Lemma~\ref{lem:volumeMrsV}. If $\Xt=\Dia$, by \eqref{eq:defLambdaL}, for each $K\in{\Mac_R^s(\Dia)}$ there exists $x\in{\Lambda_{sR}\cap B(R)}$ such that $K\cap B(x,sR)\neq\varnothing$, and thus there exists $K'\in{\Mac_{x,2sR}^{1/2}}$ such that $K'\subset K$. Therefore, abbreviating $\hat{\Mac}_{x,2sR}^{1/2}=\{K\in{\Mac_{x,2sR}^{1/2}}:\,K\cap B(x,sR)\neq\emptyset\}$, 
\begin{equation*}
\begin{split} 
\E\big[|\Mac_R^s(\Dia)|\big]&\leq\sum_{x\in{\Lambda_{sR}\cap B(R)}}\E\big[|\hat{\Mac}_{x,2sR}^{1/2}(\Dia)|\big].
\end{split}
\end{equation*}
Furthermore for all $s\in{[0,1]}$, $R\geq C/s$ and $x\in{G}$  we have by Lemma~\ref{lem:volumeMrsV} (for $2sR$ instead of $R$ and $1/2$ instead of $s$) and \eqref{eq:caseX=D} (for $\Yt=\Vol_0$, $2sR$ instead of $R$, $1/2$ instead of $t$ and  $s$)
\begin{equation*}
\begin{split} 
\E\big[|\hat{\Mac}_{x,2sR}^{1/2}(\Dia)|\big]&\leq \E\big[|\hat{\Mac}_{x,2sR}^{1/2}(\Vol_0)|\big]+\E\big[|\hat{\Mac}_{x,2sR}^{1/2}(\Dia)\setminus \Mac_{x,2sR}^{1/2}(\Vol_0)|\big]
\leq C.
\end{split}
\end{equation*}
Hence if $R\geq C/s$,
\begin{equation}
\label{eq:boundEMrsD}
\begin{split} 
\E\big[|\Mac_R^s(\Dia)|\big]\leq C|\Lambda_{sR}\cap B(R)|\leq Cs^{-\alpha},
\end{split}
\end{equation}
where we used \eqref{eq:defLambdaL} in the last inequality. Using Markov's inequality, we can conclude the case $\Xt=\Dia$.
Finally, if $\Xt=\Capa$, since any cluster with capacity at least $(sR)^{\nu}$ has diameter at least $csR$ by \eqref{eq:capBallBd}, we can conclude by the previous case $\Xt=\Dia$. 
\end{proof}

Let us now explain how to deduce Corollary~\ref{cor:macrocontainsbigloops}.
\begin{proof}[Proof of Corollary~\ref{cor:macrocontainsbigloops}]
We first  consider the case $\Xt=\Dia$. For $u>0$ to be chosen later, recalling \eqref{eq:defLambdaL}, let us define the set of "thick" subsets $K$ of $\tilde{B}(R)$ by
\begin{equation*}
\begin{split} 
\mathcal{T}=\big\{K\subset{\tilde{B}(R)}:\,\forall y\in{B(2R)\cap\Lambda_{sR}},K\cap B(y,sR)\neq\emptyset\Rightarrow\mathrm{cap}(K\cap B(y,2sR))\geq (uR)^{\nu}\big\}.
\end{split}
\end{equation*} 
Then if $s\geq C/R$ and $u\leq s\leq t/4$,
\begin{equation}
\label{eq:MrtDTc}
\begin{split} 
&\Q\big(\Mac_R^t(\Dia)\cap \mathcal{T}^{\mathsf c}\neq\emptyset\big)
\\&\leq \Q\big(\exists\,y\in{B(2R)\cap\Lambda_{sR}},\exists\,K\in{\Mac_R^t(\Dia)},K\cap B(y,sR)\neq \emptyset,\mathrm{cap}(K\cap B(y,2sR))\leq (uR)^{\nu}\big)
\\&\leq \sum_{y\in{B(2R)\cap\Lambda_{sR}}}\Q\big(\exists\, K\in{\Mac_{y,2sR}^{1/4}}(\Dia), K\notin{\Mac_{y,2sR}^{u/(2s)}(\Capa)},K\cap B(y,sR)\neq\emptyset\big)
\leq Cs^{-\alpha}e^{-cs/u},
\end{split}
\end{equation}
where we used in the third inequality that if $K\in{\Mac_R^t(\Dia)}$ and $K\cap B(y,sR)\neq\emptyset$, then $\delta(K\cap B(y,2sR))\geq sR/2=(2sR)*(1/4)$ if $sR\geq C$ and $t\geq 4s$, and the last inequality follows from \eqref{eq:defLambdaL}, the isomorphism \eqref{eq:isoloopK}, and \eqref{eq:caseX=D} for $\Yt=\Capa$, $t=1/4$ and $u/(2s)$ instead of $s$.

We can thus focus on the case $\Mac_R^t(\Dia)\subset\mathcal{T}$, that is the case where all the clusters with large diameter are thick. The idea will be to show that any thick set is hit with high probability by a long loop.  We first introduce some notation: let us denote by $\tilde{\mathcal{L}}_{R}^{s}$ the point process which consists of all the loops in $\mathcal{L}$ with diameter at most $sR$ which intersect $\tilde{B}(R)$. Let also $\mathcal{S}_{R}^s$ be the union of the traces of all the loops in $\mathcal{L}$ which have diameter at least $sR$, and are entirely included in $\tilde{B}(x,R)$. If there exists a cluster $K\in{\Mac_R^t(\Dia)\setminus \Mac_R^s(\mathcal{L})}$, then by Proposition~\ref{pro:noclusterinboundary} there exists $t'=t'(\eps,t)$ such that $K\cap B((1-t')R)\neq\emptyset$ except on an event with probability $\eps/2$. Define ${\Mac}_R^{t,s}(\Dia)$ as the set of connected components in $\tilde{B}(R)$  with diameter at least $tR$ and which intersect $B((1-t')R)$, for loops in $\tilde{\mathcal{L}}_{R}^{s}$. Note that with probability $1-\eps/2$ by definition any cluster $K$ as before avoids $\mathcal{S}_{R}^s$, and in particular $K\in{\Mac_R^{t,s}(\Dia)}$. Hence since  $\Mac_R^{t,s}(\Dia)$ is $\tilde{\mathcal{L}}_R^s$-measurable, and $\mathcal{S}_R^s$ is independent of $\tilde{\mathcal{L}}_R^s$, we have
\begin{equation}
\label{eq:boundonFu}
\begin{split} 
\Q\big(\Mac_R^t(\Dia)\setminus \Mac_R^s(\mathcal{L})\neq\emptyset,\Mac_R^t(\Dia)\subset\mathcal{T}\big)
 \leq \frac{\eps}{2}+\E_{\Q}\Big[\sum_{K\in{\Mac_R^{t,s}}(\Dia)\cap\mathcal{T}}\Q(K\cap\mathcal{S}^s_R =\emptyset)\Big].
\end{split}
\end{equation}

Let us fix some connected set $K\in{\mathcal{T}}$ with diameter at least $t'R/3$. %
Let $E_y=E_y(K)$ be the event that there is no loop in $\mathcal{L}_{B(y,\Cl{C:loopincluded}sR)}$ which intersect both $K\cap B(y,2sR)$ and $B(y,3sR)^{\mathsf c}$, for some constant $\Cr{C:loopincluded}\in{(3,\infty)}$ to be determined. Then for any $y\in{B(2R)\cap\Lambda_{sR}}$ such that $K\cap B(y,sR)\neq\emptyset$, we have with $K_{2s}= K\cap B(y,2sR)$ that
\begin{equation}
\label{eq:boundEyc}
\begin{split} 
\Q( E_y^{\mathsf c}) &\geq \Q\big(\exists\,\gamma\,\in{\mathcal{L}}:\,K_{2s}\leftrightarrow B(y,3sR)^{\mathsf c}\text{ in }\text{range}(\gamma)\big) 
\\&\quad-\Q\big(\exists\,\gamma\in{\mathcal{L}}: K_{2s}\leftrightarrow B(y,\Cr{C:loopincluded}sR)^{\mathsf c}\text{ in }\text{range}(\gamma)\big)
\\&\geq \exp\big\{-C(\Cr{C:loopincluded}sR)^{-\nu}\mathrm{cap}(K_{2s})\big\}-\exp\big\{-c(3sR)^{-\nu}\mathrm{cap}(K_{2s})\big\}\geq cu^{\nu}s^{-\nu},
\end{split}
\end{equation}
where the second inequality follows from \cite[Lemma~3.1]{DrePreRod8}, and the last one holds upon fixing $\Cr{C:loopincluded}\geq1$ large enough and using that $\mathrm{cap}(K_{2s})\geq (uR)^{\nu}$. Note that \cite[Lemma~3.1]{DrePreRod8} is stated under the additional condition $\nu\leq \alpha/2$, but the only time this condition is used in its proof is to ensure that $\mathrm{cap}(\gamma)\geq cR^{\nu}\log(R)^{-b_2}$ in \cite[(3.7)]{DrePreRod8}, which we do not use here.

Let $(x_1,\dots,x_n)$ be a path in $K$ with diameter at least $t'R/3$. Define recursively $y_1\in{\Lambda_{sR}}$ such that $x_1\in{B(y_1,sR)}$, and for all $i\geq1$, $y_{i+1}$ as the first vertex (with some arbitrary order in case on equality) in $\Lambda_{sR}$ such that $B(y_{i+1},sR)$ is visited by $(x_1,\dots,x_n)$ when first exiting $\bigcup_{1\leq j\leq i}B(y_j,3\Cr{C:loopincluded}sR)$. We iterate the procedure until the first index $p$ such that $\{x_1,\dots,x_n\}\subset\bigcup_{1\leq j\leq p}B(y_j,3\Cr{C:loopincluded}sR)$. Then since the diameter of $\{x_1,\dots,x_n\}$ is at least $t'R/3$ and the diameter of $\bigcup_{1\leq j\leq p}B(y_j,3\Cr{C:loopincluded}sR)$ is at most $6\Cr{C:loopincluded}psR$, we have $p\geq ct'/s$. Further note that $K\cap\mathcal{S}^s_R=\emptyset$ implies $E_y$ for all $y\in{B(x,(1-\Cr{C:loopincluded}s)R)}$ such that $B(y,sR)\cap K\neq\emptyset$,  that the events $E_{y_i}$, $1\leq i\leq p$, are independent since the balls $B(y_i,\Cr{C:loopincluded}sR)$, $1\leq i\leq p$, are disjoint by construction, that $K\cap B(y_i,sR)\neq\emptyset$ for all $1\leq i\leq p$, and that $y_i\in{B(K,sR)\subset B((1-\Cr{C:loopincluded}s)R)}$ if $K\subset B((1-t'/2)R)$ and $s\leq ct'$. Therefore, for any $K\in{\mathcal{T}}$ with diameter at least $t'R/3$ and such that $K\subset B((1-t'/2)R)$, we deduce from \eqref{eq:boundEyc} that  if $s\leq ct'$
\begin{equation*}
\begin{split} 
\Q(K\cap\mathcal{S}^s_R =\emptyset)\leq \prod_{i=1}^p\Q(E_{y_i})\leq (1-cu^{\nu}s^{-\nu})^{ct'/s}\leq\exp(-cu^{\nu}t's^{-\nu-1}).
\end{split}
\end{equation*}
Note that this bound remains true if $K\in{\mathcal{T}}$ has diameter at least $tR$ and intersect $B((1-t')R)$, since we can then apply it to $K\cap B((1-t'/2)R)$, which is in $\mathcal{T}$ and has diameter at least $t'R/3$. 
One can thus bound the last expectation in \eqref{eq:boundonFu} by
\begin{equation*}
\begin{split} 
 \E_{\Q}[|\Mac_R^{t,s}(\Dia)|]e^{-cu^{\nu}t's^{-\nu-1}}
&\leq \E_{\Q}[|\Mac_R^{t}(\Dia)|]e^{-cu^{\nu}t's^{-\nu-1}} \leq Ct^{-\alpha}e^{-cu^{\nu}t's^{-\nu-1}},
\end{split}
\end{equation*}
where the last inequality follows from \eqref{eq:boundEMrsD} and the isomorphism \eqref{eq:isoloopK} upon assuming $tR\geq 1$. We now choose $u=(s^{\nu+2}/t')^{\frac1{\nu+1}}$ and then by combining the previous equation with \eqref{eq:MrtDTc} and \eqref{eq:boundonFu} we have if $s\leq ct'$ (which implies $u\leq s$) and $R\geq C/s$
\begin{equation*}
\begin{split} 
\Q\big( \Mac_R^t(\Dia)\setminus \Mac_R^s(\mathcal{L})\neq\emptyset\big)\leq \frac{\eps}{2}+Cs^{-\alpha}\exp\big(-c(t'/s)^{\frac1{\nu+1}}\big)\leq \eps
\end{split}
\end{equation*}
for $s=s(t',\eps)$ small enough. When $\Xt=\Capa$, \eqref{eq:macrocontainsbigloops} follows from \eqref{eq:macrocontainsbigloops} for $\Xt=\Dia$ and the observation that $\mathrm{cap}(K)\geq tR^{\nu}$ implies $\delta(K)\geq ctR$. Finally, when $\Xt=\Vol_r$, \eqref{eq:macrocontainsbigloops} follows from Lemma~\ref{lem:bigvolumeimpliesbigradius} (for $s'$ instead of $s$) and \eqref{eq:macrocontainsbigloops} for $\Xt=\Dia$ (and for $s'$ instead of $t$), by first taking $s'=s'(\eps,t)$ small enough and then $s=s(\eps,s')$ small enough. 
\end{proof}

\begin{Rk}
\label{rk:otherchoicesforE}
\begin{enumerate}[1)]
\item\label{rk:Vr-} As the attentive reader will have noticed, the above proofs also work when additionally allowing $\Xt$ or $\Yt$ to be equal to $\Vol_r^{\geq}$ (see \eqref{eq:defVR-}) in Theorem~\ref{the:equivmacro} and Corollary~\ref{cor:macrocontainsbigloops}. In other words, every cluster with large volume has large capacity, large diameter and contains a large loop, but its volume cannot be too large. 
\item Theorem~\ref{the:equivmacro}, Corollary~\ref{cor:macrocontainsbigloops}, Proposition~\ref{cor:crossingboxseccrit} and Proposition~\ref{pro:noclusterinboundary} are still valid when replacing $\P$ by $\P^{B(R)^{\mathsf{c}}}$ and $\Q$ by the law of $\mathcal{L}_{\tilde{B}(R)}$ under $\Q$. In particular, one can also deduce Theorem~\ref{the:Hausdorff} when replacing $\P$ by $\P^{B(\sigma(n))^{\mathsf{c}}}$. %
The proofs are similar, and in fact, many of them are easier when killing on $B(R)^{\mathsf{c}}$. For instance, one does need the arguments from Section~\ref{sec:clusterinside} or Lemma~\ref{lem:boundmukilled}. We leave the details to the reader, and only note that the bounds from Proposition~\ref{cor:crossingboxseccrit} for $a=0$ and when additionally killing on $\partial B(N)$ were already obtained below \eqref{eq:crossingboxseccrit} in view of the isomorphism \eqref{eq:isoloopK}.
\end{enumerate}
\end{Rk}

\section{Scaling limits on \texorpdfstring{$\Z^d$}{Zd}}
\label{sec:proofhausdorffdim}

In this section, we specialize to $\G=\Z^d$, $d\in\{3,4,5\}$, and prove Theorem~\ref{the:Hausdorff} and Corollary~\ref{cor:loopscaling}. We start by introducing some notation in relation to sizes and distances involving subsets of $\R^d$, and then proceed in Section~\ref{sec:subdistanced} to define the distance $\mathbf{d}$ alluded to in Section~\ref{subsec:scalinglimitintro}. We also state therein that, when endowed with $\mathbf{d}$, the space $\mathcal{S}$ from above \eqref{eq:renormalizedlevelset-intro} is a Polish space (the proof is included in Appendix~\ref{App:dmetric}), and exhibit a useful collection of compacts in that space, see Lemma~\ref{L: compact}. In Section~\ref{sec:setup}, we detail the exact discrete and metric graph setup we will use, and gather some renormalized consequences of the results of Sections~\ref{sec:boxtobox}--\ref{sec:equimacro}. We then use these ingredients in Section~\ref{subsec:pf-mainresults-crit}, together with Theorem~\ref{the:equivmacro} and the other discrete results from Sections~\ref{sec:boxtobox}--\ref{sec:equimacro}, to deduce tightness of the family $(\mathscr{E}_n^{\geq0})_{n\in{\N}}$ and finish the proofs of Theorem~\ref{the:Hausdorff} and Corollary~\ref{cor:loopscaling}. In particular, this entails the Hausdorff dimension lower bounds, for which we use energy methods.

We endow $\R^d$, $d\in{\{3,4,5\}}$, with the $\ell^{\infty}$-distance, that we simply denote by $d$, denote by $B(x,r)$ 
 the closed ball in $\R^d$ centred at $x$ and with radius $r$, and abbreviate $B(r)=B(0,r)$. In particular if $x\in{\Z^d}$, we have $\tilde{B}(x,r)=B(x,r)\cap\tilde{\Z}^d$, see Section~\ref{subsec:metric}. For $K\subset\R^d$, we also denote by $B(K,r)$ the union of the sets $B(x,r)$ over all $x\in{K}$.  We  write $d(x,K)=\inf_{y\in{K}}d(y,x)$ for any $K\subset\R^d$ and $x\in \R^d$,  with the convention $d(x,\emptyset)=\infty$. For $B \subset \R^d$ we write $\partial B$ to denote the topological boundary of $B$. If $K \subset \R^d$ we write $\text{dim}_H (K)$ and $\text{dim}_M (K)$ for its Hausdorff and Minkowski (or box-counting) dimension, the latter being defined when the upper and lower Minkowski dimensions coincide; see \cite[Chapter~4]{MR2604525}. For $K,K' \subset \R^d$ we write $d_H(K,K')$ for their Hausdorff distance relative to $d$. Recall this is the infimum over $\varepsilon >0$ such that $K$ is included in the closed $\varepsilon$-neighbourhood of $K'$, and vice versa. We denote by $\delta(K)=\sup_{x,y\in{K}}d(x,y)$ the diameter of $K$ and set $\delta(\emptyset)=0$. By definition of $d_H$, one has for arbitrary $K,K' \subset \R^d$ that
\begin{equation}
\label{eq:d_H-diam}
     |\delta(K)- \delta(K')| \leq 2 d_H(K,K')
\end{equation}
(the factor $2$ is necessary, for instance consider $K=\{0\}$ and $K'=B(0,1)$). In particular, the diameter is (Lipschitz-)continuous with respect to $d_H$. The choice of sup-norm to measure diameter and Hausdorff distance is a matter of convenience. All our results remain valid for other choices of (equivalent) norm, e.g.~the Euclidean one.

\subsection{The distance \texorpdfstring{$\mathbf{d}$}{d} and compactness}
\label{sec:subdistanced}\label{sec:dist}

Recall the set $\mathcal{S}$ from above \eqref{eq:renormalizedlevelset-intro}.  For an unordered family $\mathscr{E}=(K_i)_{i\in{I}}\in \mathcal{S} $, with a slight abuse of notation we will often write $K\in{\mathscr{E}}$, or say that $K$ is an element of $\mathscr{E}$, if $K\in{\{K_i,i\in{I}\}}$. Note that the reason we define $\mathcal{S}$ as a collection of unordered families $(K_i)_{i\in{I}}$ instead of sets $\{K_i,i\in{I}\}$ is that we want to allow the same set to appear multiple times in $\mathscr{E}$, that is we allow $K_i=K_j$ for $i\neq j$. This does not occur for the family $\mathscr{E}_n^{\geq 0}$ from \eqref{eq:renormalizedlevelset-intro} that we consider in this article, but can in theory occur for its limit from Theorem~\ref{the:Hausdorff}, and occurs with positive probability for the family $\mathscr{L}_n$ from \eqref{eq:deftildeEL}. More precisely, if $\mathscr{E}^{(n)}=(K_i^{(n)})_{i\in{I^{(n)}}}$ converges to $\mathscr{E}=(K_i)_{i\in{I}}$, the multiplicity of $K_i$ in $\mathscr{E}$ will indicate the maximal number of pairwise eventually distinct sequences $(i_n)_{n\in{\N}}$ such that $K_{i_n}^{(n)}$ converges to the same set $K_i$.

We write $|A|$ for the cardinality of a set $A$, and for $\mathscr{E}=(K_i)_{i\in{I}}\in \mathcal{S} $ and $\varepsilon > 0$, we write 
\begin{equation}
\label{eq:largeeps}
I_{>\varepsilon}\stackrel{\text{def.}}{=}\{i\in{I}:\,\delta(K_i)>\varepsilon\},\ \mathscr{E}_{> \varepsilon}\stackrel{\text{def.}}{=}(K_i)_{i\in{I_{>\varepsilon}}}\text{ and }N_{>\eps}(\mathscr{E})\stackrel{\text{def.}}{=}|I_{>\eps}|
\end{equation}
for the unordered family of all sets in $\mathscr{E}$ with diameter strictly larger than $\varepsilon$, and the number of such sets. By definition of $\mathcal{S}$, $\mathscr{E}\in \mathcal{S} $ implies that $N_{>\eps}(\mathscr{E})$ is finite for all $\varepsilon > 0$. For each $i\in{I}$, since $\delta(K_i)>0$ by definition of $\mathcal{S}$, we deduce that the multiplicity of $K_i$ in $\mathscr{E}$ is finite.

We now introduce a distance $\mathbf{d}$ on $\mathcal{S}$, which turns it into a Polish space, see Lemma~\ref{L:d-metric} below for a proof. For index sets $I,J$ we call $\pi\subset I\times J$ a matching of $(I,J)$ if for each $i\in{I}$,  there exists at most one $j \in J$ such that $(i,j) \in \pi$, and  for each $j\in{J}$, there exists at most one $i \in I$ such that $(i,j) \in \pi$. 
We denote by $I^{\pi}=\{i\in{I}:\, \not\exists j \in J: (i,j) \in \pi\}$ and $J^{\pi}=\{j\in{J}:\,\, \not\exists i \in I: (i,j) \in \pi\}$ the sets of respective unmatched indices. With a slight abuse of notation, for each $i\in{I\setminus I^{\pi}}$, we will often write $\pi(i)$ for the unique $j\in{J}$ such that $(i,j)\in{\pi}$.
If $\mathscr{E}=(K_i)_{i\in{I}}$ and $\mathscr{E}'=(K'_j)_{j\in{J}}$ are two unordered families in $\mathcal{S}$, we define 
\begin{equation}
\label{eq:defdistance}
\begin{split} 
{\mathbf d}(\mathscr{E},\mathscr{E}')=\inf_{\pi}\max\Big(   2 \sup_{(i,j)\in{\pi}} d_H(K_i,K'_j),\sup_{i\in{I^{\pi}}}{\delta}(K_i),\sup_{j\in{J^{\pi}}}{\delta}(K'_j)\Big),
\end{split}
\end{equation}
where the infimum is taken over all matchings $\pi$ of $(I,J)$, and we use the convention $\sup\emptyset=0$. Note that $\mathbf{d}(\mathscr{E},\mathscr{E}')$ does not depend on the specific choice of the representative families $(K_i)_{i\in{I}}$ and $(K'_j)_{j\in{J}}$ for the unordered families $\mathscr{E}$ and $\mathscr{E}'$. The case $\pi=\emptyset$ is valid in \eqref{eq:defdistance}, and therefore $\mathbf{d}(\mathscr{E},\mathscr{E}') \leq \sup_K \delta(K)$ with the supremum ranging over all $K$ in either $\mathscr{E}$ or $\mathscr{E}'$. The factor $2$ appearing in \eqref{eq:defdistance} comes from \eqref{eq:d_H-diam}, and is necessary for the triangle inequality to be satisfied, see \eqref{eq:2appears}. We stress that \eqref{eq:defdistance} is well-defined if say $\mathscr{E}'=\emptyset$, in which case there is only one possible matching $\pi=\emptyset$, and $\mathbf{d}(\mathscr{E},\emptyset)= \sup_{K\in \mathscr{E}} \delta (K)$;  
 It is instructive to consider  the $\mathbf{d}$-distance between simple families, for instance consisting of boxes of small radius. Let $B_n^x=(x+[-1/n,1/n]^d)\cap[-1,1]^d$. Then for all finite sets $A,A'\subset [-1,1]^d$, $\mathbf{d}((B_n^x)_{x\in{A}}, (B_n^y)_{y\in{A'}})\to 0$ as $n \to \infty$, i.e.~two such families of boxes are $\mathbf{d}$-close (even though $B_n^x$ and $B_n^y$ are macroscopically separated in Hausdorff sense for large enough $n$ whenever $x \neq y$). Indeed, both families have the same limit $\emptyset$ as $n \to \infty$ relative to $\mathbf{d}$. The distance $\mathbf{d}$ thus becomes increasingly oblivious to sets of small diameter, and only describes closeness of macroscopic sets.

Recall the notation $I_{>\eps}$ from \eqref{eq:largeeps}. We will frequently use the following simple consequence of \eqref{eq:defdistance} in the sequel: for any $\varepsilon > 0$, there is a matching  $\pi=\pi(\varepsilon)$ such that the maximum on the right-hand side of \eqref{eq:defdistance} is bounded by $\mathbf{d} + \varepsilon$ where $\mathbf{d}= \mathbf{d}(\mathscr{E},\mathscr{E}')$ for $\mathscr{E}=(K_i)_{i\in{I}}$ and $\mathscr{E}'=(K'_i)_{i\in{J}}$. Moreover, for any such matching $\pi$, all indices in $I_{> \mathbf{d}+ \varepsilon}$ and $J_{> \mathbf{d}+ \varepsilon}$ must be matched, i.e.~they have empty intersection with $I^{\pi}$ and $J^{\pi}$, respectively.  For, otherwise picking an unmatched~$i_0 \in I^{\pi} \cap I_{>  \mathbf{d}+ \varepsilon} $ would entail that $\sup_{i\in{I^{\pi}}}{\delta}(K_i) \geq \delta(K_{i_0}) > \mathbf{d}+ \varepsilon$, violating the choice of $\pi$ (a similar reasoning applies to any index in $ J^{\pi} \cap J_{>  \mathbf{d}+ \varepsilon}$).

\begin{Lemme}
    \label{L:d-metric} The pair $(\mathcal{S},\mathbf{d})$ defines a complete, separable, metric space.
\end{Lemme}
The proof of Lemma~\ref{L:d-metric} is elementary but somewhat lengthy, and we defer it to Appendix~\ref{App:dmetric}.

\begin{Rk}[The choice of distance $\mathbf{d}$]
\label{rk:choiceofd}
The definition \eqref{eq:defdistance} is inspired by \cite[Section 2.7]{zbMATH07114711} which deals with planar clusters (the boundaries of which are paths), see also \cite[Section~4.1.2]{MR4091511} for a definition closer to our distance $\mathbf{d}$. In fact the distance functions defined in \cite{zbMATH07114711,MR4091511} need to be  modified by a factor $2$ as in \eqref{eq:defdistance} in order for the triangle inequality to hold. Convergence for the distance $d_H^*$ from  \cite[(1.6)]{cai2026scalinglimitcriticalloop} is equivalent to convergence for $\mathbf{d}$, at least for families with the same cardinality and a uniform positive lower bound on their diameters, as studied therein.

A different choice following \cite{MR1712629} would be to replace $\mathbf{d}$ by  $\mathbf{d}^\prime$ and $\mathscr{E}$ by $\mathscr{E}_{>\eps}$ for a fixed $\eps$, where
$$\mathbf{d}^\prime(\mathscr{E},\mathscr{E}')\stackrel{\text{def.}}{=}\sup_{K\in{\mathscr{E}}}\inf_{K'\in{\mathscr{E}'}}d_H(K,K')\vee \sup_{K'\in{\mathscr{E}'}}\inf_{K\in{\mathscr{E}}}d_H(K,K').$$
Note that in order to ensure $\mathbf{d}'(\mathscr{E},\mathscr{E}')>0$ whenever $\mathscr{E}\neq\mathscr{E}'$, one actually needs to define the distance $\mathbf{d}'$ on sets $\mathscr{E}$ which are $d_H$-closed and consist of  non-empty compact subsets of $[-1,1]^d$, instead of unordered families $\mathscr{E}$ of compacts with positive diameter for $\mathbf{d}$. The two distances $\mathbf{d}$ and $\mathbf{d}^\prime$ are of quite different nature, and $\mathbf{d}$ is designed to express `tightness' of the number of macroscopic clusters, at the cost of forgetting about microscopic ones, whereas $\mathbf{d}^\prime$ has rather opposite virtues. For instance, consider the sequence $\mathscr{E}^{(n)}=\{[\tfrac1n,1],[\tfrac1{n+1},1],\dots,[\tfrac1{2n},1]\}$ (viewed as subsets of $ \R^d$, or even assuming $d=1$ for simplicity). This sequence is $\mathbf{d}^\prime$-Cauchy, as $\mathbf{d}^\prime(\mathscr{E}^{(n)}, \mathscr{E}^{(m)}) \leq 1/(n\wedge m)$, and in fact it simply converges to $\{[0,1] \}$. In contrast, the unordered family corresponding to $\mathscr{E}^{(n)}$ has no subsequential limit under $\mathbf{d}$, and it is not $\mathbf{d}$-Cauchy since $\mathbf{d}(\mathscr{E}^{(n)}, \mathscr{E}^{(m)}) \geq \tfrac12$ for all $m>n \geq 2$. %
\end{Rk}

The following family of compact sets in $(\mathcal{S}, \mathbf{d})$ will be useful for tightness arguments; see also Remark~\ref{R:compact} below as to their necessity. Recall from \eqref{eq:largeeps} that $N_{>s}(\mathscr{E})$ denotes the number of sets in the family $\mathscr{E}$ having diameter larger than $s$. Given two sequences $s= (s_k)_{k \geq 1}\subset[0,\infty)^{\N}$ and $M=(M_k)_{k \geq 1}\subset\N_0^{\N}$ with $s_k, M_k > 0$ for all $k$, we define 
\begin{equation}
\label{eq:compact-def}
\mathcal{K}_{s,M}= \{ \mathscr{E} \in \mathcal{S}: \forall k, \, N_{>s_k}(\mathscr{E}) \leq M_k  \} .
\end{equation}
Below, compactness refers to (sequential) compactness in the metric space $(\mathcal{S}, \mathbf{d})$, cf.~Lemma~\ref{L:d-metric}.
\begin{Lemme}\label{L: compact}
For all $s,M$ such that $s_k \downarrow 0$ as $k \to \infty$, the set $\mathcal{K}_{s,M} \subset \mathcal{S}$ is compact.
\end{Lemme}
\begin{proof}
Let $(\mathscr{E}^{(n)})_n=((K_i^{(n)})_{i\in{I^{(n)}}})_n \subset  \mathcal{K}_{s,M}^{\N}$ be a sequence. Since $\mathscr{E}^{(n)}$ is an unordered family, we may assume w.l.o.g.\ that $I^{(n)}=\{1,\dots,|I^{(n)}|\}$ if $|I^{(n)}|<\infty$, and $I^{(n)}=\N$ otherwise, and that the sets in $\mathscr{E}^{(n)}$ are labelled according to decreasing diameter, i.e.\ $\delta(K^{(n)}_{1}) \geq \delta(K^{(n)}_{2}) \geq \dots$. If $|I^{(n)}|<\infty$, we further define $K^{(n)}_i=\emptyset$ for all $i\geq |I^{(n)}|+1$. Since $[-1,1]^d$ is compact (for the sup-norm), the family of compact subsets of $[-1,1]^d$, equipped with $d_H$ is also compact. Combining this with a straightforward diagonal argument, we find a subsequence $\sigma: \mathbb{N} \to \mathbb{N}$ and for each $i\in{\N}$ a compact set $K_i \subset [-1,1]^d$ such that
\begin{equation}
\label{eq:compact-pf1}
d_H(K^{(\sigma(n)) }_i, K_i) \tend{n}{\infty}0\text{ for all }i\in{\N}.
\end{equation} 
By continuity of $\delta(\cdot)$ w.r.t.~$d_H$, see \eqref{eq:d_H-diam}, one has that $\delta(K_1) \geq \delta(K_2) \geq \dots$ so the sets $K_i$ are ordered according to decreasing diameter, which will be convenient below. We now define
\begin{equation}
\label{eq:compact-pf2}
I^{(\infty)}=\{ i\in{\N} : \delta(K_i) > 0 \}\text{ and }{\mathscr{E}}^{(\infty)} = (K_i)_{i\in{I}}
\end{equation}
obtained by removing all limiting sets having vanishing diameter and proceed to verify that
\begin{align}
& \label{eq:compact-pf3} {\mathscr{E}}^{(\infty)} \in  \mathcal{K}_{s,M},  \text{ and}\\
&\label{eq:compact-pf4} \mathbf{d}(\mathscr{E}^{(\sigma(n))}, {\mathscr{E}}^{(\infty)} )\tend{n}{\infty}0,
\end{align}
by which the proof is complete.

We start with \eqref{eq:compact-pf3}. By construction in \eqref{eq:compact-pf2} it is clear that each set $K \in {\mathscr{E}}^{(\infty)} $ satisfies $\delta(K)>0$, and $K \subset [-1,1]^d$ is compact. It remains to argue that $N_{>s_k}({\mathscr{E}}^{(\infty)} ) \leq M_k$ holds for all $k$, which will imply at once together with the preceding observations that ${\mathscr{E}}^{(\infty)} \in \mathcal{S}$, and moreover that \eqref{eq:compact-pf3} holds.  Since the $K_i$'s are ordered we need to argue that $\delta(K_{M_k+1})\leq s_k$. Combining \eqref{eq:compact-pf1} and \eqref{eq:d_H-diam}, one obtains that $\delta(K^{(\sigma(n))}_{M_k+1}) \to \delta(K_{M_k+1}) \leq s_k$ as $n \to \infty$, since $N_{>s_k}( \mathscr{E}^{(\sigma(n))}) \leq M_{k}$ for all $n$ by the assumption that  $\mathscr{E}^{(\sigma(n))} \in \mathcal{K}_{s,M}$, cf.~\eqref{eq:compact-def}. %

We now show \eqref{eq:compact-pf4}. Let $\varepsilon > 0$. We need to choose $\Cl{C:compact}(\varepsilon)$
such that $\mathbf{d}(\mathscr{E}^{(\sigma(n))}, {\mathscr{E}}^{(\infty)} ) \leq \varepsilon$ whenever $n \geq \Cr{C:compact}$. Owing to \eqref{eq:defdistance} it is enough to construct for such $n$ a matching $\pi=\pi(n, \varepsilon)$ comprising all sets $K_i^{(\sigma(n))}$ and $K_i$ having diameter larger than $\varepsilon$, in such a way that 
\begin{equation}
\label{eq:compact-pf5}
\sup_{(i,j) \in \pi}d_H(K_i^{(\sigma(n))}, K_j) \leq \varepsilon/2.
\end{equation}
We first pick $k=k(\varepsilon)$ such that $s_k < \varepsilon$, which exists by assumption on $(s_k)$. This defines $M_k$ and since $\mathscr{E}^{(\sigma(n))} \in \mathcal{K}_{s,M}$ for all $n$, we have by \eqref{eq:compact-def} and our choice of $k$ that for all $n$, $\delta (K_i^{(\sigma(n))})> \varepsilon$ implies $i \leq N_{>s_k}(\mathscr{E}^{(\sigma(n))}) \leq M_k$. Similarly by \eqref{eq:compact-pf3} we have $\delta (K_i)> \varepsilon$ implies $i \leq M_k$. We now pick $\Cr{C:compact}(\varepsilon)$ large enough so that the left-hand side of \eqref{eq:compact-pf1} is at most $\varepsilon/2$ for all $i \leq M_k$ and $n \geq \Cr{C:compact}(\varepsilon)$. Finally we choose $\pi=\pi(n, \varepsilon)$ the diagonal matching comprising all pairs $(i,i)$ for indices $i\leq M_k$. Combining the previous ingredients readily yields \eqref{eq:compact-pf5}.
\end{proof}

\begin{Rk}\label{R:compact}
    Note that the space $(\mathcal{S},\mathbf{d})$ itself is not compact. For instance, if $K_i$, $i\in{\N}$, is an infinite family of disjoint and connected compact subsets of $[-1,1]^d$,  each with diameter at least $1/2$%
, and $\mathscr{E}^{(n)}=\{K_i,i\leq n\}$, then $\mathbf{d}(\mathscr{E}^{(n)},\mathscr{E}^{(m)})\geq 1/2$ for all $n\neq m$. Indeed, any matching $\pi$ between $\{1,\dots,n\}$ and $\{1,\dots,m\}$ must leave at least $|n-m| >0$ points unmatched, each corresponding to a set with diameter at least $1/2$.
\end{Rk}

\subsection{Setup and consequences of Theorem~\ref{the:equivmacro} on \texorpdfstring{$\Z^d$}{Zd}}
\label{sec:setup}\label{sec:usefulresults}
In the following we consider $\Z^d$, $d\in{\{3,4,5\}}$, respectively the associated metric graph $\tilde{\Z}^d$, as the subset of $\R^d$ consisting of points with integer coordinates, respectively with at most one non-integer coordinate. More precisely since we chose the convention $\lambda_{\cdot,\cdot}=1/2$, for any edge $e=\{x,y\}$ of $\Z^d$, the length of the cable $I_e$ in $\tilde{\Z}^d$ is $1$ (=$1/(2\lambda_{x,y})$), and so $I_e$ can be identified with the set $\{x+t(y-x),\,t\in{[0,1]}\}\subset\R^d$.  Recall our standing low-dimensional assumption \eqref{eq:lowdim}, and in particular on account of  \cite[Theorem 1.5.4]{MR1117680} and \cite[Theorem~1.1]{cai2024onearm} we have that
\begin{equation}
\label{eq:Zdinmacro}
\begin{gathered}
\text{the graph }\Z^d, d\in{\{3,4,5\}},\text{ satisfies the conditions from}
\\\text{Theorem~\ref{the:equivmacro} with }\nu=d-2\text{ and }\alpha=d\text{ therein}.
\end{gathered}
\end{equation}
In the proofs below, we can thus apply the results of Sections~\ref{sec:boxtobox}--\ref{sec:equimacro} directly for $\G=\Z^d$.

In order to study scaling limits as in Theorem~\ref{the:Hausdorff}, it will often be useful to consider the following renormalized lattice
\begin{equation}
\label{eq:Znd}
\Z_n^d\stackrel{\text{def.}}{=}\tfrac1n\Z^d\text{ and }\tilde{\Z}_n^d\stackrel{\text{def.}}{=}\tfrac1n\tilde{\Z}^d\text{ for all }n\geq1.
\end{equation}
In particular, the clusters in the families $\mathscr{E}^{\geq 0}_{n}$ and $\mathscr{E}_n$ from \eqref{eq:renormalizedlevelset-intro} and \eqref{eq:deftildeEnC} live on $\Z_n^d$. We will in fact not work directly with these families, but with slightly different ones living on $\tilde{\Z}_n^d$, which we now introduce. For $n \geq 1$ and $x\in{\Z_n^d}$, let
\begin{equation}
\label{eq:renormalizedlevelset}
\begin{split} 
\mathscr{E}^{\geq 0}_{n,x} \stackrel{\text{def.}}{=}\text{ an enumeration of }\big\{\text{$\overline{K}$: %
$K$ is a cluster of }n^{-1}\big\{y\in{\tilde{B}(xn,n)}:\,\phi_y> 0\big\}\big\},
\end{split}
\end{equation}
be the normalized clusters of the metric graph Gaussian free field under $\P$,  where $\overline{K}$ denotes the topological closure of a set $K$. Note that since the Gaussian free field on each edge $I_e$ behaves like a Brownian bridge, the closures of the connected components of $\{\phi>0\}$ are simply the connected components of $E^{\geq 0}$ with positive diameter, see \cite[Lemma 4.1]{DrePreRod3} for a proof. We also define the normalized clusters of the metric graph loop soup under $\tilde{\P}$ as
\begin{equation}
\label{eq:renormalizedlevelsetloops}
\begin{split} 
\mathscr{E}_{n,x} \stackrel{\text{def.}}{=}\text{ an enumeration of }\big\{\text{$\overline{K}$: %
$K$ is a cluster of }n^{-1}\big\{y\in{\tilde{B}(xn,n)}:\,L_y>0\big\}\big\}.
\end{split}
\end{equation}
We similarly define $\mathscr{L}_{n,x}$ as an enumeration of $\text{trace}(\gamma)/n$ among all loops $\gamma\in{\mathcal{L}_{\tilde{B}(xn,n)}}$.

As we now briefly explain, 
\begin{equation}\label{eq:new-clusters-ok}
\begin{gathered}
\text{ the family $\mathscr{E}^{\geq 0}_{n}$ from \eqref{eq:renormalizedlevelset-intro}, resp.\ $\mathscr{E}_{n}$ from \eqref{eq:deftildeEnC}, has a version  at $\mathbf{d}$-distance from }\\\mathscr{E}_{n,0}^{\geq 0},\text{ resp.\ }\mathscr{E}_{n,0},
 \text{ deterministically bounded by $\tfrac2{n}$ for each $n \ge 1$.} \end{gathered}
\end{equation}
 The relevant version $\mathscr{E}_n^{\geq 0}$ is obtained by replacing the decision rule \eqref{eq:edge-perco} with the indicator function that $\varphi \geq 0$ along the entire cable corresponding to $e$, similarly as in Section~\ref{subsec:metric}. In view of \eqref{eq:defdistance}, to deduce the desired bound on the distance between the families  $\mathscr{E}_n^{\geq 0}$ and  $\mathscr{E}_{n,0}^{\geq 0}$, one simply matches each cluster $K$ in \eqref{eq:renormalizedlevelset-intro} with the corresponding cluster $K'$ from \eqref{eq:renormalizedlevelset} that contains it (the cluster in \eqref{eq:renormalizedlevelset-intro} corresponds to the set of edges entirely contained in the latter). Since $K'$ may only further extend along edges on the outer boundary of $K$ without reaching a vertex outside $K$, one has $d_H(K,K') \leq \tfrac1n$. The family \eqref{eq:renormalizedlevelset} further contains all the clusters of $n^{-1}\{ \varphi \geq 0\} \cap [-1,1]^d$ that do not contain an entire cable, and the diameter of such clusters is bounded by $\tfrac{2}{n}$. By \eqref{eq:defdistance}, the claim \eqref{eq:new-clusters-ok} for $\mathscr{E}_n^{\geq 0}$ follows. Concerning the relevant version of $\mathscr{E}_n$, it can be obtained by replacing the decision rule to open an edge $e$ mentioned above \eqref{eq:deftildeEnC} by $\{L_x>0\text{ for all }x\in{I_e}\}$, similarly as in Section~\ref{subsec:metric}.

In light of \eqref{eq:new-clusters-ok}, as well as a similar result for $\mathscr{L}_{n,0}$ and $\mathscr{L}_n$, it is enough to prove all items of Theorem~\ref{the:Hausdorff} for the clusters $\mathscr{E}_{n,0}^{\geq 0}$ from \eqref{eq:renormalizedlevelset} instead of $\mathscr{E}_{n}^{\geq 0}$ from \eqref{eq:renormalizedlevelset-intro}, and of Corollary~\ref{cor:loopscaling} for the clusters $\mathscr{E}_{n,0}$ and $\mathscr{L}_{n,0}$ from \eqref{eq:renormalizedlevelsetloops} and below instead of $\mathscr{E}_{n}$ from \eqref{eq:deftildeEnC} and $\mathscr{L}_{n}$ from \eqref{eq:deftildeEL}. Following the convention below \eqref{eq:newnotation}, we will thus from now on forget altogether about the discrete notation from the introduction when proving the results from Section~\ref{subsec:scalinglimitintro}. In particular, from now on we simply define $\mathscr{E}_{n}^{\geq 0}=\mathscr{E}_{n,0}^{\geq 0}$, $\mathscr{E}_{n}=\mathscr{E}_{n,0},$ and $\mathscr{L}_n=\mathscr{L}_{n,0}$.%
We also define the normalized sign clusters of the Gaussian free field as
\begin{equation}
\label{eq:renormalizedlevelset0}
\begin{split} 
\mathscr{E}^{0}_{n} \stackrel{\text{def.}}{=}\text{ an enumeration of }\big\{\text{$\overline{K}$: 
$K$ is a cluster of }n^{-1}\big\{y\in{\tilde{B}(n)}:\,|\phi_y|>0\big\}\big\}.
\end{split}
\end{equation}
The interest of the sign clusters is that they are directly related to the clusters for the loop soup via the isomorphism \eqref{eq:isoloopK}, which directly implies that
\begin{equation}
\label{eq:isoclusters}
\mathscr{E}^{0}_{n}\text{ has the same law under }\P\text{ as }\mathscr{E}_n\text{ under }\tilde{\P}.
\end{equation}
This isomorphism  will allow us to transfer results from the level sets $\mathscr{E}^{\geq 0}_n$ of the Gaussian free field to the clusters $\mathscr{E}_n$ of  the loop soup.

\medskip

Let us finally introduce the notions of renormalized volume and capacity  on $\Z^d$. For each $K\subset\R^d$ and $r>0$ define the renormalized volume at scale $r$ as
\begin{equation}
\label{eq:newrenorm}
\mathrm{vol}_r(K)\stackrel{\text{def.}}=|\{x\in{r\Z^d}:\,B(x,r)\cap K\neq\emptyset\}|.
\end{equation}
The definition \eqref{eq:newrenorm} corresponds to a renormalized version of the notion introduced in \eqref{eq:vol} when $\G=\Z^d$: recalling from below \eqref{eq:defLambdaL} that $\Lambda_L=L\Z^d$ for each $L\in{\N}$ therein, for each $n\in{\N}$, $K\subset \Z_n^d$ and $r\in{\tfrac1n\N}$, we have
\begin{equation}
\label{eq:cardagree}
\mathrm{vol}_r(K){=}|nK|_{nr},\text{ where $|\cdot|_r$ is defined as in \eqref{eq:vol} with $\G=\Z^d$ therein.}
\end{equation}

Let us also define a renormalized version of the capacity of a set from \eqref{eq:cap}:  for all compacts $K\subset\R^d$
\begin{equation}
\label{eq:capren}
\mathrm{cap}_n(K)\stackrel{\text{def.}}{=}n^{2-d}\mathrm{cap}_{\tilde{\Z}^d}(nK\cap \tilde{\Z}^d).
\end{equation}
The scaling in \eqref{eq:capren} is chosen so that if $B\subset\R^d$ is a fixed ball, then $\mathrm{cap}_n(B)$, $n\in{\N}$, is bounded away from $0$ and $\infty$, see \eqref{eq:capBallBd}, and will in fact converge to the Brownian capacity $\mathrm{cap}_{\R^d}(B)$ of that ball, as follows for instance from combining \cite[Lemma~2.1 and (2.4)]{10.1214/aop/1176989003} and \cite[(1.58)]{MR2932978}. Note that the discrete Green function in \cite[p.1877]{10.1214/aop/1176989003} is $d$ times our discrete Green function due to our choice of weights $\lambda_{\cdot,\cdot}=1/2$ for $\Z^d$, but that the continuous Green function in \cite[(0.6)]{10.1214/aop/1176989003} is also $d$ times the one used in \cite[p.54]{zbMATH03642454}, which is the reference we use here to define the Brownian capacity, and so we do not need an additional normalization constant. %

In fact in order to obtain \eqref{eq:CAP}, we will need to lower bound the Brownian capacity of a cluster $K$ in the scaling limit of $\mathscr{E}_n^{\geq0}$ by the limit of the random walk capacity of the corresponding clusters in $\mathscr{E}_n^{\geq0}$. More precisely, we will use the following result from  \cite{li2014lower}, which relies on continuous approximation results as can be found e.g.~in \cite{zbMATH03642454}. 

\begin{Lemme}
Let $K^{(n)}\subset\tilde{\Z}_n^d$, $n\in{\N}$, and $K\subset\R^d$ be compact sets such that $d_H(K^{(n)},K)\rightarrow0$ as $n\rightarrow\infty$, then
\begin{equation}
\label{eq:approcap}
\mathrm{cap}_{\R^d}(K)\geq  \limsup_{n\rightarrow\infty}\mathrm{cap}_n(K^{(n)}).
\end{equation}
\end{Lemme}
\begin{proof}
 For any $\delta\geq \eta>0$ and $R,n\geq1$, we denote by $h:\R^d\rightarrow [0,1]$ the probability of entering the  Euclidean  $2\delta$-neighbourhood $B_{2}(K,{2\delta})$ of $K$ for the standard Brownian motion in $\R^d$ killed when exiting the Euclidean ball of radius $R$ around $0$, and by $h_n:\Z^d\rightarrow [0,\infty)$ an $\eta$-regularization of $h(\cdot/ n)$, see \cite[(2.1)-(2.3)]{li2014lower} with  $N= n$ and $r_U=R$ therein. Write $\overline{K}^{(n)}$ for the set of points $x$ in $\Z_n^d$ such that there exists an edge $e$ of $\Z^d$ starting in $nx$ with $nK^{(n)}\cap I_{e}\neq\emptyset$.  For $n$ large enough we have $\overline{K}^{(n)}\subset B_{2}(K,\delta)$ since $d_H(K^{(n)},K)\rightarrow0$, $\overline{K}^{(n)}\subset B(K^{(n)},1/n)$, and by norm equivalence, and hence $h_n(x)=1$ for all $x\in{n\cdot\overline{K}^{(n)}}$. By the variational formula from \cite[(1.62)]{MR2932978}, the proof of which has a version on infinite graphs, we thus have that $\mathcal{E}_{\Z^d}(h_n)\geq \mathrm{cap}_{\Z^d}( n\overline{K}^{(n)})$ for $n$ large enough, where $\mathcal{E}_{\Z^d}(f)=\tfrac14\sum_{x\sim y\in{\Z^d}}(f(x)-f(y))^2$ is the Dirichlet form associated to the simple random walk on $\Z^d$ with weights $1/2$. Furthermore for any $\eps>0$  by \cite[Proposition~2.4]{li2014lower}, there exist $\delta\in{(0,1)}$, $R\geq1$ and $\eta\in{(0,1)}$, all depending on $\eps$, such that $\mathrm{cap}_{\R^d}(K)\geq n^{2-d}\mathcal{E}_{\Z^d}(h_n)-\eps$ for $n$ large enough. Note that there is an additional factor $1/d$ therein, but with our different choice of normalization for $\mathcal{E}_{\Z^d}$ this factor becomes $1$ here, similarly as below \eqref{eq:capren}. Combining the two previous bounds with \eqref{eq:capren}, and noting that $\mathrm{cap}_{\Z^d}(n\cdot\overline{K}^{(n)})\geq \mathrm{cap}_{\tilde{\Z}^d}(n\cdot K^{(n)})$ by monotonicity of capacity, see  for instance \cite[(2.23)]{DrePreRod3} in our context, we can conclude. 
\end{proof}

\medskip

Recalling \eqref{eq:Zdinmacro}, we now collect some consequences of the results of Sections~\ref{sec:boxtobox}--\ref{sec:equimacro} that will be useful in the proofs of Theorem~\ref{the:Hausdorff} and Corollary~\ref{cor:loopscaling}.  
We start with a version of \eqref{eq:CROS} at the level of the metric graph, which is a renormalized version of Proposition~\ref{cor:crossingboxseccrit} in view of Remark~\ref{rk:othercrossingbounds},\ref{eq:crossing-infinitevol} and \eqref{eq:Zdinmacro}. Fix some $\eps\in{(0,1)}$. There exists $c=c(\eps)>0$ and $C<\infty$ such that for all $n\geq1$,  $x,y\in{\Z_n^d}$ and $R,r\in{n^{-1}\N}$ such that $d(x,y)\geq 2(R+r)$ and $B(x,r),B(y,R)\subset B(1-\eps)$,  we have
\begin{equation}
\label{eq:crossingboxseccritH}
\begin{split} 
cd(x,y)^{2-d}(Rr)^{\frac{d-2}2}\leq \P\big({B}(x,r)\leftrightarrow {B}(y,R)\text{ in }\mathscr{E}_n^{\geq 0}\big)\leq Cd(x,y)^{2-d}(Rr)^{\frac{d-2}2}.
\end{split}
\end{equation}
Moreover, \eqref{eq:crossingboxseccritH} remains true when replacing $\P$ by $\Q$ and $\mathscr{E}_n^{\geq0}$ by $\tfrac1n$ times the family of clusters of loops in $\mathcal{L}_{\tilde{B}(n)}$.%
As explained in Remark~\ref{rk:othercrossingbounds}, one can easily deduce from \eqref{eq:crossingboxseccritH} bounds on the  probability to cross an annulus, which were first obtained in \cite[Theorem~1.2]{cai2024onearm}: for all $\eps>0$, $x\in{\Z_n^d}$, $R\geq r$ with $r,R\in{n^{-1}\Z^d}$ and $B(x,R)\subset B(1-\eps)$
\begin{equation}
\label{eq:crossingseccrit}
\begin{split} 
c(\eps)(r/R)^{\frac{d-2}2}\leq \P\big({B}(x,r)\leftrightarrow \partial {B}(x,R)\text{ in }\mathscr{E}_n^{\geq 0}\big)\leq C(r/R)^{\frac{d-2}2}.
\end{split}
\end{equation}

To obtain the second part of \eqref{eq:CAP}, we will also need that the clusters of $\mathscr{E}_n^{\geq0}$ do not stay too close to $\partial [-1,1]^d$. In fact, after renormalizing Proposition~\ref{pro:noclusterinboundary}, we have for all  $t\in{(0,1]}$,
\begin{equation}
\label{eq:noclusterinboundaryH}
\begin{split}
\inf_{n\geq C/s}\P\big(\forall K\in{(\mathscr{E}_n^{\geq 0})_{>t}}, K\cap B(1-s)\neq \emptyset\big)\tend{s}{0}1.
\end{split}
\end{equation}

\subsection{Proof of Theorem~\ref{the:Hausdorff}}\label{subsec:pf-mainresults-crit}

We now prove our main theorems using the results of Section~\ref{sec:mainresults} in the special case $\G=\Z^d$ (with weights $1/2$), which corresponds to the choice of parameters in \eqref{eq:Zdinmacro}. As specified below \eqref{eq:new-clusters-ok}, we consider connected components for the Gaussian free field and loop soup on the cable system $\tilde{\Z}^d$, instead of on edges of $\Z^d$ as in Section~\ref{sec:mainresults}.

We now give a brief overview of how the results of Section~\ref{sec:usefulresults}
 will come into play in the proof of Theorem~\ref{the:Hausdorff}. The existence \eqref{eq:Hausdorffconvergence} of a subsequential scaling limit as well as the tightness \eqref{eq:TIGHT} of the number of connected components will follow from \eqref{eq:volumeMrs}.  The bound \eqref{eq:CAP} on the capacity as well as the upper bound on the Hausdorff dimension and the value of the Minkowski dimension in \eqref{eq:DIM} will follow from \eqref{eq:equivmacro}. The absence of clusters contained in $\partial[-1,1]^d$ in \eqref{eq:CAP} will follow from Proposition~\ref{pro:noclusterinboundary}. The bounds \eqref{eq:CROS} will follow by means of Proposition~\ref{cor:crossingboxseccrit}. To obtain the remaining lower bound on the Hausdorff dimension,  we will also need the following result.  We abbreviate $A,B\leftrightarrow C$ for the intersection of the events $A\leftrightarrow C$ and $B\leftrightarrow C$.

\begin{Lemme}
\label{lem:lowMink}
Assume $G=\Z^d$, $3 \leq d\leq5$ (cf.~\eqref{eq:Zdinmacro}). For all $R\geq 1$, $r,r'\in{[0,cR]}$, $x\in{G}$ and $y\in{B(x,\Cl[c]{c:lowMink}R)}$, one has that
\begin{equation}
\label{eq:lowMink}
\begin{split} 
\P\big(B(x,r),B(y,r')\leftrightarrow\partial B(x,R)\big)\leq C\left(\frac{\rr\rr'}{Rd(x,y)}\right)^{\frac\nu2}.
\end{split}
\end{equation}
\end{Lemme}

Lemma~\ref{lem:lowMink} could be deduced from \cite[Lemma~4.3]{cai2024quasimultiplicativityregularitycriticalmetric}, which relies mainly on properties of loop soups. We give a different proof in Appendix~\ref{App:lowMink}, using only properties of the  free field.

\medskip

We are now ready to give the:

\begin{proof}[Proof of Theorem~\ref{the:Hausdorff}]
We start with \eqref{eq:Hausdorffconvergence}. On account of Lemma~\ref{L:d-metric}, by Prokhorov's theorem, see for instance \cite[Theorem 5.1, p.59]{MR1700749}, it is enough to show that the laws (on $\mathcal{S}$) of $\mathscr{E}_n^{\geq0}$, $n \geq 1$, are tight. For each $\varepsilon >0$, we will show with $\mathcal{K}_{s,M}$ as defined in \eqref{eq:compact-def} that for suitable choices of sequences $s=(s_k)$ with $\lim_k s_k =0$ and $M=(M_k)$ one has
\begin{equation}
    \label{eq:tightness-pf1}
    \sup_n\P\big(\mathscr{E}_n^{\geq0} \notin \mathcal{K}_{s,M} \big) \leq \varepsilon.
\end{equation}
By Lemma~\ref{L: compact}, the bound \eqref{eq:tightness-pf1} establishes the desired tightness, thus yielding \eqref{eq:Hausdorffconvergence}. %

In order to show \eqref{eq:tightness-pf1}, recalling that we endowed the graph $\Z^d$ with the distance $d(x,y)=|x-y|_{\infty}$ (see \eqref{eq:Zdinmacro}), we have that $N_{>s}(\mathscr{E}_n^{\geq 0})\leq |\Mac_n^s(\Dia)|$, see  \eqref{eq:renormalizedlevelset}, \eqref{eq:defMRs} and %
below \eqref{eq:d_H-diam}. Hence letting $k \geq 1$, applying \eqref{eq:volumeMrs} with $\Xt=\Dia$, $R=n$, and $\eta=\eta_k(\varepsilon)$ suitably small such that the left-hand side of \eqref{eq:volumeMrs} is at least $1-\varepsilon 2^{-(k+1)}$, it follows that 
\begin{equation}
    \label{eq:tightness-pf2}
    \sup_{s\in (0,1]}\sup_{n\geq 1}\P\big(N_{>s}(\mathscr{E}_n^{\geq0}) \geq \eta_k^{-1}s^{-d} \big) \leq  \varepsilon  2^{-(k+1)}.
\end{equation}
Note that \eqref{eq:volumeMrs} actually additionally requires that $n\geq C/s$, but the case $n\leq C/s$ is trivial upon choosing $\eta_k$ small enough since $N_{>s}(\mathscr{E}_n^{\geq0})\leq Cn^{d}\leq Cs^{-d}$ in view of \eqref{eq:renormalizedlevelset}. Let $s_k > 0$ be any sequence converging to zero (for definiteness $s_k=2^{-k}$) and let $M_k= M_k(\varepsilon):= \eta_k^{-1}s_k^{-d}$. With these choices, one obtains that
$$
\P\big(\mathscr{E}_n^{\geq0} \notin \mathcal{K}_{s,M} \big) \stackrel{\eqref{eq:compact-def}}{\leq} \sum_{k \geq 0} \P\big(N_{>s_k}(\mathscr{E}_n^{\geq0}) \geq M_k \big) \stackrel{\eqref{eq:tightness-pf2}}{\leq} \varepsilon,
$$
thus implying \eqref{eq:tightness-pf1}. 

\medskip
From here onwards, we fix any $(\sigma,\P_{\sigma}^{\infty})$ such that \eqref{eq:Hausdorffconvergence} is satisfied, and let $ \mathscr{E}_{\infty}^{\geq0}=(K_i)_{i\in{I}}$ be an $\mathcal{S}$-valued random variable with law $\P_{\sigma}^{\infty}$. To ease notation, we identify the subsequence $\sigma$ with its range, an unbounded subset of $\N$, and simply write $n \in \sigma$ to refer to indices in the subsequence. Likewise, in writing `as $n \to \infty$' or `for large enough $n$' below, we always implicitly mean that $n$ varies within the given subsequence $\sigma$. For the remainder of the proof, we further assume $\mathscr{E}_{\infty}^{\geq0}$ to be coupled  with $\{ \mathscr{E}_{n}^{\geq0} : n \in \sigma\}$ in such a way that $\mathbf{d}(\mathscr{E}_{\infty}^{\geq0}, \mathscr{E}_{n}^{\geq0}) \rightarrow 0$ as $n \to \infty$. This coupling exists by application of Skorokhod's representation theorem, see for instance \cite[Theorem~5.7, p.70]{MR1700749}, which is in force by Lemma~\ref{L:d-metric}. With a slight abuse of notation, we still denote by $\P$ the measure under which this coupling is in force. We then abbreviate $\mathbf{d}_n=\mathbf{d}(\mathscr{E}_{\infty}^{\geq0},\mathscr{E}_{n}^{\geq0})$, and write $\mathscr{E}_{n}^{\geq0}=(K_i^{(n)})_{i\in{I_n}}$, for a countable set $I_n$. By \eqref{eq:defdistance}, for each $n\in \sigma$ there exists a matching $\pi_n$ of $(I,I_n)$ such that for each $i\in{I\setminus I^{\pi_n}}$, which includes each $i\in{I}$ with $\delta(K_i)\geq 2\mathbf{d}_n$, there exists $j_n^i\in{I_n}$ with $(i,j_n^i)\in{\pi_n}$ and $d_{H}(K_i,K_{j_n^i}^{(n)})\leq \mathbf{d}_n$. Up to relabeling the indices and the set $I_n$ (for simplicity, we suppose that $I_n, I \subset \mathbb{N}$ in the sequel), we will henceforth assume w.l.o.g.\ that $I\setminus I^{\pi_n}\subset I_n$ and $j_n^i=i$ for all $i\in{I\setminus I^{\pi_n}}$ to simplify notation. For all $i\in{I\setminus I_n}$, we further define $K_i^{(n)}=\emptyset$ for definiteness. Since $\delta(K_i)>0$ for all $i\in{I}$ by definition of $\mathcal{S}$ and $\mathbf{d}_n\rightarrow 0$, we deduce that
\begin{equation}
\label{eq:hausdorffgoesto0}
\begin{split}
d_H(K_i,K_i^{(n)})\rightarrow 0\text{ a.s.\ as }n (\in \sigma)\rightarrow\infty\text{ for all }i\in{I}.
\end{split}
\end{equation}

\medskip
We now show \eqref{eq:TIGHT}. Write $I_{\geq\eps}=\{i\in{I}:{\delta}(K_i)\geq \eps\}$, which is finite by definition of $\mathcal{S}$, and recall the definition of $\Mac_R^{\eps}(\Dia)$ from \eqref{eq:defMRs}. By \eqref{eq:hausdorffgoesto0} a.s.\ for $n$ large enough, we have $d_H(K_i,K_i^{(n)})\leq\eps/4$ for all $i\in{I_{\geq\eps}}$, which implies that the diameter of $K_i^{(n)}$ is at least $\eps/2$ by \eqref{eq:d_H-diam}. Hence $|I_{\geq\eps}|\leq |\Mac^{\eps/2}_{n}(\Dia)|$ for $n (\in \sigma)$ large enough. Similarly consider $i\in{I_n}$ such that $\delta(K_i^{(n)})\geq 2\eps$, then since $\mathbf{d}_n\rightarrow 0$ we have for $n$ large enough that each such $i$ is paired via $\pi_n$ to some point in $I$, which implies $i\in{I}$ in view of the convention above \eqref{eq:hausdorffgoesto0}, and similarly as before we deduce from \eqref{eq:hausdorffgoesto0} that $|\Mac^{2\eps}_{n}(\Dia)|\leq |I_{\geq\eps}|$ for $n$ large enough. Combining these observations with \eqref{eq:volumeMrs} for $\Xt=\Dia$ yields that \eqref{eq:volumeMrs} for $s=\eps$ still holds when replacing $\Mac_{R}^{\eps}(\Dia)$ by $I_{\geq \eps}$, and we can conclude.

\medskip

We now turn to the proof of \eqref{eq:DIM}, and start with the upper bounds. Note that \eqref{eq:TIGHT} implies that $I$ is infinite, and it is countable by definition of $\mathcal{S}$. In writing $ \mathscr{E}_{\infty}^{\geq0}=(K_i)_{i\in{I}}$, we assume without loss of generality from here on that $I=\N$, $0<{\delta}(K_{i+1})\leq{\delta}(K_i)$ for all $i\in{\N}$, and \eqref{eq:hausdorffgoesto0} is satisfied. Fix some $i\in{\N}$ and $\eps>0$, and let $t=t(\eps,i)>0$ be such that $\P({\delta}(K_i)\leq 2t)\leq \eps$. Using \eqref{eq:hausdorffgoesto0}, we deduce from this that 
    \begin{equation}
\label{eq:deltaKinlarge}
\begin{split}
\P(\delta(K_i^{(n)})\leq t)\leq 2\eps\text{ for }n \in \sigma \text{ large enough},
\end{split}
\end{equation}
or, in other words, that on an event having probability $1-2\varepsilon$, the set $nK_i^{(n)}$ belongs to $\Mac^t_n(\Dia)$. Combining this, \eqref{eq:defMRs} and \eqref{eq:equivmacro}, applied for $\Xt=\Dia$ and $\Yt=\Vol_{\eta n}$, we obtain that for all $s\leq c(\eps,i)$ we have 
        \begin{equation}
    \label{eq:boundsizeKjsigman}
\begin{split} 
    \P\big((\eta/s)^{-\alpha+\frac\nu2}\leq | nK_i^{(n)}|_{\eta n}\leq (s\eta)^{-\alpha+\frac\nu2}\big)\geq 1-Cs^{c}\text{ for $n \in \sigma$ large enough }.
\end{split}
\end{equation}
Let $n$ large enough so that \eqref{eq:boundsizeKjsigman} is satisfied and $d_H(K_i,K_i^{(n)})\leq \eta$, see \eqref{eq:hausdorffgoesto0}. It follows that on an event with probability at least $1-Cs^c$, one can cover $nK_i^{(n)}$ using at most $(s\eta)^{-1-\frac{d}2}$  $\ell^{\infty}$-balls of radius $\eta n$, and so with probability $1-Cs^c$, $M_{2\eta}(K_i)\leq C(s\eta)^{-1-\frac{d}2}$. Taking $\eta=2^{-(k-1)}$ and $s=k^{-C}$ for a large enough constant $C$, and using Borel-Cantelli lemma, we deduce that a.s.
\begin{equation}
\label{eq:MkKi}
M_{2^{-k}}(K_i)\leq Ck^C(2^{-k})^{-1-\frac{d}2}\text{ for }k\text{ large enough}.
\end{equation}
In particular, a.s.\ for $t$ large enough, choosing $k\in{\N}$ so that $2^{-k}\leq t\leq 2^{-(k-1)}$, we deduce that $M_{t}(K_i)\leq M_{2^{-k}}(K_i)\leq Ck^C(2^{-k})^{-1-\frac{d}2}\leq Ct^{-1-\frac{d}2}\ln(1/t)^{C}$.%
 This concludes the proof of the upper bound $\text{dim}_H(K_i)\leq \overline{\text{dim}}_M(K_i)\leq 1+\frac{d}{2}$, where $\overline{\text{dim}}_M(K_i)$ denotes the upper Minkowski dimension.

    Let us now turn to the lower bounds. One could obtain the lower bound on the lower Minkowski dimension by a similar reasoning as around \eqref{eq:MkKi}, but since $\underline{\text{dim}}_M(K_i)\geq \text{dim}_H(K_i) $, where $\underline{\text{dim}}_M(K_i)$ denotes the lower Minkowski dimension, it will be enough to prove that $\text{dim}_H(K_i)\geq 1+d/2$. %
We use an energy method. Still for given $i \in \mathbb{N}$ as above, fix an arbitrary point $x_i^{(n)}\in{K_i^{(n)}}$ and write 
\begin{equation}\label{eq:nu-i}
\nu_{i,\eta}^{(n)}(\cdot)\stackrel{\text{def.}}{=} \eta^{-\frac\nu2} \lambda\vert_{ V_i^{(n)} }, \qquad V_i^{(n)} = B(K_i^{(n)},\eta)\cap B(x_i^{(n)},\tfrac{\Cr{c:lowMink}t}{8}),
\end{equation}
with $t=t(\eps,i)$ as declared above \eqref{eq:deltaKinlarge}, $\Cr{c:lowMink}$ as in Lemma~\ref{lem:lowMink}, and where $\lambda$ denotes the Lebesgue measure on $[-1,1]^d$ and $\lambda\vert_S(\cdot)= \lambda(\cdot \cap S)$ its restriction to the set $S$. Note that for $\eta\leq c  =c(\varepsilon, i)$, the Lebesgue measure of the ball $B(x,\eta/2)$ is larger than $c\eta^{\alpha}$ for each $x\in{(\frac1{n}\Lambda_{\eta n/2})\cap B(x_i^{(n)},\tfrac{\Cr{c:lowMink}t}{16})}$ such that $B(x,\eta/2)\cap K_i^{(n)}\neq\emptyset$, and moreover, %
a positive proportion of these balls are disjoints,  and they are all contained in the set $V_i^{(n)}$ from \eqref{eq:nu-i}. Then using \eqref{eq:boundsizeKjsigman} (with $\eta/2$ in place of $\eta$), we obtain in view of  \eqref{eq:nu-i} that for $\eta\leq c$,
\begin{equation}
\label{eq:nunetapos}
\begin{split}
\P\big(c\leq \nu_{i,\eta}^{(n)}([-2,2]^d)\leq C\big)\geq 1-3\eps\text{ for $n$ large enough},
\end{split}
\end{equation}
where $c=c(\eps,i)>0$ and $C=C(\eps,i)<\infty)$. Let $A_n =\Lambda_{\eta n}\cap B(n x_i^{(n)},\tfrac{\Cr{c:lowMink}tn}{4})$. Note that by \eqref{eq:defLambdaL} for $\eta\leq c$ and on the event $\delta(K_i^{(n)})\geq t$, the set $V_i^{(n)}$ is included in the union of the balls $B({x}/{n},\eta)$ with $x$ ranging over all $x\in{A_n}$ such that $B(x,2\eta n)\leftrightarrow \partial B(nx_i^{(n)},tn/2)$. 
Hence, writing 
$$I_{\eps}(\nu,\mu)\stackrel{\text{def.}}{=}\iint|x-y|_{\infty}^{-\alpha+\frac\nu2+\eps}\nu(\mathrm{d}x)\mu(\mathrm{d}y) \quad (\in [0,\infty])$$
 for any measures $\nu,\mu$ on $[-1,1]^d$, and $I_{\eps}(\nu,\nu)=I_{\eps}(\nu)$, we deduce by subadditivity that
\begin{equation}
\label{eq:mathcalE1}
\begin{split}
 &\E\big[I_{\eps}(\nu_{i,\eta}^{(n)})1\big\{{\delta}(K_i^{(n)})\geq t\big\}\big]
    \\&\leq \eta^{-\nu}\E\Big[\sum_{x,y\in{A_n}}I_{\eps}\big(\lambda_x,\lambda_y\big)1\big\{B(x,2\eta n),B(y,2\eta n)\leftrightarrow \partial B(n x_i^{(n)},{tn/2})\big\}\Big],
\end{split}
\end{equation}
where $\lambda_x$ denotes the restriction of the Lebesgue measure $\lambda$ to %
 $B({x}/{n},\eta)$. Let $A'_n=\Lambda_{\eta n}\cap [-2n,2n]^d$. We further have for all $x,y\in{A'_n}$ with $|x-y|_{\infty}\leq 3\eta n$, 
$$I_{\eps}\big(\lambda_{x},\lambda_{y}\big)\leq I_{\eps}\big(\lambda_{|B(x/n,4\eta)}\big) \leq C \lambda (B(x/n,4\eta)) \int_0^{8 \eta} r^{\alpha-1}r^{-\alpha+\frac\nu2+\eps} \mathrm{d}r  \leq C\eta^{\alpha+\frac\nu2+\eps}.$$ 
 Moreover for all $x,y\in{A'_n}$ with $|x-y|_{\infty}\geq 3\eta n$, since $|z-z'|_{\infty}\geq c|x-y|_{\infty}/ n$ for all $z\in{B(x/ n,\eta)}$ and $z'\in{B(y/ n,\eta)}$, we obtain that
\begin{equation*}
\begin{split}
I_{\eps}\big(\lambda_{x},\lambda_{y}\big)\leq C\eta^{2\alpha}\Big(\frac{|x-y|_{\infty}}{ n}\Big)^{-\alpha+\frac\nu2+\eps}.
\end{split}
\end{equation*}
Notice also that for all $x,y\in{A_n}$ the event $B(y,\eta n)\leftrightarrow \partial B( nx_i^{(n)},t n/2)$ implies $B(y,\eta n)\leftrightarrow \partial B(x,t n/4)$ since $d(x,nx_i^{(n)}) \leq \tfrac{tn}{4}$ and $d(x,y)\leq \tfrac{tn}{4}-\eta n$ upon assuming w.l.o.g.\ that $\Cr{c:lowMink}\leq 1/4$ and $\eta\leq t/8$. Combining these observations with \eqref{eq:defLambdaL} and noticing that $A_n\subset A'_n$ by definition, it follows that the last line of \eqref{eq:mathcalE1} is bounded by
    \begin{equation}
\label{eq:mathcalE2}
\begin{split} 
  &C\eta^{-\nu}  %
  \Big(\sum_{x\in{A'_n}}\eta^{\alpha+\frac\nu2+\eps}\P\big(B(x,\eta n)\leftrightarrow \partial B(x,t n/4)\big)
    \\&\quad+ \eta^{2\alpha}\sum_{x\neq y\in{A'_n}} \Big(\frac{|x-y|_{\infty}}{ n}\Big)^{-\alpha+\frac\nu2+\eps}\P\big(B(x,\eta n),B(y,\eta n)\leftrightarrow \partial B(x,t n/4)\big)\Big)
    \\&\leq  C \Big( \eta^{\alpha+\eps}|A'_n|+\eta^{2\alpha}
    \sum_{x\neq y\in{A'_n}} n^{\alpha-\eps}d(x,y)^{-\alpha+\eps} \Big)
\end{split}
\end{equation}
 for a constant $C=C(i,\eps)$ and $\eta \leq c= c(i\varepsilon)$, 
 where the last inequality follows by %
 combining the following estimates on relevant connection probabilities: Lemma~\ref{lem:lowMink}, which implies on account of \eqref{eq:Zdinmacro} that the quantity $\P\big(B(x,\eta n),B(y,\eta n)\leftrightarrow \partial B(x,t n/4)\big)$ is bounded by $C\eta^\nu n^{\nu/2} d(x,y)^{-\nu/2}$ for all $\eta\leq c$ when $x \neq y$, and Proposition~\ref{cor:crossingboxseccrit}, which implies that the same quantity is bounded by $C \eta^{\nu/2}$ when $x=y$. To see this last bound, one chooses  $N = Cn$ in \eqref{eq:crossingboxseccrit} and paves $\partial B(x,t n/4)$ with a constant number of boxes of radius $t n/16$, all at distance exceeding $t n/8$ from $B(x,\eta n)$ for small enough $\eta$, and applies a union bound using \eqref{eq:crossingboxseccrit}. Since $|A'_n| \leq C\eta^{-\alpha}$ by \eqref{eq:defLambdaL}, the term in parenthesis on the last line of \eqref{eq:mathcalE2} is further bounded by 
\begin{equation}
\label{eq:mathcalE3}
\begin{split}
 C \eta^{\varepsilon}+ C\eta^{2\alpha} \sum_{k \geq 0} \sum_{\substack{ x\neq y\in{A'_n}: \\ 2^{k}  \leq \frac{ d(x,y) }{c\eta n} < 2^{k+1}  }}  (2^k \eta )^{-\alpha+\eps} \leq  C\eta^{\varepsilon}+ C \eta^{\alpha + \varepsilon}|A_n'| \sum_{k = 0}^{\log_2 (C\eta^{-1})} 2^{k \alpha} 2^{k( \varepsilon-\alpha)} \leq C' 
\end{split}
\end{equation} 
using that the last sum in \eqref{eq:mathcalE3} is bounded by $C\eta^{-\varepsilon}$.
Combining \eqref{eq:defLambdaL}, \eqref{eq:mathcalE1}, \eqref{eq:mathcalE2} and \eqref{eq:mathcalE3}, we finally deduce that
\begin{equation}
\label{eq:matchalE4}
\begin{split}
    \E\big[I_{\eps}(\nu_{i,\eta}^{(n)})1\big\{{\delta}(K_i^{(n)})\geq t\big\}\big]\leq C.
\end{split}
\end{equation}
Combining \eqref{eq:deltaKinlarge}, \eqref{eq:nunetapos}, \eqref{eq:matchalE4} and Markov’s inequality one can find $c=c(\eps,i)>0$ and $C=C(\eps,i)>0$ such that
\begin{equation}
\label{eq:limsupnuietan}
\begin{split}
\P&\big(\limsup\{c\leq \nu_{i,\eta}^{(n)}([-2,2]^d)\leq C,I_{\eps}(\nu_{i,\eta}^{(n)})\leq C\}\big)
\\&\geq \limsup\P\big(c\leq \nu_{i,\eta}^{(n)}([-2,2]^d)\leq C,I_{\eps}(\nu_{i,\eta}^{(n)})\leq C\big)\geq 1-6\eps,
\end{split}
\end{equation}
where the above $\limsup$ are limits as $n\rightarrow\infty$, $n\in{\sigma}$, and then $\eta\rightarrow0$. Let $E$ be the event appearing on the first line of \eqref{eq:limsupnuietan}. On the event $E$, there exist (random) sequences $(n_{\ell})_{\ell\in{\N}}$ (part of $\sigma$) and $(\eta_k)_{k\in{\N}}$ so that $c\leq \nu_{i,\eta_k}^{(n_{\ell})}([-2,2]^d)\leq C$ and $I_{\eps}(\nu_{i,\eta_k}^{(n_{\ell})})\leq C$ for all $k,\ell\in{\N}$. Let $\nu_i$ be any subsequential weak limit of the measures $\nu_{i,\eta_k}^{(n_{\ell})}$ as $\ell \rightarrow\infty$ and then $k\rightarrow\infty$, which exists on $E$. With a slight abuse of notation, we still denote by $(\eta_k)$ and $(n_{\ell})$ the subsequence along which $\nu_{i,\eta_k}^{(n_{\ell})}$ converge to $\nu_i$.  Then on the event $E$, $\nu_i$ satisfies 
$$
\iint |x-y|^{-\alpha+\frac\nu2+\eps}1_{|x-y|\geq \delta}\, \nu_i(\mathrm{d}x)\nu_i(\mathrm{d}y)\leq C
$$
 for any $\delta>0$,  which follows readily by combining weak convergence of the product measures (along a chosen subsequence) with a straightforward approximation argument, thus replacing the integrand by a suitable bounded continuous function and applying dominated convergence.  Letting $\delta\rightarrow0$ we thus find that $I_{\eps}(\nu_i)<\infty$ on the event $E$ event with probability $1-4\eps$. Further note that on the event $E$, for any $x\in{K_i^{\mathsf{c}}}$ and any $\delta>0$ small enough so that $K_i\cap B(x,3\delta)=\emptyset$, $\nu_i((x-\delta,x+\delta)^d)\leq \liminf\nu_{i,\eta_k}^{(n_\ell)}((x-\delta,x+\delta)^d)=0$ for $\delta$ small enough on account of \eqref{eq:nu-i} and since  $B(x,2\delta)\cap K_i^{(n)}=\emptyset$  for all $n$ large enough by \eqref{eq:hausdorffgoesto0}; and that $c\leq \nu_i([-2,2]^d)=\lim\nu_{i,\eta_k}^{(n_\ell)}([-2,2]^d)\leq C$ since $[-2,2]^d$ is a boundary set (in $[-2,2]^d$, where all our measures are supported), where the limits are first taken as $\ell \rightarrow\infty$ and then as $k\rightarrow\infty$. Therefore, on the event $E$, $\nu_i$ is a positive and finite  measure supported on $K_i$ with $I_{\eps}(\nu_i)<\infty$. Hence in view of \eqref{eq:limsupnuietan}, we have $\text{dim}_H(K_i)\geq 1+\frac{d}{2}-\eps$ with probability at least $1-6\eps$  by application e.g.~of \cite[Theorem 4.27]{MR2604525}, which finishes the proof of \eqref{eq:DIM} upon letting $\eps\rightarrow0$. 

\medskip

Let us now turn to \eqref{eq:CAP}. %
For any $\eps,\eta>0$, in view of \eqref{eq:approcap} and \eqref{eq:hausdorffgoesto0} we have that 
\begin{equation}\label{eq:caplb-pf}
\P(\mathrm{cap}_{\R^d}(K_i)\geq \mathrm{cap}_n(K_i^{(n)})-\eta)\geq 1-\eps, \text{ for $n$ large enough,}
\end{equation}
It thus follows by combining \eqref{eq:capren} and \eqref{eq:caplb-pf} with \eqref{eq:equivmacro} for $\Xt=\Dia$ and $\Yt=\Capa$ and using \eqref{eq:deltaKinlarge} that there exists $s=s(\eps,i)$ such that $\P(\mathrm{cap}_{\R^d}(K_i)\geq s -\eta)\geq 1-4\eps$ for $n$ large enough. Letting $\eta\rightarrow 0$ and then $\eps\rightarrow0$, we conclude that $\mathrm{cap}_{\R^d}(K_i) >0$ a.s.

We finally show that $K_i\cap(-1,1)^d\neq\emptyset$ a.s. Fix $i\in{\N}$ and $\eps>0$, and let $t=t(\eps,i)$ be as above \eqref{eq:deltaKinlarge}. By \eqref{eq:noclusterinboundaryH}, there exists $s=s(\eps,t)\in{(0,1)}$ such that for all $n\geq C/s$, with probability at least $1-\eps$, every $K\in{(\mathscr{E}_n^{\geq0})_{>t}}$ intersects $B(1-s)$. Combining with \eqref{eq:deltaKinlarge}, the event $F_n=\{K_i^{(n)}\cap B(1-s)\neq\emptyset\}$ has probability at least $1-3\eps$ for all $n\in{\sigma}$ large enough, and thus $\P(\limsup_n F_n)\geq 1-3\eps$. On the event $\limsup_nF_n$, there exist a (random) subsequence $\sigma'$ of $\sigma$ and $x_n\in{K_i^{(n)}\cap B(1-s)}$, $n\in{\sigma'}$; by compactness of $B(1-s)$ we may assume that $x_n\to x\in{B(1-s)}$, and since $d_H(K_i^{(n)},K_i)\to0$ a.s.\ by \eqref{eq:hausdorffgoesto0} and $K_i$ is closed, we obtain $x\in{K_i}$. Hence $\P(K_i\cap(-1,1)^d\neq\emptyset)\geq 1-3\eps$, and we conclude by letting $\eps\to0$.

\medskip

Finally, it remains to prove \eqref{eq:CROS}. Let $x,y\in{[-1,1]^d}$ and $r,R>0$ be such that $B(x,r)\cap B(y,R)=\emptyset$, which implies that there exists $t\in  (0 ,1]$ such that $d(x,y)\geq r+R+2t$. Fix such $t$. If we have $B(x,r)\leftrightarrow B(y,R)$ in $\mathscr{E}_{\infty}^{\geq 0}$, then by \eqref{eq:hausdorffgoesto0}, except on an event of probability at most $\varepsilon$, for all $n (\in \sigma)$ with $n \ge C(\varepsilon)$ we have $B(x,r+t)\leftrightarrow B(y,R+t)$ in $\mathscr{E}_{n}^{\geq 0}$.  We now apply Proposition~\ref{cor:crossingboxseccrit}. Combining the previous observation with the upper bound from \eqref{eq:crossingboxseccrit} obtained after letting $N \uparrow \infty$, for $x,y,r$ and $R$ appearing therein replaced by $xn$, $yn$, $(r+t)n$ and $(R+t)n$ for any \emph{fixed}  value of $n \ge C(\varepsilon)$ (with $n$ belonging to the subsequence $\sigma$), we obtain the upper bound in \eqref{eq:CROS} after letting $\eps\rightarrow0$.

Similarly, up to an event of probability at most $\varepsilon$, for large enough $n (\in \sigma)$, if $B(x,rn/2)\leftrightarrow B(y,Rn/2)$ in $\mathscr{E}_{n}^{\geq 0}$, then $B(x,r)\leftrightarrow B(y,R)$ in $\mathscr{E}_{\infty}^{\geq 0}$. Moreover, the assumptions on $B(x,r)$ and $B(y,R)$ above \eqref{eq:CROS} imply the ones relevant to the application of Proposition~\ref{cor:crossingboxseccrit}, with $N :=n$. In much the same way as for the upper bound, the previous inclusion and the lower bound in  \eqref{eq:crossingboxseccrit} thus yield the lower bound in \eqref{eq:CROS}. 
\end{proof}

\begin{Rk}[Extension to sign clusters]\label{R:ext-signclusters}
Combining the symmetry of the law $\P$ of the free field under $\varphi \mapsto -\varphi$ with the results of Theorem~\ref{the:Hausdorff} and using straightforward adaptations of some of the arguments in the above proof, one readily deduces that all the conclusions of Theorem~\ref{the:Hausdorff} continue to hold with $\mathscr{E}_n^{\geq 0}$ in \eqref{eq:renormalizedlevelset} replaced by the collection of sign clusters $\mathscr{E}_n^{0}$ from \eqref{eq:renormalizedlevelset0}.
\end{Rk}

The interest of Remark~\ref{R:ext-signclusters} is that via the isomorphism \eqref{eq:isoclusters}, it can be rephrased in terms of the loop soup clusters on $\tilde{\Z}^d$, thus letting us finish the

\begin{proof}[Proof of  Corollary~\ref{cor:loopscaling}]
By the isomorphism theorem \eqref{eq:isoclusters}, symmetry of the Gaussian free field, and \eqref{eq:tightness-pf1}, the clusters $(\mathscr{E}_n)_n$ are tight in $\mathcal{S}$, yielding \eqref{eq:Hausdorffconvergenceloops}. The fact that the family  $(K_i)_{i\in{I}}$ satisfies \eqref{eq:TIGHT}--\eqref{eq:CROS} follows from \eqref{eq:isoclusters} and Remark~\ref{R:ext-signclusters}. It thus remains to prove \eqref{eq:loopbuildclusters}, that is that each loop cluster in the scaling limit is the closure of the union of scaling limits of loops. We first prove that each loop cluster in the scaling limit contains the scaling limit of some loop in $\mathscr{L}_{n}$, that is for any $\sigma$ so that \eqref{eq:Hausdorffconvergenceloops} holds,
\begin{equation}
\label{eq:loopbuildclusters1}
\tilde{\P}_\sigma^{\infty}\text{ a.s.\ for each }i\in{I}\text{ there exists }j\in{J}\text{ such that }K'_j\subset K_i.
\end{equation}
Fix some $\eps>0$ and $i\in{I}$, then there exists $t=t(\eps,i)>0$ such that $\tilde{\P}_{\sigma}^{\infty}(\delta(K_i)>2t)\geq 1-\eps$. In view of \eqref{eq:Hausdorffconvergenceloops} and by a reasoning similar to \eqref{eq:hausdorffgoesto0} under an appropriate coupling $\eta_n:=\mathbf{d}_2((\mathscr{E}_{n},\mathscr{L}_{n}),(\mathscr{E},\mathscr{L}))\rightarrow0$ a.s.\ as $n\rightarrow\infty$, $n\in{\sigma}$, and so there exists $K_i^{(n)}\in{\mathscr{E}_{n}}$, $n\in{\sigma}$, so that  $d_H(K_i^{(n)},K_i)\leq \eta_n$  for $n\in{\sigma}$ large enough, which implies particular on the event $\{\delta(K_i)>2t\}$ that $\delta(K_i^{(\sigma(n))})>t$ for all $n\in{\sigma}$ large enough.  In particular on that event, $nK_i^{(n)}\in{\Mact_{0,n}^t(\Dia)}$ for all $n\in{\sigma}$ large enough, see \eqref{eq:deftildeMRS}, and thus by \eqref{eq:macrocontainsbigloops} there exists $s=s(\eps,i)$ so that the probability of  $nK_i^{(n)}\in{\Mact_{0,n}^s(\mathcal{L})}$ is larger than $1-3\eps$ for all $n\in{\sigma}$ with $n\geq C(\eps,i)$. In view of  \eqref{eq:defMRSL} and \eqref{eq:deftildeEL}, this implies that for $n\in{\sigma}$ larger than $C(i,\eps)$,  there exists $\Gamma^{(n)}\in{\mathscr{L}_n}$ so that on an event with probability at least $1-3\eps$ we have $\Gamma^{(n)}\subset K_i^{(n)}$ and $\delta(\Gamma^{(n)})\geq s$. For $n\in{\sigma}$ large enough we have $\mathbf{d}(\mathscr{L}_n,\mathscr{L})<s/2$, and thus by \eqref{eq:defdistance} there exists $j_n\in{J}$ so that on an event with probability at least $1-4\eps$  we have $\delta(K'_{j_n})\geq s/2$ and $d_H(\Gamma^{(n)},K'_{j_n})\leq  \eta_n$. In particular for $n\in{\sigma}$ large enough, on an event with probability at least $1-4\eps$, $K'_{j_n}\subset B(\Gamma^{(n)},\eta_n)\subset B(K_i^{(n)},\eta_n)\subset B(K_i,2\eta_n)$.

Since $\eta_n\rightarrow 0$ a.s, we have proved that for each $\eps>0$ and $i\in{I}$, there exists a sequence $(\delta_n)_n$ with $\delta_n\rightarrow 0$, $s=s(\eps,i)$ and a random sequence $(j_n)_{n\in{\N}}$ so that $\limsup\tilde{\P}_{\sigma}^{\infty}(\delta(K'_{j_n})\geq s/2,K'_{j_n}\subset B(K_i,\delta_n))\geq 1-5\eps$ as $n\in{\sigma}\rightarrow\infty$, and so $\tilde{\P}_{\sigma}^{\infty}(\limsup\{\delta(K'_{j_n})\geq s/2,K'_{j_n}\subset B(K_i,\delta_n)\})\geq 1-5\eps$. In particular, on an event with probability at least $1-5\eps$, there is a (random) sequence $\sigma'\subset\sigma$ so that $\delta(K'_{j_{n}})\geq s/2$  and $K'_{j_n}\subset B(K_i,\delta_n)$ for all $n\in{\sigma'}$. Since $\mathscr{L}\in{\mathcal{S}}$ we have $|J_{>s/4}|<\infty$, and since $j_n$ is in this set for all $n\in{\sigma'}$ there is a (random) subsequence $\sigma''$ of $\sigma'$ so that $j_{\sigma''(n)}$ is constant equal to some $j\in{J}$. Therefore on an event with probability at least $1-5\eps$, $K'_j\subset B(K_i,\delta_n)$ for all $n\in{\sigma''}$. Since $\delta_n\rightarrow0$, we deduce that on the previous event we have $K'_j\subset K_i$, which finishes the proof of \eqref{eq:loopbuildclusters1} by letting $\eps\rightarrow0$.

\medskip

We now explain how to deduce \eqref{eq:loopbuildclusters} from \eqref{eq:loopbuildclusters1}. Let us denote by $A_i$ the union of all the sets $K'_j$ with $K'_j\subset K_i$, $j\in{J}$. We proceed by contradiction: if \eqref{eq:loopbuildclusters} does not hold, then on an event with positive probability, there exists a random $x\in{K_i}$ such that $x\notin{\bar{A_i}}$, and thus there exists a random $\eta>0$ so that $B(x,2\eta)\cap A_i=\emptyset$. Fix $P\in{\N}$ so that $\eta\geq 1/P$ and $\delta(K_i)> 1/(2P)$ with positive probability, then there exists some $y\in{(1/P)\Z^d\cap (-1,1)^d}$ such that $\tilde{\P}_\sigma^{\infty}(E_{y,P}^i)>0$, where 
$$E_{y,P}^i:=\big\{\exists\,x\in{K_i}:\,B(x,\tfrac1{2P})\cap [-1,1]^d\subset B(y,\tfrac1P),\delta(K_i)>\tfrac1{2P},B(y,\tfrac1P)\cap A_i=\emptyset\big\}.$$
Abbreviate $n_P=\lfloor n/P\rfloor$ and let $y_n\in{(1/(Pn_P))\Z^d}$ so that $B(Pn_Py_n,n_P)\subset B(0,n)$ and $d(y_n,y)\leq 1/n_P$. In case $B(y,1/P)\cap\partial[-1,1]^d\neq\emptyset$, we further choose such $y_n$ with the additional condition that $d_H(B(y,1/P)\cap \partial[-1,1]^d,B(Pn_Py_n/n,n_P/n)\cap\partial [-1,1]^d)\leq 3/n_P$. %
By tightness (see the proof of \eqref{eq:Hausdorffconvergenceloops})  and translation invariance, there is a subsequence $\sigma'\subset\sigma$ so that $(\mathscr{E}_{n_P,Py_n}-Py_n,\mathscr{L}_{n_P,Py_n}-Py_n)$ converges weakly as $n\rightarrow\infty$, $n\in{\sigma'}$ toward some limiting law, and let $(\mathscr{E}_{y},\mathscr{L}_{y})$ be random variables so that $(P\mathscr{E}_{y}-Py,P\mathscr{L}_{y}-Py)$ has this law. 

Since $y_n\rightarrow y$ as $n\rightarrow\infty$, and up to changing $\sigma'$ to a subsequence of $\sigma'$ if necessary to obtain joint convergence, one can now consider a coupling so that $\mathbf{d}((1/P)\mathscr{E}_{n_P,y_n},\mathscr{E}_{y})\rightarrow 0$ and $\mathbf{d}(\mathscr{E}_{n},\mathscr{E})\rightarrow 0$ a.s.\ as $n\rightarrow\infty$, $n\in{\sigma'}$,  as well as $\mathbf{d}((1/P)\mathscr{L}_{n_P,Py_n},\mathscr{L}_{y})\rightarrow0$ and $\mathbf{d}(\mathscr{L}_n,\mathscr{L})\rightarrow0$.  On the event $E_{y,P}^i$, there is an $x\in{K_i}$ such that $B(x,1/(2P))\cap[-1,1]^d\subset B(y,1/P)$, and $\delta(K_i)>1/(2P)$. Let $K_i^{(n)}\in{\mathscr{E}_n}$ so that $d_H(K_i^{(n)},K_i)\rightarrow 0$ a.s. For $n$ large enough by our choice of $y_n$, there exists $x_n\in{K_i^{(n)}}$ with $B(x_n,1/(4P))\cap[-1,1]^d\subset B(Pn_Py_n/n,n_P/n)$ and $\delta(K_i^{(n)})\geq 1/(4P)$. Note that $n_P\mathscr{E}_{n_P,Py_n}$ corresponds  to the clusters of $\mathcal{L}$ within $B(Pn_Py_n,n_P)$, which is included in $B(0,n)$. In particular on the event $E_{y,P}^i$ and for $n$ large enough, $K_i^{(n)}$ contains a cluster of $(n_P/n)\mathscr{E}_{n_P,Py_n}$ with diameter at least $1/(4P)$. 

Letting $n\rightarrow\infty$, so that $n_P/n\rightarrow 1/P$, we deduce that on the event $E_{y,P}^i$, there is a cluster $B\in{\mathscr{E}_{y}}$ such that $B\subset K_i$ . Applying \eqref{eq:loopbuildclusters1} to the sequence $\{n_P, n\in{\sigma'}\}$ instead of $\sigma$ and using translation invariance, we further know that there exists a cluster $B'\in{\mathscr{L}_{y}}$ so that $B'\subset B$. In particular on the event $E_{y,P}^i$ we have $B'\subset K_i$. Furthermore since any loop in $n_P\mathscr{L}_{n_P,Py_n}$ is also a loop in $n\mathscr{L}_n$ we deduce by letting $n\rightarrow\infty$, $n\in{\sigma'}$, that $B'\subset A_i$ on $E_{y,P}^i$. Since $B'\subset B(y,1/P)$, we thus have that $A_i\cap B(y,1/P)\neq\emptyset$ on $E_{y,P}^i$, which is a contradiction.
\end{proof}

\textbf{Acknowledgements.}  The research of AD has been supported by the Deutsche
Forschungsgemeinschaft (DFG) grant DR 1096/2-1. AP has been funded by the Deutsche
Forschungsgemeinschaft (DFG, German Research Foundation) – Projektnummer 552316285 as well as
CRC 1720 – 539309657. The research of PFR is supported by the European Research Council
(ERC) under the European Union’s Horizon Europe research and innovation program (grant agreement
No 101171046).
\appendix

\section{Proof of Lemma~\ref{L:d-metric}}
\label{App:dmetric}

In this appendix, we prove that $(\mathcal{S},\mathbf{d})$ is a Polish space, see Section~\ref{sec:subdistanced} for notation. Recall in particular the consequence of \eqref{eq:defdistance} stated above Lemma~\ref{L:d-metric}, which we use repeatedly below.

\begin{proof}[Proof of Lemma~\ref{L:d-metric}]
We first show that $(\mathcal{S},\mathbf{d})$ is a metric space. Recall that the empty family $\emptyset$ (corresponding to empty index set) is in $\mathcal{S}$, and we note once and for all that all the proofs below will also trivially work for that specific family.  Symmetry of $\mathbf{d}(\cdot,\cdot)$ is obvious from \eqref{eq:defdistance}, using that ${\pi}$ defines a matching of $J$ and $I$ via $\{(j,i): (i,j)\in \pi\}$ and using symmetry of the distance $d_H$. 
    
    We now show that $\mathbf{d}(\mathscr{E},\mathscr{E}')=0$ implies that $\mathscr{E}=\mathscr{E}'$ for any two families $\mathscr{E}=(K_i)_{i\in{I}}$ and $\mathscr{E}'=(K'_j)_{j\in{J}}$ in $\mathscr{S}$, the reverse implication being straightforward. Noting that  elements of $\mathcal{S}$ do not contain sets of vanishing diameter, %
    it is enough to show that for all $\varepsilon > 0$, one has $\mathscr{E}_{> \varepsilon} = \mathscr{E}_{> \varepsilon}'$ (see the paragraph above \eqref{eq:defdistance} for notation).    
    Since $\mathbf{d}(\mathscr{E},\mathscr{E}')=0$, for every $n\in{\N}$ there exists a matching $\pi_n$ such that all three items appearing in the max in \eqref{eq:defdistance} are each at most $1/n$. In particular, this implies by the above reasoning that whenever $1/n \leq \varepsilon$,
\begin{equation}
    \label{eq:d-distpf1}
    I^{\pi_n} \cap I_{> \varepsilon} = J^{\pi_n} \cap J_{> \varepsilon}=\emptyset.
\end{equation}
By combining \eqref{eq:d-distpf1} and \eqref{eq:defdistance}, it follows that for every $n\in{\N}$ satisfying $2/n \leq \varepsilon$ and all $i \in I_{> \varepsilon}$, we have that $d_H(K_i,K'_{\pi_n(i)})\leq 1/(2n)$, and in particular $\pi_n(i)\in{ J_{> \varepsilon/2}}$ by \eqref{eq:d_H-diam}. Since $|I_{> \varepsilon}|<\infty$ and $|J_{>\eps/2}|<\infty$ by definition of $\mathcal{S}$, there exists an infinite set $\sigma\subset \N\cap[2/\eps,\infty)$ such that, for each  $i\in{I_{> \varepsilon}}$,  $\pi_n(i)$ is constant equal to some $\pi(i)\in{J}$ for all $n\in{\sigma}$. In particular, $d_H(K_i,K_{\pi(i)}')=d_H(K_i,K_{\pi_n(i)}')\leq 1/n$ for all $n\in{\sigma}$, and letting $n\rightarrow\infty$, $n\in{\sigma}$, we deduce $K_{\pi(i)}'=K_i$. Since for $i\neq j\in{I_{>\varepsilon}}$ we have $\pi(i)=\pi_n(i)\neq\pi_n(j)=\pi(j)$ for any $n\in{\sigma}$, we deduce that $\mathscr{E}_{> \varepsilon} \subset \mathscr{E}'_{> \varepsilon}$ (as unordered families), and the other inclusion follows in the same way.

It remains to show the triangle inequality. For $k=1,2,3$, let $\mathscr{E}^k =(K_i^k)_{i \in I^k} \in \mathcal{S}$, and $\varepsilon >0$. We aim to show that 
\begin{equation}\label{eq:d-distpf-triangle}
d_{13}\leq d_{12} + d_{23}+ \varepsilon,
\end{equation}
where $d_{ij}= \mathbf{d}(\mathscr{E}^i,\mathscr{E}^j)$.  The claim then follows upon letting $\varepsilon \downarrow 0$.  To deduce \eqref{eq:d-distpf-triangle}, it is enough to exhibit a matching $\pi=\pi(\varepsilon)$ of sets in $\mathscr{E}^1,\mathscr{E}^3$ such that 
\begin{equation}
\label{eq:d-distpf2}
   ( I^k)^{\pi} \cap I^k_{> d_{12}+ d_{23} +\varepsilon} =\emptyset, \text{ for $k=1,3,$}
\end{equation}
and
\begin{equation}
    \label{eq:d-distpf3}
2 \sup_{(i,j)\in{\pi}}d_H(K_i^1,K^3_j) \leq d_{12}+ d_{23} +\varepsilon. 
\end{equation}

We proceed to construct $\pi$ with the desired properties.
For $(ij)=(12), (23)$, let $\pi_{ij}$ denote a matching between $I^i \times I^j$ achieving $d_{ij}$ up to an error $\varepsilon/2$. In particular, $\pi_{ij}$ matches all indices in $I^i_{>d_{ij}+\varepsilon/2}$ and $I^j_{>d_{ij}+\varepsilon/2}$, and for any matched indices $(k,p)\in{\pi_{ij}}$ we have $2d_H(K_{k}^i,K_p^j)\leq d_{ij}+\eps/2$. Defining $\pi=\{(i_1,i_3)\in{I^1\times I^3}:\,\exists\, i_2\in{I_2}\text{ with }(i_1,i_2)\in{\pi_{12}}\text{ and }(i_2,i_3)\in{\pi_{23}}\}$, we now verify \eqref{eq:d-distpf2} and \eqref{eq:d-distpf3}. For each $i_1\in{I^1_{>d_{12} + d_{23}+\varepsilon} (\subset I^1_{>d_{12} +\varepsilon/2} )}$, there exists $i_2\in{I^2}$ such that $(i_1,i_2)\in{\pi_{12}}$. Moreover, the diameter of $K_{i_2}^2$ cannot be too small: by \eqref{eq:d_H-diam}, 
\begin{equation}
\label{eq:2appears}
\delta(K^2_{i_2}) \geq \delta(K^1_{i_1})- 2d_H(K^1_{i_1},K^2_{i_2}) > d_{12} + d_{23}+\varepsilon - (d_{12} +\varepsilon/2) = d_{23}+\varepsilon/2,
\end{equation}
i.e.~$i_2 \in I^2_{> d_{23}+\varepsilon/2}$. Therefore there exists $i_3\in{I^3}$ such that $(i_2,i_3)\in{\pi_{23}}$. In particular by construction we have $(i_1,i_3)\in{\pi}$, which yields \eqref{eq:d-distpf2} for $k=1$. When $k=3$, one repeats the previous construction ``backwards,'' starting with an arbitrary index $i_3 \in I^3_{> d_{12}+ d_{23} +\varepsilon}$, and using the matchings $\pi_{23}$ and then $\pi_{12}$ to match $i_3$ with some $i_1\in{I^1}$ so that $(i_1,i_3)\in{\pi}$. It remains to verify \eqref{eq:d-distpf3}: for any $(i_1,i_3)\in{\pi}$, letting $i_2\in{I^2}$ so that $(i_1,i_2)\in{\pi_{12}}$ and $(i_2,i_3)\in{\pi_{23}}$, we have $2d_H(K_{i_1}^1,K_{i_3}^3)\leq 2d_H(K_{i_1}^1,K_{i_2}^2)+2d_H(K_{i_2}^2,K_{i_3}^3)\leq d_{12}+ d_{23} +\varepsilon$.

\medskip

Next, we argue that $(\mathcal{S},\mathbf{d})$ is complete.
Thus let $(\mathscr{E}^{(n)})_n=((K_i^{(n)})_{i\in{I^{(n)}}})_n\in \mathcal{S}^{\N}$ denote a Cauchy sequence with respect to $\mathbf{d}$. We first construct a candidate limit $\mathscr{E}^{(\infty)}$. For each $k \geq 0$, let $N_k$ denote the smallest integer larger than $N_{k-1}+1$ such that 
\begin{equation}\label{eq:Cauchy1}
\mathbf{d}(\mathscr{E}^{(n)}, \mathscr{E}^{(m)})< 2^{-k}, \text{ for all $n,m \geq N_k$.} 
\end{equation}
In particular, applying \eqref{eq:Cauchy1} with the choices $n=N_k$, $m=N_{k+1}$ yields via \eqref{eq:defdistance} a matching $\pi_{k,k+1}$ which to each index $i\in{I^{(N_k)}_{>2^{-k}}}$  associates an index $\pi_{k,k+1}(i)\in{I^{(N_{k+1})}}$ such that $d_H(K_i^{(N_k)},K_{\pi_{k,k+1}(i)}^{(N_{k+1})})\leq 2^{-k-1}$. 

Essentially, the set $\mathscr{E}^{(\infty)}$ will consist of all limit points $K_i$, ${i\in{I^{(\infty)}}}$, of certain (convergent, as it will turn out) sequences $((K^{(N_k)}_{\pi_i(k)})_{k\geq M_i})$, $i\in{I^{(\infty)}}$, as $k\rightarrow\infty$, with $M_i\in{\N}$ and ${\pi_i(k)} \in{I^{(N_k)}}$  for each $k\geq M_i$ and $i\in{I^{(\infty)}}$, constructed inductively as follows. Fix some $i\geq 1$, assume that $(\pi_j(k))_{k\in{\N}}$ and $M_j$ are constructed for all $1\leq j\leq i-1$, and write $I^{(N_k)}_{i-1}=I^{(N_k)}\setminus \{\pi_j(k):\,1\leq j\leq i-1\}$ for all $k\geq M_{i-1}$, with $M_0=1$ (in particular $I^{(N_k)}_0=I^{(N_k)}$). We then define $M_i$ as the first $k\geq M_{i-1}$ such that there exists $p\in{I^{(N_k)}_{i-1}}$ with $\delta(K_p^{(N_k)})\geq 2^{-p+1}$ (when such a $k$ exists), and let $\pi_{i}(M_i)$ as the index $p\in{I^{(N_{M_i})}_{i-1}}$ such that $\delta(K_p^{(N_{M_i})})$ is maximal, with some arbitrary rule in case of equality. One can then define $\pi_{i}(k+1)=\pi_{k,k+1}(\pi_i(k))$ for all $k\geq M_i$, and abbreviate $\tilde{K}_i^{(k)}=K_{\pi_i(k)}^{(N_k)}$. For $\pi_i(k+1)$ to be well defined, as explained below \eqref{eq:Cauchy1} it is enough that $\pi_i(k)\in{I^{(N_k)}_{>2^{-k}}}$ for all $k\geq M_i$, and we now prove the stronger statement $\delta(\tilde{K}_i^{(k)})\geq 2^{-k+1}$ recursively in $k\geq M_i$. The initialization is true by definition of $M_i$ and $\pi_i(M_i)$. Assume that $\delta(\tilde{K}_i^{(k)})\geq 2^{-k+1}$ for some $k\geq M_i$, then $\pi_i(k)\in{I_{>2^{-k}}^{(N_k)}}$, and so as explained below \eqref{eq:Cauchy1}  
\begin{equation}
\label{eq:hausdorffbetween2sets}
d_H\big(\tilde{K}_{i}^{(k)},\tilde{K}_{i}^{({k+1})}\big)=d_H\big(K_{\pi_i(k)}^{(N_k)},K_{\pi_{k,k+1}(\pi_i(k))}^{(N_{k+1})}\big)\leq 2^{-k-1}.
\end{equation}
In particular, one can deduce from \eqref{eq:d_H-diam} and our recursive assumption  that $\delta(\tilde{K}_{i}^{({k+1})})\geq 2^{-k+1}-2\cdot 2^{-k-1}\geq 2^{-(k+1)+1}$.  This finishes the recursive construction of $(\pi_i(k))_{k\geq M_i}$ and $M_i$. The only problem is that it is possible for some $i\geq1$ that $\delta(K_p^{(N_k)})< 2^{-p+1}$ for all $p\in{I_{i-1}^{(N_k)}}$ and $k\geq M_{i-1}$, in which case we cannot construct $M_i$ as before, and we then stop the recursive construction at this step and let $M=i$. Otherwise, we continue the recursion for each $i\in{\N}$, and let $M=\infty$.

 It is easy to see that any sequence $(\tilde{K}_{i}^{(k)})_{k\geq M_i}$ obtained in this way is Cauchy w.r.t.~$d_H$. Indeed, by \eqref{eq:hausdorffbetween2sets} and the triangle inequality we have $d_H(\tilde{K}_{i}^{(k)},\tilde{K}_{i}^{(n)})\leq 2^{-n}$ for all $i<M$ and $n\geq k\geq M_i$. But the set of compact subsets of $[-1,1]^d$ being $d_H$-complete, it follows that there exists $K_{i} \subset [-1,1]^d$ compact such that $d_H(\tilde{K}_{i}^{(k)},K_{i})\to 0$ as $k \to \infty$. In fact letting $n\rightarrow\infty$ in the previous inequality, we have $d_H(\tilde{K}_{i}^{(k)},K_{i})\leq 2^{-k}$ for all $k\in{\N}$. We then define $I^{(\infty)}=\{i\in{\N}:\,i<M\text{ and }\delta(K_i)>0\}$ and $\mathscr{E}^{(\infty)}=(K_i)_{i\in{I^{(\infty)}}}$.

 We first proceed to verify that $\mathscr{E}^{(\infty)} \in \mathcal{S}$. By definition $\mathscr{E}^{(\infty)}$ is an unordered family of compact sets of positive diameter. Let $\eta > 0$. It remains to argue that 
 \begin{equation}\label{eq:Cauchy2}
 N_{> \eta}(\mathscr{E}^{(\infty)})< \infty.
 \end{equation}
Since  $d_H(\tilde{K}_{i}^{(k)},K_{i})\leq 2^{-k}$, it follows from \eqref{eq:d_H-diam} that $\delta(\tilde{K}_i^{(k)})>\eta -2^{-k+1}\geq \eta/2$ for all $k\geq C(\eta)$ and $i\in{I^{(\infty)}_{>\eta}}$. Hence for $k\geq C(\eta)$ one can bound the left-hand side of \eqref{eq:Cauchy2} by $N_{>\eta/2}(\mathscr{E}^{(N_k)})<\infty$ since $\mathscr{E}^{(N_k)}\in{\mathcal{S}}$.

It remains to argue that $\mathscr{E}^{(\infty)}$ is indeed the desired limit point, i.e.~that $\mathbf{d}(\mathscr{E}^{(n)},\mathscr{E}^{(\infty)}) \to 0$ as $n \to \infty$. Fix some $n\in{\N}$, and let $k=k(n)\in{\N}$ be such that $N_k\leq n<N_{k+1}$, which always exists since we constructed $N_k$ as a strictly sequence, see above \eqref{eq:Cauchy1}. By \eqref{eq:Cauchy1} there exists a matching $\tilde{\pi}_{n,k}\subset I^{(n)}\times I^{(N_k)}$ which matches each index in $I^{(n)}_{>2^{-k}}$, as well as each index in $I^{(N_k)}_{>2^{-k}}$, and such that for each $(j,p)\in{\tilde{\pi}_{n,k}}$ we have $d_H(K_{j}^{(n)},K_p^{(N_k)})\leq 2^{-k-1}$. We further define
\begin{equation}
\label{eq:deftildepin}
\tilde{\pi}_n\stackrel{\text{def.}}{=}\{(j,i)\in{I^{(n)}\times I^{(\infty)}}:\,k\geq M_i\text{ and }(j,\pi_i(k))\in{\tilde{\pi}_{n,k}}\},
\end{equation}
where $\pi_i$ and $M_i$ were constructed above  \eqref{eq:hausdorffbetween2sets}, and $k=k(n)$ as above. Note that for each $(j,i)\in{\tilde{\pi}_n}$ since $K_{\pi_i(k)}^{(N_k)}=\tilde{K}_i^{(k)}$   we have 
\begin{equation}
\label{eq:dHjni}
d_H(K_j^{(n)},K_{i})\leq d_H(K_j^{(n)},K_{\pi_i(k)}^{(N_k)})+d_H(\tilde{K}_i^{(k)},K_{i})\leq 2^{-k-1}+2^{-k}\leq  2^{-k+1}.
\end{equation}
We now argue that for each index $j\in{I^{(n)}_{>2^{-k+2}}}$, resp.\ $i\in{I^{(\infty)}_{>2^{-k+2}}}$ there exists $i\in{I^{(\infty)}}$, resp.\ $j\in{I^{(n)}}$, such that $(j,i)\in{\tilde{\pi}_n}$. Once this is proven, it follows from \eqref{eq:defdistance} for $\tilde{\pi}_n$ and \eqref{eq:dHjni} that $\mathbf{d}(\mathscr{E}^{(n)},\mathscr{E}^{(\infty)})\leq 2^{-k+2}$, and since $k=k(n)\rightarrow\infty$ as $n\rightarrow\infty$, we can conclude.

Let us first take some index $j\in{I^{(n)}_{>2^{-k+2}}}$, then by definition of $\tilde{\pi}_{n,k}$ there exists $p\in{I^{(N_k)}}$ such that $(j,p)\in{\tilde{\pi}_{n,k}}$, and then $d_H(K_j^{(n)},K_p^{(N_k)})\leq 2^{-k-1}$. In particular, \eqref{eq:d_H-diam} yields that $\delta(K_p^{(N_k)})> 2^{-k+2}-2\cdot 2^{-k-1}= 3\cdot2^{-k}>2^{-k+1}$. Hence $p$ must be equal to $\pi_i(k)$ for some $i<M$ and $k\geq M_i$, where $\pi_i$ is the sequence constructed above \eqref{eq:hausdorffbetween2sets}. Indeed otherwise by definition of $M_i$ we would have $M_i\leq k$ for all $i<M$, and so if $M<\infty$ this would be in contradiction with the definition of $M$ below \eqref{eq:hausdorffbetween2sets}; and if $M=\infty$ we would have by construction that $\pi_i(k)$, $i\in{\N}$, are disjoint elements of $I^{(N_k)}_{>2^{-k}}$, see around  \eqref{eq:hausdorffbetween2sets}, which is also a contradiction as $|I^{(N_k)}_{>2^{-k}}|<\infty$. Further note by \eqref{eq:d_H-diam} that $\delta(K_{i})\geq \delta(K_p^{(N_k)})-d_H(\tilde{K}_i^{(k)},K_i)\geq 3\cdot 2^{-k}-2^{-k}>0$, and so $i\in{I^{(\infty)}}$. We have thus verified that $(j,i)\in{\tilde{\pi}_n}$, see \eqref{eq:deftildepin}.

We now take some index $i\in{I^{(\infty)}_{>2^{-k+2}}}$, and let us first prove that $k\geq M_i$ by contradiction. If $k<M_i$, which implies $i\in{I^{(\infty)}_{>2^{-M_i+1}}}$ then as explained below \eqref{eq:Cauchy1} there exists $p\in{I^{(N_{M_i-1})}}$ which is paired by $\pi_{M_i-1,M_i}$ with $\pi_{i}(M_i)$, and then $d_H(K_p^{(N_{M_i-1})},K_{\pi_i(M_i)}^{(N_{M_i})})\leq 2^{-M_i}$. Therefore by \eqref{eq:d_H-diam} we would have $$\delta(K_p^{(N_{M_i-1})})\geq \delta(K_{\pi_i(M_i)}^{(N_{M_i})})-2\cdot 2^{-M_i}\geq 2^{-k+2}-2\cdot 2^{-M_i}-2\cdot 2^{-M_i}\geq 2^{-M_i+2}>2^{-(M_i-1)+1},$$ where we used that $d_H(K_{\pi_i(M_i)}^{(N_{M_i})},K_i)=d_H(\tilde{K}_i^{(M_i)},K_i)\leq 2^{-M_i}$. Noting that $p\in{I_{i-1}^{(N_{M_{i}-1})}}$, since otherwise $\pi_i(M_i)=\pi_{M_{i}-1,M_i}(p)$ would also not be in $I_{i-1}^{(N_{M_i})}$ by construction, this is a contradiction by definition of $M_i$ above \eqref{eq:hausdorffbetween2sets}. This finishes the proof of $k\geq M_i$, and therefore $\pi_i(k)$ is well-defined and $d_H(K_{\pi_i(k)}^{(N_k)},K_i)=d_H(\tilde{K}_i^{(k)},K_i)\leq 2^{-k}$. In particular by \eqref{eq:d_H-diam} we have $\delta(K_{\pi_i(k)}^{(N_k)})\geq 2^{-k+2}-2^{-k}\geq 3\cdot 2^{-k}>2^{-k}$, and so $\pi_i(k)\in{I^{(N_k)}_{>2^{-k}}}$. Therefore, there exists $j\in{I^{(n)}}$ such that $(j,\pi_i(k))\in{\tilde{\pi}_{n,k}}$, and we conclude that $(j,i)\in{\tilde{\pi}_n}$, see \eqref{eq:deftildepin}. As explained below \eqref{eq:dHjni}, this finishes the proof of completeness for $(\mathcal{S},\mathbf{d})$.

\medskip

Finally, one needs to verify that $(\mathcal{S},\mathbf{d})$ is separable. Let us consider  the collection $\mathcal{B}$ which consists of any finite union of boxes of the form $[x,-r, x+r]^d \cap [-1,1]^d$, with rational centers $x\in{[-1,1]^d}$ and rational radii $r>0$. Then $\mathcal{B}$ is countable, and one easily checks that for each $\eps>0$, any compact $K\subset [-1,1]^d$ is at Hausdorff distance at most $\eps$ from some element of $\mathcal{B}$. Consider $\mathcal{S}'$ the set of unordered families $(K'_i)_{i\in{I}}$ with $K'_i\in{\mathcal{B}}$ for each $i\in{I}$ and $|I|<\infty$. Then $\mathcal{S}'$ is countable since $\mathcal{B}$ is countable. Furthermore, let $\mathscr{E}=(K_i)_{i\in{I}}\in{\mathcal{S}}$ and $\eps>0$. For each $i\in{I_{>\eps}}$, there exists $K'_i\in{\mathcal{B}}$ such that $d_H(K'_i,K_i)\leq \eps/2$. Therefore letting $\mathscr{E}'=(K'_i)_{i\in{I_{>\eps}}}$ we have $\mathscr{E}'\in{\mathcal{S}'}$ and $\mathbf{d}(\mathscr{E},\mathscr{E}')\leq \eps$ by \eqref{eq:defdistance} for $\pi(i)=i$ for each $i\in{I_{>\eps}}$, and we can conclude. 
\end{proof}

\section{Proof of Lemma~\ref{lem:lowMink}}
\label{App:lowMink}

In this appendix, we give the proof of Lemma~\ref{lem:lowMink}, used in the context of deriving our main results to obtain lower bounds on cluster dimensions.

The proof of Lemma~\ref{lem:lowMink} relies on a connection bound at (small) negative levels, Lemma~\ref{lem:connectionnegative} below, which we derive using random interlacements \cite{MR2680403} and their link to the Gaussian free field via an isomorphism theorem. Random interlacements is a Poisson point process of doubly-infinite trajectory modulo time-shift introduced on $\Z^d$ in \cite{MR2680403}, and extended to $\tilde{\Z}^d$ in \cite{MR3502602}. We will not need the full construction of interlacements here, and refer the reader to \cite[Section~2.5]{DrePreRod3} for the definition we use here. Instead, we will only give the following description of the restriction of their law to a compact. Let 
\begin{equation}
\label{eq:DefmuK}
P_{e_K}=\sum_{x\in{\partial K}}e_K(x)P_x\text{ for all connected compacts }K\subset\tilde{\Z}^d,
\end{equation}
where we recall from Section~\ref{sec:setup} that $P_x$ denotes the law of the Brownian motion on $\tilde{\Z}^d$ starting in $x$, and $e_K$ the equilibrium measure of $K$. Let $\OP$ be a probability measure under which the interlacement process at level $u$ is defined for each $u\in{(0,\infty)}$, then the restriction of this process to forward trajectories hitting $K$, and started the first time they hit $K$, is a Poisson point process with intensity $uP_{e_K}$. Let us denote by $\I^u$ the interlacement set, that is the set of points in $\tilde{\Z}^d$ intersected by some trajectories in the interlacement process, and write $\I_n^u=n^{-1}\I^u\subset\tilde{\Z}_n^d$. One readily deduces from this description and \eqref{eq:capren} that for any connected compact $K \subset \tilde\G $,
 \begin{equation}
 \label{eq:defIu}
\begin{split} 
\OP(\I^u_n\cap K=\emptyset)=\exp\big(-un^{d-2}\mathrm{cap}_n(K)\big).
\end{split}
\end{equation}

 We now discuss the isomorphism theorem involving interlacements and the Gaussian free field, see \cite{MR2892408,MR3502602,MR3492939,DrePreRod3}.
Under the product measure $\OP\otimes \P$, we denote by $\mathcal{C}^u$ the  union of the clusters of $\{x\in{\tilde{\G}}:|\phi_x|>0\}$ which intersect $\I^u$ (an open set). %
The isomorphism from  \cite[Theorem~2.4]{MR3492939}, see also \cite[(Isom)]{DrePreRod3}, is in force in view of \cite[Lemma~3.4 and Theorem~1.1]{DrePreRod3}. This implies that for all $u\geq0$,
\begin{align}
&\label{eqcouplingintergffszn}
	\begin{array}{l}
	(\mathcal{C}^u)^{\mathsf{c}}\cap E^{\geq0} \text{ has the same law under }\overline{\P}\otimes\P \text{ as }E^{\geq \sqrt{2u}}\text{ under }\P,
	\end{array}
\end{align}
where we recall the notation $E^{\geq a}$ from \eqref{eq:newnotation}. The isomorphism \eqref{eqcouplingintergffszn} lets us express the law of $\{\varphi \geq a \}$ in terms of random interlacements and \emph{critical} clusters of the Gaussian free field for any $a>0$, and in fact also for $a<0$. Indeed, taking complements in \eqref{eqcouplingintergffszn} and using the symmetry of the free field, we obtain from \eqref{eqcouplingintergffszn} that $E^{\geq -\sqrt{2u}}$ under $\P$ has the same law as the closure of $\mathcal{C}^u\cup \{x\in{\tilde{\Z}^d}:\,\phi_x >0\}$, that is $\overline{\mathcal{C}^u}\cup E^{\geq0}$, under $\overline{\P}\otimes\P$. %
Here, we used that the closure of $\{x\in{\tilde{\Z}^d}:\,\phi_x>a\}$ is a.s.~$E^{\geq a}$ for all $a\in{\R}$, for a similar reason as below \eqref{eq:renormalizedlevelset}.

\begin{Lemme}
\label{lem:connectionnegative}
For all  $a\geq0$, $x,y\in{\Z^d}$ and $R,r\geq1$ such that $d(x,y)\geq 2(r +R)$ 
\begin{equation}
\label{eq:crossingboxneg}
\begin{split} 
 \P&\big({B}(x,r)\leftrightarrow {B}(y,R)\text{ in }E^{\geq -a}\big)
 \\&\leq C(a\vee d(x,y)^{-\frac{d-2}2})\Big(\big(R^{\frac{d-2}2}r ^{\frac{d-2}2}(a\vee d(x,y)^{-\frac{d-2}2})\big)\wedge r^{\frac{d-2}2}\wedge R^{\frac{d-2}2}\Big),
\end{split}
\end{equation}
and if $R\geq r$ 
\begin{equation}
\label{eq:crossingannulineg}
\begin{split} 
 \P\big({B}(x,r)\leftrightarrow \partial{B}(x,R)\text{ in }E^{\geq -a}\big)\leq C(a\vee R^{-\frac{d-2}2})r ^{\frac{d-2}2}.
\end{split}
\end{equation}
\end{Lemme}
\begin{proof}
When $R\geq 3r$, the upper bound in \eqref{eq:crossingannulineg} follows by a union bound from  \eqref{eq:crossingboxneg} for $y\in{cR\Z^d\cap B(x,R)\setminus B(x,3R/4)}$ and $cR$ instead of $R$ for a small enough constant $c$. Since this bound is also trivial when $r\leq R\leq 3r,$ we can focus only on the proof of \eqref{eq:crossingboxneg}. Let $u=a^2/2$,  by monotonicity we can assume w.l.o.g.\ that $u\geq (d(x,y)/5)^{2-d}$, and by symmetry that $R\geq r$. As explained below \eqref{eqcouplingintergffszn}, the event ${B}(x,r)\leftrightarrow {B}(y,R)$ in $E^{\geq -a}$ has the same probability as the union of the events ${B}(x,r)\leftrightarrow {B}(y,R)\text{ in }E^{\geq 0}$, ${B}(x,r),{B}(y,R)\leftrightarrow\I^u$ in $E^{\geq0}$, and ${B}(x,r),{B}(y,R)\leftrightarrow\I^u$ in $(E^{\geq0})^{\mathsf c}$. Let $K_x$, resp.\ $K_y$, be the largest $k\in{\N_0}$ such that ${B}(x,r)\leftrightarrow \partial{B}(x,2^kr )$, resp.\ ${B}(y,R)\leftrightarrow \partial{B}(y,2^{k'}R)$ (recall that by convention above \eqref{eq:clustersDef} this means that these sets are connected in $E^{\geq0}$). Then by symmetry of the law of $E^{\geq 0}$ we have
\begin{equation}
\label{eq:decompKx}
\begin{split} 
\P&\big({B}(x,r)\leftrightarrow {B}(y,R)\text{ in }E^{\geq -a}\big)\leq 2\P\big({B}(x,r)\leftrightarrow {B}(y,R)\big)
\\&+2\P\big({B}(x,r)\centernot\leftrightarrow {B}(y,R),\I^u\cap {B}(x,2^{K_x+1}r )\neq\emptyset,\I^u\cap {B}(y,2^{K_y+1}R)\neq\emptyset\big).
\end{split}
\end{equation}
We first give bounds on $K_x$ and $K_y$ using the isomorphism \eqref{eq:isoclusters} with the loop soup $\mathcal{L}$. Let $\tilde{\K}(x,r)$ denote the cluster of $B(x,r)$ for the loops in $\mathcal{L}$. Conditionally on $\tilde{\K}(x,r)$ and on the event ${B}(x,r)\centernot\leftrightarrow {B}(y,R)$ in $\mathcal{L}$, that is $\tilde{\K}(x,r)\cap {B}(y,R)=\emptyset$, the 
event $K_y\geq k'$ for the loop soup is equal to the event that $B(y,R)\leftrightarrow \partial B(y,2^{k'}R)$ for the loops in $\mathcal{L}$ which avoid $\tilde{\K}(x,r)$, and by properties of Poisson point processes this conditional event is  stochastically dominated by the event that $B(y,R)\leftrightarrow \partial B(y,2^{k'}R)$ by loops in $\mathcal{L}$ under $\Q$. The probability of this last event can be bounded by $C2^{-\frac{k'(d-2)}2}$ using \eqref{eq:isoclusters} and \eqref{eq:crossingseccrit} as $n\rightarrow\infty$, and using \eqref{eq:crossingseccrit}  again but this time to bound ${\Q}(K_x\geq k)$, we obtain
\begin{equation*}
\begin{split} 
\P\big(K_x\geq k,K_y\geq k',{B}(x,r)\centernot\leftrightarrow {B}(y,R)\big)\leq C2^{-\frac{k(d-2)}2}2^{-\frac{k'(d-2)}2}.
\end{split}
\end{equation*}
Furthermore, it follows from \eqref{eq:defIu} and \cite[Lemma~2.1]{MR2953562} that
\begin{equation*}
\begin{split} 
\overline{\P}&\big(\I^u\cap {B}(x,2^{k+1}r )\neq\emptyset,\I^u\cap {B}(y,2^{k'+1}R)\neq\emptyset\big)
\\&\leq \Big(1-e^{-u\mathrm{cap}(B(x,2^{k+1}r ))}\Big)\Big(1-e^{-u\mathrm{cap}(B(y,2^{k'+1}R))}\Big)
\\&\quad+Cu\mathrm{cap}(B(x,2^{k+1}r ))\mathrm{cap}(B(y,2^{k'+1}R))d(x,y)^{-(d-2)}
\\&\leq Cu^2r^{d-2}R^{d-2}2^{k(d-2)}2^{k'(d-2)},
\end{split}
\end{equation*}
where we used \eqref{eq:capBallBd}, the inequality $1-e^{-t}\leq t$ and our assumption $u\geq (d(x,y)/5)^{2-d}$ in the last inequality. Abbreviate $M_{u,p}=\lfloor\log_2(u^{-1/(d-2)}/p)\rfloor$. Let us first assume that $u^{-1/(d-2)}\geq R(\geq r),$ then combining the last two equations with \eqref{eq:defIu} and \eqref{eq:capBallBd}, and summing on the events $K_x=k$ for $k\leq M_{u,r}-1$ and $K_x\geq M_{u,r}$, and similarly for $K_y$, one can thus bound the second line of \eqref{eq:decompKx} by
\begin{equation*}
\begin{split} 
&\sum_{k=0}^{M_{u,r}-1}\sum_{k'=0}^{M_{u,R} -1}Cu^2r^{d-2}R^{d-2}2^{\frac{k(d-2)}2}2^{\frac{k'(d-2)}2}
+C2^{-\frac{d-2}2M_{u,r}}2^{-\frac{d-2}2M_{u,R}}
\\&\quad+\sum_{k=0}^{M_{u,r}-1}Cur^{d-2}2^{\frac{k(d-2)}2}2^{-\frac{d-2}2M_{u,R}}
+\sum_{k'=0}^{M_{u,R} -1}CuR^{d-2}2^{k'\frac{d-2}2}2^{-\frac{d-2}2M_{u,r}}
\\&\leq Cur^{\frac{d-2}2}R^{\frac{d-2}2},
\end{split}
\end{equation*}
where in the last inequality we used that under our assumption $u^{-1/(d-2)}\geq R$ we have $cu^{-\frac12}p^{-\frac{d-2}{2}}\leq 2^{\frac{d-2}{2}M_{u,p}}\leq Cu^{-\frac12}p^{-\frac{d-2}{2}}$ for $p\in{\{r,R\}}$.
Combining this with \eqref{eq:decompKx} and \eqref{eq:crossingboxseccritH} (with $n\rightarrow\infty$ therein) and recalling that $u\geq (d(x,y)/5)^{2-d}$, we can conclude the case $u^{-1/(d-2)}\geq R$. It remains to consider the case $u^{-1/(d-2)}\leq R$. Note that $\P(K_x\geq k)\leq C2^{-k(d-2)/2}$ by \eqref{eq:crossingseccrit} as $n\rightarrow\infty$.  
Summing on the events $K_x=k$ for $k\leq \lfloor\log_2(u^{-1/(d-2)}/r)\rfloor-1$ and $K_x\geq \lfloor\log_2(u^{-1/(d-2)}/r)\rfloor$, we thus have by \eqref{eq:defIu} if $u^{-1/(d-2)}\geq r$
\begin{equation*}
\begin{split} 
\P\otimes\overline{\P}\big(\I^u\cap {B}(x,2^{K_x+1}r)\neq\emptyset\big)&\leq \sum_{k=0}^{M_{u,r}-1 }Cur^{d-2}2^{k(d-2) }2^{-\frac{k(d-2)}2}+2^{-\frac{d-2}2M_{u,r}}
 C\sqrt{u}r^{\frac{d-2}2}.
\end{split}
\end{equation*}
Note that this last bound remains clearly true if $u^{-1/(d-2)}\leq r$. One can thus bound the second line of \eqref{eq:decompKx} by $C\sqrt{u}r^{\frac{d-2}2}$ and combining this with \eqref{eq:crossingboxseccritH}, we can conclude the case $u^{-1/(d-2)}\leq R$.
\end{proof}

\begin{proof}[Proof of Lemma~\ref{lem:lowMink}] Throughout the proof, we abbreviate $\rho=d(x,y)$.
Since $\G=\Z^d$, it follows from \cite[Corollary~4.5]{MR3417515} and symmetry of the Gaussian free field that there exist constants $\Cl[c]{cboundeta}\in{(0,1)}$ and $\Cl{Cboundeta}<\infty$ such that for all $h>0$, $z\in{G}$, $n\geq \Cr{Cboundeta}h^{-2/\nu}$ and $K\subset \partial B(z,n)\cup B(z,n)^{\mathsf{c}}$, 
\begin{equation}
\label{eq:boundeta}
\begin{split} 
\P\big(\sup_{z'\in{\tilde{B}(z,\Cr{cboundeta}n)}}\eta_{K}(z')\geq h\big)\leq \exp\big(-ch^2n^{\nu}\big),
\end{split}
\end{equation}
where the field $\eta_{\partial B(x,n)}$ was introduced in \eqref{eq:defeta}. Note that \cite[Corollary~4.5]{MR3417515} actually implies \eqref{eq:boundeta} for $B(z,\Cr{cboundeta}n)$ instead of $\tilde{B}(z,\Cr{cboundeta}n)$ and $K=\partial B(z,n)$, but the supremum in \eqref{eq:boundeta} is anyway supported on $\partial B(z,\Cr{cboundeta}n)$ by the Markov property, and the proof from \cite{MR3417515} can easily be adapted for any $K\subset \partial B(z,n)\cup B(z,n)^{\mathsf{c}}$.
Let $\bar{\eta}_{\partial B(x,\rho/4)}$ be an independent copy of the field $\eta_{\partial B(x,\rho/4)}$, and abbreviate $\bar{\phi}=\phi-\eta_{\partial B(x,\rho/4)}+\bar{\eta}_{\partial B(x,\rho/4)}$ as well as $\bar{E}^{\geq h}=\{z\in{\tilde{\G}}:\,\bar{\phi}_z\geq h\}.$ Letting $\gamma^x=\sup_{z\in{\tilde{B}(x,\Cr{cboundeta}\rho/4)}}\eta_{\partial B(x,\rho/4)}(z)$ and $\bar{\gamma}^x=\inf_{z\in{\tilde{B}}(x,\rho/4)}\bar{\eta}_{\partial B(x,\rho/4)}(z)$, we then have $\bar{\phi}_z\geq {\phi}_z-\gamma^x+\bar{\gamma}^x$ for all $z\in{\tilde{B}(x,\rho/4)}$. By the Markov property \eqref{eq:Markov2}, $\bar{\phi}$ has the same law as $\phi$ and is independent of $\phi_{|\tilde{B}(x,\rho/4)^{\mathsf c}}$ and  $\gamma^x$. Therefore,
\begin{equation}
\label{eq:conditcrossing}
\begin{split} 
&\P\big(B(x,r)\leftrightarrow\partial B(x,\Cr{cboundeta}\rho/4)\,|\,\phi_{|\tilde{B}(x,\rho/4)^{\mathsf c}}\big)\P(\bar{\gamma}^x\geq -(\Cr{Cboundeta}\rho/4)^{-\frac\nu2})
\\& \quad \leq \P\big(B(x,r)\leftrightarrow\partial B(x,\Cr{cboundeta}\rho/4)\text{ in }\bar{E}^{\geq -\gamma^x-(\Cr{Cboundeta}\rho/4)^{-\frac\nu2}}\,|\,\phi_{|\tilde{B}(x,\rho/4)^{\mathsf c}}\big) \leq C(|\gamma^x|+\rho^{-\frac\nu2})r_1^{\frac\nu2},
\end{split}
\end{equation}
where in the last inequality we used \eqref{eq:crossingannulineg}, which is in force if $\rho\geq Cr_1$. Since $\P(\bar{\gamma}^x\geq -(\Cr{Cboundeta}\rho/4)^{-\frac\nu2})\geq c$ by \eqref{eq:boundeta} and $\gamma^x$ is measurable with respect to $\phi_{|\tilde{B}(y,\rho/4)^{\mathsf c}}$ by definition, we deduce that if $\rho\geq C(r_1+r'_1)$.
\begin{equation*}
\begin{split} 
&\P\big(B(x,r),B(y,r')\leftrightarrow\partial B(x,2\rho)\,|\,\phi_{|\tilde{B}(x,3\rho)^{\mathsf c}}\big)
\\&\leq Cr_1^{\frac\nu2}\E\big[\big(|\gamma^x|+\rho^{-\frac\nu2}\big)\P\big(B(y,r')\leftrightarrow \partial B(y,\Cr{cboundeta}\rho/4)\,|\,\phi_{|\tilde{B}(y,\rho/4)^{\mathsf c}}\big)\,|\,\phi_{|\tilde{B}(x,3\rho)^{\mathsf c}}\big]
\\&\leq Cr_1^{\frac\nu2}(r'_1)^{\frac\nu2}\big(|\gamma^x|+\rho^{-\frac\nu2}\big)\big(|\gamma^y|+\rho^{-\frac\nu2}\big),
\end{split}
\end{equation*}
where in the last inequality we used a reasoning similar to \eqref{eq:conditcrossing} replacing respectively $x$ by $y$ and $r$ by $r'$, and $\gamma^y$ is defined similarly as $\gamma^x$ but exchanging the roles of $x$ and $y$ therein.  Recalling that $\sup q<\infty$, see below \eqref{eq:defMRs}, and combining this with \eqref{eq:mainboxtobox}, whose hypotheses are satisfied for $\mathcal{O}=\emptyset$, $A=\tilde{B}(x,5\rho/\Cr{cboundeta})$, $A'=\partial B(x,R)$, $R/4$ instead of $R$, $r=5\rho/\Cr{cboundeta}$ and $R'=R$ upon assuming $\rho\leq cR$ and $R\wedge \rho\geq C$ (in particular $R'\geq Cr$ and $\delta(A)\leq 2r+1\leq 3r$, as required in Lemma~\ref{lem:mainboxtobox}), and noting that $\mathcal{H}:=\mathcal{H}^{\emptyset}_{\tilde{B}(x,5\rho/\Cr{cboundeta}),\partial B(x,R)}(5\rho/\Cr{cboundeta},R/4)$ (see \eqref{eq:defcalH}) is measurable with respect to $\phi_{|\tilde{B}(x,3\rho)^{\mathsf c}}$, we deduce that if $C(r_1+r'_1)\leq \rho\leq cR$ 
\begin{equation}
\label{eq:maindoubleconnection}
\begin{split} 
\P\big(B(x,r),B(y,r')\leftrightarrow\partial B(x,R)\big)\leq C\big({r_1r'_1}/{R}\big)^{\frac\nu2}\E\big[\big(|\gamma^x|+\rho^{-\frac\nu2}\big)\big(|\gamma^y|+\rho^{-\frac\nu2}\big)\mathcal{H}\big]
\end{split}
\end{equation}
It remains to bound this last expectation, and for that purpose let $\eta'_{B(x,2\rho)}$ be an independent copy of $\eta_{B(x,2\rho)}$, and $\phi'=\phi-\eta_{B(x,2\rho)}+\eta'_{B(x,2\rho)}$. Let also $\gamma=\sup_{z\in{\tilde{B}(x,4\rho/\Cr{cboundeta})^{\mathsf c}}}\eta_{\tilde{B}(x,2\rho)}(z)$ and $\gamma'=\inf_{z\in{\tilde{B}(x,4\rho/\Cr{cboundeta})^{\mathsf c}}}\eta'_{\tilde{B}(x,2\rho)}(z)$. By the Markov property \eqref{eq:Markov2}, $\phi'$ has the same law as $\phi$ and is independent of $\gamma^x$ and $\gamma^y$. Defining $\mathcal{H}'$ similarly as $\mathcal{H}^{\emptyset}_{\tilde{B}(x,5\rho/\Cr{cboundeta}),\partial B(x,R)}(5\rho/\Cr{cboundeta},R/4)$ in \eqref{eq:defcalH} but for $\phi'$ instead of $\phi$, we have $\mathcal{H}\leq \mathcal{H}'+C(\gamma-\gamma')\rho^{\nu}$ in view of \eqref{eq:expH2} and \eqref{eq:capBallBd}, and hence one can bound the last expectation in \eqref{eq:maindoubleconnection} by 
\begin{equation*}
\begin{split} 
&\E\big[\big(|\gamma^x|+\rho^{-\frac\nu2}\big)\big(|\gamma^y|+\rho^{-\frac\nu2}\big)\big]\E\big[\mathcal{H}'\big] +C\rho^{\nu}\E\big[\big(|\gamma^x|+\rho^{-\frac\nu2}\big)\big(|\gamma^y|+\rho^{-\frac\nu2}\big)(\gamma-\gamma')\big] 
\\&  \leq \E\big[\big(|\gamma^x|+\rho^{-\frac\nu2}\big)^2\big]^{\frac12}\E\big[\big(|\gamma^y|+\rho^{-\frac\nu2}\big)^2\big]^{\frac12}\E\big[\mathcal{H}\big] 
\\&\quad+C\rho^{\nu}\E\big[\big(|\gamma^x|+\rho^{-\frac\nu2}\big)^3\big]^{\frac13}\E\big[\big(|\gamma^y|+\rho^{-\frac\nu2}\big)^3\big]^{\frac13}\E[|\gamma|^3]^{\frac13},
\end{split}
\end{equation*}
where we used the (generalized) Hölder's inequality. Note that \eqref{eq:boundeta} still hold when replacing $\tilde{B}(z,\Cr{cboundeta}n)$ by $\tilde{B}(z',n)$ for any $z'$ such that $d(z,z')\geq n/\Cr{cboundeta}$ (up to decreasing the constant $\Cr{cboundeta}$). Indeed, the proof of \cite[Corollary~4.5]{MR3417515} can easily be adapted in this context, noticing in particular that \cite[(4.23) and (4.28)]{MR3417515} are still satisfied under this hypothesis. By the Markov property and \eqref{eq:defLambdaL}, we further have $\gamma\leq \sup_{z\in{\partial B(x,4\rho/\Cr{cboundeta})^{\mathsf c}}}\eta_{\tilde{B}(x,2\rho)}(z)\leq \sup_{z'\in{\Lambda_{\rho}\cap B(x,6\rho/\Cr{cboundeta})\setminus B(x,3\rho/\Cr{cboundeta})}}\sup_{z\in{\tilde{B}(z',\rho)}}\eta_{\tilde{B}(x,2\rho)}(z)$. Combining these observations with \eqref{eq:defLambdaL} and \eqref{eq:boundeta}, with $z,$ $n$ and $K$ therein being either $x$, $\rho/4$ and $\partial B(x,\rho/4)$, or $y$, $\rho/4$ and $\partial B(y,\rho/4)$, or $z'$, $\rho/\Cr{cboundeta}$ and $\tilde{B}(x,2\rho)$ for $z'\in{\Lambda_{\rho}\cap B(x,6\rho/\Cr{cboundeta})\setminus B(x,3\rho/\Cr{cboundeta})}$,  yields
\begin{equation*}
\begin{split} 
\sup_{k\in{\{1,2,3\}}}\E\big[|\gamma^x|^k\big]^{\frac1k}\vee\E\big[|\gamma^y|^k\big]^{\frac1k}\vee\E\big[|\gamma|^k\big]^{\frac1k}\leq \rho^{-\frac{\nu}2}.
\end{split}
\end{equation*}
If $C(r_1+r'_1)\leq \rho\leq cR$, we can readily conclude by combining the last two displays with \eqref{eq:maindoubleconnection}, and noting that by \eqref{eq:expH} and since $\sup q<\infty$,  we have  $\E[\mathcal{H}]\leq  \rho^{\nu/2}$. Finally,  if $\rho\leq C(r_1+r'_1)$, \eqref{eq:lowMink} follows easily from \eqref{eq:crossingannuli}.
\end{proof}

\begin{Rk}
\label{rk:CSbounds}
The proof of Lemma~\ref{lem:lowMink} relies on bounds of Cauchy--Schwarz type, see the use of Hölder's inequality above. These allow to control the effect of the random shifts $\gamma^x$, $\gamma^y$ and $\gamma$ of the level multiplicatively, thereby avoiding additive error terms as well as the cost of working at a fixed near-critical level. This is resourceful. Another example  can be found in \cite{rodriguez2025anomalous}, to the effect of exhibiting near-critical costs in the context of the two-dimensional Gaussian free field.
\end{Rk}

\bibliography{bibliographie}
\bibliographystyle{abbrv}

\end{document}